\documentclass[11pt]{article}
\usepackage[english]{babel}
\usepackage{amssymb}
\usepackage{amsmath}
\usepackage{amsthm}
\usepackage{graphicx}
\usepackage{MnSymbol}
\newtheorem{theorem}{Theorem}[section]
\newtheorem{lemma}[theorem]{Lemma}

\newtheorem{corollary}[theorem]{Corollary}

{\theoremstyle{definition}\newtheorem{notation}[theorem]{Notation}}
{\theoremstyle{definition}\newtheorem{definition}[theorem]{Definition}}
{\theoremstyle{definition}\newtheorem{remark}[theorem]{Remark}}

\numberwithin{equation}{section}

\def\N{{\mathbb N}}
\def\Z{{\mathbb Z}}
\def\R{{\mathbb R}}
\def\Q{{\mathbb Q}}
\def\K{{\mathbb K}}

\def\epsilon{\varepsilon}
\def\kappa{\varkappa}
\def\phi{\varphi}
\def\leq{\leqslant}
\def\geq{\geqslant}

\def\dim{{\rm dim}\,}

\def\spann{\hbox{\tt span}\,}

\def\deg{\hbox{\tt deg}\,}

\def\hhh{\vrule height10pt depth5pt width0pt}
\def\tAbb#1#2#3#4#5#6{\noindent{\scalebox{0.7}{\hbox{\vrule\,{\makebox[0.77cm][l]{\,#1}\,\vrule\,
\makebox[5.45cm][l]{#2}\,\vrule\,\makebox[8.5cm][l]{#3}\,\vrule\,\makebox[2.7cm][l]{#4}\,\vrule\,
\makebox[5cm][l]{#5}\,\vrule\,\makebox[0.7cm][l]{#6}\,\vrule}}}}\hrule}
\def\taAbb#1#2#3#4#5#6{\noindent{\scalebox{0.7}{\hbox{\vrule\,{\makebox[0.77cm][l]{\,#1}\,\vrule\,
\makebox[5.45cm][l]{#2}\,\vrule\,\makebox[8.5cm][l]{#3}\,\vrule\,\makebox[4.7cm][l]{#4}\,\vrule\,
\makebox[3cm][l]{#5}\,\vrule\,\makebox[0.7cm][l]{#6}\,\vrule}}}}\hrule}
\def\tABb#1#2#3#4#5{\noindent{\scalebox{0.7}{\hbox{\vrule\,{\makebox[0.77cm][l]{\,#1}\,\vrule\,
\makebox[3.05cm][l]{#2}\,\vrule\,\makebox[9.2cm][l]{#3}\,\vrule\,\makebox[5.25cm][l]{#4}\,\vrule\,
\makebox[5cm][l]{#5}\,\vrule}}}}\hrule}
\def\hhhhh{\vrule height12pt depth5pt width0pt}
\def\tabbb#1#2#3#4#5{\noindent{\scalebox{0.7}{\hbox{\vrule\,{\makebox[0.75cm][l]{\,#1}\,\vrule\,
\makebox[7.25cm][l]{#2}\,\vrule\,\makebox[7.3cm][l]{#3}\,\vrule\,\makebox[3cm][l]{#4}\,\vrule\,
\makebox[5cm][l]{#5}\,\vrule}}}}\hrule}
\def\tabbbb#1#2#3#4#5{\noindent{\scalebox{0.7}{\hbox{\vrule\,{\makebox[0.75cm][l]{\,#1}\,\vrule\,
\makebox[5.25cm][l]{#2}\,\vrule\,\makebox[8.3cm][l]{#3}\,\vrule\,\makebox[6cm][l]{#4}\,\vrule\,
\makebox[3cm][l]{#5}\,\vrule}}}}\hrule}
\def\hhh{\vrule height12pt depth5pt width0pt}

\DeclareMathSymbol{\varkappa}{\mathord}{AMSb}{"7B}

\title{Classification of quadratic and cubic PBW algebras on three generators
}

\author{Natalia Iyudu and Stanislav Shkarin}

\date{}

\begin{document}

\maketitle

\begin{abstract}
We give a complete classification of quadratic algebras $A$, with Hilbert series $H_A=(1-t)^{-3}$, which is the Hilbert series of commutative polynomials on 3 variables. Koszul algebras as well as algebras with quadratic Gr\"obner basis among them are identified. We also give a complete classification of cubic algebras $A$ with Hilbert series $H_A=(1+t)^{-1}(1-t)^{-3}$. These two classes of algebras contain all Artin--Schelter regular algebras of global dimension $3$. As far as the latter are concerned, our results extend well-known results of  Artin and Schelter by providing a classification up to an algebra isomorphism.
\end{abstract}

\small \noindent{\bf MSC:} \ \ 17A45, 16A22

\noindent{\bf Keywords:} \ \ Quadratic algebras, Cubic algebras, Koszul algebras, Hilbert series, Sklyanin algebras, PBW-algebras, potential algebras  \normalsize

\section{Introduction \label{s1}}\rm

Throughout this paper $\K$ is an algebraically closed field of characteristic different from $2$ or $3$ (some arguments break down in the absence of any of these assumptions). 
If $B$ is a $\Z_+$-graded vector space, $B_m$ stands for the $m^{\rm th}$ component of $B$. We always assume that each $B_m$ is finite dimensional, which allows to consider the obvious generating function of the sequence of dimensions of graded components called the {\it Hilbert series} of $B$:
$$
\textstyle H_B(t)=\sum\limits_{m=0}^\infty \dim B_m\,\,t^m\in\Z[\![t]\!].
$$
If $V$ is an $n$-dimensional vector space over $\K$, then $F=F(V)$ is the tensor algebra of $V$. For any choice of a basis $x_1,\dots,x_n$ in $V$, $F$ is naturally identified with the free algebra $\K\langle x_1,\dots,x_n\rangle$, always assumed to be degree graded. If $R$ is a subspace of the $n^2$-dimensional space $V^2$, then the quotient $A$ of $F(V)$ by the ideal generated by $R$ is called a {\it quadratic algebra} and denoted $A(V,R)$. If $R$ is a subspace of the $n^3$-dimensional space $V^3$, then the quotient $A$ of $F(V)$ by the ideal generated by $R$ is called a {\it cubic algebra} and denoted $B(V,R)$. In both cases, the ideal generated by $R$ is known as the {\it ideal of relations} of $A$. We say that
$$
\text{$A$ is a {\it $PBW_S$-algebra} if $H_A(t)=H_{\K[x_1,\dots,x_n]}(t)=(1-t)^{-n}$.}
$$
These algebras are called PBW (Poincare-Birkhoff-Witt) by a number of authors, for example, Odesskii \cite{ode}. In the book by Polishchuk and Positselski   \cite{popo}, however, the term PBW algebra is reserved for a quadratic algebra with a quadratic Gr\"obner basis in the ideal of relations. We call them here $\rm PBW_B$-algebras. More precisely, a quadratic algebra $A=A(V,R)$ is a {\it $PBW_B$-algebra} if there are linear bases $x_1,\dots,x_n$ and $g_1,\dots,g_m$ in $V$ and $R$ respectively such that with respect to some compatible with multiplication well-ordering on the monomials in $x_j$, $g_1,\dots,g_m$ form a Gr\"obner basis of the ideal of relations of $A$. References to Poincare-Birkhoff-Witt properties are relevant in both cases,  in both cases we deal with generalisations of the PBW theorem on the Hilbert series of the universal enveloping of a Lie algebra:
the series concept refers to the conclusion, while the basis concept refers to the method of the proof of this theorem.
 However, one has to keep in mind that none of $\rm PBW_S$ or $\rm PBW_B$ yield the other: a $\rm PBW_B$ algebra may very well have exponential growth or be finite dimensional, while a $\rm PBW_S$ algebra may fail to even have a finite Gr\"obner basis in the ideal of relations.

Another concept playing an important role in this paper is Koszulity. A quadratic algebra $A=A(V,R)$ is called {\it Koszul} if $\K$ as a graded right $A$-module has a free resolution $\dots\to M_m\to\dots\to M_1\to A\to\K\to 0$, where the second last arrow is the augmentation map and the matrices of the maps $M_m\to M_{m-1}$ with respect to some free bases consist of homogeneous elements of degree $1$. If we pick a basis $x_1,\dots,x_n$ in $V$, we get a bilinear form on the free algebra $\K\langle x_1,\dots,x_n\rangle$ defined by $[u,v]=\delta_{u,v}$ for every monomials $u$ and $v$. The quadratic algebra $A^!=A(V,R^\perp)$, where $R^\perp=\{u\in V^2:[r,u]=0\ \text{for each}\ r\in R\}$, is called the {\it dual algebra} of $A$. Note that up to an isomorphism $A^!$ does not depend on the choice of a basis in $V$. We shall use the following well-known properties of Koszul algebras:
\begin{equation}\label{stm2}
\begin{array}{c}
\text{every $\rm PBW_B$-algebra is Koszul;\ \ $A$ is Koszul$\iff A^!$ is Koszul;}\\ \text{if $A$ is Koszul, then $H_A(-t)H_{A^!}(t)=1$.}\end{array}
\end{equation}

Artin and Schelter \cite{AS} characterize the regular algebras of global dimension $3$. These naturally split into  two classes: some quadratic algebras $A$ satisfying $H_A=(1-t)^{-3}$ and some cubic algebras $A$ with $H_A=(1+t)^{-1}(1-t)^{-3}$. As far as quadratic algebras are concerned, Artin and Schelter characterize a subclass of $\rm PBW_S$ quadratic algebras on three generators (additional properties imposed are Gorenstein and global dimension $3$). The purpose of this article is to complete their characterization to incorporate all quadratic $\rm PBW_S$ algebras on three generators, identifying Koszul algebras on the way. We also characterize all cubic algebras with the Hilbert series $(1+t)^{-1}(1-t)^{-3}$. We push it to the limit providing a canonical form up to isomorphism. For the sake of brevity we denote
$$
\Omega=\{A:A\ \text{is a quadratic algebra satisfying $H_A=(1-t)^{-3}$}\}.
$$
We split the class $\Omega$ into three disjoint parts: $\Omega=\Omega^0\cup \Omega^+\cup\Omega^-$, where
$$
\begin{array}{l}
\Omega^0=\{A\in\Omega:\text{$A$ is $\rm PBW_B$}\},\\
\Omega^+=\{A\in\Omega:\text{$H_{A^!}=(1+t)^3$, but $A$ is not $\rm PBW_B$}\},\\
\Omega^-=\{A\in\Omega:\text{$H_{A^!}\neq (1+t)^3$}\}.
\end{array}
$$
Note also that $\Omega\subset \Omega'$, where
$$
\Omega'=\{A:A\ \text{is a quadratic algebra satisfying $\dim A_1=3$, $\dim A_2=6$ and $\dim A_3=10$}\}.
$$
Observe that $A\in \Omega^-$ can not be Koszul and therefore can not be $\rm PBW_B$ since the equality $H_A(-t)H_{A^!}(t)=1$ fails.

As for cubic algebras we denote
$$
\Lambda=\{A:A\ \text{is a cubic algebra satisfying $H_A=(1+t)^{-1}(1-t)^{-3}$}\}.
$$
Note that $\Lambda\subset \Lambda'$, where
$$
\Lambda'=\{A:A\ \text{is a cubic algebra satisfying $\dim A_1=2$, $\dim A_2=4$, $\dim A_3=6$ and $\dim A_4=9$}\}.
$$

Let us mention, that as a consequence of this classification we were able to answer an old question of Ufnarovski, namely to provide an example of automaton algebra (one from the family N1 in Theorem~\ref{main}), which does not have a finite Gr\"obner basis. The proof of this result one can find in \cite{autom}.

Before stating the main result, we would like to say a few words about the key idea as well as to introduce some further notations, which will be used throughout the paper.

\subsection{Quasipotentials}

\begin{notation}\label{qpnot} If $V$ is a finite dimensional vector space over $\K$, $k\geq 2$ and $Q\in V^{k+1}=V^{\otimes (k+1)}$, then there are the smallest (in the inclusion sense) subspaces
$$
\text{$E_j=E_j(Q)$ of $V$ and $F_j=F_j(Q)$ of $V^{k}$ such that $Q\in E_1\otimes F_1$ and $Q\in F_2\otimes E_2$.}
$$
Clearly,
$$
\text{$n_1(Q)=\dim E_1=\dim F_1$ is the rank of $Q$ as an element of $V\otimes V^{k}$,}
$$
while
$$
\text{$n_2(Q)=\dim E_2=\dim F_2$ is the rank of $Q$ as an element of $V^{k}\otimes V$.}
$$
We also denote
$$
R_Q=F_1+F_2,\quad \text{which is a subspace of $V^{k}$.}
$$
\end{notation}

\begin{lemma}\label{quasi1} Let $n,k$ be integers such that $n,k\geq2$, $V$ be an $n$-dimensional vector space over $\K$ and $R$ be an $n$-dimensional subspace of the $n^k$-dimensional space $V^k$. Assume also that $A=F(V)/I$, where $I$ is the ideal generated by $R$. Then $\dim A_{k+1}=n^{k+1}-2n^2+1$ if and only if $\dim(RV\cap VR)=1$.
\end{lemma}

\begin{proof} Obviously, $\dim A_{k+1}=n^{k+1}-\dim I_{k+1}$ and $I_{k+1}=RV+VR$. Since $\dim RV=\dim VR=n^2$, we have $\dim I_{k+1}=2n^2-\dim(RV\cap VR)$. The result immediately follows.
\end{proof}

\begin{definition}\label{quasi} Let $n,k$ be integers such that $n,k\geq2$, $V$ be an $n$-dimensional vector space over $\K$, $R$ be an $n$-dimensional subspace of the $n^k$-dimensional space $V^k$ and $I$ be the ideal in $F(V)$ generated by $R$. The algebra $A=F(V)/I$ is called a {\it quasipotential algebra} if $\dim A_{k+1}=n^{k+1}-2n^2+1$. By Lemma~\ref{quasi1}, $RV\cap VR$ is one-dimensional and therefore is spanned by a single degree $k+1$ homogeneous element $Q$ of $F(V)$. That is, $RV\cap VR=\spann\{Q\}$. We call $Q$ a {\it quasipotential} for $A$. We call $Q\in V^{k+1}$ {\it a quasipotential} if it is a quasipotential for some algebra.
\end{definition}

\begin{remark}\label{quasi2} Note that to be a quasipotential algebra is an isomorphism invariant. Moreover, the quasipotential $Q$ of an algebra $A$ is unique up to a scalar multiple and every linear substitution providing an isomorphism between two quasipotential algebras must transform the quasipotential of the first into the quasipotential of the second up to a non-zero scalar multiple.
\end{remark}

\begin{remark}\label{quasi3} By Lemma~\ref{quasi1}, all algebras in $\Omega'$ are quasipotential, each with a degree $3$ quasipotential, while all algebras in $\Lambda'$ are quasipotential, each with a degree $4$ quasipotential.
\end{remark}

\begin{remark}\label{quasi4} Let $Q$ be a quasipotential for a quasipotential algebra $A$ and $V$, $R$ be as in Definition~\ref{quasi}. We easily have that $R_Q\subseteq R$, where $R_Q=F_1(Q)+F_2(Q)$ is introduced in Notation~\ref{qpnot}. In particular, if $R_Q$ happens to be $n$-dimensional, we must have $R_Q=R$. That is, if $Q\in V^{k+1}$ is a quasipotential and $\dim R_Q=n$, then there is exactly one algebra for which $Q$ is the quasipotential.
\end{remark}

The bulk of the paper is devoted to providing a canonical form of quasipotentials in the cases $(n,k)=(3,2)$ and $(n,k)=(2,3)$ under the natural action of $GL_n(\K)$ by linear substitutions. In the case $(n,k)=(3,2)$, this task can in a way be treated as an extension of the canonical form results for ternary cubics (these go way back to Weierstrass). It turns out that only in the case $n_1(Q)=n_2(Q)=1$, there are multiple algebras with desired Hilbert series corresponding to the same quasipotential. This case stands out a lot.

As usual, an {\it invariant} is some characteristic of an algebra from a given class, which remains the same when we replace an algebra by an isomorphic one. By Remark~\ref{quasi2},
\begin{equation}\label{rank1}
\text{the ranks $n_1(Q)$ and $n_2(Q)$ are invariants for quasipotential algebras.}
\end{equation}

\begin{definition}\label{twisted}
Let $n,k$ be integers such that $n,k\geq2$, $V$ be an $n$-dimensional vector space over $\K$ and $Q\in V^{k+1}$ be a quasipotential. We call $Q$ a {\it twisted potential} if $n_1(Q)=n_2(Q)=n$. If a twisted potential $Q$ is cyclicly invariant (that is, invariant under the linear map $C:V^{k+1}\to V^{k+1}$ defined by $x_0x_1\dots x_k\mapsto x_1\dots x_kx_0$), then $Q$ is called a {\it potential}.
\end{definition}

\begin{remark}\label{n1-3} Assume that $\dim V=n$ and $Q\in V^{k+1}$ is a twisted potential. Then for every linear basis $X=\{x_1,\dots,x_n\}$ in $V$, we have
$$
\textstyle Q=\sum\limits_{j=1}^n x_jf_j=\sum\limits_{j=1}^n g_jx_j,
$$
where both $\{f_1,\dots,f_n\}$ and $\{g_1,\dots,g_n\}$ are linear bases in the $n$-dimensional space $R=R_Q$. Thus we have a matrix $M_Q(X)\in GL_n(\K)$ of coefficients of $g_j$ with respect to the basis $f_1,\dots,f_n$. Now if $Y=\{y_1,\dots,y_n\}$ is another linear basis in $V$ and $C\in GL_n(\K)$ is the matrix of coefficients of $x_j$ with respect to the basis $y_1,\dots,y_n$, then a routine computation shows that $M_Q(Y)=BM_Q(X)B^{-1}$, where $B=C^T$ is the transpose of $C$. This observation yields that the Jordan normal form of $M_Q(X)$ is an invariant for the class of twisted potential algebras. Note also that a twisted potential $Q$ is cyclicly invariant if and only if $M_Q(X)$ is the identity matrix for some (=for any) basis $X$ in $V$. That is, $Q$ is a potential if and only if $n_1(Q)=n_2(Q)=n$ and $M_Q(X)$ is the identity matrix for some (=for any) basis $X$ in $V$.
\end{remark}

\begin{remark}\label{rerer} The concepts of potential and twisted potential algebras go beyond degree graded algebras. However, in the case of degree graded algebras with certain non-degeneracy assumed, our definition is equivalent to the original definition of potential algebras of Kontsevich \cite{kon} (see \cite{BW,xx} for alternative equivalent definitions). What we call twisted potential algebras here and in \cite{SKL} was first introduced under the name of algebras defined by multilinear forms by Dubois-Violette \cite{DV1,DV2}. Since we intend to never wander outside the class $\Omega'\cup\Lambda'$, we stick with the above definitions (ignore the non-graded case as well as the degenerate potentials).
\end{remark}

Not all pairs of numbers between $1$ and $n$ occur as $(n_1(Q),n_2(Q))$ for a quasipotential $Q$.

\begin{lemma}\label{n1n2} Let $n,k$ be integers such that $n,k\geq2$, $V$ be an $n$-dimensional vector space over $\K$ and $Q\in V^{k+1}$ be a quasipotential. Then $n_1(Q)=n\iff n_2(Q)=n$.
\end{lemma}

\begin{proof} Assume the contrary: $\min\{n_1(Q),n_2(Q)\}<\max\{n_1(Q),n_2(Q)\}=n$. Reversing the order of letters in each of the monomials featuring in $Q$ yields another quasipotential $Q'$ with $n_1(Q')=n_2(Q)$ and $n_2(Q')=n_1(Q)$. This allows us, without loss of generality, assume that $n_1(Q)=n$ and $n_2(Q)=m<n$. Let $x_1,\dots,x_m$ be a linear basis in $W=E_2(Q)$. Pick any $x_{m+1},\dots,x_n$ such that $x_1,\dots,x_n$ form a basis in $V$. Then
\begin{equation}\label{baba}
Q=\sum_{j=1}^n x_jf_j=\sum_{j=1}^m g_jx_j,
\end{equation}
where $f_1,\dots,f_n$ form a basis in the $n$-dimensional space $F_1(Q)\subset V^k$, while $g_1,\dots,g_m$ form a basis in the $m$-dimensional space $F_2(Q)\subset V^k$. Since $\dim(F_1(Q)+F_2(Q))\leq n$ (see Remark~\ref{quasi4}), we have $F_2(Q)\subset F_1(Q)$ and therefore,
\begin{equation}\label{baba1}
\text{each $g_j$ is a linear combination of $f_1,\dots,f_n$.}
\end{equation}
Now by (\ref{baba}), $f_j\in V^{k-1}W$ for each $j$. By (\ref{baba1}), $g_j\in V^{k-1}W$ for each $j$. Plugging this back into (\ref{baba}), we get $f_j\in V^{k-2}W^2$ for each $j$. According to (\ref{baba1}), $g_j\in V^{k-2}W^2$ for each $j$. Iterating this procedure, we eventually see that $f_j\in W^k$ and $g_j\in W^k$ for all $j$. The latter plugged into the equality $Q=g_1x_1+{\dots}+g_mx_m$ of (\ref{baba}), yields $Q\in W^{k+1}$, which is incompatible with the first equality in (\ref{baba}) since $f_j$ are linearly independent. This contradiction completes the proof.
\end{proof}

\vfill\break

The proof of the main results goes along the following lines. Assuming $A\in\Omega'$, for the quasipotential $Q$ for $A$, we have that one of following mutually exclusive options:
\begin{itemize}\itemsep=-2pt
\item $Q$ is a cube (of a degree $1$ element);
\item $n_1(Q)=n_2(Q)=1$ and $Q$ is not a cube;
\item $n_1(Q)=1$ and $n_2(Q)=2$;
\item $n_1(Q)=2$ and $n_2(Q)=1$;
\item $n_1(Q)=n_2(Q)=2$;
\item $n_1(Q)=n_2(Q)=3$.
\end{itemize}
All possibilities are covered since, $n_1(Q)=3\iff n_2(Q)=3$ for $A\in\Omega'$ by Lemma~\ref{n1n2}. Since $n_j(Q)$ are invariants, we do not have to worry of algebras corresponding to different items of the above list being isomorphic to each other. In the case when $Q$ is a cube: $Q=zzz$ for a non-zero $z\in V$, we stratify further by isomorphism classes of the algebra $A_0=A/I$, where $I$ is the ideal generated by $z$ ($A_0$ is uniquely determined by $A\in\Omega'$ up to an isomorphism). In each particular case, we use the Gr\"obner basis technique to figure out which of the algebras have the same series as the commutative polynomials. We also identify Koszul and $\rm PBW_B$ algebras among them.

Similarly, if $A\in\Lambda'$, for the quasipotential $Q$ for $A$, we have that one of following mutually exclusive options:
\begin{itemize}\itemsep=-2pt
\item $Q$ is a fourth power (of a degree $1$ element);
\item $n_1(Q)=n_2(Q)=1$ and $Q$ is not a fourth power;
\item $n_1(Q)=n_2(Q)=2$.
\end{itemize}
Again, all possibilities are covered since, $n_1(Q)=2\iff n_2(Q)=2$ for $A\in\Lambda'$ by Lemma~\ref{n1n2}.

Before formulating main results, we introduce some more notation. Everywhere afterwards, $\theta$ and $i$ are fixed elements of $\K$ satisfying
\begin{equation}\label{thixi}
\theta^3=1\neq\theta\ \ \ \text{and}\ \ \ i^2=-1.
\end{equation}
Note that $\theta$ does exist since $\K$ is algebraically closed and has characteristic different from $3$, $i$ does exist since $\K$ is algebraically closed and $i\notin\{1,-1\}$ since ${\rm char}\,\K\neq 2$.

\subsection{Main results}

The results are presented in tables. The first column provides a label for further references. Generators of algebras from $\Omega$ are denoted $x,y,z$, while generators of algebras from $\Lambda$ are denoted $x,y$. We use the letters $a,b,c,d$ for the parameters from $\K$ (we never need more than 4 parameters). The exceptions column says which values of the parameters are excluded. The isomorphism column provides generators of a group action on the space of parameters such that corresponding algebras are isomorphic precisely when the parameters are in the same orbit. All isomorphisms are meant in the graded algebras sense. For shortness, we occasionally use the following notation. If $u_1,\dots,u_n\in V$, then ${u_1\dots u_n}^\rcirclearrowleft$ stands for the sum of all $n$ cyclic permutations of the word $u_1\dots u_n$. The PBW column indicates whether the algebra is PBW$_{\rm B}$ or not: the Y entry stands for the algebra being PBW$_{\rm B}$, while the N entry for the opposite. Algebras featuring with {\bf different labels} are {\bf non-isomorphic}.

\vfill\break

\begin{theorem}\label{main} {\bf I.} \ An algebra $A$ belongs to $\Omega$ and its quasipotential $Q=Q_A$ is the cube of a degree one element if and only if $A$ is isomorphic to an algebra from the following table. All such algebras are {\bf NON}-Koszul and therefore non-$PBW_B$.

\medskip

\hrule
\tABb{\hhh}{\rm \!\!Quasipotential $Q_A$}{\rm Defining Relations of $A$}{\rm Exceptions}{\rm Isomorphisms}
\tABb{\hhh\rm R1}{$z^3$}{$xy{+}yx{+}zx{+}azy;\ \ y^2{-}xz{+}(1{-}a)yz{-}zx{-}azy;\ \ z^2$}{\rm none}{\rm trivial}
\tABb{\hhh\rm R2}{$z^3$}{$xy{+}yx;\ \ y^2{-}xz{-}yz{-}zx{-}zy;\ \ z^2$}{\rm none}{\rm trivial}
\tABb{\hhh\rm R3}{$z^3$}{$xy{+}yx;\ \ y^2{-}xz{-}zx;\ \ z^2$}{\rm none}{\rm trivial}
\tABb{\hhh\rm R4}{$z^3$}{$xy-yx;\ \ y^2-xz-zx;\ \ z^2$}{\rm none}{\rm trivial}
\tABb{\hhh\rm R5}{$z^3$}{$\begin{array}{l}{\vrule height12pt depth0pt width0pt}axy{-}a^2yx{-}a^2(a^2{-}1)zx{-}(a{-}1)(a^3{+}1)zy;\\
y^2{-}xz{+}(a{-}2)yz{-}a^2zx{-}(a^2{-}a{+}1)zy;\ \ z^2\end{array}$}{$a\neq0$, $a^2\neq 1$}{\rm trivial}
\tABb{\hhh\rm R6}{$z^3$}{${\vrule height12pt depth0pt width0pt}xy{-}ayx;\ \ y^2{-}xz{-}a^2zx;\ \ z^2$}{$a\neq0$, $a^2\neq 1$}{\rm trivial}
\tABb{\hhh\rm R7}{$z^3$}{$x^2{-}xy{-}yz;\ \ yx{-}azx{-}bzy;\ \ z^2$}{${\vrule height12pt depth0pt width0pt}a(1{+}{\dots}{+}b^k)\neq 1\ \text{for all}\ k\in\Z_+$}{\rm trivial}
\tABb{\hhh\rm R8}{$z^3$}{$x^2{-}xy;\ \ yx{-}azx{+}zy;\ \ z^2$}{$a\neq0$, $a\neq -1$}{\rm trivial}
\tABb{\hhh\rm R9}{$z^3$}{$yx{-}bxz{-}azx{-}azy;\ \ y^2{-}zx{-}zy;\ \ z^2$}{\rm none}{\rm trivial}
\tABb{\hhh\rm R10}{$z^3$}{$yx{-}axz{-}zx;\ \ y^2{-}zx;\ \ z^2$}{\rm none}{\rm trivial}
\tABb{\hhh\rm R11}{$z^3$}{$yx{-}xz;\ \ y^2{-}zx;\ \ z^2$}{\rm none}{\rm trivial}
\tABb{\hhh\rm R12}{$z^3$}{$yx;\ \ y^2{-}zx;\ \ z^2$}{\rm none}{\rm trivial}
\tABb{\hhh\rm R13}{$z^3$}{$xy{-}bzx{-}axz{-}ayz;\ \ y^2{-}xz{-}yz;\ \ z^2$}{\rm none}{\rm trivial}
\tABb{\hhh\rm R14}{$z^3$}{$xy{-}azx{-}xz;\ \ y^2{-}xz;\ \ z^2$}{\rm none}{\rm trivial}
\tABb{\hhh\rm R15}{$z^3$}{$xy{-}zx;\ \ y^2{-}xz;\ \ z^2$}{\rm none}{\rm trivial}
\tABb{\hhh\rm R16}{$z^3$}{$xy;\ \ y^2{-}xz;\ \ z^2$}{\rm none}{\rm trivial}
\tABb{\hhh\rm R17}{$z^3$}{$xy{-}yz+zx;\ \ yx{-}xz{+}zy;\ \ z^2$}{\rm none}{\rm trivial}
\tABb{\hhh\rm R18}{$z^3$}{$xy{-}yz-zx;\ \ yx{-}xz{-}zy;\ \ z^2$}{\rm none}{\rm trivial}
\tABb{\hhh\rm R19}{$z^3$}{$xy{-}azx{-}zy;\ \ yx{-}xz;\ \ z^2$}{${\vrule height12pt depth0pt width0pt}(1{+}a{+}{\dots}{+}a^k)\neq 0\ \text{for all}\ k\in\Z_+$}{\rm trivial}
\tABb{\hhh\rm R20}{$z^3$}{$x^2{+}yz{+}azy;\ \ y^2{+}xz{+}\frac{1}{a} zx;\ \ z^2$}{$a\neq 0$}{${\vrule height12pt depth7pt width0pt}a\mapsto\frac1a$}
\tABb{\hhh\rm R21}{$z^3$}{$xy{-}yx{-}y^2{-}zx{-}czy;\ \ xz{-}azx{-}bzy;\ \ z^2$}{\rm $na+b\neq 0$ for all $n\in\Z_+$}{\rm trivial}
\tABb{\hhh\rm R22}{$z^3$}{$xy{-}yx{-}y^2{-}zy;\ \ xz{-}azx{-}bzy;\ \ z^2$}{\rm $na+b\neq 0$ for all $n\in\Z_+$}{\rm trivial}
\tABb{\hhh\rm R23}{$z^3$}{$xy{-}yx{-}y^2;\ \ xz{-}azx{-}bzy;\ \ z^2$}{\rm $na+b\neq 0$ for all $n\in\Z_+$}{\rm trivial}
\tABb{\hhh\rm R24}{$z^3$}{$xy{-}yx{-}y^2{-}zx-bzy;\ \ zx{-}azy;\ \ z^2$}{$a\neq 0$}{\rm trivial}
\tABb{\hhh\rm R25}{$z^3$}{$xy{-}yx{-}y^2{-}zy;\ \ zx{-}azy;\ \ z^2$}{$a\neq 0$}{\rm trivial}
\tABb{\hhh\rm R26}{$z^3$}{$xy{-}yx{-}y^2;\ \ zx{-}azy;\ \ z^2$}{$a\neq 0$}{\rm trivial}
\tABb{\hhh\rm R27}{$z^3$}{$xy{-}ayx{-}zx{-}zy;\ \ xz{+}byz{+}czx{+}dzy;\ \ z^2$}{$\begin{array}{l}a\neq 0,\ a\neq 1,\ b\neq 0\\ d\neq a^nbc\ \text{for all $n\in\Z_+$}\end{array}$}{$(a,b,c,d)\mapsto(\frac1a,\frac1b,\frac{d}{b},\frac{c}{b})$}
\tABb{\hhh\rm R28}{$z^3$}{$xy{-}ayx{-}zx{-}zy;\ \ xz{+}bzx{+}czy;\ \ z^2$}{$a\neq 0$,\ $a\neq 1$,\ $c\neq 0$}{\rm trivial}
\tABb{\hhh\rm R29}{$z^3$}{$xy{-}ayx{-}zx;\ \ xz{+}yz{+}bzx{+}czy;\ \ z^2$}{$\begin{array}{l}a\neq 0,\ a\neq 1\\ c\neq a^nb\ \text{for all $n\in\Z_+$}\end{array}$}{\rm trivial}
\tABb{\hhh\rm R30}{$z^3$}{$xy{-}ayx{-}zx;\ \ xz{+}bzx{+}zy;\ \ z^2$}{$a\neq 0$,\ $a\neq 1$}{\rm trivial}
\tABb{\hhh\rm R31}{$z^3$}{$xy{-}ayx{-}zx;\ \ yz{+}zx{+}bzy;\ \ z^2$}{$a\neq 0$,\ $a\neq 1$}{\rm trivial}
\tABb{\hhh\rm R32}{$z^3$}{$xy{-}ayx;\ \ xz{+}yz{+}bzx{+}czy;\ \ z^2$}{$\begin{array}{l}a\neq 0,\ a\neq 1\\ c\neq a^nb\ \text{for all $n\in\Z_+$}\end{array}$}{\rm trivial}
\tABb{\hhh\rm R33}{$z^3$}{$xy{-}ayx;\ \ yz{+}bzx{+}zy;\ \ z^2$}{$a\neq 0$,\ $a\neq 1$}{\rm trivial}
\tABb{\hhh\rm R34}{$z^3$}{$xy{-}yx{-}zx{-}azy;\ \ xz{-}zy;\ \ z^2$}{\rm none}{\rm trivial}
\tABb{\hhh\rm R35}{$z^3$}{$xy{-}yx{-}zy;\ \ xz{-}zy;\ \ z^2$}{\rm none}{\rm trivial}
\tABb{\hhh\rm R36}{$z^3$}{$xy{-}yx;\ \ xz{-}zy;\ \ z^2$}{\rm none}{\rm trivial}
\tABb{\hhh\rm R37}{$z^3$}{$xy{-}yz{-}czx;\ \ xz{+}ayz{+}bzx{+}zy;\ \ z^2$}{$ab\neq 1$}{\rm trivial}
\tABb{\hhh\rm R38}{$z^3$}{$xy{-}zx;\ \ xz{+}ayz{+}bzx{+}zy;\ \ z^2$}{$ab\neq 1$}{\rm trivial}
\tABb{\hhh\rm R39}{$z^3$}{$xy;\ \ xz{+}ayz{+}bzx{+}zy;\ \ z^2$}{$ab\neq 1$}{\rm trivial}

\vfill\break

\noindent{\bf II.} \ An algebra $A$ belongs to $\Omega$ and its quasipotential $Q=Q_A$ satisfies $n_1(Q)=1$, $n_2(Q)=2$ if and only if $A$ is isomorphic to an algebra from the following table. All such algebras are Koszul.

\medskip
\hrule
\tAbb{\hhh}{\rm Quasipotential $Q_A$}{\rm Defining Relations of $A$}{\rm Exceptions}{\rm Isomorphisms}{\hskip-1pt$\scriptstyle{\rm PBW}$}
\tAbb{\hhh\rm M1}{$xyz-axzy$}{$xy;\ \ xz;\ \ yz-azy$}{$a\neq 0$}{${\vrule height12pt depth7pt width0pt}a\mapsto \frac1a$}{\rm Y}
\tAbb{\hhh\rm M2}{$xyz-axzy+xzx$}{$xy;\ \ xz;\ \ yz-azy+zx$}{$a\neq 0$}{\rm trivial}{\rm Y}
\tAbb{\hhh\rm M3}{$xyz-axzy+xzx+xyx$}{$xy;\ \ xz;\ \ yz-azy+zx+yx$}{$a\neq 0$, $a\neq 1$}{${\vrule height12pt depth7pt width0pt}a\mapsto \frac1a$}{\rm Y}
\tAbb{\hhh\rm M4}{$xyz-xzy-xz^2+xyx$}{$xy;\ \ xz;\ \ yz-zy-z^2+yx$}{\rm none}{\rm trivial}{\rm Y}
\tAbb{\hhh\rm M5}{$xyz-xzy-xz^2+xzx$}{$xy;\ \ xz;\ \ yz-zy-z^2+zx$}{\rm none}{\rm trivial}{\rm Y}
\tAbb{\hhh\rm M6}{$xyz-xzy-xz^2$}{$xy;\ \ xz;\ \ yz-zy-z^2$}{\rm none}{\rm trivial}{\rm Y}
\tAbb{\hhh\rm M7}{$xzy+xyx$}{$xy;\ \ xz;\ \ zy+yx$}{\rm none}{\rm trivial}{\rm Y}
\tAbb{\hhh\rm M8}{$xzy+xyx-xzx$}{$xy;\ \ xz;\ \ zy+yx-zx$}{\rm none}{\rm trivial}{\rm Y}

\medskip

\noindent{\bf III.} \ An algebra $A$ belongs to $\Omega$ and its quasipotential $Q=Q_A$ satisfies $n_1(Q)=2$, $n_2(Q)=1$ if and only if $A$ is isomorphic to an algebra from the following table. All such algebras are Koszul.

\medskip
\hrule
\tAbb{\hhh}{\rm Quasipotential $Q_A$}{\rm Defining Relations of $A$}{\rm Exceptions}{\rm Isomorphisms}{\hskip-1pt$\scriptstyle{\rm PBW}$}
\tAbb{\hhh\rm L1}{$zyx-ayzx$}{$yx;\ \ zx;\ \ zy-ayz$}{$a\neq 0$}{${\vrule height12pt depth7pt width0pt}a\mapsto \frac1a$}{\rm Y}
\tAbb{\hhh\rm L2}{$zyx-ayzx+xzx$}{$yx;\ \ zx;\ \ zy-ayz+xz$}{$a\neq 0$}{\rm trivial}{\rm Y}
\tAbb{\hhh\rm L3}{$zyx-ayzx+xzx+xyx$}{$yx;\ \ zx;\ \ zy-ayz+xz+xy$}{$a\neq 0$}{${\vrule height12pt depth7pt width0pt}a\mapsto \frac1a$}{\rm Y}
\tAbb{\hhh\rm L4}{$zyx-yzx-z^2x+xyx$}{$yx;\ \ zx;\ \ zy-yz-z^2+xy$}{\rm none}{\rm trivial}{\rm Y}
\tAbb{\hhh\rm L5}{$zyx-yzx-z^2x+xzx$}{$yx;\ \ zx;\ \ zy-yz-z^2+xz$}{\rm none}{\rm trivial}{\rm Y}
\tAbb{\hhh\rm L6}{$zyx-yzx-z^2x$}{$yx;\ \ zx;\ \ zy-yz-z^2$}{\rm none}{\rm trivial}{\rm Y}
\tAbb{\hhh\rm L7}{$yzx+xyx$}{$yx;\ \ zx;\ \ yz+xy$}{\rm none}{\rm trivial}{\rm Y}
\tAbb{\hhh\rm L8}{$yzx+xyx-xzx$}{$yx;\ \ zx;\ \ yz+xy-xz$}{\rm none}{\rm trivial}{\rm Y}

\medskip

\noindent{\bf IV.} \ An algebra $A$ belongs to $\Omega$ and its quasipotential $Q=Q_A$ satisfies $n_1(Q)=n_2(Q)=2$ if and only if $A$ is isomorphic to an algebra from the following table. All such algebras are Koszul. The algebra in {\rm (S18)} is an odd one out. It is well-defined and belongs to $\Omega'$ whenever ${\rm char}\,\K\neq 2$. However it is in $\Omega$ only when $\K$ is of characteristic zero. This explains the weird entry in the exceptions column.

\medskip

\hrule
\taAbb{\hhh}{\rm Quasipotential $Q_A$}{\rm Defining Relations of $A$}{\rm Exceptions}{\rm Isomorphisms}{\hskip-1pt$\scriptstyle{\rm PBW}$}
\taAbb{\hhh\rm S1}{$xzy+xyx+yxy$}{$xy;\ \ zy+yx;\ \ xz+yx$}{\rm none}{\rm trivial}{\rm Y}
\taAbb{\hhh\rm S2}{$bxyz+xzy+xz^2+x^2z+ayxz$}{$xz;\ \ byz+zy+z^2;\ \ bxy+ayx+x^2$}{$ab\neq 0$}{\rm trivial}{\rm Y}
\taAbb{\hhh\rm S3}{$x^2z+axyz+xzy+byxz$}{$xz;\ \ ayz+zy;\ \ axy+byx+x^2$}{$ab\neq 0$}{\rm trivial}{\rm Y}
\taAbb{\hhh\rm S4}{$xz^2+axyz+yxz+bxzy$}{$xz;\ \ ayz+bzy+z^2;\ \ axy+yx$}{$ab\neq 0$}{\rm trivial}{\rm Y}
\taAbb{\hhh\rm S5}{$axyz+xzy+byxz$}{$xz;\ \ ayz+zy;\ \ axy+byx$}{$ab\neq 0$}{\rm trivial}{\rm Y}
\taAbb{\hhh\rm S6}{$xyz+ayxz+yzy+yz^2$}{$yz;\ \ axz+zy+z^2;\ \ xy+ayx$}{$a\neq 0$}{\rm trivial}{\rm Y}
\taAbb{\hhh\rm S7}{$xyz+ayxz+yzy$}{$yz;\ \ axz+zy;\ \ xy+ayx$}{$a\neq 0$}{\rm trivial}{\rm Y}
\taAbb{\hhh\rm S8}{$xyz-yxz+ay^2z+yzy+yz^2$}{$yz;\ \ xz-zy-z^2;\ \ xy-yx+ay^2$}{$a\neq 0$}{\rm trivial}{\rm Y}
\taAbb{\hhh\rm S9}{$xyz-yxz+ay^2z+yzy$}{$yz;\ \ xz-zy;\ \ xy-yx+ay^2$}{$a\neq 0$}{\rm trivial}{\rm Y}
\taAbb{\hhh\rm S10}{$xyz+axzy+yxy+x^2y$}{$xy;\ \ yz+azy;\ \ axz+yx+x^2$}{$a\neq 0$}{\rm trivial}{\rm Y}
\taAbb{\hhh\rm S11}{$xyz+axzy+yxy$}{$xy;\ \ yz+azy;\ \ axz+yx$}{$a\neq 0$}{\rm trivial}{\rm Y}
\taAbb{\hhh\rm S12}{$xyz-xzy+axy^2+yxy+x^2y$}{$xy;\ \ yz-zy+ay^2;\ \ xz-yx-x^2$}{$a\neq 0$}{\rm trivial}{\rm Y}
\taAbb{\hhh\rm S13}{$xyz-xzy+axy^2+yxy$}{$xy;\ \ yz-zy+ay^2;\ \ xz-yx$}{$a\neq 0$}{\rm trivial}{\rm Y}
\taAbb{\hhh\rm S14}{$x^2z+axyz+xzy+yxz+y^3$}{$ayz+zy-y^2;\ \ xz+y^2;\ \ axy+yx+x^2$}{$a\neq 0$}{\rm trivial}{\rm Y}
\taAbb{\hhh\rm S15}{$xz^2+axyz+yxz+xzy+y^3$}{$ayz+zy+z^2;\ \ xz+y^2;\ \ axy+yx-y^2$}{$a\neq 0$}{\rm trivial}{\rm Y}
\taAbb{\hhh\rm S16}{$axyz+xzy+yxz+y^3$}{$ayz+zy;\ \ xz+y^2;\ \ axy+yx$}{$a\neq 0$}{\rm trivial}{\rm Y}
\taAbb{\rm S17}{\!\!\!\!\hskip-1pt$\begin{array}{l}\hhh yxz{-}ayxy{-}azxz{+}ay^2z{-}yz^2{+}az^3\\{+}a^2zxy-a^2y^3-a^2z^2y+a^3zy^2\\ +(a{-}1{-}a^3)zyz+(a^3{-}a^2{+}1)yzy\end{array}$}{$\!\!\begin{array}{l}{\vrule height12pt depth0pt width0pt}axy-xz+a^2y^2-(a^3+1)zy+z^2;\\ yx-azx+ay^2-(a^3+1)zy+az^2;\ \ yz-azy\end{array}$}{$a\neq0,\ \ a^k\neq 1\ \text{for all}\ k\in\N$}{$a\mapsto\frac1a$}{\rm N}
\taAbb{\rm S18}{$\begin{array}{l}\hhh 4yxy+4y^3+4zxy+4zy^2\\ -2z^2y-2yz^2+2zyz-z^3\end{array}$}{$yx{+}y^2{+}zx{+}zy{-}\frac12z^2;\ \ xy{+}y^2{-}\frac12 z^2;\ \ yz{-}zy{+}\frac12z^2$}{${\rm char}\,\K=0$}{\rm trivial}{\rm N}
\taAbb{\rm S19}{$\begin{array}{l}axy^2-axyz-xzy+x^2y{\vrule height0pt depth9pt width0pt} \\ -azxy-\frac{b^2{-}(a{+}1)^2}{4}z^3\end{array}$}{$\begin{array}{l}{\vrule height16pt depth9pt width0pt} x^2{-}xz{-}azx{-}\frac{b^2{-}(a{+}1)^2}{4}z^2;\ \ xy{+}\frac{b^2{-}(a{+}1)^2}{4a}z^2;\\ {\vrule height0pt depth9pt width0pt} y^2{-}yz{-}\frac1a zy{-}\frac{b^2{-}(a{+}1)^2}{4a^2}z^2\end{array}$}{$\!\!\!\begin{array}{l} {\vrule height0pt depth9pt width0pt} a{\neq}0,\ a{\neq}-1,\ b^2{\neq}(a{+}1)^2\!,\\ \frac{\scriptstyle (1{-}a{-}b)^n}{\scriptstyle 1{+}a{+}b}{\neq}\frac{\scriptstyle (1{-}a{+}b)^n}{\scriptstyle  1{+}a{-}b}\ \scriptstyle\text{for all}\ n\in\Z_+\end{array}$}{$(a,b)\mapsto (a,-b)$}{\rm N}
\taAbb{\rm S20}{$\begin{array}{l}axy^2-axyz-xzy+x^2y{\vrule height0pt depth9pt width0pt} \\ -azxy+\frac{(a{+}1)^2}{4}z^3\end{array}$}{$\begin{array}{l}{\vrule height16pt depth9pt width0pt} x^2{-}xz{-}azx{+}\frac{(a{+}1)^2}{4}z^2;\ \  xy{-}\frac{(a{+}1)^2}{4a}z^2;\\ {\vrule height0pt depth9pt width0pt} y^2{-}yz{-}\frac1a zy{+}\frac{(a{+}1)^2}{4a^2}z^2\end{array}$}{$\!\!\!\begin{array}{l}{\vrule height0pt depth9pt width0pt} a{\neq}0,\ a{\neq}-1,\\ na\neq n+2\ \text{for all}\ n\in\N\end{array}$}{\rm trivial}{\rm N}

\vfill\break

\noindent{\bf V.} \ There is no algebra $A\in \Omega$ such that its quasipotential $Q=Q_A$ satisfies $\max\{n_1(Q),n_2(Q)\}=3$ and $\min\{n_1(Q),n_2(Q)\}<3$. There is no $A\in\Lambda$ such that its quasipotential $Q=Q_A$ satisfies $\max\{n_1(Q),n_2(Q)\}=2$ and $\min\{n_1(Q),n_2(Q)\}<2$.

\bigskip

\noindent{\bf VI.} \ An algebra $A$ belongs to $\Omega$ and is potential if and only if $A$ is isomorphic to an algebra from the following table. All such algebras are Koszul.

\medskip

\hrule
\tAbb{\hhh}{\rm The potential $Q_A$}{\rm Defining Relations of $A$}{\rm Exceptions}{\rm Isomorphisms}{\hskip-1pt$\scriptstyle{\rm PBW}$}
\tAbb{\hhh\rm P1}{$x^3+y^3+z^3+axyz^\rcirclearrowleft+bxzy^\rcirclearrowleft$}{$\begin{array}l \hhh x^2+ayz+bzy;\\ y^2+azx+bxz;\\ z^2+axy+byx\end{array}$}
{$\begin{array}l \hhh (a,b)\neq(0,0)\\ (a^3,b^3)\neq(1,1)\\ (a+b)^3+1\neq 0\end{array}$}{$\begin{array}l (a,b)\mapsto (\theta a,\theta b)\\ (a,b)\mapsto \bigl(\frac{\theta a+\theta^2b+1}{a+b+1},\frac{\theta^2 a+\theta b+1}{a+b+1}\bigr)\end{array}$}{\rm N}
\tAbb{\hhh\rm P2}{$xyz^\rcirclearrowleft+axzy^\rcirclearrowleft$}{$yz+azy;\ \ zx+axz;\ \ xy+ayx$}
{$a\neq0$}{$a\mapsto a^{-1}$}{\rm Y}
\tAbb{\hhh\rm P3}{$(y+z)^3+xyz^\rcirclearrowleft+axzy^\rcirclearrowleft$}{$\begin{array}l \hhh yz+azy;\ \ axz+zx+(y+z)^2;\\ xy+ayx+(y+z)^2\end{array}$}
{$a\neq 0$, $a\neq -1$}{$a\mapsto a^{-1}$}{\rm Y}
\tAbb{\hhh\rm P4}{$z^3+xyz^\rcirclearrowleft+axzy^\rcirclearrowleft$}{$yz+azy;\ \ axz+zx;\ \ xy+ayx+z^2$}
{$a\neq 0$}{$a\mapsto a^{-1}$}{\rm Y}
\tAbb{\hhh\rm P5}{$y^3+{xz^2}^\rcirclearrowleft+xyz^\rcirclearrowleft-xzy^\rcirclearrowleft$}{$yz-zy+z^2;\ \ -xz+zx+y^2;\ \ xy-yx+xz+zx$}{\rm none}{\rm trivial}{\rm Y}
\tAbb{\hhh\rm P6}{${xz^2}^\rcirclearrowleft+{y^2z}^\rcirclearrowleft+xyz^\rcirclearrowleft-xzy^\rcirclearrowleft$}{$\begin{array}l yz-zy+z^2;\ \ -xz+zx+yz+zy;\\ xy-yx+xz+zx+y^2\end{array}$}{\rm none}{\rm trivial}{\rm Y}
\tAbb{\hhh\rm P7}{$y^3+z^3+xyz^\rcirclearrowleft-xzy^\rcirclearrowleft$}{$yz-zy;\ \ -xz+zx+y^2;\ \ xy-yx+z^2$}{\rm none}{\rm trivial}{\rm Y}
\tAbb{\hhh\rm P8}{${yz^2}^\rcirclearrowleft+xyz^\rcirclearrowleft-xzy^\rcirclearrowleft$}{$yz-zy;\ \ -xz+zx+z^2;\ \ xy-yx+yz+zy$}{\rm none}{\rm trivial}{\rm Y}

\bigskip

\noindent{\bf VII.} \ An algebra $A$ belongs to $\Omega$ and is twisted potential and non-potential if and only if $A$ is isomorphic to an algebra from the following table. All such algebras are Koszul.

\medskip
\hrule
\tAbb{\hhh}{\rm Twisted potential $Q_A$}{\rm Defining Relations of $A$}{\rm Exceptions}{\rm Isomorphisms}{\hskip-1pt$\scriptstyle{\rm PBW}$}
\tAbb{\rm T1}{$\!\!\begin{array}{l}bxyz +ayzx+czxy\\ -abyxz-bcxzy-aczyx\end{array}$}{$xy-ayx;\ \ zx-bxz;\ \ yz-czy$}{$\!\!\!\begin{array}{l}abc\neq0\\
(a{-}b,a{-}c){\neq}(\hskip-1pt0,0\hskip-1pt)
\end{array}$}
{$\begin{array}l (a,b,c)\mapsto (b,c,a)\\ (a,b,c)\mapsto(a^{-1},c^{-1},b^{-1})\end{array}$}{\rm Y}
\tAbb{\rm T2}{$\!\!\begin{array}{l}axyz +byzx+azxy\\ -abyxz-a^2xzy-abzyx-az^3\end{array}$}{$xy-byx-z^2;\ \ zx-axz;\ \ yz-azy$}{$\begin{array}{l}ab\neq0\\ a\neq b\end{array}$}{$(a,b)\mapsto(a^{-1},b^{-1})$}{\rm Y}
\tAbb{\rm T3}{$\!\!\begin{array}{l} xzy^\rcirclearrowleft{-}xyz^\rcirclearrowleft{-}\frac{1{+}a}{2}yzy\\ +a(xz^2{+}z^2x{+}z^2y)
\\+\frac{1-a}{2}(y^2z{+}zy^2{-}2zxz{-}zyz){\vrule height0pt depth7pt width0pt}\end{array}$}{$\!\!\begin{array}l \hhh yz-zy-az^2;\\ xz-zx-azy+\frac{a(1-a)}{2}z^2;\\
xy{-}yx{+}(1{-}2a)zx{+}\frac{a{-}1}2 y^2{+}\frac{(1{+}a)(1{-}2a)}{4}zy{+}\frac{a^2(1{-}a)}{2}z^2 {\vrule height0pt depth7pt width0pt} \end{array}$}{$a\neq\frac13$}{\rm trivial}{\rm Y}
\tAbb{\rm T4}{$\!\!\begin{array}{l} {\vrule height12pt depth8pt width0pt} xzy^\rcirclearrowleft{-}xyz^\rcirclearrowleft{+}\frac13xz^2{+}\frac13z^2x
\\{-}\frac23zxz+\frac13y^2z{+}\frac13zy^2{-}\frac23yzy\\ {+}\frac13z^2y{-}\frac{1}{3}zyz{+}\frac{a}{27}z^3 {\vrule height12pt depth7pt width0pt}\end{array}$}{$\begin{array}l {\vrule height12pt depth8pt width0pt} yz-zy-\frac13z^2;\\ xz-zx-\frac13zy-\frac19z^2;\\ xy-yx-\frac13y^2+\frac13zx+\frac29zy+\frac{1-a}{27}z^2 {\vrule height12pt depth7pt width0pt} \end{array}$}{\rm none}{\rm trivial}{\rm Y}
\tAbb{\rm T5}{$\!\!\begin{array}{l} \hhh zyx{+}byxz{+}b^2xzy{-}bzxy{-}b^2xyz\\-yzx{+}(ab{-}1)zxz{+}az^2x{+}ab^2xz^2\end{array}$}{$\begin{array}{l} bxy{+}(1{-}ab)xz{-}yx{-}azx;\ \ bxz-zx;\\
yz-zy-az^2\end{array}$}{$b\neq 0$}{\rm trivial}{\rm Y}
\tAbb{\rm T6}{$\!\!\begin{array}{l} yxz-xzy+zyx+yzx\\-xyz-zxy+(a-1)yzy\\+ay^2z+azy^2+z^3\end{array}$}{$\begin{array}{l}\hhh -xy+yx+ay^2+z^2;\\ xz+zx+(a{-}1)zy+ayz;\\
yz+zy\end{array}$}{\rm none}{\rm trivial}{\rm Y}
\tAbb{\rm T7}{$\!\!\begin{array}{l} \hhh xzy^\rcirclearrowleft-xyz^\rcirclearrowleft-yzy\\+ay^2z^\rcirclearrowleft+by^3+z^3\end{array}$}{\!\!$\begin{array}{l} \hhh -xy+yx+ay^2+z^2;\ \ yz-zy;\\ xz{+}by^2{+}ayz{-}zx{+}(a{-}1)zy\end{array}$}{\rm none}{$(a,b)\mapsto (a,-b)$}{\rm Y}
\tAbb{\rm T8}{$xzy^\rcirclearrowleft{-}xyz^\rcirclearrowleft{-}yzy{+}{yz^2}^\rcirclearrowleft{+}ay^3$}{$\begin{array}{l} -xy+yx+yz+zy;\ \ yz-zy;\\ xz+ay^2-zx-zy+z^2\end{array}$}{$a\neq 0$}{\rm trivial}{\rm Y}
\tAbb{\rm T9}{$\!\!\begin{array}{l} {\vrule height11pt depth0pt width0pt}a^2xyz{+}yzx{+}azxy{-}a^2xzy{-}zyx\\ -ayxz+a^2xz^2+zyz+azxz\end{array}$}{$axy-yx+2zx;\ \ axz-zx;\ \ yz-zy+z^2$}{$a\neq 0$}{\rm trivial}{\rm Y}
\tAbb{\rm T10}{$\!\!\begin{array}{l} xyz-yzx+zxy-yxz+xzy\\-zyx+y^2z-yzy+zy^2+z^3\end{array}$}{$xy-yx+y^2+z^2;\ \ xz+zx+2zy;\ \ yz+zy$}{\rm none}{\rm trivial}{\rm Y}
\tAbb{\rm T11}{$\!\!\begin{array}{l} \hhh x^2z+axzx+a^2 zx^2\\ +y^2z-ayzy+a^2zy^2\end{array}$}{$xz+azx;\ \ yz-azy;\ \ x^2+y^2$}{$a\neq 0$}{$a\mapsto -a$}{\rm Y}
\tAbb{\rm T12}{$\!\!\begin{array}{l} \hhh z^2y+izyz-yz^2+y^2x\\-yxy+xy^2+x^3\end{array}$}{$x^2+y^2;\ \ xy-yx+z^2;\ \ zy+iyz$}{\rm none}{\rm trivial}{\rm N}
\tAbb{\rm T13}{$\!\!\begin{array}{l} \hhh z^2y-izyz-yz^2+y^2x\\-yxy+xy^2+x^3\end{array}$}{$x^2+y^2;\ \ xy-yx+z^2;\ \ zy-iyz$}{\rm none}{\rm trivial}{\rm N}
\tAbb{\rm T14}{$\!\!\begin{array}{l} xyx+yxy+zyx+yzy+zyz\\+\theta xzy+\theta zxz+\theta^2xzx+\theta^2yxz \end{array}$}{$yx+\theta zy+\theta^2zx;\ \ xy+zy+\theta^2xz;\ \  yx+yz+\theta xz$}{\rm none}{\rm trivial}{\rm N}
\tAbb{\rm T15}{$\!\!\begin{array}{l} xyx+yxy+zyx+yzy+zyz\\+\theta^2 xzy+\theta^2 zxz+\theta xzx+\theta yxz \end{array}$}{$yx+\theta^2 zy+\theta zx;\ \ xy+zy+\theta xz;\ \  yx+yz+\theta^2 xz$}{\rm none}{\rm trivial}{\rm N}
\tAbb{\hhh\rm T16}{${y^2z}^\rcirclearrowleft+z^3+x^2z-xzx+zx^2$}{$x^2+y^2+z^2;\ \  xz-zx;\ \ yz+zy$}{\rm none}{\rm trivial}{\rm Y}
\tAbb{\hhh\rm T17}{${xy^2}^\rcirclearrowleft+y^3+xz^2-zxz+z^2x$}{$xz-zx;\ \ xy+yx+y^2;\ \  y^2+z^2$}{\rm none}{\rm trivial}{\rm Y}
\tAbb{\hhh\rm T18}{$y^3{+}{yz^2}^\rcirclearrowleft{+}az^3{+}x^2z{-}xzx{+}zx^2$}{$xz-zx;\ \ yz+zy+x^2+az^2;\ \  y^2+z^2$}{$a^2+4\neq 0$}{$a\mapsto -a$}{\rm N}

\vfill\break

\noindent{\bf VIII.} \ An algebra $A$ belongs to $\Omega$ and its quasipotential $Q=Q_A$ satisfies $n_1(Q)=n_2(Q)=1$ with $Q$ {\bf NOT} being a cube of a degree one element if and only if $A$ is isomorphic to an algebra from the following table. All such algebras are Koszul.

\bigskip
\hrule
\taAbb{\hhh}{\rm Quasipotential $Q_A$}{\rm Defining Relations of $A$}{\rm Exceptions}{\rm Isomorphisms}{\hskip-1pt$\scriptstyle{\rm PBW}$}
\taAbb{\hhh\rm N1}{$xyz$}{$xy;\ \ yz;\ \ x^2+xz+azx+yx+by^2+czy$}{\rm none}{\rm trivial}{\rm Y}
\taAbb{\hhh\rm N2}{$xyz$}{$xy;\ \ yz;\ \ x^2+xz+azx+by^2+zy$}{\rm none}{\rm trivial}{\rm Y}
\taAbb{\hhh\rm N3}{$xyz$}{$xy;\ \ yz;\ \ x^2+xz+azx+y^2$}{\rm none}{\rm trivial}{\rm Y}
\taAbb{\hhh\rm N4}{$xyz$}{$xy;\ \ yz;\ \ x^2+xz+azx$}{\rm none}{\rm trivial}{\rm Y}
\taAbb{\hhh\rm N5}{$xyz$}{$xy;\ \ yz;\ \ xz+azx+yx+by^2+czy+z^2$}{\rm none}{\rm trivial}{\rm Y}
\taAbb{\hhh\rm N6}{$xyz$}{$xy;\ \ yz;\ \ xz+azx+by^2+zy+z^2$}{\rm none}{\rm trivial}{\rm Y}
\taAbb{\hhh\rm N7}{$xyz$}{$xy;\ \ yz;\ \ xz+azx+y^2+z^2$}{\rm none}{\rm trivial}{\rm Y}
\taAbb{\hhh\rm N8}{$xyz$}{$xy;\ \ yz;\ \ xz+azx+z^2$}{\rm none}{\rm trivial}{\rm Y}
\taAbb{\hhh\rm N9}{$xyz$}{$xy;\ \ yz;\ \ xz+azx+yx+by^2+zy$}{\rm none}{\rm trivial}{\rm Y}
\taAbb{\hhh\rm N10}{$xyz$}{$xy;\ \ yz;\ \ xz+azx+yx+y^2$}{\rm none}{\rm trivial}{\rm Y}
\taAbb{\hhh\rm N11}{$xyz$}{$xy;\ \ yz;\ \ xz+azx+yx$}{\rm none}{\rm trivial}{\rm Y}
\taAbb{\hhh\rm N12}{$xyz$}{$xy;\ \ yz;\ \ xz+azx+y^2+zy$}{\rm none}{\rm trivial}{\rm Y}
\taAbb{\hhh\rm N13}{$xyz$}{$xy;\ \ yz;\ \ xz+azx+zy$}{\rm none}{\rm trivial}{\rm Y}
\taAbb{\hhh\rm N14}{$xyz$}{$xy;\ \ yz;\ \ xz+azx+y^2$}{\rm none}{\rm trivial}{\rm Y}
\taAbb{\hhh\rm N15}{$xyz$}{$xy;\ \ yz;\ \ xz+azx$}{\rm none}{\rm trivial}{\rm Y}
\taAbb{\rm N16}{$xyz$}{$\begin{array}{l}xy;\ \ yz;\\ x^2-xz-azx-yx-cy^2-dzy-\frac{b^2{-}(a{+}1)^2}{4}z^2\end{array}$}{$\!\!\!\begin{array}{l} {\vrule height12pt depth7pt width0pt} a{\neq}-1,\ b^2{\neq}(a{+}1)^2\!,\\ {\vrule height0pt depth8pt width0pt} \frac{\scriptstyle (1{-}a{-}b)^n}{\scriptstyle 1{+}a{+}b}{\neq}\frac{\scriptstyle (1{-}a{+}b)^n}{\scriptstyle  1{+}a{-}b}\ \scriptstyle\text{for all}\ n\in\Z_+\end{array}$}{$(a,b)\mapsto (a,-b)$}{\rm N}
\taAbb{\rm N17}{$xyz$}{$\begin{array}{l}xy;\ \ yz;\\ x^2-xz-azx-cy^2-zy-\frac{b^2{-}(a{+}1)^2}{4}z^2\end{array}$}{$\!\!\!\begin{array}{l} {\vrule height12pt depth7pt width0pt} a{\neq}-1,\ b^2{\neq}(a{+}1)^2\!,\\ {\vrule height0pt depth8pt width0pt}\frac{\scriptstyle (1{-}a{-}b)^n}{\scriptstyle 1{+}a{+}b}{\neq}\frac{\scriptstyle (1{-}a{+}b)^n}{\scriptstyle  1{+}a{-}b}\ \scriptstyle\text{for all}\ n\in\Z_+\end{array}$}{$(a,b)\mapsto (a,-b)$}{\rm N}
\taAbb{\rm N18}{$xyz$}{$\begin{array}{l}xy;\ \ yz;\\ x^2-xz-azx-y^2-\frac{b^2{-}(a{+}1)^2}{4}z^2\end{array}$}{$\!\!\!\begin{array}{l} {\vrule height12pt depth7pt width0pt} a{\neq}-1,\ b^2{\neq}(a{+}1)^2\!,\\ {\vrule height0pt depth8pt width0pt}\frac{\scriptstyle (1{-}a{-}b)^n}{\scriptstyle 1{+}a{+}b}{\neq}\frac{\scriptstyle (1{-}a{+}b)^n}{\scriptstyle  1{+}a{-}b}\ \scriptstyle\text{for all}\ n\in\Z_+\end{array}$}{$(a,b)\mapsto (a,-b)$}{\rm N}
\taAbb{\rm N19}{$xyz$}{$\begin{array}{l}xy;\ \ yz;\\ x^2-xz-azx-\frac{b^2{-}(a{+}1)^2}{4}z^2\end{array}$}{$\!\!\!\begin{array}{l} {\vrule height12pt depth7pt width0pt} a{\neq}-1,\ b^2{\neq}(a{+}1)^2\!,\\ {\vrule height0pt depth8pt width0pt}\frac{\scriptstyle (1{-}a{-}b)^n}{\scriptstyle 1{+}a{+}b}{\neq}\frac{\scriptstyle (1{-}a{+}b)^n}{\scriptstyle  1{+}a{-}b}\ \scriptstyle\text{for all}\ n\in\Z_+\end{array}$}{$(a,b)\mapsto (a,-b)$}{\rm N}
\taAbb{\rm N20}{$xyz$}{$\begin{array}{l}xy;\ \ yz;\\ x^2-xz-azx-yx-cy^2-dzy+\frac{(a{+}1)^2}{4}z^2{\vrule height0pt depth8pt width0pt}\end{array}$}{$a{\neq}{-}1,\ na{\neq}n{+}2\ \text{for all}\ n\in\N$}{\rm trivial}{\rm N}
\taAbb{\rm N21}{$xyz$}{$\begin{array}{l}xy;\ \ yz;\\ x^2-xz-azx-cy^2-zy+\frac{(a{+}1)^2}{4}z^2{\vrule height0pt depth8pt width0pt}\end{array}$}{$a{\neq}{-}1,\ na{\neq}n{+}2\ \text{for all}\ n\in\N$}{\rm trivial}{\rm N}
\taAbb{\rm N22}{$xyz$}{$\begin{array}{l}xy;\ \ yz;\\ x^2-xz-azx-y^2+\frac{(a{+}1)^2}{4}z^2{\vrule height0pt depth8pt width0pt}\end{array}$}{$a{\neq}{-}1,\ na{\neq}n{+}2\ \text{for all}\ n\in\N$}{\rm trivial}{\rm N}
\taAbb{\rm N23}{$xyz$}{$\begin{array}{l}xy;\ \ yz;\\ x^2-xz-azx+\frac{(a{+}1)^2}{4}z^2{\vrule height0pt depth8pt width0pt}\end{array}$}{$a{\neq}{-}1,\ na{\neq}n{+}2\ \text{for all}\ n\in\N$}{\rm trivial}{\rm N}

\bigskip

\noindent {\bf IX.} \ An algebra $A$ belongs to $\Lambda$ and is potential if and only if $A$ is isomorphic to an algebra from the following table.

\bigskip
\hrule
\tabbb{\hhh}{\rm Potential $Q_A$}{\rm Defining relations of $A_Q$}{\rm Exceptions}{\rm Isomorphisms}
\tabbb{\rm F1\hhh}{$x^4+a{x^2y^2}^\rcirclearrowleft+bxyxy^\rcirclearrowleft+y^4$}{$\begin{array}l x^3+axy^2+ay^2x+2byxy;\\ ax^2y+ayx^2+2bxyx+y^3\end{array}$}{$\begin{array}l \hhh 4(a+b)^2\neq1\\ (a,b)\neq(0,0)\\ (a,b)\neq\pm(1,1/2)\end{array}$}{$\begin{array}l (a,b)\mapsto (-a,-b)\\ (a,b)\mapsto \bigl(\frac{1-2b}{1+2a+2b},\frac{1-2a+2b}{2(1+2a+2b)}\bigr)\end{array}$}
\tabbb{\rm F2\hhh}{${x^2y^2}^\rcirclearrowleft+\frac{a}{2}xyxy^\rcirclearrowleft$}{$\begin{array}l \hhh xy^2+y^2x+ayxy;\\ x^2y+yx^2+axyx\end{array}$}
{\rm none}{\rm trivial}
\tabbb{\rm F3\hhh}{$x^4+{x^2y^2}^\rcirclearrowleft+\frac{a}{2}xyxy^\rcirclearrowleft$}{$\begin{array}l \hhh x^3+xy^2+y^2x+ayxy;\\ x^2y+yx^2+axyx\end{array}$}{\rm none}{\rm trivial}
\tabbb{\rm F4\hhh}{$x^3y^\rcirclearrowleft+{x^2y^2}^\rcirclearrowleft-xyxy^\rcirclearrowleft$}{$\begin{array}l \hhh x^2y^\rcirclearrowleft+{xy^2}^\rcirclearrowleft-3yxy;\\ x^3+x^2y+yx^2-2xyx\end{array}$}{\rm none}{\rm trivial}

\vfill\break

\noindent{\bf X.} \ An algebra $A$ belongs to $\Lambda$ and is twisted potential and non-potential if and only if $A$ is isomorphic to an algebra from the following table.

\bigskip
\hrule
\tabbb{\hhh}{\rm Twisted potential $Q_A$}{\rm Defining relations of $A_Q$}{\rm Exceptions}{\rm Isomorphisms}
\tabbb{\rm G1\hhh}{$\begin{array}l \hhh x^2y^2+a^2y^2x^2+axy^2x\\+ayx^2y+bxyxy+abyxyx\end{array}$}{$\begin{array}l \hhh a^2yx^2+ax^2y+ab xyx;\\ xy^2+ay^2x+byxy\end{array}$}{$\begin{array}{l}a\neq 0\\ a\neq 1\end{array}$}{\rm trivial}
\tabbb{\rm G2\hhh}{$\begin{array}l \hhh x^2y^2+y^2x^2-xy^2x-yx^2y+(a-1)x^2yx\\ +(1-a)xyx^2+ayx^3-ax^3y+\frac{a}2x^4\end{array}$}{$\begin{array}l \hhh xy^2{-}y^2x{+}(a{-}1)xyx{+}(1{-}a)yx^2{-}ax^2y{+}\frac{a}2x^3;\\ yx^2-x^2y+ax^3\end{array}$}{\rm none}{\rm trivial}
\tabbb{\rm G3\hhh}{$\begin{array}{l}{x^2y^2}^\rcirclearrowleft-xyxy^\rcirclearrowleft+ayx^3\\ +ax^3y+(a-1)xyx^2+(a+1)x^2yx\end{array}$}{$\!\!\begin{array}l xy^2{+}y^2x{-}2yxy{+}ax^2y{+}(a{-}1)yx^2{+}(a{+}1)xyx;\\ ax^3+x^2y+yx^2-2xyx\end{array}$}{\rm none}{\rm trivial}
\tabbb{\rm G4\hhh}{${x^2y^2}^\rcirclearrowleft-xyxy^\rcirclearrowleft-xyx^2+x^2yx+ax^4$}{$\begin{array}l \hhh x^2y+yx^2-2xyx;\\ xy^2+y^2x-2yxy-yx^2+xyx+ax^3\end{array}$}{\rm none}{\rm trivial}
\tabbb{\rm G5\hhh}{$x^2y^2+a^2y^2x^2+axy^2x-ayx^2y$}{$a^2yx^2-ax^2y;\ \ xy^2+ay^2x$}{\rm $a\neq 0$}{\rm trivial}
\tabbb{\rm G6\hhh}{$x^3y+yx^3+\theta xyx^2+\theta^2x^2yx+y^4$}{$x^2y+\theta yx^2+\theta^2xyx;\ \ x^3+y^3$}{\rm none}{\rm trivial}
\tabbb{\rm G7\hhh}{$x^3y+yx^3+\theta^2xyx^2+\theta x^2yx+y^4$}{$x^2y+\theta^2yx^2+\theta xyx;\ \ x^3+y^3$}{\rm none}{\rm trivial}
\tabbb{\rm G8\hhh}{$\begin{array}l \hhh x^4-iyx^3-y^2x^2+iy^3x\\+y^4+xy^3+x^2y^2+x^3y\end{array}$}{$x^3+x^2y+xy^2+y^3;\ \ x^3-iyx^2-y^2x+iy^3$}{\rm none}{\rm trivial}
\tabbb{\rm G9\hhh}{$\begin{array}l \hhh x^4+iyx^3-y^2x^2-iy^3x\\+y^4+xy^3+x^2y^2+x^3y\end{array}$}{$x^3+x^2y+xy^2+y^3;\ \ x^3+iyx^2-y^2x-iy^3$}{\rm none}{\rm trivial}
\tabbb{\rm \!G10\hhh}{$\begin{array}{l}\hhh x^2y^2-yx^2y+y^2x^2-xy^2x\\ +y^3x-xy^3+yxy^2-y^2xy\end{array}$}{$x^2y-yx^2-y^2x-xy^2+yxy;\ \ xy^2-y^2x-y^3$}{\rm none}{\rm trivial}

\bigskip

\noindent {\bf XI.} \ An algebra $A$ belongs to $\Lambda$ and its quasipotential $Q=Q_A$ is the fourth power of a degree $1$ element if and only if $A$ is isomorphic to an algebra from the following table.

\bigskip
\hrule
\tabbbb{\hhh}{\rm Quasipotential $Q_A$}{\rm Defining relations of $A_Q$}{\rm Exceptions}{\rm Isomorphisms}
\tabbbb{\rm Z1\hhh}{$y^4$}{$y^3$;\ \ $x^3-xy^2-ayxy-y^2x$}{$a\neq 0$}{\rm trivial}
\tabbbb{\rm Z2\hhh}{$y^4$}{$y^3$;\ \ $x^2y+xyx+yx^2-yxy$}{\rm none}{\rm trivial}
\tabbbb{\rm Z3\hhh}{$y^4$}{$y^3$;\ \ $x^2y-xyx+yx^2-yxy$}{\rm none}{\rm trivial}
\tabbbb{\rm Z4\hhh}{$y^4$}{$y^3$;\ \ $x^2y-yxy-ay^2x$}{$a\neq 0$}{\rm trivial}
\tabbbb{\rm Z5\hhh}{$y^4$}{$y^3$;\ \ $yx^2-yxy-axy^2$}{$a\neq 0$}{\rm trivial}
\tabbbb{\rm Z6\hhh}{$y^4$}{$y^3$;\ \ $x^2y-ayx^2-yxy-by^2x$}{$\begin{array}{l}\hhh a\neq 0,\ a+a^2+{\dots}+a^k+b\neq 0\\ \text{for all $k\in\N$}\end{array}$}{\rm trivial}
\tabbbb{\rm Z7\hhh}{$y^4$}{$y^3$;\ \ $x^2y-ayx^2-y^2x$}{\rm none}{\rm trivial}
\tabbbb{\rm Z8\hhh}{$y^4$}{$y^3$;\ \ $yx^2-xy^2$}{\rm none}{\rm trivial}
\tabbbb{\rm Z9\hhh}{$y^4$}{$y^3$;\ \ $x^2y-axyx+a^2yx^2$}{$a\neq 0$}{\rm trivial}
\tabbbb{\rm \!Z10\hhh}{$y^4$}{$y^3$;\ \ $x^2y-xyx+yx^2-xy^2-ayxy+y^2x$}{\rm none}{\rm trivial}

\bigskip

\noindent{\bf XII.} \ An algebra $A$ belongs to $\Lambda$ and its quasipotential $Q=Q_A$ satisfies $n_1(Q)=n_2(Q)=1$ with $Q$ {\bf NOT} being a fourth power of a degree one element if and only if $A$ is isomorphic to an algebra from the following table.

\bigskip
\hrule
\tabbbb{\hhh}{\rm Quasipotential $Q_A$}{\rm Defining relations of $A_Q$}{\rm Exceptions}{\rm Isomorphisms}
\tabbbb{\rm Y1\hhh}{$(x-by)(xy-ayx)(x-y)$}{$\begin{array}l \hhh x^2y-axyx+byxy-aby^2x;\\ xyx+xy^2-ayx^2-ayxy\end{array}$}{$a\neq 0,\ a\neq 1,\ a^kb\neq 1\ \text{for all $k\in\Z_+$}$}{$(a,b)\mapsto(\frac1a,\frac1b)$}
\tabbbb{\rm Y2\hhh}{$(x-y)(xy-yx-ayy)x$}{$\begin{array}l \hhh x^2y-xyx-axy^2-yxy+y^2x+ay^3;\\ xyx+xy^2-yx^2\end{array}$}{$a\neq 0,\ na+1\neq0\ \text{for all $n\in\N$}$}{\rm trivial}
\tabbbb{\rm Y3\hhh}{$x(xy-yx)y$}{$x^2y-xyx;\ \ xy^2-yxy$}{\rm none}{\rm trivial}
\tabbbb{\rm Y4\hhh}{$(x-ay)xy(x-y)$}{$x^2y-ayxy;\ \ xyx-xy^2$}{$a\neq 1$}{\rm trivial}
\tabbbb{\rm Y5\hhh}{$(x-y)(xy-ayx)y$}{$x^2y-axyx+yxy-ay^2x;\ \ xy^2-ayxy$}{$a\neq 1$}{\rm trivial}
\tabbbb{\rm Y6\hhh}{$x(xy-ayx)y$}{$x^2y-axyx;\ \ xy^2-ayxy$}{$a\neq 1$}{\rm trivial}
\tabbbb{\rm Y7\hhh}{$x(xy-yx-yy)y$}{$x^2y-xyx-xy^2;\ \ xy^2-yxy-y^3$}{\rm none}{\rm trivial}
\tabbbb{\rm Y8\hhh}{$y(xy-yx-yy)x$}{$yxy-y^2x-y^3;\ \ xyx-yx^2-y^3$}{\rm none}{\rm trivial}
\end{theorem}

\vfill\break

Note, that graded 3-Calabi-Yau algebras are known to be potential \cite{Ginz}. It is also true that the potential complex for a 3-Calabi-Yau algebra must be exact. In case of quadratic algebras on three generators and cubic algebras on two generators, this yields  that the Hilbert series of the Hilbert series of the algebra must be $(1-t)^{-3}$ or $(1+t)(1-t)^{-3}$ respectively. It is easy to check that algebras $P1-P8$ and $F1-F4$ of Theorem~\ref{main} are  3-Calabi-Yau. Since they are the only potential in the above tables, they form of a complete list of quadratic 3-Calabi-Yau algebras on three generators  $P1-P8$ and cubic  3-Calabi-Yau algebras on two generators  $F1-F4$. Thus this part of classification coincide with lists given by \cite{Smi} and \cite{Smi2} respectively.

Let us note  that the geometric meaning of our classification is that we classify orbits of the natural action of $GL_3(\K)$ on the Grassmann variety $Gr(3,9)$.
Fix a basis $x,y,z$ in a $3$-dimensional vector space $V$ over $\K$. Quadratic algebras $A(V,R)$ with $R$ being 3-dimensional (=given by 3 linearly independent quadratic relations) can be interpreted as points in the $18$-dimensional Grassmanian manifold $G$ of 3-dimensional subspaces of the 9-dimensional space $V^2$. This turns
$$
\Omega''=\{A:A\ \text{is a quadratic algebra satisfying $\dim A_1=3$, $\dim A_2=6$ and $\dim A_3\geq 10$}\}
$$
into an algebraic subvariety of $G$. Note that $\Omega''$ is in a sense almost the same as $\Omega'$: $\Omega'$ is Zarisski open in $\Omega''$. The natural action of $GL_3(\K)$ cuts $\Omega''$ into orbits with algebras from $\Omega''$ being isomorphic precisely when they are in the same orbit. What we do is the following: we determine which orbits correspond to algebras from $\Omega$ and pick a single element (a canonical form) in each orbit corresponding to an algebra from $\Omega$. Note that $\Omega''\setminus\Omega$, although not Zarisski closed, is nearly like that: it is the union of countably many subvarieties of $\Omega''$. Similar interpretation is available for the  part of the above theorem dealing with cubic algebras.

\begin{lemma}\label{cube} Let $A=A(V,R)\in\Omega'$, $Q$ be the corresponding quasipotential and $u\in V\setminus\{0\}$. Then $Q$ is a scalar multiple of $u^3$ if and only if $u^2\in R$.
\end{lemma}

\begin{proof} If $Q$ is a scalar multiple of $u^3$, then $u^2\in R_Q\subseteq R$. Now assume that $u^2\in R$. Then $u^3\in RV\cap VR$. Since $RV\cap VR$ is one-dimensional and $Q\in VR\cap RV$, $Q$ is a scalar multiple of $u^3$.
\end{proof}

Here we also give a related statement about the dual algebras.

\begin{lemma}\label{LL2} Let $A=A(V,R)$ be a quadratic algebra and $u,v$ be non-zero elements of $V$. Assume that $uv,vu\in R$. Then $A^!$ is infinite dimensional.
\end{lemma}

\begin{proof} Choose a basis $x_1,\dots,x_n$ in $V$ such that $u=x_1$ and $v=ax_j$ with $a\in\K^*$ and $j\in\{1,2\}$. Since $x_1x_j,x_jx_1\in R$, $x_1x_j$ and $x_jx_1$ do not feature at all in any of the defining relations of $A^!$ written in generators $x_k$. It follows that $(x_1x_j)^n$ for $n\in\N$ are linearly independent in $A^!$ and therefore $A^!$ is infinite dimensional.
\end{proof}

The above two lemmas explain why algebras with the labels containing the letter R in Theorem~\ref{main} can not be Koszul. They all fall into $\Omega^-$. Thus Theorem~\ref{main} yields the following funny corollary.

\begin{corollary}\label{koome} If a quadratic algebra $A=A(V,R)$ over a field whose characteristic is different from $2$ or $3$ satisfies $H_A=(1-t)^{-3}$, then
$$
\text{$A$ is Koszul}\iff H_{A^!}=(1+t)^3\iff \text{$u^2\notin R$ for every non-zero $u\in V.$}
$$
\end{corollary}

\begin{remark}\label{asr} The quadratic algebras among the Artin--Schelter regular algebras of global dimension $3$  are precisely the twisted potential algebras (including the potential ones) in Theorem~\ref{main}.
The classiffication of Artin--Schelter does not provide a canonical form up to an isomorphism. As it is observed in \cite{AS}, different sets of parameters in their description may lead to isomorphic algebras and when this actually happens was left a mystery.
\end{remark}

\begin{remark}\label{groups} The groups featuring in the isomorphism column of the tables in Theorem~\ref{main} are all finite, most being trivial. The largest order 24 occurs for Sklyanin algebras (P1).
\end{remark}

Throughout the paper we perform linear substitutions. When describing a substitution, we keep the same letters for both old and new variables. We introduce a substitution by showing by which linear combination of (new) variables must the (old) variables be replaced. For example, if we write $x\to x+y+z$, $y\to z-y$ and $z\to 7z$, this means that all occurrences of $x$ (in the relations, potential etc.) are replaced by $x+y+z$, all occurrences of $y$ are replaced by $z-y$, while $z$ is swapped for $7z$. A {\bf scaling} is a linear substitution with a diagonal matrix. That is it swaps each variable with it own scalar multiple. For example, the substitution $x\to 2x$, $y\to -3y$ and $z\to iz$ is a scaling.

\section{Auxiliary results}

In this section, we prepare some tools needed for proving the main result. The following lemma is very useful in dealing with algebras from $\Omega^0$.

\begin{lemma}\label{ome0} Let $A=A(V,R)$ be a quadratic algebra. Then the following statements are equivalent$:$
\begin{itemize}\itemsep=-2pt
\item[\rm (1)]$A\in\Omega^0;$
\item[\rm (2)]$A\in \Omega'$ and $A$ is $\rm PBW_B;$
\item[\rm (3)]$\dim V=\dim R=3$, $\dim A_3\geq 10$ and there is a basis $x,y,z$ in $V$ and a well-ordering on $x,y,z$ monomials compatible with multiplication, with respect to which the set of leading monomials of elements of a basis in $R$ is one of $\{xy,xz,yz\}$, $\{xy,xz,zy\}$, $\{xy,zx,zy\}$, $\{yx,yz,xz\}$, $\{yx,yz,zx\}$ or $\{yx,zy,zx\}$.
\item[\rm (4)]$\dim V=\dim R=3$, $\dim A_3\geq 10$ and there is a basis $x,y,z$ in $V$ and a well-ordering on $x,y,z$ monomials compatible with multiplication, with respect to which the set of leading monomials of elements of a basis in $R$ is $\{xy,xz,yz\}$.
\end{itemize}
\end{lemma}

\begin{proof} The implication (1)$\Longrightarrow$(2) is obvious. Next, assume $A$ is $PBW_B$ and $A\in\Omega'$. Then $\dim V=\dim R=3$ and $\dim A_3=10$. Let $x,y,z$ be a PBW-basis for $A$, while $f,g,h$ be corresponding PBW-generators. Since $f$, $g$ and $h$ form a Gr\"obner basis of the ideal of relations of $A$, it is easy to see that $\dim A_3$ is $9$ plus the number of overlaps of the leading monomials $\overline{f}$, $\overline{g}$ and $\overline{h}$ of $f$, $g$ and $h$. Since $\dim A_3=10$, the monomials $\overline{f}$, $\overline{g}$ and $\overline{h}$ must produce exactly one overlap. Now it is a straightforward routine check that if at least one of three degree 2 monomials in 3 variables is a square, these monomials overlap at least twice. The same happens, if the three monomials contain $uv$ and $vu$ for some distinct $u,v\in\{x,y,z\}$. Finally, the triples $\{xy,yz,zx\}$ and $\{yx,xz,zy\}$ produce 3 overlaps apiece. The only option left, is for $(\overline{f},\overline{g},\overline{h})$ to be one of the triples listed in (3). This completes the proof of implication (2)$\Longrightarrow$(3). The implication (3)$\Longrightarrow$(4) follows from the observation that the group $S_3$ of permutations of the 3-element set $\{x,y,z\}$ acts transitively on the $6$-element set of triples from (3). Finally, assume that (4) is satisfied. Then the leading monomials of defining relations have exactly one overlap. If this overlap produces a non-trivial degree $3$ element of the Gr\"obner basis of the ideal of relations of $A$, then $\dim A_3=9$, which contradicts the assumptions. Hence, the overlap resolves. That is, a linear basis in $R$ is actually a Gr\"obner basis of the ideal of relations of $A$. Then $A$ is $\rm PBW_B$. Furthermore, the leading monomials of the defining relations are the same as for $\K[x,y,z]$ for the left-to-right or right-to-left degree lexicographical ordering with $x>y>z$. Hence $A$ and $\K[x,y,z]$ have the same Hilbert series: $H_A=(1-t)^{-3}$. That is, $A\in\Omega^0$. This completes the proof of implication (4)$\Longrightarrow$(1).
\end{proof}

\subsection{One and two-dimensional subspaces of $V^2$ with $\dim V=2$}

The following observations are very well-known. We sketch the proofs for the sake of convenience.

\begin{lemma}\label{1-dim} Let $\K$ be an arbitrary algebraically closed field $($characteristics $2$ and $3$ are allowed here$)$, $V$ be a $2$-dimensional vector space over $\K$ and $S$ be a $1$-dimensional subspace of $V^2=V\otimes V$. Then $S$ satisfies exactly one of the following conditions$:$

there is a basis $x,y$ in $V$ such that
\begin{itemize}\itemsep=-2pt
\item[\rm (I.1)] $S=\spann\{yy\};$
\item[\rm (I.2)] $S=\spann\{yx\};$
\item[\rm (I.3)] $S=\spann\{xy-\alpha yx\}$ with $\alpha\in\K^*;$
\item[\rm (I.4)] $S=\spann\{xy-yx-yy\}$.
\end{itemize}
Furthermore, if $S=\spann\{xy-\alpha yx\}=\spann\{x'y'-\beta y'x'\}$ with $\alpha\beta\neq 0$ for two different bases $x,y$ and $x',y'$ in $V$, then either $\alpha=\beta$ or $\alpha\beta=1$.
\end{lemma}

\begin{proof} If $V$ is spanned by a rank one element, then $S=\spann\{uv\}$, where $u,v$ are non-zero elements of $V$ uniquely determined by $S$. If $u$ and $v$ are linearly independent, we set $y=u$ and $x=v$ to see that (I.2) is satisfied. If $u$ and $v$ are linearly dependent, we set $y=u$ and pick an arbitrary $x\in V$ such that $y$ and $x$ are linearly independent. In this case (I.1) is satisfied. Obviously, (I.1) and (I.2) can not happen simultaneously. Since $S$ in (I.3) and (I.4) are spanned by rank 2 elements, neither of them can happen together with either (I.1) or (I.2).

Now let $u,v$ be an arbitrary basis in $V$ and $S$ be spanned by a rank $2$ element $f=auu+buv+cvu+dvv$ with $a,b,c,d\in\K$. A linear substitution, which keeps $u$ intact and replaces $v$ by $v+su$ with an appropriate $s\in\K$ turns $a$ into $0$ (one must use the fact that $f$ has rank $2$ and that $\K$ is algebraically closed: $s$ is a solution of a quadratic equation). Thus we can assume that $a=0$. Since $f$ has rank $2$, it follows that $bc\neq 0$. If $b+c\neq 0$, we set $x=u+\frac{dv}{b+c}$ and $y=bv$ to see that (I.3) is satisfied with $\alpha=\frac{c}{b}\neq 1$. Note also that the only linear substitutions which send $xy-\alpha yx$  to $xy-\beta yx$ (up to a scalar multiple) with $\alpha\beta\in\K^*$, $\alpha\neq 1$ are the scalings of the variables and the scalings composed with swapping $x$ and $y$. In the first case $\alpha=\beta$. In the second case $\alpha\beta=1$. Finally, if $b+c=0$, then we have two options. If, additionally, $d=0$, $S$ is spanned by $xy-yx$ with $x=u$ and $y=v$, which falls into (I.3) with $\alpha=1$.  Note that any linear substitution keeps the shape of $xy-yx$ up to a scalar multiple. If $d\neq 0$, we set $x=u$ and $y=\frac{dv}{b}$ to see that $S$ is spanned by $xy-yx-yy$ yielding (I.4). The remarks on linear substitutions, we have thrown on the way complete the proof.
\end{proof}

\begin{lemma}\label{2-dim} Let $\K$ be an arbitrary algebraically closed field $($characteristics $2$ and $3$ are allowed here$)$, $V$ be a $2$-dimensional vector space over $\K$ and $S$ be a $2$-dimensional subspace of $V^2=V\otimes V$. Then Then $S$ satisfies exactly one of the following conditions$:$

there is a basis $x,y$ in $V$ such that
\begin{itemize}\itemsep=-2pt
\item[\rm (II.1)]$S=\spann\{xx,yy\};$
\item[\rm (II.2)]$S=\spann\{xx-yx,yy\};$
\item[\rm (II.3)]$S=\spann\{xy,yy\};$
\item[\rm (II.4)]$S=\spann\{yx,yy\};$
\item[\rm (II.5)]$S=\spann\{xy-\alpha yx,yy\}$ with $\alpha\in\K^*;$
\item[\rm (II.6)]$S=\spann\{xy,yx\};$
\item[\rm (II.7)]$S=\spann\{xx-xy,yx\};$
\item[\rm (II.8)]$S=\spann\{xx-axy-yy,yx\}$ with $a\in\K$, $a^2+1\neq 0.$
\end{itemize}
Furthermore, $\alpha$ in {\rm (II.5)} is uniquely determined by $S$. Finally, if $S=\spann\{xx-axy-yy,yx\}=\spann\{x'x'-bx'y'-y'y',y'x'\}$ with $(a^2+1)(b^2+1)\neq 0$ for two different bases $x,y$ and $x',y'$ in $V$, then either $a=b$ or $a+b=0$.
\end{lemma}

\begin{proof} {\bf Case 1:} $S$ contains $yy$ for some non-zero $y\in V$.

Picking a basis $y,w$ in $V$, we see that $S$ is spanned by $\{yy,aww+bwy+cyw\}$ for some non-zero $(a,b,c)\in\K^3$. First, assume that $a\neq 0$. Replacing $w$ by $\alpha x+\beta y$ with appropriately chosen $\alpha\in\K^*$ and $\beta \in\K$, we can turn $S$ into the span of either $\{xx,yy\}$ or $\{xx-yx,yy\}$. If $a=0$, we just set $x=w$ and see that $S$ is the span of either $\{xy,yy\}$, or $\{yx,yy\}$, or $\{xy-\alpha yx,yy\}$ with $\alpha\in\K^*$. Thus we fall into one of (II.1--II.5). The spaces $S$ in (II.6--II.9) are easily seen to contain no element of the shape $uu$ with $u\in V\setminus\{0\}$ (in a manner of speaking, $S$ is square-free). Thus each of (II.1--II.5) is incompatible with each of (II.6--II.9). Next, (II.1) is singled out by $S$ having two linearly independent squares. As for $S$ from (II.2--II.5), it contains just one square: $yy$. Thus any linear substitution transforming one such $S$ into another must send $y$ to its own scalar multiple. Without loss of generality it sends $y$ to itself, while sending $x$ to $sx+ty$ with $s\in\K^*$, $t\in\K$. An easy cases by case consideration shows that each substitution like this preserves each $S$ from (II.5), only the substitutions with $t=0$ are eligible for $S$ from (II.3) and (II.4), preserving it, while only the identity substitution is eligible for $S$ from (II.2). Thus different conditions from the list (II.1--II.5) can not happen simultaneously and $\alpha$ in {\rm (II.5)} is uniquely determined by $S$.

{\bf Case 2:} $S$ contains no square of a non-zero element of $V$.

It is easy to see that for every finite dimensional space $V$ and for every $2$-dimensional subspace $S_0$ of $V\otimes V$, there is a non-zero element in $S_0$ of rank strictly less than the dimension of $V$ (this follows easily from the fact that the spectrum of a matrix over an algebraically closed field is always non-empty). Hence there is an element of rank 1 in $S$. That is, there are non-zero $u,v\in V$ such that $uv\in S$. By the assumption of Case~2, $u$ and $v$ are linearly independent and therefore form a basis of $V$. Since $S$ is $2$-dimensional, it is spanned by $uv$ and $\alpha uu+\beta vu+\gamma vv$ with non-zero $(\alpha,\beta,\gamma)\in\K^3$. First, consider the case $\gamma=0$. If $\alpha=0$, then we set $x=u$, $y=v$ and observe that $S$ is spanned by $\{xy,yx\}$. If $\beta=0$, $S$ contains $uu$, contradicting the assumption of Case~2. If $\alpha\beta\neq 0$, we set $y=u$ and $x=-\frac{\beta}{\alpha}v$, we see that $S=\spann\{xx-xy,yx\}$. It remains to consider the case $\gamma\neq 0$. Without loss of generality, $\gamma=1$. If $\alpha=\beta=0$, then $S$ contains $vv$, contradicting the assumption of Case~2. If $\alpha=0$, $\beta\neq 0$, we set $y=-\beta u$ and $x=v$ to see that $S=\spann\{xx-xy,yx\}$. Finally, assume $\alpha\neq 0$. Since $\K$ is algebraically closed, there is $s\in\K$ such that $\alpha=-s^2$. Setting $x=v$ and $y=su$, we see that $S=\spann\{xx-axy-yy,yx\}$ for $a=\frac{\beta}{s}\in\K$. Finally, if $a^2+1=0$, then $xx-axy-ayx-yy=(x-ay)(x-ay)\in S$ and the assumption of Case~2 is violated. Thus $a^2+1\neq 0$. Hence at least one of (II.6--II.8) is satisfied. Incompatibility of different conditions from (II.6--II.8) and the fact that $a$ in (II.8) is uniquely determined by $S$ up to the sign is another easy case by case consideration.
\end{proof}

\subsection{A remark on Koszulity}

Let $A=A(V,R)$ be a quadratic algebra. Fix a basis $x_1,\dots,x_n$ in $V$. Recall that there is a specific complex of free right $A$-modules, called the Koszul complex, whose exactness is equivalent to the Koszulity of $A$:
\begin{equation}\label{koco1}
\dots\mathop{\longrightarrow}^{d_{k+1}} (A^!_k)^*\otimes A\mathop{\longrightarrow}^{d_k} (A^!_{k-1})^*\otimes A
 \mathop{\longrightarrow}^{d_{k-1}}\dots \mathop{\longrightarrow}^{d_1} (A^!_{0})^*\otimes A=A\longrightarrow \K\to 0,
\end{equation}
where the tensor products are over $\K$, the second last arrow is the augmentation map, each tensor product carries the natural structure of a free right $A$-module and $d_k$ are given by $d_k(\phi \otimes u)=\sum\limits_{j=1}^n \phi_j\otimes x_ju$, where $\phi_j\in (A^!_{k-1})^*$, $\phi_j(v)=\phi(x_jv)$. Although $A^!$ and the Koszul complex seem to depend on the choice of a basis in $V$, it is not really the case up to the natural equivalence \cite{popo}. If we know a Gr\"obner basis in the ideal of relations of $A^!$, we know the multiplication table (=structural constants) of $A^!$ with respect to the basis of normal words, which allows us to explicitly write the matrices of the maps $d_k$.

Note that the Koszul complex is finite (=bounded) precisely when $A^!$ is finite dimensional. Some authors have remarked (see \cite{SKL} for detailed proof) that if $A^!$ is finite dimensional and the Koszul complex (\ref{koco1}) of $A$ is exact at all entries with one possible exception, then it is exact (equivalently, $A$ is Koszul) if and only if the duality formula $H_A(t)H_{A^!}(-t)=1$ is satisfied. It is also well-known that for every quadratic algebra $A$, the Koszul complex (\ref{koco1}) is automatically exact at its three rightmost entries. These observations lead to the following lemma, which is our main tool in proving Koszulity.

\begin{lemma}\label{koal3} Let $A=A(V,R)$ be a quadratic algebra such that $A^!_3\neq \{0\}$ and $A^!_4=\{0\}$. Then $A$ is Koszul if and only if $H_A(t)H_{A^!}(-t)=1$ and the map $d_3$ in the Koszul complex
$$
0\longrightarrow (A^!_3)^*\otimes A\mathop{\longrightarrow}^{d_3} (A^!_{2})^*\otimes A
 \mathop{\longrightarrow}^{d_{2}} (A^!_{1})^*\otimes A\mathop{\longrightarrow}^{d_1} (A^!_{0})^*\otimes A=A\longrightarrow \K\to 0
$$
is injective.
\end{lemma}

\begin{proof} If $A$ is Koszul, then $H_A(t)H_{A^!}(-t)=1$ and the above complex is exact. In particular, $d_3$ is injective. Now assume that $H_A(t)H_{A^!}(-t)=1$ and $d_3$ is injective. This injectivity is the same as the exactness of the above complex at its leftmost entry. As we have already mentioned the complex is exact at its three rightmost entries. Thus the exactness can potentially break at one entry $(A^!_{2})^*\otimes A$ only. As we have mentioned above, the equality $H_A(t)H_{A^!}(-t)=1$ now yields that $A$ is Koszul.
\end{proof}

\subsection{Remarks on generic algebras}

\begin{definition}\label{gene} If $W$ is an irreducible affine algebraic variety of positive dimension over an uncountable algebraically closed field $\K$, then we say that {\it generic} points of $W$ have a property $P$ if $P$ holds for all $x\in W$ outside the union of countably many proper subvarieties of $W$. Since the field is uncountable, such a union can never cover the whole of $W$. Obviously, if each of $P_1,\dots,P_n$ holds for generic $x\in W$, then $P_1\wedge {\dots}\wedge P_n$ holds for generic $x\in W$.
\end{definition}

We sketch the proof of the following elementary and known fact about the varieties of quadratic algebras for the sake of convenience.

\begin{definition}\label{vari}
Let $n,m,d\in\N$, $\K$ be any field, $V$ be an $n$-dimensional vector space over $\K$ and $W\subseteq \K^m$ be an affine algebraic variety. For $1\leq j\leq d$, let $q_j:\K^m\to V^{k_j}$ be a polynomial map, where $k_j\geq 2$. For $b\in W$, we set $A^b=F(V)/I_b$, where $I_b$ is the ideal generated by $q_1(b),\dots,q_d(b)$. The family $\{A^b:b\in W\}$ will be called {\it a variety of graded algebras}. If each $k_j$ equals $2$, every $A^b$ is quadratic and we say that  $\{A^b:b\in W\}$ is a {\it variety of quadratic algebras}.
\end{definition}

\begin{lemma}\label{minhs}
Let $\{A^b:b\in W\}$ be a variety of graded algebras. For $k\in\Z_+$, let $h_k=\min\{\dim A_k^b:b\in W\}$. Then the non-empty set $\{b\in W:\dim A_k^b=h_k\}$ is Zarissky open in $W$.
\end{lemma}

\begin{proof} We can assume that $k\geq 2$ (for $k\in\{0,1\}$, the set in question is the entire $W$). Pick $c\in W$ such that $\dim A_k^c=h_k$. Then $\dim I_k^c=n^k-h_k$. Note that since $I^b_k$ is the linear span of $u q_j(b) v$, where $1\leq j\leq d$, $u,v$ are monomials and $\deg(uv)+k_j=k$, $\dim I_k^b$ is exactly the rank of the rectangular $\K$-matrix $M(b)$ of the coefficients of all such $u q_j(b) v$. Let $M_1(b),\dots,M_N(b)$ be all $(n^k-h_k)\times(n^k-h_k)$ submatrices of $M(b)$. For each $j$, let $\delta_j(b)$ be the determinant of the matrix $M_j(b)$. Clearly, each $\delta_j$ is a (commutative) polynomial in the variables $b=(b_1,\dots,b_m)$. Obviously,
$$
G=\{b\in W:\dim A^b_k>h_k\}=\{b\in\K^m:\dim I_k^b<n^k-h_k\}=\{b\in\K^m:\delta_1(b)={\dots}=\delta_N(b)=0\}
$$
is Zarissky closed. Then $U=W\setminus G=\{b\in W:\dim A^b_k\leq h_k\}$ is Zarisski open. By definition of $h_k$, $U=\{b\in W:\dim A^b_k=h_k\}$ and the result follows.
\end{proof}

The following result of Drinfeld \cite{dr} features also as Theorem~2.1 in Chapter~6 in \cite{popo}. To explain it properly, we need to remind the characterization of Koszulity in terms of the distributivity of lattices of vector spaces. Let $A=A(V,R)$ be a quadratic algebra. For $n\geq 3$, let $L_n(V,R)$ be the finite lattice of subspaces of $V^n$ generated by the spaces $V^kRV^{n-2-k}$ for $0\leq k\leq n-2$ (as usual, the lattice operations are sum and intersection). Then $A$ is Koszul if and only if $L_n(V,R)$ is distributive for each $n\geq 3$ (see \cite[Chapter~3]{popo}). The mentioned result of Drinfeld is as follows.

\begin{lemma}\label{dri}
Let $\{A^b:b\in W\}$ be a variety of quadratic algebras and $U$ is a non-empty Zarissky open subset of $W$ such that $\dim A_2^b$ and $\dim A_3^b$ do not depend on $b$ for $b\in U$. Then for each $k\geq 3$, the set
$$
\{b\in U:L_j(V,R_b)\ \ \text{for $3\leq j\leq k$ are distributive}\}
$$
is Zarissky open in $W$.
\end{lemma}

The proof of the above lemma is a classical blend of the same argument as in the proof of Lemma~\ref{minhs} with an appropriate inductive procedure. We need the following corollary of Lemmas~\ref{minhs} and~\ref{dri}.

\begin{lemma}\label{hsdri} Let $\{A^b:b\in W\}$ be a variety of quadratic algebras. Assume also that $\K$ is uncountable and algebraically closed and $W$ is irreducible and has positive dimension. As above, let $h_k=\min\{\dim A_k^b:b\in W\}$ for $k\in\Z_+$. Then for generic $b\in W$, $H_{A^b}(t)=\sum\limits_{k=0}^\infty h_kt^k$. Furthermore, exactly one of the following statements holds true$:$
\begin{itemize}\itemsep=-2pt
\item[\rm (1)] $A^b$ is non-Koszul for every $b\in W$ satisfying $\dim A^b_3=h_3$ and $\dim A^b_2=h_2;$
\item[\rm (2)] $A^b$ is Koszul for generic $b\in W.$
\end{itemize}
\end{lemma}

\begin{proof} By Lemma~\ref{minhs}, $\{b\in W:\dim A_k^b\neq h_k\}$ is a proper subvariety of $W$ for each $k\in\N$. By definition, $H_{A^b}=\sum\limits_{n=0}^\infty h_nt^n$ for generic $b\in W$. Next, by Lemma~\ref{minhs},
$$
U=\{b\in W:\dim A^b_3=h_3\ \text{and}\ \dim A^b_2=h_2\}
$$
is a non-empty Zarissky open subset of $W$. If $A^b$ is non-Koszul for every $b\in U$, (1) is satisfied. Assume now that (1) fails. Then there is $c\in U$ for which $A^c$ is Koszul. By Lemma~\ref{dri}, $W_k=\{b\in U:L_j(V,R_b)\ \ \text{for $3\leq j\leq k$ are distributive}\}$ is Zarissky open in $W$. Since $A^c$ is Koszul, $c\in W_k$ for every $k\geq 3$. Since for $b$ from the intersection of $W_k$ with $k\geq 3$, $A^b$ is Koszul and each $W_k$ is Zarissky open and non-empty, (2) is satisfied. Obviously, (1) and (2) are incompatible.
\end{proof}

We adjust the above results to the situation we are interested in.

\begin{lemma}\label{dri3} Let $\{A^b:b\in W\}$ be a variety of quadratic algebras. Assume also that $\K$ is uncountable and algebraically closed, $W$ is irreducible and has positive dimension, there is a non-empty Zarisski open subset $U$ of $W$ such that for each $b\in U$, $H_{(A^b)^!}=(1+t)^3$ and there is at least one $c\in U$ for which $A^c$ is Koszul.
Then for every $n\in\Z_+$ and every $b\in W$, $\dim A_n^b\geq \frac{(n+1)(n+2)}{2}$. Moreover, for generic $b\in W$, $H_{A^b}=(1-t)^{-3}$ and $A^b$ is Koszul.
\end{lemma}

\begin{proof} As above, let $h_k=\min\{\dim A_k^b:b\in W\}$. Since $H_{(A^b)^!}=(1+t)^3$, we have that $\dim A^b_1=3$, $\dim A_b^2=6$ and $\dim A^b_3=10$ for each $b\in U$. Since $W$ is irreducible and $U\subseteq W$ is non-empty and Zarisski open, Lemma~\ref{minhs} yields $h_1=3$, $h_2=6$ and $h_3=10$. By Lemma~\ref{hsdri}, $H_{A^b}=\sum\limits_{n=0}^\infty h_nt^n$ and $A^b$ is Koszul for generic $b\in W$. On the other hand, the duality formula in (\ref{stm2}), implies that $H_{A_b}=(1-t)^{-3}$ whenever $b\in U$ and $A^b$ is Koszul. Thus $H_{A_b}=(1-t)^{-3}$ for generic $b\in W$. Hence $(1-t)^{-3}=\sum\limits_{n=0}^\infty h_nt^n$ and therefore $h_n=\frac{(n+1)(n+2)}{2}$ for all $n\in\Z_+$. By definition of $h_n$, $\dim A_n^b\geq \frac{(n+1)(n+2)}{2}$ for all $b\in W$.
\end{proof}

\subsection{Degenerate algebras}

Here we provide a sufficient condition for a quadratic algebra to fall outside $\Omega$. More specifically, we show that if an algebra is given by generators $x,y,z$ and three quadratic relations, of which at least two are devoid of $x$, then $A\notin\Omega$.

\begin{lemma}\label{2-2} Let $\K$ be an arbitrary field $($characteristics $2$ and $3$ are allowed here$)$ and $A=A(V,R)\in\Omega$. Then $\dim R\cap M^2\leq 1$ for every $2$-dimensional subspace $M$ of $V$.
\end{lemma}

\begin{proof} It is easy to see that if this lemma holds for a ground field $\K$ then it holds when $\K$ is replaced by any subfield. Thus, without loss of generality, we can assume that $\K$ is algebraically closed. Since $\dim A_1=3$ and $\dim A_2=6$, we have $\dim V=\dim R=3$.

Assume the contrary. Then there is a $2$-dimensional subspace $M$ of $V$ such that $\dim R\cap M^2\geq 2$. First, we get rid of the trivial case $R\subseteq M^2$. If $x,y,z$ is a basis in $V$ such that $y,z$ span $M$, then, provided $R\subseteq M^2$, $R^\perp$ is spanned by $xx$, $xy$, $xz$, $yx$, $zx$, $h$, where $h\in M^2$ is non-zero. By Lemma~\ref{2-dim}, making a $y,z$ linear substitution, we can turn $h$ into one of the following forms $yy$, $yz-zy-zz$ or $yz-azy$ with $a\in\K$. Whatever the case $xx$, $xy$, $xz$, $yx$, $zx$ and $h$ form a Gr\"obner basis of the ideal of relations of $A^!$ with respect to the left-to-right degree lexicographical ordering assuming $x>y>z$. Thus $A^!$ is Koszul and we can compute its Hilbert series: $H_{A^!}=1+3t+3t^2+4t^3+5t^4+{\dots}=\frac{1{+}t-2t^2{+}t^3}{(1-t)^2}$ if $h=yz-zy-zz$ or $h=yz-azy$ and $H_{A^!}=1+3t+3t^2+5t^3+8t^4+{\dots}=\frac{1{+}2t{-}t^2{-}t^3}{1-t-t^2}$ if $h=yy$. In both cases the duality formula in (\ref{stm2}) yields $\dim A_3>10$, which contradicts the assumptions. Thus for the rest of the proof we can assume that $S=R\cap M^2$ is two-dimensional.

By Lemma~\ref{2-dim}, there is a basis $y,z$ in $M$ such that $S$ is spanned by one of the following sets$:$ $\{yy,zz\}$, $\{yy-zy,zz\}$, $\{yz,zz\}$, $\{zy,zz\}$, $\{yz-\alpha zy,zz\}$ with $\alpha\in\K^*$, $\{yz,zy\}$, $\{yy-yz,zy\}$, $\{yy-zy,yz\}$ or $\{yy-\alpha yz-zz,zy\}$ with $\alpha\in\K$, $\alpha^2+1\neq 0$. Pick $x\in V\setminus M$. Then $x,y,z$ is a basis in $V$. For the rest of this proof we use the following ordering on monomials in $x,y,z$. First, a monomial of bigger degree is considered to be bigger. For two monomials of equal degree, the one with bigger $x$-degree ($x$ occurs more times) is bigger. Finally we use the left-to-right lexicographical ordering with $x>y>z$ to break the ties. One easily sees that this order is compatible with multiplication and therefore allows to use the Gr\"obner basis technique. Since $R$ is $3$-dimensional, $R$ is spanned by $S$ together with one element $f\in R\setminus S$.

{\bf Case 1:} \ $S$ is the span of one of $\{yy,zz\}$, $\{yz,zz\}$, $\{zy,zz\}$, $\{yz,zy\}$, or $\{yz-\alpha zy,zz\}$ with $\alpha\in\K^*$.

In this case, regardless what shape $f$ has, at least two overlaps of leading monomials of the defining relations resolve without producing a degree 3 element of the Gr\"obner basis of the ideal of relations of $A$. This yields $\dim A_3>10$, contradicting the assumption $A\in\Omega$.

{\bf Case 2:} \ $S=\spann\{yy-yz,zy\}$ or $S=\spann\{yy-zy,yz\}$.

These two options reduce to each other by passing to the opposite multiplication. Thus we can assume that $S=\spann\{yy-yz,zy\}$. Then $R$ is spanned by $yy-yz$, $zy$ and $f=a_1xx+a_2xy+a_3xz+a_4yx+a_5zx+a_6yz+a_7zz$ with $a=(a_1,\dots,a_7)\in\K^7$. The overlaps $yyy$ and $zyy$ produce just one degree $3$ element of the Gr\"obner basis of the ideal of relations of $A$ being $yzz$.

If $a_1\neq 0$, the leading monomial of $f$ is $xx$. Any substitution of the form $x\to x+\alpha y+\beta z$, $y\to y$, $z\to z$ with $\alpha,\beta\in\K$ leaves the relations $yy-yz$ and $zy$ intact and preserves the shape of $f$. Furthermore, $\alpha$ and $\beta$ can be chosen to turn $a_2$ and $a_3$ into $0$. After such a substitution the defining relations of $A$ turn into $xx-byx-czx-dyz-pzz$, $yy-yz$ and $zy$ with $b,c,d,p\in\K$. Provided $b\neq 0$, the only overlap (other than $yyy$ and $zyy$) $xxx$ produces the second degree $3$ element $g$ of the Gr\"obner basis of the ideal of relations of $A$:
$$
g=bxyx+cxzx+dxyz+pxzz-(d+b^2+bc)yzx-(p+c^2)zzx-cpzzz
$$
with the leading monomial being $xyx$. Degree $4$ overlaps produce just one degree $4$ element of the Gr\"obner basis (comes from the overlap $xxyx$):
$$
h=(b^2+bc+d)xyzx+(p+c^2)xzzx+cpxzzz-(c^3+2cp)zzzx-(c^2p+p^2)zzzz.
$$
In the case $b\neq 0$ and $h\neq 0$, we obtain $\dim A_4=16$. Thus $\dim A_4=16$ for Zarisski generic $(b,c,d,p)\in\K^4$. By Lemma~\ref{minhs}, $\dim A_4\geq 16$ for all $(b,c,d,p)\in\K^4$. Thus $\dim A_4\geq 16$ whenever $a_1\neq 0$. Hence $\dim A_4\geq 16$ for Zarisski generic $a\in\K^7$. By Lemma~\ref{minhs}, $\dim A_4\geq 16$ for all $a\in\K^7$. This contradicts the assumption $A\in\Omega$ and concludes Case~2.

{\bf Case 3:} \ $S=\spann\{yy-zy,zz\}$.

As above, $R$ is spanned by $yy-zy$, $zz$ and $f=a_1xx+a_2xy+a_3xz+a_4yx+a_5zx+a_6yz+a_7zy$ with $a=(a_1,\dots,a_7)\in\K$. The overlaps $yyy$ and $zzz$ produce just one degree $3$ element of the Gr\"obner basis of the ideal of relations of $A$ being $yzy$. We proceed as in the second case. If $a_1\neq 0$, the leading monomial of $f$ is $xx$. Making a substitution $x\to x+\alpha y+\beta z$, $y\to y$, $z\to z$ with $\alpha,\beta\in\K$ with appropriate $\alpha,\beta\in\K$, we turn the defining relations of $A$ turn into $xx-byx-czx-dyz-pzy$, $yy-zy$ and $zz$ with $b,c,d,p\in\K$. Provided $b\neq 0$, we can scale $x$ to turn $b$ into $-1$. Assuming $d-c\neq 0$ and $pc+c-c^2+cd-d\neq 0$, we compute the Gr\"obner basis of the ideal of relations of $A$, which turns out finite and consisting of $xx+yx-czx-dyz-pzy$, $yy-zy$, $zz$, $yzy$, $-xyx+cxzx+dxyz+pxzy+(c-d)yzx+(c-p-1)zyx+d(1-c)zyz$, $(d-c)xyzx+(p+1-c)xzyx+d(c-1)xzyz+((d-c)(1-c)-cp)zyzx$ and $xzyzx$. Now we easily get $\dim A_7=37$. Thus $\dim A_7\geq 37$ for Zarisski generic $(b,c,d,p)\in\K^4$. By Lemma~\ref{minhs}, $\dim A_7\geq 37$ for all $(b,c,d,p)\in\K^4$. Thus $\dim A_7\geq 37$ whenever $a_1\neq 0$. Hence $\dim A_7\geq 37$ for Zarisski generic $a\in\K^7$. By Lemma~\ref{minhs}, $\dim A_7\geq 37$ for all $a\in\K^7$. This contradicts the assumption $A\in\Omega$. Indeed, one has $\dim A_7=36$ for $A\in\Omega$. This concludes Case~3.

{\bf Case 4:} \ $S=\spann\{yy-\alpha yz-zz,zy\}$ with $\alpha^2+1\neq 0$.

Then $R$ is spanned by $yy-ayz-zz$, $zy$ and $f=a_1xx+a_2xy+a_3xz+a_4yx+a_5zx+a_6yz+a_7zz$ with $a=(a_1,\dots,a_7)\in\K$. Since $R$ is not contained in $M^2$, we have $(a_1,a_2,a_3,a_4,a_5)\neq (0,0,0,0,0)$. The overlaps $yyy$ and $zyy$ produce two linearly independent degree $3$ element of the Gr\"obner basis of the ideal of relations of $A$ being $yzz$ and $zzz$.

First, consider the case $a_1\neq 0$. Using the same substitution as in Case~2, we can turn the defining relations of $A$ into $xx-byx-czx-dyz-pzz$, $yy-\alpha yz-zz$ and $zy$ with $b,c,d,p\in\K$. The only overlap (other than $yyy$ and $zyy$) $xxx$ produces
$$
g=bxyx+cxzx+dxyz+pxzz-(d+\alpha b^2+bc)yzx-bczyx-(p+b^2+c^2)zzx.
$$
If $g\neq 0$, we have $\dim A_3=9$, which is impossible since we know that $\dim A_3=10$. Thus $g=0$, or equivalently, $b=c=d=p=0$. In the latter case, $xx$, $yy-ayz-zz$, $zy$, $yzz$ and $zzz$ form a Gr\"obner basis of the ideal of relations of $A$, yielding $\dim A_4=18>15$, which contradicts the assumptions. Thus $a_1=0$.

Now we consider the case $a_1=0$ and $a_2\neq 0$. In this case the defining relations of $A$ have the shape $xy-bxz-cyx-dzx-pyz-qzz$, $yy-\alpha yz-zz$ and $zy$ with $b,c,d,p,q\in\K$. The only overlap (other than $yyy$ and $zyy$) $xyy$ produces
$$
g=-(\alpha b+1)xzz+c(b-\alpha)yxz+c(d+\alpha c)yzx+d(b-\alpha)zxz+(d^2+c^2)zzx.
$$
If $g\neq 0$, we have $\dim A_3=9$, which is impossible. Thus $g=0$. Using the fact that $\alpha^2+1\neq 0$, we see that this only happens if $c=d=0$ and $\alpha b+1=0$. Thus $\alpha\neq 0$ and the defining relations of $A$ have the shape $xy+\frac1{\alpha}xz-pyz-qzz$, $yy-\alpha yz-zz$ and $zy$ with $p,q\in\K$. Now the defining relations together with $zzz$ and $yzz$ form a Gr\"obner basis of the ideal of relations of $A$. This yields $\dim A_4=19$, which contradicts the assumptions. Thus $a_2=0$.

Next, consider the case $a_1=a_2=0$ and $a_3\neq 0$. In this case the defining relations of $A$ have the shape $xz-byx-czx-pyz-qzz$, $yy-\alpha yz-zz$ and $zy$ with $b,c,p,q\in\K$. The only overlap (other than $yyy$ and $zyy$) $xzy$ produces $g=byxy+czxy$. If $g\neq 0$, we have $\dim A_3=9$, which is impossible. Thus $b=c=0$. Then the defining relations of $A$ have the form $xz-pyz+qzz$, $yy-\alpha yz+zz$ and $zy$ with $p,q\in\K$.  Now the defining relations together with $zzz$ and $yzz$ form a Gr\"obner basis of the ideal of relations of $A$. This yields $\dim A_4=19$, which contradicts the assumptions. Thus $a_3=0$.

In the case $a_1=a_2=a_3=0$ and $a_4\neq 0$, the defining relations of $A$ have the shape $yx-bzx-cyz-dzz$, $yy-\alpha yz-zz$ and $zy=0$ with $b,c,d\in\K$. The two overlaps (other than $yyy$ and $zyy$) $yyx$ and $zyx$ produce $g=(b-\alpha)yzx-zzx$ and $h=bzzx$. If $g$ and $h$ are linearly independent, we have $\dim A_3=9$, which is impossible. Thus $g$ and $h$ are linearly dependent. This only happens if either $b=\alpha$ or $b=0$. If $b=\alpha$, then $yx-\alpha zx-cyz-dzz$, $yy-\alpha yz-zz$, $zy$, $zzz$, $yzz$ and $zzx$ form a Gr\"obner basis of the ideal of relations of $A$. If $b=0$, then $yx-cyz-dzz$, $yy-\alpha yz-zz$, $zy$, $zzz$, $yzz$ and $\alpha yzx+zzx$ form a Gr\"obner basis of the ideal of relations of $A$. In both cases $\dim A_4=19$, which contradicts the assumptions. Thus $a_4=0$.

Since $(a_1,a_2,a_3,a_4,a_5)\neq (0,0,0,0,0)$, the only option left is $a_1=a_2=a_3=a_4=0$ and $a_5\neq 0$. Then the defining relations of $A$ have the shape $zx-byz-czz$, $yy-\alpha yz-zz$ and $zy$ with $b,c\in\K$. In this case $zx-byz-czz$, $yy-\alpha yz-zz$, $zy$, $zzz$ and $yzz$ form a Gr\"obner basis of the ideal of relations of $A$, which yields $\dim A_3=9$. This contradicts the assumptions and concludes the last case.
\end{proof}

\begin{remark}\label{redege} What we have actually proved in the above lemma, is that if $A=A(V,R)$ is a quadratic algebra, $\dim V=\dim R=3$ and $\dim R\cap M^2\geq 2$ for some $2$-dimensional subspace $M$ of $V$, then the first $8$ terms of $H_A$ can not be the same as for an algebra from $\Omega$: $H_A\neq 1+3t+6t^2+10t^3+15t^4+21t^5+28t^6+36t^7+{\dots}$ However, they can coincide up to the $t^6$ term inclusively.
\end{remark}

\section{Absence of $\rm PBW_B$ condition}

In order to prove Theorem~\ref{main}, we need to show that certain algebras are not $\rm PBW_B$. In this section, we deal with this.

\begin{lemma}\label{non-pbwb-all}
Let $A$ be the quadratic algebra given by the generators $x,y,z$ and three quadratic relations $r_1$, $r_2$ and $r_3$ from the following list
\begin{itemize}\itemsep=-2pt
\item[\rm (A1)] $r_1=xy$, $r_2=yz$, $r_3=a_0xx{+}a_1xz{+}a_2yx{+}a_3yy{+}a_4zx{+}a_5zy{+}a_6zz$ for $a_j\in\K$ such that $a_0a_6\neq 0;$
\item[\rm (A2)] $r_1=xx{-}xz{-}azx{-}bzz$, $r_2=xy{+}\frac bazz$, $r_3=yy{-}yz{-}\frac1azy{-}\frac b{a^2}zz$, where $a,b\in\K^*$, $(a,b)\neq(1,-1);$
\item[\rm (A3)] $r_1=xy{+}byz{+}zz$, $r_2=zx{+}(b{-}1)yz{-}bzy{-}zz$, $r_3=yy{+}yz{+}bzy{+}zz$, where $b\in\K$, $b\neq 0$, $b\neq 1.$
\end{itemize}
Then either $\dim A_3\neq10$ or $A$ is non-$PBW_B$.
\end{lemma}

\begin{proof} Assume the contrary. That is, $\dim A_3=10$ and $A$ is $\rm PBW_B$. Since $A$ is $\rm PBW_B$ and $A\in\Omega'$, Lemma~\ref{ome0} yields that there exist a well-ordering $\leq$ on the $x,y,z$ monomials compatible with multiplication and satisfying $x>y>z$ (this we can acquire by permuting the variables) and a non-degenerate linear substitution $x\mapsto ux+\alpha_1 y+\beta_1z$, $y\mapsto vx+\alpha_2y+\beta_2 z$, $z\mapsto wx+\alpha_3y+\beta_3 z$ such that the leading monomials $m_1,m_2,m_3$ of the new space of defining relations satisfy
\begin{equation}\label{lemo1}
\!\!\!\{m_1,m_2,m_3\}{\in}\bigl\{\!\{xy,xz,yz\},\{xy,xz,zy\},\{xy,zx,zy\},
\{yx,yz,xz\},\{yx,yz,zx\},\{yx,zy,zx\}\!\bigr\}
\end{equation}
Since $xx$ is the biggest degree $2$ monomial,
\begin{equation}\label{noxx}
\text{$xx$ is absent in $r_j$ after the substitution.}
\end{equation}
Since $<$ satisfies $x>y>z$ and is compatible with multiplication,
\begin{equation}\label{4big2}
\text{four biggest degree $2$ monomials are either $xx,xy,yx,xz$ or $xx,xy,yx,zx$}
\end{equation}
(not necessarily in this order). Since each of the classes of algebras with relations from (A1), (A2) or (A3) is closed (up to isomorphism) with respect to passing to the opposite multiplication and the two options in (\ref{4big2}) reduce to one another via passing to the opposite multiplication, for the rest of the proof we can assume that
\begin{equation}\label{4big}
\text{the set of four biggest degree $2$ monomials is $\{xx,xy,yx,xz\}$.}
\end{equation}

{\bf Case 1:} \ $r_j$ are given by (A1).

Since $a_0a_6\neq 0$, by scaling we can without loss of generality assume that $a_0=a_6=1$ from the start (that is, $r_3=xx+a_1xz+a_2yx+a_3yy+a_4zx+a_5zy+zz$ to begin with). Then (\ref{noxx}) reads  $0=uv=vw=u^2+(a_1+a_4)uw+a_3v^2+w^2$.

{\bf Case 1a:} \ $v=0$. Since our substitution is non-degenerate $(v,w)\neq (0,0)$. Since $0=u^2+(a_1+a_4)uw+w^2$, we get $uw\neq 0$. By scaling $x$ (this does not effect the leading monomials), we can assume that $u=1$. Then $w$ is a solution of the quadratic equation $w^2+(a_1+a_4)w+1=0$. The following table shows the coefficients in $r_j$ in front of certain monomials (after substitution). The coefficients whose shape we do not care about are replaced by $*$.
$$
\begin{matrix}
&\vrule&xx&xy&yx&xz&zx&yy&yz\\
\noalign{\hrule}
r_1\hhhhh&\vrule&0&\alpha_2&0&\beta_2&0&\alpha_1\alpha_2&\alpha_1\beta_2\\
r_2&\vrule&0&0&\alpha_2 w&0&\beta_2 w&\alpha_2\alpha_3&\alpha_2\beta_3\\
r_3&\vrule&0&*&*&*&*&*&*
\end{matrix}
$$
If both $xy$ and $yx$ columns of the above matrix vanish, then (\ref{lemo1}) is violated. Thus at least one of these columns is non-zero. First, assume that $\alpha_2\neq 0$. Then the $yx$ column of the above matrix is not in the linear span of the $xy$ and $xz$ columns. Using
(\ref{4big}), the linear independence of the $xy$ and $yx$ columns and the the obvious inequality $xy>xz$, we have that both $xy$ and $yx$ are among the leading monomials of the defining relations. Since this contradicts (\ref{lemo1}), we must have $\alpha_2=0$. Since $v=0$ and  our substitution is non-degenerate, we have $\beta_2\neq 0$. Now by (\ref{lemo1}) and (\ref{4big}), the two biggest leading monomials $m_1$ and $m_2$ of the defining relations are either $xy$ and $xz$ or $yx$ and $xz$. Now using $\alpha_2=0$ and $\beta_2\neq 0$, we see that both $yy$ and $yz$ columns are in the linear span of $m_1$ and $m_2$ columns, while the $zx$ column is not in the said span. Since $zx>zy>zz$, the third leading monomial must be $zx$. Again, we have arrived to a contradiction with (\ref{lemo1}). This concludes Case~1a.

{\bf Case 1b:} \ $v\neq 0$. As above, without loss of generality, $v=1$. The equations $0=uv=vw=u^2+(a_1+a_4)uw+a_3v^2+w^2$ now yield $u=w=a_3=0$. Taking this into account, we write the following table of the coefficients in $r_j$ in front of certain monomials (after substitution):
$$
\begin{array}{c}
\displaystyle
\begin{matrix}
&\vrule&xx&xy&yx&xz&zx&yy&yz\\
\noalign{\hrule}
r_1\hhhhh&\vrule&0&0&\alpha_1&0&\beta_1&\alpha_1\alpha_2&\alpha_1\beta_2\\
r_2&\vrule&0&\alpha_3&0&\beta_3&0&\alpha_2\alpha_3&\alpha_2\beta_3\\
r_3&\vrule&0&\alpha_1a_2&\alpha_3a_5&\beta_1a_2&\beta_3a_5&q_1&q_2
\end{matrix}
\\
\phantom{0}
\\
\text{where}\ \ q_1=\alpha_1^2+\alpha_1\alpha_3(a_1+a_2+a_4)+\alpha_2\alpha_3a_5+\alpha_3^2,
\\
q_2=\alpha_1\beta_1+\alpha_1\beta_3a_1+\alpha_2\beta_1a_2+\alpha_3\beta_1a_4+\alpha_3\beta_2a_5
+\alpha_3\beta_3.
\end{array}
$$
If $\alpha_1\alpha_3\neq 0$ or $\alpha_1a_2\neq 0$ or $\alpha_1=0$ and $\alpha_3a_5\neq 0$, then
the $yx$ column of the above matrix is not in the linear span of the $xy$ and $xz$ columns. Using
(\ref{4big}) and the inequality $xy>xz$, we see that both $xy$ and $yx$ are among the leading monomials of the defining relations. Since this contradicts (\ref{lemo1}), we must have $\alpha_1a_2=\alpha_1\alpha_3=0$ and $\alpha_3a_5=0$ if $\alpha_1=0$.

First, assume that $\alpha_1=0$. Since our substitution is non-degenerate, this yields $\beta_1\alpha_3\neq0$. Since $\alpha_3a_5=0$, we have $a_5=0$. Then the $yx$ column is zero. If $a_2\neq 0$, the $xy$ and $xz$ columns are linearly independent, which makes $xy$ and $xz$ the first two leading monomials of the defining relations. One easily sees that $yy$ and $yz$ columns are in the linear span of $xy$ and $xz$ columns, while the $zx$ column is not in the said span. Since $zx>zy>zz$, $zx$ is the last leading monomial. Thus both $xz$ and $zx$ are among the leading monomials, which contradicts (\ref{lemo1}). This contradiction implies $a_2=0$. In this case the $yx$ column is zero, the $xy$ column is non-zero and the $xz$ column is a scalar multiple of the $xy$ one. Thus the biggest leading monomial is $xy$. Moreover, using $\beta_1\alpha_3\neq0$, one easily checks that $xy$, $zx$ and $yy$ columns are linearly independent. Since $yy>zy>zz$ and $yy>yz$, this implies that $yy$ is among the leading monomials of the defining relations, which contradicts (\ref{lemo1}).

The above contradiction yields $\alpha_1\neq 0$. Since $\alpha_1a_2=\alpha_1\alpha_3=0$, we have $a_2=\alpha_3=0$. Since our substitution is non-degenerate, we must have $\beta_3\neq 0$. Now the $xy$ column is zero, while the $yx$ and $xz$ columns are linearly independent. By (\ref{4big}), this makes $yx$ and $xz$ leading monomials. By (\ref{lemo1}), the third leading monomial must be either $yz$ or $zy$. Since $yy>yz$ and $yy>zy$, this means that the $yy$ column must be in the linear span of $yx$ and $xz$ columns. However, it is easily seen to be not the case. This contradiction completes the proof in Case~1.

{\bf Case 2:} \ $r_j$ are given by (A2).

In this case (\ref{noxx}) reads
$$
0=u^2-(a+1)uw-bw^2=(av)^2-(a+1)(av)w-bw^2=u(av)+bw^2.
$$
Since $(u,v,w)\neq (0,0,0)$, it easily follows that $w\neq 0$. Normalizing, we can assume $w=1$. Then the equation in the above display is satisfied if and only if $t^2-(a+1)t-b=(t-u)(t-av)$ in $\K[t]$. This yields $u+av=a+1$ and $auv+b=0$.

{\bf Case 2a:} \ $b\neq -a$. In this case, one easily sees that $v\neq 1$, $u\neq 1$ and $uv\neq 1$ and the system $u+av=a+1$, $auv+b=0$, gives $a=\frac{u-1}{1-v}$, $b=uv\frac{1-u}{1-v}$. Since $b\neq 0$, we also have $uv\neq 0$. Now we rewrite the defining relations as $r_1=xx-xz+\frac{1-u}{1-v}zx-uv\frac{1-u}{1-v}zz$, $r_2=xy-uvzz$, $r_3=yy-yz+\frac{1-v}{1-u}zy-uv\frac{1-v}{1-u}zz$. We split the substitution $x\mapsto ux+\alpha_1 y+\beta_1z$, $y\mapsto vx+\alpha_2y+\beta_2 z$, $z\mapsto wx+\alpha_3y+\beta_3 z$ into two consecutive subs: first $x\mapsto ux$, $y\mapsto vx+y$, $z\mapsto x+z$ and, second, $x\mapsto x+\alpha_1 y+\beta_1z$, $y\mapsto \alpha_2y+\beta_2 z$, $z\mapsto \alpha_3y+\beta_3 z$ ($\alpha_j,\beta_j$ are not the same). After the first substitution, the space $R$ of defining relations is spanned by (new) $r_j$ with $r_1=(1-uv)xy-v(2-u-v)zx-v(1-v)zz$, $r_2=(1-uv)xz-(1-u)(1-v)zx+v(1-u)zz$, $r_3=-(1-u)(1-v)yx+(1-u)yy-(1-u)yz+v(2-u-v)zx+(1-v)zy+v(1-u)zz$. After performing the second sub, we arrive to table of the coefficients in $r_j$ in front of certain monomials (the coefficients we do not care about are replaced by $*$):
$$
\begin{array}{c}
\displaystyle
\begin{matrix}
&\vrule&xx&xy&yx&xz&zx\\
\noalign{\hrule}
r_1\hhhhh&\vrule&0&\alpha_2(1-uv)&*&\beta_2(1-uv)&*\\
r_2&\vrule&0&\alpha_3(1-uv)&*&\beta_3(1-uv)&*\\
r_3&\vrule&0&0&-\alpha_2(1-u)(1-v)+v(2-u-v)\alpha_3&0&q
\end{matrix}
\\
\phantom{0}
\\
\text{with $q=\alpha_1({-}\alpha_2(1{-}u)(1{-}v){+}v(2{-}u{-}v)\alpha_3){+}(1{-}u)\alpha_2^2{+}(u{-}v)\alpha_2\alpha_3{+}
v(1{-}u)\alpha_3^2$.}
\end{array}
$$
If $\alpha_2\neq \frac{v(2-u-v)}{(1-u)(1-v)}\alpha_3$, then the $3\times 3$ matrix of $xy$, $yx$ and $xz$ coefficients is non-degenerate. Hence the leading monomials of defining relations are $xy$, $yx$ and $xz$, contradicting (\ref{lemo1}). Hence $\alpha_2=\frac{v(2-u-v)}{(1-u)(1-v)}\alpha_3$. Since the second sub must be non-degenerate, $\alpha_3\neq 0$. Using the last equality, we get
$$
q=\alpha_3^2\Bigl(\frac{v^2(2-u-v)^2}{(1-v)^2}+\frac{(u-v)(2-u-v)}{1-v}+v(1-u)^2\Bigr)=
\alpha_3^2\frac{v(1-uv)^2}{(1-u)(1-v)^2}\neq 0.
$$
Since $q\neq 0$, we easily see that either $yy$ is among leading monomials or both $xz$ and $zx$ are. Either way, we get a contradiction with (\ref{lemo1}).

{\bf Case 2b:} \ $b=-a$. In this case $u=v=1$ and our sub takes the shape $x\mapsto x+\alpha_1 y+\beta_1z$, $y\mapsto x+\alpha_2y+\beta_2 z$, $z\mapsto x+\alpha_3y+\beta_3z$. After performing this sub, we arrive to table of the coefficients in $r_j$ in front of certain monomials:
$$
\begin{matrix}
&\vrule&xx&xy&yx&xz&zx&yy\\
\noalign{\hrule}
r_1\hhhhh&\vrule&0&(1{-}a)(\alpha_1{-}\alpha_3)&0&(1{-}a)(\beta_1{-}\beta_3)&0&(\alpha_1{-}\alpha_3)(\alpha_1{-}a\alpha_3)\\
r_2&\vrule&0&\alpha_2{-}\alpha_3&\alpha_1{-}\alpha_3&\beta_2{-}\beta_3&\beta_1{-}\beta_3&\alpha_1\alpha_2{-}\alpha_3^2\\
r_3&\vrule&0&(1{-}\frac1a)(\alpha_2{-}\alpha_3)&0&(1{-}\frac1a)(\beta_2{-}\beta_3)&0&(\alpha_2{-}\alpha_3)(\alpha_2{-}\frac1a\alpha_3)
\end{matrix}
$$
If $\alpha_1\neq \alpha_3$, then the $3\times 3$ matrix of $xy$, $yx$ and $xz$ coefficients is non-degenerate. Hence the leading monomials of defining relations are $xy$, $yx$ and $xz$, contradicting (\ref{lemo1}). Thus $\alpha_1=\alpha_3$. In this case both $3\times 3$ matrices of $xy$, $xz$ and $zx$ coefficients and of $xy$, $xz$ and $yy$ coefficients are non-degenerate (to prove all these invertibilities, we use the fact that our sub is non-degenerate). This yields that either $yy$ or both $xz$ and $zx$ are among the leading monomials, which is incompatible with (\ref{lemo1}). This contradiction completes Case~2.

{\bf Case 3:} \ $r_j$ are given by (A3).

In this case (\ref{noxx}) reads $0=uv+bvw+ww^2=uw-vw-w^2=v^2+(b+1)vw+w^2$.

{\bf Case 3a:} $w=0$. Since $(u,v,w)\neq (0,0,0)$, it easily follows that $v=0$ and $u\neq 0$. Normalizing, we can assume $u=1$. After performing the sub, we arrive to table of the coefficients in $r_j$ in front of certain monomials (the coefficients we do not care about are replaced by $*$):
$$
\begin{array}{c}
\displaystyle
\begin{matrix}
&\vrule&xx&xy&yx&xz&zx&yy\\
\noalign{\hrule}
r_1\hhhhh&\vrule&0&\alpha_2&0&\beta_2&0&*\\
r_2&\vrule&0&0&\alpha_3&0&\beta_3&*\\
r_3&\vrule&0&0&0&0&0&q
\end{matrix}
\\
\phantom{0}
\\
\text{with $q=\alpha_2^2+\alpha_3^2+(b+1)\alpha_2\alpha_3$.}
\end{array}
$$
If $\alpha_2\alpha_3\neq 0$, one easily sees that both $xy$ and $yx$ must be among the leading monomials, contradicting (\ref{lemo1}). Thus $\alpha_2\alpha_3=0$. Since $(\alpha_2,\alpha_3)\neq (0,0)$, it follows that  $q\neq 0$. This forces $yy$ to be a leading monomial, which contradicts (\ref{lemo1}) and concludes Case~3a.

{\bf Case 3b:} \ $w\neq 0$. Normalizing, we can assume $w=1$. Now the equations $0=uv+bvw+ww^2=uw-vw-w^2=v^2+(b+1)vw+w^2$ are satisfied precisely when $v^2+(b+1)v+1=0$ and $u=v+1$. In this case $b=-1-v-\frac1v$ and the relations read: $r_1=xy-(1+v+\frac1v)yz+zz$, $r_2=zx-(2+v+\frac1v)yz+(1+v+\frac1v)zy-zz$ and $r_3=yy+yz-(1+v+\frac1v)zy+zz$. Now our sub takes form
$x\mapsto (1+v)x+\alpha_1 y+\beta_1z$, $y\mapsto vx+\alpha_2y+\beta_2 z$, $z\mapsto x+\alpha_3y+\beta_3 z$.
After performing this sub, one easily sees that the vectors of $xy$ and $xz$ coefficients are always linearly independent, while of the vectors of $yx$ and $yy$ coefficients, at least one is not in the span of the vectors of $xy$ and $xz$ coefficients (more precisely, the opposite happens only if $v=-1$, which corresponds to the excluded case $b=-1$). The formulas for the coefficients in this case are rather unwieldy, so we skip them leaving the verification to the reader. The above relations between vectors of coefficients imply that among the leading coefficients of the defining relations we find either both $xy$ and $yx$ or both $xz$ and $zx$ or $yy$. Either way, (\ref{lemo1}) is violated. This contradiction completes the proof.
\end{proof}

\section{Parts V--VII, IX and X of Theorem~\ref{main} and general comments}

It is elementary to verify that
\begin{equation}\label{omepr}
\text{each $Q$ in Theorem~\ref{main} is indeed the quasipotential for the corresponding algebra $A$.}
\end{equation}
Moreover, each $Q$ has the properties claimed in Theorem~\ref{main}: has the declared $n_1(Q)$ and $n_2(Q)$, is or os not cyclicly invariant etc. We skip this elementary and routine linear algebra exercise. It is very easy in each case, but since we have so many of them, we do not spell out this verification.

Since $n_1(Q)$, $n_2(Q)$ and the Jordan normal form of $M_Q$ (in case $n_1(Q)=n_2(Q)=n$) are invariants, we immediately see that
\begin{equation}\label{omepr1}
\text{algebras from Theorem~\ref{main} with different letters in their labels can not be isomorphic.}
\end{equation}
By (\ref{omepr1}), parts of Theorem~\ref{main} are indeed independent.

Parts VI, VII, IX and X of Theorem~\ref{main} follow from the description by the authors of all potential and twisted potential quadratic algebras on three generators and cubic algebras on two generators \cite[Theorems~1.6 and 1.7]{SKL}.  Part~V of Theorem~\ref{main} is a direct corollary of Lemma~\ref{n1n2}. Apart from mentioning this, we prove the following lemma. Although it slightly overlaps with Parts VI and VII of Theorem~\ref{main}, we give it in full.

\begin{lemma}\label{pbwb} Every algebra $A=A(V,R)$ from {\rm (P2--P8), (T1--T11), (T16--T17), (S1--S16), (M1--M8), (L1--L8)} and {\rm (N1--N15)} $($in short, every algebra in Theorem~$\ref{main}$ specified as $PBW_B)$ belongs to $\Omega^0$. That is, $A$ is indeed $PBW_B$ and $H_A=(1-t)^{-3}$.
\end{lemma}

\begin{proof} By (\ref{omepr}) and Lemma~\ref{ome0}, it suffices to show that there is a basis $x,y,z$ in $V$ and a well-ordering on $x,y,z$ monomials compatible with multiplication, with respect to which the set of leading monomials of elements of a basis in $R$ is one of $\{xy,xz,yz\}$, $\{xy,xz,zy\}$, $\{xy,zx,zy\}$, $\{yx,yz,xz\}$, $\{yx,yz,zx\}$ or $\{yx,zy,zx\}$. If we keep the original basis $x,y,z$ and use the left-to-right degree-lexicographical ordering with $x>y>z$, then the triple of leading monomials for $R$ is $\{xy,xz,yz\}$ for $A$ from (P2--P8), (T1--T10), (S15--S16), (M1--M2), (M5--M6) and (N5--N15). If we keep the original basis $x,y,z$ and use the right-to-left degree-lexicographical ordering with $z>y>x$, then the triple of leading monomials for $R$ is $\{xy,xz,yz\}$ for $A$ from (M3), (S4), (N1--N4) and (S14). With respect to the same order, the triple of leading monomials for $R$ is $\{yx,zx,yz\}$ for $A$ from (L1--L3), (L7--L8) and is $\{xy,xz,zy\}$ for $A$ from (M7--M8). If we keep the original basis $x,y,z$ and use the left-to-right degree-lexicographical ordering with $y>z>x$, then the triple of leading monomials for $R$ is $\{yx,zx,yz\}$ for $A$ from (L4--L6) and is $\{yz,zx,xy\}$ for $A$ from (M4). If we keep the original basis $x,y,z$ and use the left-to-right degree-lexicographical ordering with $y>x>z$, then the triple of leading monomials for $R$ is $\{xz,yz,yx\}$ for $A$ from (S2--S9). If we keep the original basis $x,y,z$ and use the left-to-right degree-lexicographical ordering with $z>x>y$, then the triple of leading monomials for $R$ is $\{zy,xz,xy\}$ for $A$ from (S10--S13).  These considerations take care of all algebras in question except for (T11), (T16) and (T17), for which the original basis does not work regardless which order on monomials we consider. Thus a change of basis is needed.

For $A$ from (T11), we perform the substitution $x\to x$, $y\to y+ix$, $z\to z$, which turns the defining relations into $xz+azx$, $yz-azy-2aizx$ and $xy+yx-iyy$. For $A$ from (T16), we perform the same substitution $x\to x$, $y\to y+ix$, $z\to z$, which turns the defining relations into $xz-zx$, $yz+zy-2izx$ and $xy+yx-iyy-izz$. For $A$ from (T17), we first swap $y$ and $z$ turning the defining relations into $yy+zz$, $xy-yx$ and $xz+zx+zz$. Next, we follow up with the substitution $x\to x$, $y\to y$ and $z\to z+iy$, turning the defining relations into $xy-yx$, $yz+zy-izz$ and $xz+zx+2iyx-yy$. In the new $x,y,z$ basis the leading monomials for $R$ for each of these three algebras are $\{xy,xz,yz\}$ with respect to left-to-right degree-lexicographical ordering with $x>y>z$. As we have mentioned at the start, an application of Lemma~\ref{ome0} completes the proof.
\end{proof}

\section{Proof of Part VIII of Theorem~\ref{main}}

We start by studying the family {\bf (W)} of quadratic algebras $A^a$ given by the generators $x,y,z$ and relations $xy$, $yz$ and $a_1xx+xz+a_2yx+a_3yy+a_4zx+a_5zy+a_6zz$ for $a\in\K^6$. Note that this family includes all of (N1--N24).

\begin{lemma}\label{isoN} Let $a,b\in\K^6$ and $A^a$, $A^b$ be the corresponding algebras from {\rm (W)}. Then $A^a$ and $A^b$ are isomorphic if and only if there exist $\alpha,\beta\in\K^*$ such that $b=\bigl(\alpha^2 a_1,\alpha\beta a_2,\beta^2 a_3,a_4,\frac{\beta a_5}{\alpha},\frac{a_6}{\alpha^2}\bigr)$.
\end{lemma}

\begin{proof} Let $A^a=A(V,R)$ be an algebra from (W). It is easy to see that $y$ is up to a scalar multiple the only non-zero element $v$ of $V$ for which there exist non-zero $u,w\in V$ satisfying $uv,vw\in R$. Next, up to a scalar multiple, $x$ is the only non-zero element $u$ of $V$ for which $uy\in R$. Finally, up to a scalar multiple, $z$ is the only non-zero element $u$ of $V$ for which $yu\in R$. These observations imply that every linear substitution providing an isomorphism of $A^a$ and $A^b$ must send each of $x,y,z$ to its own scalar multiple (=is a scaling). Now, applying a scaling to relations of $A^a$, we immediately see that $A^a$ and $A^b$ are isomorphic if and only if there exist $\alpha,\beta\in\K^*$ such that $b=\bigl(\alpha^2 a_1,\alpha\beta a_2,\beta^2 a_3,a_4,\frac{\beta a_5}{\alpha},\frac{a_6}{\alpha^2}\bigr)$.
\end{proof}

\begin{lemma} \label{q2dual} Let $A$ be the quadratic algebra given by the generators $x,y,z$ and the relations $xy$, $yz$ and $xx-xz-\alpha yx-\beta yy-azx-\gamma zy-bzz$ for some $(a,b,\alpha,\beta,\gamma)\in\K^5$. Then $A\in\Omega^0$ if $b=0$ and the following statements hold true$:$
\begin{itemize}\itemsep-2pt
\item[\rm(\ref{q2dual}.1)]$\dim A_n\geq \frac{(n+1)(n+2)}{2}$ for all $n\in\Z_+;$
\item[\rm(\ref{q2dual}.2)]if $(a,b)\neq (1,-1)$, then $H_{A^!}=(1+t)^3;$
\item[\rm(\ref{q2dual}.3)]if $b\neq 0$ and $(a,b)\neq (1,-1)$, then the $($right module$)$ Koszul complex of $A$ is given by
\begin{equation}\label{q2kos}
\begin{array}{l}
\displaystyle \qquad\qquad0\longrightarrow A\mathop{\longrightarrow}^{d_3} A^3
\mathop{\longrightarrow}^{d_{2}} A^3\mathop{\longrightarrow}^{d_1} A\mathop{\longrightarrow}^{d_0} \K\to 0
\\
\text{where $d_0$ is the augmentation,}\ \ d_1(u,v,w)=xu+yv+zw,\ \ d_3(u)=(zu,0,0)
\\
\text{and}\ \ d_2(u,v,w)=(yu+(z-x)w,zv+(\alpha x+\beta y)w,(ax+\gamma y+bz)w).
\end{array}
\end{equation}
\item[\rm(\ref{q2dual}.4)]if $(a,b)=(1,-1)$, then $\dim A_5\geq 24$ and therefore $A\notin\Omega;$
\item[\rm(\ref{q2dual}.5)]if $b\neq 0$, $(a,b)\neq(1,-1)$ and there exists a non-zero homogeneous of degree $k$ element $u$ of $A$ such that $zu=0$ in $A$, then $\dim A_j>\frac{(j+1)(j+2)}{2}$ for some $j\leq k+2$.
\end{itemize}
\end{lemma}

\begin{proof} If $b=0$, then with respect to the right-to-left degree lexicographical ordering satisfying $z>y>x$, the leading monomials of the defining relations of $A$ are $xy$, $xz$ and $yz$. By Lemma~\ref{ome0}, $A\in\Omega^0$.

Throughout the rest of the proof, we use the left-to-right degree lexicographical ordering on $x,y,z$ monomials assuming $x>y>z$. Since proving the above statements in the case when $\K$ is replaced by any field extension will yield their validity for original $\K$, we can without loss of generality, assume that $\K$ is uncountable.

Assuming $(a,b)=(1,-1)$, we make the substitution leaving $y$ and $z$ as they were and replacing $x$ by $x+z$. The defining relations of $A$ take the shape $xx-\alpha yx-\beta yy-\gamma zy$, $xy+zy$, $yz$. A direct computation shows that  for $(\alpha,\beta,\gamma)$ from a Zarisski open subset of $\K^3$, the Gr\"obner basis of the ideal of relations of $A$ is finite and the leading monomials of its members are $xx$, $xy$, $yz$, $yyx$, $xzy$, $zyyy$, $yyyy$, $zzzyx$, $zzyxz$ and $xzzyx$. This yields $\dim A_5=24$ for $(\alpha,\beta,\gamma)$ from a Zarisski open subset of $\K^3$. By Lemma~\ref{minhs}, $\dim A_5\geq 24$ for all $\alpha,\beta,\gamma$ justifying (\ref{q2dual}.4).

As we have already observed, $A\in \Omega^0$ if $b=0$. In particular, $A$ is Koszul, $H_A=(1-t)^{-3}$ and $H_{A^!}=(1=t)^3$ provided $b=0$. Now assume $b\neq 0$. In this case, the defining relations of $A^!$ take the shape $xx+\frac1b zz$, $xz-\frac1b zz$, $zx-\frac{a}{b} zz$, $yx-\frac{\alpha}{b} zz$, $yy-\frac{\beta}{b} zz$ and $zy-\frac{\gamma}{b} zz$. A direct computation shows that provided $(a,b)\neq (1,-1)$, the defining relations together with $yzz$ and $zzz$ form a Gr\"obner basis of the ideal of relations of $A^!$. The complete list of normal words for $A^!$ now is $1$, $x$, $y$, $z$, $xy$, $yz$, $zz$ and $xyz$ yielding $H_{A^!}=(1+t)^3$ and justifying (\ref{q2dual}.2). Furthermore, the normal words in $A^!$ furnish us with a graded linear basis in $A^!$, while the above Gr\"obner basis provides the corresponding structural constants. Now a routine computation yields (\ref{q2dual}.3). By the already verified (\ref{q2dual}.2), $H_{A^!}=(1+t)^3$ provided $(a,b)\neq (1,-1)$. We also know that $A$ is Koszul if $b=0$. By Lemma~\ref{dri3} (the corresponding variety $W$ is the affine space $\K^5$), (\ref{q2dual}.1) is satisfied.

Now we shall verify (\ref{q2dual}.5). Assume that $b\neq 0$, $(a,b)\neq(1,-1)$ and there exists a non-zero homogeneous of degree $k$ element $u$ of $A$ such that $zu=0$ in $A$. Without loss of generality, we can assume that $k$ is the minimal positive integer for which such an $u$ exists. Towards a contradiction assume that (\ref{q2dual}.5) fails. That is, $\dim A_j\leq a_j$ for all $j\leq k+2$, where $a_j=\frac{(j+1)(j+2)}{2}$. According to the already verified (\ref{q2dual}.1), $\dim A_j=a_j$ for all $j\leq k+2$. Now we consider the following graded 'slice' of the Koszul complex (\ref{q2kos}) of $A$:
\begin{equation}\label{q2kos1}
0\longrightarrow A_{k-1}\mathop{\longrightarrow}^{\delta_3} A_k^3
\mathop{\longrightarrow}^{\delta_{2}} A_{k+1}^3\mathop{\longrightarrow}^{\delta_1} A_{k+2}\to 0,
\end{equation}
where $\delta_1$, $\delta_2$ and $\delta_3$ are restrictions of $d_1$, $d_2$ and $d_3$ to $A_{k+1}^3$, $A_k^3$ and $A_{k-1}$ respectively. Since (\ref{q2kos}) (as a Koszul complex of a quadratic algebra) is exact at its three rightmost terms, (\ref{q2kos1}) is exact at $A_{k+1}^3$ and at $A_{k+2}$. Since $\dim A_j=a_j$ for all $j\leq k+2$, this yields that the dimension of the kernel of $\delta_2$ is $3a_k-3a_{k+1}+a_{k+2}=a_{k-1}$ (the last equality follows from the definition of $a_j$). The minimality of $k$ and the shape of $d_3$ yields that $\delta_3$ is injective. Hence (\ref{q2kos1}) is exact at $A_{k-1}$ and the image of $\delta_3$ has the dimension $a_{k-1}$. Since the latter coincides with the dimension of the kernel of $\delta_2$ the said image and kernel coincide. On the other hand, one easily sees that $\delta_2(0,u,0)=d_2(0,u,0)=0$ (since $zu=0$), while $(0,u,0)$ is clearly not in the image of $\delta_3$. This contradiction concludes the proof of (\ref{q2dual}.5).
\end{proof}

We also need two sequences of polynomials. For $n\in\Z_+$, let $p_n,q_n\in\K[x,y]$ be defined by
\begin{equation}\label{pnqn}
\left(\begin{array}{c}p_n\\ q_n\end{array}\right)=\left(\begin{array}{cc}1&y\\ 1&-x\end{array}\right)^n\left(\begin{array}{c}1\\ 1\end{array}\right).
\end{equation}
Equivalently, polynomials $p_n$ and $q_n$ can be defined recurrently by $p_0=q_0=1$, $p_{n+1}=p_n+yq_n$, $q_{n+1}=p_n-xq_n$. On few occasions, we will have to deal with the condition $p_n(a,b)\neq 0$ for all $n\in\N$ imposed on $(a,b)\in\K^2$. The following remark simplifies this condition.

\begin{remark}\label{pnqnrem}
By expressing the matrix
$$
M=\left(\begin{array}{cc}1&b\\ 1&-a\end{array}\right)
$$
with $a,b\in\K$ as a conjugate of a matrix in Jordan normal form (the eigenvalues of $M$ are $\alpha=\frac{1-a+d}{2}$ and $\beta=\frac{1-a-d}{2}$ with $d^2=(a+1)^2+4b$), one easily gets explicit formulas for $p_n(a,b)$ and $q_n(a,b)$:
\begin{equation}\label{pnqn1}
\begin{array}{lll}
{\vrule height0pt depth8pt width0pt}q_n(a,b)=\frac{(1-a+d)^{n+1}-(1-a-d)^{n+1}}{2^nd},&{\vrule height0pt depth8pt width0pt}p_n(a,b)=q_n(a,b)+\frac{(a+b)((1-a+d)^{n}-(1-a-d)^{n})}{2^{n-1}d}
&\text{if $d\neq 0;$}
\\
q_n(a,b)=\frac{(n+1)(1-a)^n}{2^n},&p_n(a,b)=\frac{(n+2-na)(1-a)^n}{2^{n-1}}&\text{if $d=0.$}
\end{array}
\end{equation}
Furthermore, if $p_n(a,b)=0$ and $d\neq 0$, then  $(1-a+d)(1+a-d)\neq 0$ and one easily concludes that
\begin{equation}\label{pnqn2}
\begin{array}{ll}
p_n(a,b)=0\iff \Bigl(\frac{1-a-d}{1-a+d}\Bigr)^{n+1}=\frac{1+a+d}{1+a-d}&\text{if $d\neq 0;$}
\\
p_n(a,b)=0\iff na=n+2&\text{if $d=0.$}
\end{array}
\end{equation}
\end{remark}

\begin{lemma}\label{gbq2} Let $A$ be the quadratic algebra given by the generators $x,y,z$ and the relations $xy$, $yz$ and $xx-xz-\alpha yx-\beta yy-azx-\gamma zy-bzz$ for some $a,b,\alpha,\beta,\gamma\in\K^5$ such that $b\neq 0$ and $(a,b)\neq(1,-1)$. Let also $n\in\N$ be such that $p_k(a,b)\neq 0$ for $0\leq k<n$ and $q_k(a,b)\neq 0$ for $0\leq k\leq n$. Then the equalities
\begin{equation}\label{gb21}
xz^kx-s_kxz^{k+1}-t_ky^{k+1}x-h_ky^{k+2}-zu_k=0\ \ \text{and}\ \ xz^ky-r_ky^{k+2}-zv_k=0
\end{equation}
hold in $A$ for $0\leq k\leq n$, where $s_k=\frac{p_k(a,b)}{q_k(a,b)}$, while $t_k,h_k,r_k\in\K$ and $u_k,v_k\in A_{k+1}$ are defined inductively by $u_0=ax+\gamma y+bz$, $v_0=0$, $t_0=\alpha$, $h_0=\beta$, $r_0=0$, $t_{k+1}=\frac{\alpha(r_k-t_k)-h_k}{s_k-a}$, $h_{k+1}=\frac{\beta s_k(r_k-t-k)-\gamma h_k}{s_k(s_k-a)}$, $r_{k+1}=-\frac{h_k}{s_k}$, $u_{k+1}=\frac1{s_k-a}\bigl(u_k(z-x-\frac{\gamma}{s_k}y)+v_k(\alpha x+\beta y)\bigr)$ and $v_{k+1}=-\frac1{s_k}uy$ for $0\leq k<n$ $($note that by the conditions imposed upon $a,b$ not only $s_k\neq0$, but also $s_k\neq a$ for $0\leq k<n)$. Moreover, the equalities
\begin{equation}\label{gb22}
\begin{array}{l}
(s_n-a)xz^{n+1}x-(s_n+b)xz^{n+2}+(\alpha t_n+h_n-\alpha r_n)y^{n+2}x+\beta(t_n-r_n)y^{n+3}-\gamma xz^{n+1}y
\\
\qquad\qquad-z(u_n(z-x)+v_n(\alpha x+\beta y))=0\ \ \text{and}\ \ s_nxz^{n+1}y+h_ny^{n+3}+zu_ny=0
\end{array}
\end{equation}
hold in $A$.
Furthermore, the left-hand sides of $(\ref{gb21})$ for $0\leq k\leq n$ together with $yz$ are all the elements of degree up to $n+2$ of the reduced Gr\"obner basis of the ideal of relations of $A$ and $\dim A_j=\frac{(j+1)(j+2)}{2}$ for $0\leq j\leq n+2$.
\end{lemma}

\begin{proof} We shall prove (\ref{gb21}) for $0\leq k\leq n$ inductively. The defining relations $xy=0$ and $xx-xz-\alpha yx-\beta yy-azx-\gamma zy-bzz=0$ justify (\ref{gb21}) for $k=0$, providing the basis of induction. Assume that $0\leq k\leq n-1$ and (\ref{gb21}) holds for $k$. We shall verify (\ref{gb21}) for $k$ replaced by $k+1$, which will complete the inductive proof. We do this by resolving certain overlaps as in the Buchberger algorithm using the induction hypothesis and the defining relations of $A$ in the process. The monomials to which we apply reduction are indicated by underlining:
$$
\underline{xz^k\phantom{x}}\!\!\!\!\underline{xy} \to 0=s_kxz^{k+1}y+t_ky^{k+1}\underline{xy}+h_ky^{k+3}+zu_ky=
s_kxz^{k+1}y+h_ky^{k+3}+zu_ky,
$$
which yields
\begin{equation}\label{eq1}
s_kxz^{k+1}y+h_ky^{k+3}+zu_ky=0\ \ \text{in $A$}.
\end{equation}
Since $s_k\neq 0$, we can divide by $s_k$ and use the definitions of $r_j$ and $v_j$ to see that
$$
xz^{k+1}y-r_{k+1}y^{k+3}-zv_{k+1}=0\ \ \text{in $A$},
$$
which is the second equality in (\ref{gb21}). Now we proceed in the same manner with another overlap:
$$
\begin{array}{l}
\underline{xz^k\phantom{y}}\!\!\!\!\underline{xx} \to 0=s_kxz^{k+1}x+t_ky^{k+1}\underline{xx}+h_ky^{k+2}x+zu_kx-\underline{xz^kx}z
\\
\qquad-axz^{k+1}x-bxz^{k+2}-\alpha\underline{xz^ky}x-\beta\underline{xz^ky}y-\gamma xz^{k+1}y
\\
=s_kxz^{k+1}x+t_ky^{k+1}xz+at_ky^k\underline{yz}x+bt_ky^k\underline{yz}z+\alpha t_ky^{k+2}x+\beta t_ky^{k+3}
+\gamma t_ky^k\underline{yz}y
\\
\qquad +h_ky^{k+2}x+zu_kx-s_kxz^{k+2}-t_ky^{k+1}xz-h_ky^{k+2}\underline{yz}-zu_kz-axz^{k+1}x
\\
\qquad-bxz^{k+2}-\alpha r_ky^{k+2}x-\alpha zv_kx-\beta r_ky^{k+3}-\beta zv_ky -\gamma xz^{k+1}y.
\end{array}
$$
After obvious cancelations and rearrangements, we get
\begin{equation}\label{eq2}
\begin{array}{l}
(s_k-a)xz^{k+1}x-(s_k+b)xz^{k+2}+(\alpha t_k+h_k-\alpha r_k)y^{k+2}x+\beta(t_k-r_k)y^{k+3}\\
\qquad\qquad-\gamma xz^{k+1}y-z(u_k(z-x)+v_k(\alpha x+\beta y))=0\ \ \text{in $A$.}
\end{array}
\end{equation}
Using the already verified equality $xz^{k+1}y-r_{k+1}y^{k+3}-zv_{k+1}$, dividing by $s_k-a$ and using the obvious equality $s_{k+1}=\frac{s_k+b}{s_k-a}$ together with the definitions of $t_j$, $h_j$ and  $u_j$, we obtain the first equality in (\ref{gb21}). This completes the inductive proof of (\ref{gb21}) for $0\leq k\leq n$. Now (\ref{gb22}) is the combination of (\ref{eq1}) and (\ref{eq2}) for $k=n$, which hold since (\ref{gb21}) holds for $k=n$.

Now the only degree $\leq n+2$ monomials, which do not contain $yz$ as well as $xz^kx$ and $xz^ky$ for $0\leq k\leq n$ as submonomials are the words of the shape $z^jy^m$ and $z^jy^mxz^p$ of degree $\leq n+2$ with $j,m,p\in\Z_+$. The number of such monomials of degree $k$ is exactly $\frac{(k+1)(k+2)}{2}$. Since by Lemma~\ref{q2dual}, $\dim A_k\geq \frac{(k+1)(k+2)}{2}$, $yz$ together with the left-hand sides of $(\ref{gb21})$ for $0\leq k\leq n$ must comprise all the elements of degree up to $n+2$ of the reduced Gr\"obner basis of the ideal of relations of $A$ and we must have $\dim A_k=\frac{(k+1)(k+2)}{2}$ for $0\leq k\leq n+2$. \end{proof}

\begin{lemma}\label{gbq2-1} Let $A$ be the quadratic algebra given by the generators $x,y,z$ and the relations $xy$, $yz$ and $xx-xz-\alpha yx-\beta yy-azx-\gamma zy-bzz$ for some $a,b,\alpha,\beta,\gamma\in\K^5$ such that $b\neq 0$. Assume also that $p_k(a,b)\neq 0$ for all $k\in\N$. Then $A\in\Omega^+$, $A$ is Koszul and $A$ is non-$PBW_B$. In particular, algebras from {\rm (N16--N19)} belong to $\Omega^+$ are Koszul and non-$PBW_B$.
\end{lemma}

\begin{proof} Note that $p_1(a,b)\neq 0$ yields $b\neq -1$ and therefore $(a,b)\neq (1,-1)$.

{\bf Case 1:} $q_k(a,b)\neq 0$ for all $k\in\N$.

By Lemma~\ref{gbq2}, $H_A=(1-t)^{-3}$ and the leading monomials of the reduced Gr\"obner basis of the ideal of relations of $A$ are $yz$, $xz^kx$ and $xz^ky$ for $k\in\Z_+$. Since none of these monomials starts with the smallest variable $z$, we have $zu\neq 0$ for all non-zero $u\in A$. The exact shape of the Koszul complex of $A$ provided by Lemma~\ref{q2dual} now allows us to say that this complex is exact at its left-most term ($d_3$ is injective). By Lemma~\ref{koal3}, $A$ is Koszul.

{\bf Case 2:} $q_k(a,b)=0$ for some $k\in\N$.

Let $n\in\Z_+$ be the minimal non-negative integer for which $q_{n+1}(a,b)=0$. By Lemma~\ref{gbq2}, (\ref{gb21}) is satisfied for $0\leq k\leq n$ and (\ref{gb22}) holds. Since $q_{n+1}(a,b)=0$, we have $s_n=a$ for $s_j$ defined in By Lemma~\ref{gbq2}. Next, observe that $a+b\neq 0$ and $a\neq 0$. Indeed, if $a+b=0$, then $q_k(a,b)=p_k(a,b)\neq 0$, while if $a=0$, then $q_k(a,b)=1$ for all $k\in\Z_+$ thus contradicting the assumption of Case~2. Taking this into account, we can rewrite (\ref{gb22}) as follows:
\begin{equation}\label{gb22-2}
\begin{array}{l}
\textstyle xz^{n+2}-\frac{\alpha t_n+h_n-\alpha r_n}{a+b}y^{n+2}x-\frac{\beta(t_n-r_n)}{a+b}y^{n+3}+\frac{\gamma}{a+b} xz^{n+1}y-zu=0;\\ xz^{n+1}y+\frac{h_n}{a}y^{n+3}+zv=0\ \ \ \text{for some $u,v\in A_{n+2}$.}
\end{array}
\end{equation}

Now we observe that the number of monomials of degree $m$, which do not contain any of $yz$, $xz^ky$ for $0\leq k\leq n+1$, $xz^kx$ for $0\leq k\leq n$ and $xz^{n+2}$ as submonomials is exactly $\frac{(m+1)(m+2)}{2}$ for every $m\in\Z_+$. Since by Lemma~\ref{q2dual}, $\dim A_m\geq \frac{(m+1)(m+2)}{2}$, $yz$ together with the left-hand sides of (\ref{gb21}) for $0\leq k\leq n$ and the left-hand sides of (\ref{gb22-2}) must form the reduced Gr\"obner basis of the ideal of relations of $A$ and we must have $H_A=(1-t)^{-3}$. Since none of the leading monomials of this basis starts with the smallest variable $z$, we have $zu\neq 0$ for all non-zero $u\in A$. Exactly as in Case~1, we can now use  Lemma~\ref{q2dual} and Lemma~\ref{koal3} to conclude that $A$ is Koszul. Note that in the second case the Gr\"obner basis turns out to be finite. Finally, $A$ is non-PBW$_{\rm B}$ according to Lemma~\ref{non-pbwb-all}. The comment about algebras in (N16--N19) is a direct corollary of the above.
\end{proof}

\begin{lemma}\label{gbq2-2} Let $A$ be the quadratic algebra given by the generators $x,y,z$ and the relations $xy$, $yz$ and $xx-xz-\alpha yx-\beta yy-azx-\gamma zy-bzz$ for some $a,b,\alpha,\beta,\gamma\in\K^5$ such that $p_k(a,b)=0$ for some $k\in\N$. Then $A\notin\Omega$.
\end{lemma}

\begin{proof} As on few occasions above, we can without loss of generality assume that $\K$ is uncountable. By Lemma~\ref{q2dual}, $A\notin\Omega$ if $(a,b)=(1,-1)$. Thus for the rest of the proof we can assume that $(a,b)\neq(1,-1)$. Let $n$ be the smallest positive integer for which $p_n(a,b)=0$. Consider the variety $W_0=\{(s,t)\in\K^2:p_n(s,t)=0\}$. We strongly suspect that $W_0$ is irreducible, however, we do not see an easy way to demonstrate it. Thus we take an irreducible component $W_1$ of $W_0$, which contains $(a,b)$. As $W_0$ is defined by one non-trivial polynomial equation, $W_1$ must be one-dimensional. Then $W=W_1\times \K^3$ is a $4$-dimensional irreducible affine variety. The plan of the proof is the following. Instead of dealing with the given specific algebra, we shall demonstrate that for generic $(a,b,\alpha,\beta,\gamma)\in W$, $H_A\neq (1-t)^{-3}$. On the other hand by Lemma~\ref{hsdri}, for such generic $A$, $H_A=H_{\min}$, where $H_{\min}$ is coefficient-wise minimum of $H_A$ for $A$ with parameters from $W$. Furthermore, by Lemma~\ref{q2dual}, $\dim A_j\geq \frac{(j+1)(j+2)}{2}$ for all $j\in\Z_+$. Combining these, we see that if
$H_A\neq (1-t)^{-3}$ for generic $(a,b,\alpha,\beta,\gamma)\in W$, then $H_A\neq (1-t)^{-3}$ for all  $(a,b,\alpha,\beta,\gamma)\in W$, which will include the original specific algebra. Thus the proof will be complete if we demonstrate that $H_A\neq (1-t)^{-3}$ for generic $(a,b,\alpha,\beta,\gamma)\in W$.

Using the explicit description of $p_n$ and $q_n$ given in Remark~\ref{pnqnrem}, one easily sees that $q_k(a,b)\neq 0$ for all $k\in\Z_+$ for all $(a,b)\in W_0$ with countably many possible exceptions. Thus $q_k(a,b)\neq 0$ for all $k\in\Z_+$ for generic $(a,b,\alpha,\beta,\gamma)\in W$.
Let $s_k$, $t_k$, $h_k$, $r_k$, $u_k$ and $v_k$ be as in Lemma~\ref{gbq2}. From their definition, one easily sees that $h_n\neq 0$ for generic  $(a,b,\alpha,\beta,\gamma)\in W$. Now if $\alpha=\beta=\gamma=0$, then from the recurrent definition of $u_k$, $v_k$ it follows that $v_k=0$ for $0\leq k\leq n$, while $u_k=z^k(a_kx+b_kz)$ for $0\leq k\leq n$, where $a_0=a$, $b_0=b$, $a_{k+1}=-\frac{aa_k+b_k}{s_k-a}$ and $b_{k+1}=\frac{b_k-ba_k}{s_k-a}$ for $0\leq k\leq n$. It is straightforward to check that $b_n=0$ only for finitely many $(a,b)\in W_0$. From the inequality $\dim A_j\geq \frac{(j+1)(j+2)}{2}$ it follows that
provided $h_n\neq 0$, the left-hand sides of (\ref{gb21}) and ({\ref{gb22}}) form the degree up to $n+3$ part of a Gr\"obner basis of the ideals of relations of $A$. Thus, if $\alpha=\beta=\gamma=0$, $u_nyy=b_nz^{n+1}yy\neq 0$ provided $b_n\neq 0$ with the latter being true with finitely many exceptions. Now one easily sees that $u_nyy\neq 0$ for generic $(a,b,\alpha,\beta,\gamma)\in W$. By ({\ref{gb22}}), $h_ny^{n+3}+zu_ny=0$ in $A$. Multiplying by $y$ on the left and using the defining relations, we get $h_ny^{n+4}=0$ in $A$. Now multiplying $h_ny^{n+3}+zu_ny=0$ by $y$ on the right and using $h_ny^{n+4}=0$, we get $zu_nyy=0$. Thus for generic $(a,b,\alpha,\beta,\gamma)\in W$, the non-zero degree $n+3$ element $u_nyy$ satisfies $zu_nyy=0$ in $A$. By Lemma~\ref{q2dual}, there is $j\leq n+5$ for which $\dim A_j>\frac{(j+1)(j+2)}{2}$. Thus for generic $(a,b,\alpha,\beta,\gamma)\in W$, $H_A\neq (1-t)^{-3}$, which completes the proof.
\end{proof}

\begin{lemma}\label{algN} All algebras in {\rm (N1--N23)} belong to $\Omega$, all these algebras are Koszul, algebras from {\rm (N1--N15)} are $PBW_B$, while algebras from {\rm (N1--N15)} are non-$PBW_B$. Algebras in {\rm (N1--N23)} with different labels are non-isomorphic and the isomorphism conditions of Theorem~$\ref{main}$ within algebras with a given label from {\rm (N1--N23)} are satisfied.
\end{lemma}

\begin{proof} The isomorphism statement follows easily from Lemma~\ref{isoN}. Algebras from (N1--N15) belong to $\Omega^0$ by Lemma~\ref{pbwb}. Algebras in (N16--N23) belong to $\Omega^+$ and are Koszul according to Lemma~\ref{gbq2-1} and Remark~\ref{pnqnrem}. These algebras are non-PBW$_{\rm B}$ by Lemma~\ref{non-pbwb-all}.
\end{proof}

\begin{lemma}\label{omepl1-1} Let $A=A(V,R)\in\Omega$ and the corresponding quasipotential $Q$ is not a cube and satisfies $n_1(Q)=n_2(Q)=1$. Then there is a basis $x,y,z$ in $V$ such that $Q=xyz$ and $R$ is spanned by $xy$, $yz$, $a_0xx+a_1xz+a_2yx+a_3yy+a_4zx+a_5zy+a_6zz$ with some $a=(a_0,\dots,a_6)\in\K^7$ satisfying $a_1\neq 0$.
\end{lemma}

\begin{proof} Since $n_1(Q)=n_2(Q)=1$, $Q=xuv$, where $x,u,v$ are non-zero elements of $V$. Considering possible linear dependencies between $x$, $u$ and $v$, one easily sees that $y,z\in V$ can be  chosen in such a way that $x,y,z$ is a basis in $V$ and $Q$ has one of the following forms: $xxx$, $xxz$, $xzz$, $xzx$, $x(x-z)z$, $xyz$. If $Q=xxx$, $Q=xxz$ or $Q=xzz$, $R_Q\subseteq R$ contains a square and therefore $Q$ must be a cube according to Lemma~\ref{cube}. Since this is not the case, $Q$ can not be $xxx$ or $xxz$ or $xzz$. If $Q=xzx$ or $Q=x(x-z)z$, $R_Q$ intersects $M^2$ by at 2-dimensional space, where $M$ is spanned by $x$ and $z$. By Lemma~\ref{2-2}, $A\notin \Omega$, which is a contradiction. This leaves only one option:
$$
Q=xyz.
$$
In this case $R_Q$ is spanned by $xy$ and $yz$. Thus $R$ must be spanned by $xy$, $yz$ and $a_0xx+a_1xz+a_2yx+a_3yy+a_4zx+a_5zy+a_6zz$ for some non-zero $(a_0,\dots,a_6)\in\K^7$. It remains to show that $a_1\neq 0$. Assume the contrary. That is, $a_1=0$. If $a_0a_2a_3\neq 0$, we can scale to get $a_0=a_2=-a_3=1$. The defining relations of $A$ now take the shape $xy=yz=0$ and $xx=pyx-yy+zx+azy+bzz$ with $a,b,p\in\K$. We use the usual left-to-right degree lexicographical ordering on monomials assuming $x>y>z$. A direct computation shows that for $(a,b,p)$ from a Zarisski open subset of $\K^3$, the leading monomials of the members of the Gr\"obner basis of the ideal of relations of $A$ of degrees up to $5$ are $xx$, $xy$, $yz$, $yyy$, $xzx$, $xzyy$, $zzyy$, $xzyx$, $zzzzy$ and $xzzyx$. More specifically, for this to be the case, one needs $a$, $b$, $b+a^2$, $a+p$ and $p^2b-pa-b-a^2$ to be non-zero. In this case we get $\dim A_5=22$. By Lemma~\ref{minhs}, $\dim A_5\geq 22$ for all $a,b,p\in\K$. This means that $\dim A_5\geq 22$ whenever $a_0a_2a_3\neq 0$ (and $a_1=0$ still). Again, Lemma~\ref{minhs} yields that $\dim A_5\geq 22$ if $a_1=0$ regardless what other $a_j$ are. Since this inequality is incompatible with $A\in\Omega$, we arrive to a contradiction, which proves that $a_1\neq 0$.
\end{proof}

We are ready to prove Part~VIII of Theorem~\ref{main}. According to Lemma~\ref{algN}, the only thing which we have to verify is that every $A=A(V,R)\in\Omega$ such that the corresponding quasipotential $Q$ is not a cube and satisfies $n_1(Q)=n_2(Q)=1$ is isomorphic to one of the algebras in (N1--N23). Using Lemma~\ref{omepl1-1} and a suitable scaling to turn $a_1$ into $1$, we can assume that $A$ is given by generators $x,y,z$ and relations $xy$, $yz$ and $a_1xx+xz+a_2yx+a_3yy+a_4zx+a_5zy+a_6zz$ for some $a\in\K^6$. Now using Lemma~\ref{isoN}, Lemma~\ref{gbq2-2},  Remark~\ref{pnqnrem} and considering possible distributions of zeros among the parameters $a_j$, it is easy to see that $A$ is isomorphic to one the algebras in (N1--N23). Indeed, Lemma~\ref{isoN} describes possible isomorphisms within our family of algebras, Lemma~\ref{gbq2-2} pinpoints the exceptional parameters and Remark~\ref{pnqnrem} translates the description of these exceptions to the form used in Theorem~\ref{main} (that is, without polynomials $p_k$).

\section{Proof of parts II and III of Theorem~\ref{main}}

\begin{lemma}\label{algML} The algebras in {\rm (M1--M8)} and {\rm (L1--L8)} belong to $\Omega^0$ and therefore are $PBW_B$ and Koszul. Algebras in {\rm (M1--M8)} and {\rm (L1--L8)} with different labels are non-isomorphic and the isomorphism conditions of Theorem~$\ref{main}$ within algebras with a given label from {\rm (M1--M8)} and {\rm (L1--L8)}are satisfied.
\end{lemma}

\begin{proof} Algebras in {\rm (M1--M8)} and {\rm (L1--L8)} belong to $\Omega^0$ according to Lemma~\ref{pbwb}. It remains to deal with isomorphisms. Recall that we already know that algebras from Theorem~\ref{main} with different letters in the labels are non-isomorphic. Thus the families (M1--M8) and (L1--L8) can be treated separately. First, observe that for every algebra in (M1--M8), the rank one quadratic relations are non-zero members of the linear span of $xy$ and $xz$. Thus any linear substitution providing an isomorphism between two algebras from (M1--M8) must preserve the linear span of $x$ as well as the linear span of $y$ and $z$. Now for every $A$ from (M1--M8), the quotient $B=A/I$ by the ideal generated by $x$ is given by generators $y,z$ and one quadratic relation: $yz-azy$ with $a\in\K^*$ for $A$ from (M1--M3), $yz-zy-zz$ for $A$ from (M4--M5) and $zy$ for $A$ from (M7--M8). Since our substitution sends $x$ into its own scalar multiple, it must provide an isomorphism of the corresponding algebras $B$ as well. By Lemma~\ref{1-dim}, we now have that algebras from different groups (M1--M3), (M4--M6) and (M7--M8) can not be isomorphic. Furthermore, the only automorphisms of each of the two algebras $\K\langle y,z\rangle/{\rm Id}(zy)$ and
$\K\langle y,z\rangle/{\rm Id}(yz-zy-zz)$ are given by scalings. Just by looking at the sets of monomials involved in defining relations, it now becomes obvious that algebras in (M4--M8) are pairwise non-isomorphic. These considerations take (M4--M8) out of the picture, leaving us to deal with (M1--M3). Now for each presentation (M1--M3), it is easy to see that a scaling either throws it outside the class (M1--M3) or provides an automorphism. Other than scalings, the only other $y,z$ substitutions that transform $yz-azy$ with $a\in\K^*$ into $yz-bzy$ with $b\in\K^*$ are certain scalings composed with the swapping of $y$ and $z$. Such a substitution (combined with a scaling of $x$) throws every presentation from (M2) outside the class (M1--M3). As for each of (M1) and (M3), such a substitution provides an isomorphism of $A$ and an algebra from the same class with the parameter $a$ replaced by $\frac1a$. This completes the proof of the isomorphism statement for (M1--M8). Since algebras in (L1--L8) are isomorphic to algebras from (M1--M8) with the opposite multiplication, the isomorphism statement for (L1--L8) follows as well.
\end{proof}

\begin{lemma}\label{Q12} Let $A=A(V,R)\in\Omega$ be such that the corresponding quasipotential $Q=Q_A$ satisfies $n_1(Q)=1$ and $n_2(Q)=2$. Then $A$ is isomorphic to
to a $\K$-algebra given by generators $x,y,z$ and three quadratic relations from {\rm (M1--M8)} of Theorem~$\ref{main}$.
\end{lemma}

\begin{proof}Since $n_1(Q)=1$ and $n_2(Q)=2$, there is $x\in V$ such that $Q=xf$, where $f\in V^2$ has rank $2$. Let $M$ be the (unique) $2$-dimensional subspace of $V$ such that $f\in VM$.

{\bf Case 1:} \ $x\notin M$. Then $Q=xxu+xg$, where $u\in M$ and $g\in M^2$. Clearly, $g\neq 0$: otherwise $f$ has rank 1. By Lemma~\ref{1-dim}, a basis $y,z$ in $M$ can be chosen in such a way that $g\in\{zz,yz,yz-\alpha zy,yz-zy-zz\}$, where $\alpha\in\K^*$. Clearly, $x,y,z$ is a basis in $V$.

{\bf Case 1a:} \ $g=zz$ or $g=yz$. In this case $Q=pxxy+qxxz+xwz$ with $p,q\in\K$, where $w\in\{y,z\}$. Then $R_Q$ is spanned by $pxy+qxz+wz$, $pxx$ and $qxx+xw$. If $p\neq 0$, then $xx\in R_Q\subseteq R$. By Lemma~\ref{cube}, $Q$ is a scalar multiple of $xxx$, which is obviously not the case. Hence $p=0$. Then $f=qxz+wz$ has rank $1$. This contradiction completes Case~1a.

{\bf Case 1b:} \ $g=yz-\alpha zy$ with $\alpha\in\K^*$. In this case $Q=pxxy+qxxz+xyz-\alpha xzy$ with $p,q\in\K$. After the linear substitution $x\to x$, $y\to y-qx$, $z\to \frac{px-z}{\alpha}$, $Q$ acquires the shape $Q=-\frac1\alpha xyz+xzy-qxzx+\frac p\alpha xyx$. Depending on which of $p,q$ is or is not zero, a scaling turns $Q$ into one of the following forms: $Q=xyz-\alpha xzy$, $Q=xyz-\alpha xzy+xyx$, $Q=xyz-\alpha xzy+xzx$ or $Q=xyz-\alpha xzy+xyx+xzx$. The families of quasipotentials  $Q=xyz-\alpha xzy+xyx$ for $\alpha\in\K^*$ and $Q=xyz-\alpha xzy+xzx$ for $\alpha\in\K^*$ turn into each other by swapping $y$ and $z$ together with a scaling. Thus we end up with three families of quasipotentials: $Q=xyz-\alpha xzy$, $Q=xyz-\alpha xzy+xzx$ and $Q=xyz-\alpha xzy+xyx+xzx$ with $\alpha\in\K^*$. If $\alpha=1$, the quasipotential $Q=xyz-\alpha xzy+xyx+xzx$ transforms into $Q=xyz-xzy+xzx$ by means of the substitution  $x\to x$, $y\to y-z$, $z\to z$. Thus we can exclude $\alpha=1$ from the last family. For the three families, we have obtained, $R_Q=R$ is spanned by $\{xy,xz,yz-\alpha zy\}$, $\{xy,xz,yz-\alpha zy+zx\}$ and $\{xy,xz,yz-\alpha zy+zx+yx\}$ respectively. Thus $A$ is isomorphic to an algebra from (M1--M3).

{\bf Case 1c:} \ $g=yz-zy-zz$. In this case $Q=pxxy+qxxz+xyz-xzy-xzz$ with $p,q\in\K$. After the substitution $x\to x$, $y\to y+z-qx$, $z\to z+px$, $Q$ acquires the shape $Q=xyz-xzy-xzz+qxzx+pxyx-pxzz-p^2xxx$. If $p\neq 0$, a substitution $x\to x$, $y\to y+sz$, $z\to z$ with an appropriate $s\in\K$ kills $q$. After that, a scaling turns $p$ into $1$. This yields $Q=xyz-xzy-xzz+xyx-xzz-xxx$. After the sub $x\to x$, $y\to x+y+z$, $z\to z$ we arrive to $Q=xyz-xzy-xzz+xyx$, the quasipotential of (M4). If $p=0$ and $q\neq 0$, a scaling turns $q$ into $1$: $Q=xyz-xzy-xzz+xzx$, which is the quasipotential from (M5). Finally, if $p=q=0$, we have $Q=xyz-xzy-xzz$, the quasipotential from (M6). This concludes Case~1.

{\bf Case 2:} \ $x\in M$.

Now we can pick $y,z\in V$ such that $x,y,z$ is a basis in $V$, while $x,y$ form a basis in $M$. Then $Q=xzu+xg$, where $u\in M$ and $g\in M^2$.

{\bf Case 2a:} \ $u=0$. Then $g=f$ must have rank 2. Hence $R_Q\subset R\cap M^2$ is at least $2$-dimensional. By Lemma~\ref{2-2}, $A\notin\Omega$, which is a contradiction.

{\bf Case 2b:} \ $x$ and $u$ are linearly independent. Without loss of generality, we can then assume that $y=u$. Then $Q=xzy+axxx+bxxy+cxyx+dxyy$ with $a,b,c,d\in\K$. The substitution $x\to x$, $y\to y$, $z\to z+sx+ty$ with appropriate $s,t\in\K$ kills $b$ and $d$. This makes $Q=xzy+axxx+cxyx$ with $a,c\in\K$. If $c=0$, then $a\neq 0$ (otherwise $n_2(Q)=1$) and therefore $xx\in R_Q\subseteq R$. This, however, can not happen according to Lemma~\ref{cube}. Scaling, we can make $c=1$ and $a\in\{0,1\}$.
If $a=0$, $Q=xzy+xyx$ and $R$ is spanned by $\{zy+yx,xz,xy\}$. If $a=1$, $Q=xzy+xyx+xxx$ and $R$ is spanned by $\{zy+yx+xx,xz,xy+xx\}$. The first algebra is (M7), while the second is isomorphic to (M8): just use the sub $x\to x$, $y\to y-x$, $z\to z$.

{\bf Case 2c:} \ $u$ and $x$ are linearly dependent and $u\neq 0$. Without loss of generality, we can assume $u=x$.
Then $Q=xzx+axxx+bxxy+cxyx+dxyy$ with $a,b,c,d\in\K$. The substitution $x\to x$, $y\to y$, $z\to z+sx+ty$ with appropriate $s,t\in\K$ kills $a$ and $c$. This makes $Q=xzx+bxxy+dxyy$ with $b,d\in\K$.
If $d=0$, we have $xx\in R$, which can not happen according to Lemma~\ref{cube}. Thus $d\neq 0$. A normalization turns $d$ into $1$: $Q=xzx+bxxy+xyy$. The sub $x\to x$, $y\to y-bx$, $z\to z$ turns $Q$ into $Q=xzx+xyy-bxyx$. Then $R$ is spanned by $xy$, $xz$ and $yy+zx-byx$.
If $b\neq 0$, we use the usual left-to-right degree-lexicographical ordering assuming $x>y>z$ to see that $xy$, $xz$, $yy-byx+dzx$, $yyy$ and $yyz$ form a Gr\"obner basis of the ideal of relations of $A$. This yields $\dim A_4=17$ provided $b\neq 0$. By Lemma~\ref{minhs}, $\dim A\geq 17$ for all $b$, contradicting $A\in\Omega$. This concludes Case~2.
\end{proof}

Part II of Theorem~\ref{main} follows straight away from Lemmas~\ref{algML} and \ref{Q12}. Part III of Theorem~\ref{main} is equivalent to Part II: just pass to the opposite multiplication.

\section{Proof of Part IV of Theorem~\ref{main}}

First we prove the following lemma, which deals with a family of algebras containing (S19) and (S20)

\begin{lemma}\label{q45-1} Let $A=A^{a,b}$ be the quadratic algebra given by the generators $x,y,z$ and the relations $xx-xz-azx-bzz$, $xy+\frac{b}{a}zz$ and $yy-yz-\frac1a zy-\frac{b}{a^2}zz$ with $a\in\K^*$ and $b\in\K$. Then $A^{a,b}$ and $A^{a',b'}$ are non-isomorphic provided $(a,b)\neq (a',b')$. If $b=0$, then $A\in \Omega^0$. If $b\neq 0$ and $p_n(a,b)=0$ for some $n\in\N$ $($again, $p_n$ are polynomials defined in $(\ref{pnqn}))$, then $A\notin\Omega$. If $b\neq0$ and $p_n(a,b)\neq 0$ for all $n\in\N$, then $A\in\Omega^+$, $A$ is Koszul and $A$ is non-$PBW_B$.
\end{lemma}

In order to prove above lemma we need some preparation.

\begin{lemma}\label{q3dual} Let $A$ be the quadratic algebra given by the generators $x,y,z$ and the relations $xx-xz-azx-bzz$, $xy+\frac{b}{a}zz$ and $yy-yz-\frac1a zy-\frac{b}{a^2}zz$ with $a\in\K^*$ and $b\in\K$. Then the following statements hold$:$
\begin{itemize}\itemsep-2pt
\item[\rm(1)]$\dim A_n\geq \frac{(n+1)(n+2)}{2}$ for all $n\in\Z_+;$
\item[\rm(2)]if $(a,b)\neq (1,-1)$, then $H_{A^!}=(1+t)^3;$
\item[\rm(3)]if $b\neq 0$ and $(a,b)\neq (1,-1)$, then the $($right module$)$ Koszul complex of $A$ is given by
\begin{equation}\label{q3kos}
\!\!\!\!\!\!\!\!\!\!\!\!\!\!\!\!\!\!\!\begin{array}{l}
\displaystyle \qquad\qquad0\longrightarrow A\mathop{\longrightarrow}^{d_3} A^3
\mathop{\longrightarrow}^{d_{2}} A^3\mathop{\longrightarrow}^{d_1} A\mathop{\longrightarrow}^{d_0} \K\to 0
\\
\text{where $d_0$ is the augmentation,}\ \ d_1(u,v,w)=xu+yv+zw,\ \ d_3(u)=(yu,a(y-z)u,0)
\\
\text{and}\ \ d_2(u,v,w)=((x-z)u+yv,(y-z)w,-(ax+bz)u+\frac bazv-(\frac1ay+\frac{b}{a^2}z)w).
\end{array}
\end{equation}
\item[\rm(4)]if $(a,b)=(1,-1)$, then $\dim A_3=12$ and therefore $A\notin\Omega;$
\item[\rm(5)]if $b=0$, then $A\in\Omega^0$.
\end{itemize}
\end{lemma}

\begin{proof} Since proving the above statements in the case when $\K$ is replaced by any field extension will yield their validity for original $\K$, we can without loss of generality, assume that $\K$ is uncountable.

If $b=0$, then with respect to the right-to-left degree lexicographical ordering assuming $z>y>x$, the leading monomials of the defining relations are $yz$, $xz$ and $xy$. By Lemma~\ref{ome0}, $A\in\Omega^0$. This verifies (5).

If $b\neq 0$ and $(a,b)\neq (1,-1)$, then the defining relations of $A^!$ are $xx+\frac1azx$, $xy+zx+zy-\frac abzz$, $xz-\frac1azx$, $yx$, $yy+azy$ and $yz-azy$. A direct computation shows that with respect to the
left-to-right degree lexicographical ordering assuming $x>y>z$, the reduced Gr\"obner basis of the ideal of relations of $A^!$ consists of the defining relations together with $zzx$, $zzy$ and $zzzz$. The complete list of normal words consists of $1$, $x$, $y$, $z$, $zx$, $zy$, $zz$ and $zzz$ yielding $H_{A^!}=(1+t)^3$. If $b=0$, we already know that $A\in\Omega^0$ giving $H_{A^!}=(1+t)^3$. This takes care of (2).
Furthermore, the normal words in $A^!$ furnish us with a graded linear basis in $A^!$, while the above Gr\"obner basis provides the corresponding structural constants. Now a routine computation yields (3). By the already verified (2), $H_{A^!}=(1+t)^3$ provided $(a,b)\neq (1,-1)$. We also know that $A$ is Koszul if $b=0$. By Lemma~\ref{dri3} (the corresponding variety $W$ is the affine space $\K^2$), (1) is satisfied.

Finally, (4) is easily verified by a direct computation.
\end{proof}

\begin{lemma}\label{gbq3} Let $A$ be the quadratic algebra given by the generators $x,y,z$ and the relations $xx-xz-azx-bzz$, $xy+\frac{b}{a}zz$ and $yy-yz-\frac1a zy-\frac{b}{a^2}zz$ with $a,b\in\K^*$ such that $(a,b)\neq(1,-1)$. Let also $n\in\N$ be such that $p_k(a,b)\neq 0$ for $0\leq k<n$ and $q_k(a,b)\neq 0$ for $0\leq k\leq n$. Then the equalities
\begin{equation}\label{gb31}
\begin{array}{l}
xz^kx-\alpha_kxz^{k+1}-u_kz^{k+1}x-s_kz^{k+2}=0;
\\
xz^ky-\beta_kxz^{k+1}-v_kz^{k+1}y-t_kz^{k+2}=0;
\\
yz^ky-\gamma_kyz^{k+1}-w_kz^{k+1}y-r_kz^{k+2}=0
\end{array}
\end{equation}
hold in $A$ for $0\leq k\leq n$, where $\alpha_k,\beta_k,\gamma_k,u_k,v_k,w_k,s_k,t_k,r_k\in\K$,  $\alpha_k=\frac{p_k(a,b)}{q_k(a,b)}$, $\gamma_k=\frac{p_k(1/a,b/a^2)}{q_k(1/a,b/a^2)}$, while the rest of the numbers are defined inductively by $u_0=a$, $w_0=\frac1a$, $s_0=b$, $r_0=\frac b{a^2}$, $t_0=-\frac ba$, $\beta_0=v_0=0$, $u_{k+1}=-\frac{s_k+au_k}{\alpha_k-a}$, $s_{k+1}=\frac{s_k-bu_k}{\alpha_k-a}$, $w_{k+1}=-\frac{a^2\gamma_k+b}{a^2\gamma_k-a}$, $r_{k+1}=\frac{a^2r_k-bw_k}{a^2\gamma_k-a}$, $\beta_{k+1}=-\frac{b}{a\alpha_k}$, $v_{k+1}=-\frac{s_k}{\alpha_k}$ and $t_{k+1}=\frac{bu_k}{a\alpha_k}$ for $0\leq k<n$ $($note that by the conditions imposed upon $a,b$, we never divide by zero in these fractions$)$. Moreover, the equalities
\begin{equation}\label{gb32}
\begin{array}{l}
(\alpha_n-a)xz^{n+1}x-(\alpha_n+b)xz^{n+2}+(s_n+au_n)z^{n+2}x+(bu_n-s_n)z^{n+3}=0;
\\ \textstyle
\alpha_nxz^{n+1}y+\frac{b}{a}xz^{n+2}+s_nxz^{n+2}y-\frac ba u_nz^{n+3}=0;
\\ \textstyle
(\gamma_n-\frac1a)yz^{n+1}y-(\gamma_n+\frac{b}{a^2})yz^{n+2}+(r_n+\frac{w_n}{a})z^{n+2}y
+(\frac{bw_n}{a^2}-r_n)z^{n+3}=0
\end{array}
\end{equation}
hold in $A$.
Furthermore, the left-hand sides of $(\ref{gb31})$ for $0\leq k\leq n$ together with $yz$ are all the elements of degree up to $n+2$ of the reduced Gr\"obner basis of the ideal of relations of $A$ and $\dim A_j=\frac{(j+1)(j+2)}{2}$ for $0\leq j\leq n+2$.
\end{lemma}

\begin{proof} We shall prove (\ref{gb31}) for $0\leq k\leq n$ inductively. The defining relations justify (\ref{gb31}) for $k=0$, providing the basis of induction. Assume that $0\leq k\leq n-1$ and (\ref{gb31}) holds for $k$. We shall verify (\ref{gb31}) for $k$ replaced by $k+1$, which will complete the inductive proof. We do this by resolving certain overlaps as in the Buchberger algorithm using the induction hypothesis and the defining relations of $A$ in the process. The monomials to which we apply reduction are indicated by underlining:
$$
\begin{array}{l}
\underline{xz^k\phantom{y}}\!\!\!\!\underline{xx} \to 0=\alpha_kxz^kx+u_kz^{k+1}\underline{xx}+s_kz^{k+2}x-\underline{xz^kx}z-axz^{k+1}x-bxz^{k+2}
\\
=(\alpha_k-a)xz^kx+s_kz^{k+2}x-bxz^{k+2}+u_kz^{k+1}xz+au_kz^{k+2}x+bu_kz^{k+3}
\\
\quad-\alpha_kxz^{k+2}-u_kz^{k+1}xz-s_kz^{k+3},
\end{array}
$$
which yields
\begin{equation}\label{eq21}
(\alpha_k-a)xz^{k+1}x-(\alpha_k+b)xz^{k+2}+(s_k+au_k)z^{k+2}x+(bu_k-s_k)z^{k+3}=0\ \ \text{in $A$}.
\end{equation}
From the definition of $\alpha_k$ and the recurrent formulas for polynomials $p_k$ and $q_k$ it follows that
$\alpha_{k+1}=\frac{\alpha_k+b}{\alpha_k-a}$. By the assumptions (recall that $k<n$) it follows that $\alpha_k\neq a$. Thus dividing (\ref{eq21}) by $\alpha_k-a$ and using the definition of $u_j$ and $s_j$, we see that the first equality in (\ref{gb31}) is satisfied. We proceed in the same manner:
$$
\begin{array}{l}\textstyle
\underline{yz^k\phantom{y\hhh}}\!\!\!\!\underline{yy} \to 0=\gamma_kyz^ky+w_kz^{k+1}\underline{yy}+r_kz^{k+2}y-\underline{yz^ky}z-\frac1ayz^{k+1}y-\frac{b}{a^2}yz^{k+2}
\\  \textstyle
=(\gamma_k-\frac1a)yz^ky+r_kz^{k+2}y-\frac{b}{a^2}yz^{k+2}+w_kz^{k+1}yz+\frac{w_k}{a}z^{k+2}y
\\
\quad+\frac{bw_k}{a^2}z^{k+3}-\gamma_kxz^{k+2}-w_kz^{k+1}xz-r_kz^{k+3},
\end{array}
$$
which yields
\begin{equation}\label{eq22}\textstyle
(\gamma_k-\frac1a)yz^{k+1}y-(\gamma_k+\frac{b}{a^2})yz^{k+2}+(r_k+\frac{w_k}{a})z^{k+2}y+(\frac{bw_k}{a^2}
-r_k)z^{k+3}=0\ \ \text{in $A$}.
\end{equation}
From the definition of $\gamma_k$ and the recurrent formulas for polynomials $p_k$ and $q_k$ it follows that
$\gamma_{k+1}=\frac{a^2\gamma_k+b}{a^2\gamma_k-a}$. By the assumptions (recall that $k<n$) it follows that $\gamma_k\neq \frac1a$. To see this, one has to observe that from the explicit formulae for $p_j$ and $q_j$ given in Remark~\ref{pnqnrem} it follows that $q_j(a,b)=0\iff q_j(1/a,b/a^2)=0$. Thus dividing (\ref{eq22}) by $\gamma_k-\frac1a$ and using the definition of $w_j$ and $r_j$, we see that the third equality in (\ref{gb31}) is satisfied. Finally, we deal with the third overlap:
$$
\begin{array}{l} \textstyle
\underline{xz^k\phantom{x}}\!\!\!\!\underline{xy} \to 0=\alpha_kxz^{k+1}y+u_kz^{k+1}\underline{xy}+s_kz^{k+2}y
+\frac{b}{a}xz^{k+2}
\\
=\alpha_kxz^{k+1}y-\frac{bu_k}{a}z^{k+3}+s_kz^{k+2}y
+\frac{b}{a}xz^{k+2},
\end{array}
$$
which yields
\begin{equation}\label{eq32}\textstyle
\alpha_kxz^{k+1}y+\frac{b}{a}xz^{k+2}+s_kz^{k+2}y-\frac{bu_k}{a}z^{k+3}
=0\ \ \text{in $A$}.
\end{equation}
By assumptions $\alpha_k\neq 0$. Dividing (\ref{eq32}) by $\alpha_k$ and using the definition of $\beta_j$, $v_j$ and $t_j$, we see that the second equality in (\ref{gb31}) is satisfied. This completes the inductive proof of (\ref{gb31}) for $0\leq k\leq n$. Now (\ref{gb32}) is the combination of (\ref{eq21}), (\ref{eq22}) and (\ref{eq32}) for $k=n$, which hold since (\ref{gb31}) holds for $k=n$.

Now the only degree $\leq n+2$ monomials, which do not contain $xz^kx$, $yz^ky$ and $xz^ky$ for $0\leq k\leq n$ as submonomials are the words of the shape $z^j$, $z^jxz^m$, $z^jyz^m$ and $z^jyz^mxz^p$ of degree $\leq n+2$ with $j,m,p\in\Z_+$. The number of such monomials of degree $k$ is exactly $\frac{(k+1)(k+2)}{2}$. Since by Lemma~\ref{q2dual}, $\dim A_k\geq \frac{(k+1)(k+2)}{2}$, the left-hand sides of $(\ref{gb31})$ for $0\leq k\leq n$ must comprise all the elements of degree up to $n+2$ of the reduced Gr\"obner basis of the ideal of relations of $A$ and we must have $\dim A_k=\frac{(k+1)(k+2)}{2}$ for $0\leq k\leq n+2$. \end{proof}

\begin{lemma}\label{gbq3-1} Let $A$ be the quadratic algebra given by the generators $x,y,z$ and the relations $xx-xz-azx-bzz$, $xy+\frac{b}{a}zz$ and $yy-yz-\frac1a zy-\frac{b}{a^2}$ with $a,b\in\K^*$. If $p_k(a,b)\neq 0$ for all $k\in\N$, then $A\in\Omega^+$ and $A$ is Koszul. On the other hand, if $p_k(a,b)=0$ for some $k\in\N$,  then $A\notin\Omega$.
\end{lemma}

\begin{proof} First, assume that $p_k(a,b)\neq 0$ for all $k\in\N$. Note that $p_1(a,b)\neq 0$ yields $b\neq -1$ and therefore $(a,b)\neq (1,-1)$.

{\bf Case 1:} $q_k(a,b)\neq 0$ for all $k\in\N$.

By Lemma~\ref{gbq3}, $H_A=(1-t)^{-3}$ and the leading monomials of the reduced Gr\"obner basis of the ideal of relations of $A$ are $xz^ky$, $xz^kx$ and $yz^ky$ for $k\in\Z_+$. Since none of these monomials starts with the smallest variable $z$, we have $zu\neq 0$ for all non-zero $u\in A$. The shape of the Koszul complex of $A$ provided by Lemma~\ref{q3dual} now allows us to say that this complex is exact at its left-most term ($d_3$ is injective). By Lemma~\ref{koal3}, $A$ is Koszul.

{\bf Case 2:} $q_k(a,b)=0$ for some $k\in\N$.

Let $n\in\Z_+$ be the minimal non-negative integer for which $q_{n+1}(a,b)=0$. By Lemma~\ref{gbq3}, (\ref{gb31}) is satisfied for $0\leq k\leq n$ and (\ref{gb32}) holds. Since $q_{n+1}(a,b)=0$, we have $\alpha_n=a$ for $\alpha_j$ defined in Lemma~\ref{gbq3}. As we have already mentioned, from the explicit formula for $p_n$ and $q_n$ provided by Remark~\ref{pnqnrem} it follows that $q_{n+1}(1/a,b/a^2)=0$, from which one sees that $\gamma_n=\frac1a$. Next, observe that $a+b\neq 0$. Indeed, if $a+b=0$, then $q_k(a,b)=p_k(a,b)\neq 0$, contradicting the assumption of Case~2. Taking this into account, we can rewrite (\ref{gb32}) as follows:
\begin{equation}\label{gb32-2}
\begin{array}{l}\textstyle
{\vrule height0pt depth8pt width0pt}xz^{n+2}-\frac{s_n+au_n}{a+b}z^{n+2}x-\frac{bu_n-s_n}{a+b}z^{n+3}=0;
\\ \textstyle
{\vrule height0pt depth8pt width0pt}xz^{n+1}y+\frac{b}{a^2}xz^{n+2}+\frac{s_n}{a}xz^{n+2}y-\frac{bu_n}{a^2}z^{n+3}=0;
\\ \textstyle
yz^{n+2}-\frac{a^2r_n+aw_n}{a+b}z^{n+2}y+\frac{bw_n-a^2r_n}{a+b}z^{n+3}=0
\end{array}
\end{equation}
Now we observe that the number of monomials of degree $m$, which do not contain any of $yz^ky$, $xz^kx$ for $0\leq k\leq n$, $xz^ky$ for $0\leq k\leq n+1$, $xz^{n+2}$ and $yz^{n+2}$ as submonomials is exactly $\frac{(m+1)(m+2)}{2}$ for every $m\in\Z_+$. Since by Lemma~\ref{q2dual}, $\dim A_m\geq \frac{(m+1)(m+2)}{2}$, $yz$ together with the left-hand sides of (\ref{gb31}) for $0\leq k\leq n$ and the left-hand sides of (\ref{gb32-2}) must form the reduced Gr\"obner basis of the ideal of relations of $A$ and we must have $H_A=(1-t)^{-3}$. Since none of the leading monomials of this basis starts with the smallest variable $z$, we have $zu\neq 0$ for all non-zero $u\in A$. Exactly as in Case~1, we can now use  Lemma~\ref{q2dual} and Lemma~\ref{koal3} to conclude that $A$ is Koszul. Note that in the second case the Gr\"obner basis turns out to be finite.

Since $A$ is non-$\rm PBW_B$ according to Lemma~\ref{non-pbwb-all}, $A\in\Omega^+$ in both cases. Thus $A\in\Omega^+$ and $A$ is Koszul provided $p_k(a,b)\neq 0$ for all $k\in\N$.

Assume now that $p_k(a,b)=0$ for some $k\in\N$. To complete the proof, we have to demonstrate that $A\notin\Omega$. Note that $p_1(1,-1)=0$ and that by Lemma~\ref{q3dual}, $A\notin\Omega$ if $(a,b)=(1,-1)$. Thus we can assume $(a,b)\neq (1,-1)$. Let $n$ be the smallest non-negative integer satisfying $p_{n+1}(a,b)=0$. Using the explicit formulas for $p_j$ and $q_j$ given in Remark~\ref{pnqnrem}, we see that $q_j(a,b)\neq 0$ and $q_j(1/a,b/a^2)\neq 0$ for $0\leq j\leq n+1$ ($(1,-1)$ is the only exception for this rule and was excluded because of that). Consider the variety $W_0=\{(s,t)\in\K^2:p_{n+1}(s,t)=0\}$. Clearly, $(a,b)\in W_0$. By Lemma~\ref{gbq3}, (\ref{gb31}) is satisfied for $0\leq k\leq n$ and (\ref{gb32}) holds. Furthermore, $\alpha_n=-b$. For $(a,b)\in W_0$ with finitely many exceptions, $\alpha_n\neq a$ and $\gamma_n\neq \frac1a$. By Lemma~\ref{gbq3}, the left hand sides of (\ref{gb31}) with $0\leq k\leq n$ provide all elements of degree $\leq n+2$ of the reduced Gr\"obner basis of the ideal of relations of $A$. The corresponding leading monomials are $xz^kx$, $xz^ky$ and $yz^ky$ with $0\leq k\leq n$. By Lemma~\ref{q3dual}, $\dim A_{n+3}\geq \frac{(n+4)(n+5)}{2}$. It follows that the left-hand sides of (\ref{gb32}) give the degree $n+3$ elements of the said Gr\"obner basis. With finitely many exceptions (in $W_0$) the new leading monomials are $xz^{n+1}x$, $yz^{n+1}y$ and $xz^{n+2}$. Proceeding the same way as in the proof of Lemma~\ref{gbq3}, we can get few more steps of the Gr\"obner basis. With finitely many exceptions (in $W_0$) the leading monomials turn out to be: $z^{n+2}yx$ and $yz^{n+2}y$ in degree $n+4$ and $z^{n+2}yzx$, $yz^{n+3}y$ and $yz^{n+3}x$ in degree $n+5$. This still is in agreement with the $\rm PBW_S$-series: $\dim A_j=\frac{(j+1)(j+2)}{2}$ for $j\leq n+5$. At degree $n+6$, we have to deal with all $13$ overlaps, to find out that (still with finitely many exceptions) the leading monomials of degree $n+6$ of the reduced Gr\"obner basis are $z^{n+2}yzzx$, $yz^{n+4}y$ and $yz^{n+4}x$. This yields $\dim A_{n+6}=\frac{(n+7)(n+8)}{2}+1$. Since this holds for all $(a,b)\in W_0$ with finitely many exceptions, Lemma~\ref{minhs} (applied to all irreducible components of $W_0$) yields $\dim A_{n+6}>\frac{(n+7)(n+8)}{2}$ for all $(a,b)\in W_0$, which includes the original pair $(a,b)$. Hence $A\notin\Omega$, as required.
\end{proof}

\begin{proof}[Proof of Lemma~$\ref{q45-1}$]
The fact that $A\in\Omega^0$ if $b=0$ follows from Lemma~\ref{q3dual}. That $A\notin\Omega$ if $b\neq 0$ and $p_n(a,b)=0$ for some $n\in\N$ and that $A\in\Omega^+$ and is Koszul if $b\neq 0$ and $p_n(a,b)\neq 0$ for all $n\in\N$ follows from Lemma~\ref{gbq3-1}. Now if $A=A^{a,b}$ and $A'=A^{a',b'}$ are isomorphic, then the isomorphism (a linear substitution) must turn the corresponding quasipotentials $Q$ and $Q'$ one to the other. It is easy to check that $Q$ and $Q'$ satisfy $E_1(Q)=E_1(Q')=\spann\{x,z\}$ and $E_2(Q)=E_2(Q')=\spann\{y,z\}$. Thus our substitution must have both $\spann\{x,z\}$ and $\spann\{y,z\}$ as invariant subspaces. Hence it is given by $x\to \alpha x+sz$, $y\to \beta y+tz$, $z\to \gamma z$ with $\alpha,\beta,\gamma\in\K^*$ and $s,t\in\K$. Without loss of generality, $\gamma=1$. Applying this kind of a sub to the defining relations, we see that we have the relation of the form $xy+czz$ with $c\in\K$ present only if $s=t=0$. Thus our sub has the form $x\to\alpha x$, $y\to \beta y$, $z\to z$. Now it is easy to see that the shape of defining relations is not preserved unless $\alpha=\beta=1$. This leaves the identity map only and therefore we must have $(a,b)=(a',b')$.
\end{proof}

Next, we move to the following family of algebras, which as we shall see, contains (S17) and (S18).

\begin{lemma}\label{q8-1iso} Let $A=A^b$ be the quadratic algebra given by the generators $x,y,z$ and the relations $xy+byz+zz$, $zx+(b-1)yz-bzy-zz$ and $yy+yz+bzy+zz$ with $b\in\K$. Then $A^b$ is non-isomorphic to $A^{b'}$ provided $b\neq b'$.
\end{lemma}

\begin{proof} The value $b=1$ is the only one for which $\dim A_3\neq 10$ (easily checked using Gr\"obner basis). Thus we can assume that $b\neq 1$ and $b'\neq 1$. For $b$ other than $1$, $A^b\in \Omega'$ and the corresponding quasipotential $Q^b$ satisfies $E_1(Q^b)=E_2(Q^b)=M=\spann\{y,z\}$. Thus every linear sub providing an isomorphism between $A^b$ and $A^{b'}$ must keep $M$ invariant. Analyzing the space $R^b$ of quadratic relations of $A^b$ (pay attention to how $x$ occurs), one sees that our shape of relations is not preserved unless the sub sends each of $y$ and $z$ to their scalar multiples. Without loss of generality (just normalize the relations), it sends $y$ to $y$ and $z$ to $az$  with $a\in\K^*$. Moreover, $R^b\cap M^2$ is the one-dimensional space spanned by $yy+yz+bzy+zz$. Thus our sub must transform $yy+yz+bzy+zz$ into a scalar multiple of $yy+yz+b'zy+zz$. The latter only happens when $a=1$ and $b=b'$. The result follows.
\end{proof}

\begin{lemma}\label{q6dual} Let $A=A^b$ be the quadratic algebra given by the generators $x,y,z$ and the relations $xy+byz+zz$, $zx+(b-1)yz-bzy-zz$ and $yy+yz+bzy+zz$ with $b\in\K$. Then the following statements hold true$:$
\begin{itemize}\itemsep-2pt
\item[\rm(1)]$\dim A_n\geq \frac{(n+1)(n+2)}{2}$ for all $n\in\Z_+;$
\item[\rm(2)]if $b\notin\{0,1\}$, then $H_{A^!}=(1+t)^3$ and the $($right module$)$ Koszul complex of $A$ is given by
\begin{equation}\label{q6kos}
\!\!\!\!\!\!\!\!\!\!\!\!\!\!\!\!\begin{array}{l}
\displaystyle \qquad\qquad0\longrightarrow A\mathop{\longrightarrow}^{d_3} A^3
\mathop{\longrightarrow}^{d_{2}} A^3\mathop{\longrightarrow}^{d_1} A\mathop{\longrightarrow}^{d_0} \K\to 0
\\
\text{where $d_0$ is the augmentation,}\ \ d_1(u,v,w)=xu+yv+zw,\ \ d_3(u)=(0,yu,(y+z)u)
\\
\text{and}\ \ d_2(u,v,w)=(yu,bzu+(b-1)zv+(y+z)w,zu+(x-by-z)v+(by+z)w).
\end{array}
\end{equation}
\item[\rm(3)]if $b=1$, then $A\notin\Omega.$
\end{itemize}
\end{lemma}

\begin{proof} Since proving the above statements in the case when $\K$ is replaced by any field extension will yield their validity for original $\K$, we can without loss of generality, assume that $\K$ is uncountable. In this proof we always use the left-to-right degree lexicographical ordering assuming $x>y>z$. The validity of (3) is a matter of routine verification: a direct Gr\"obner basis calculation yields $\dim A_3=11$.

Next, it is a matter of an easy calculation to see that $A^!$ is given by generators $x,y,z$ and the relations $xx$, $bxy+zy-bzz$, $xz$, $yx$, $byy-bzx-zy$ and $byz-b^2zx+(b-1)zy-b^2zz$ provided $b\notin\{0,1\}$. A direct computation shows that the reduced Gr\"obner basis of the ideal of relations of $A^!$ consists of the defining relations together with $zzx$, $zzy$ and $zzzz$. The complete list of normal words is: $1$, $x$, $y$, $z$, $zx$, $zy$, $zz$ and $zzz$ yielding $H_{A^!}=(1+t)^3$. Furthermore, the normal words in $A^!$ furnish us with a graded linear basis in $A^!$, while the above Gr\"obner basis provides the corresponding structural constants. Now a routine computation yields (\ref{q6kos}), which completes the proof of (2).

It remains to prove (1). First, observe that it is enough to prove (1) in the case $b\neq 0$: Lemma~\ref{minhs} fills the gap. If $b\neq 0$, we perform the sub $x\to bx$, $y\to y$, $z\to z$ turning the defining relations of $A^b$ into $bxy+byz+zz$, $bzx+(b-1)yz-bzy-zz$ and $yy+yz+bzy+zz$. Now replacing $b$ by $1/a$, we see that $A^b=A^{1/a}$ is isomorphic to the algebra $B=B^a$, given by the generators $x,y,z$ and the relations $xy+yz+azz$, $zx+(1-a)yz-zy-azz$ and $ayy+ayz+zy+azz$ with $a\in\K$. We already know that for $a\notin\{0,1\}$, $H_{B^!}=(1+t)^3$. Next, for $a=0$, the leading monomials of the defining relations of $B$ are $xy$, $zx$ and $zy$ with respect to the left-to-right degree lexicographical ordering assuming $x>z>y$. By Lemma~\ref{ome0}, $B^0\in\Omega^0$ and therefore $B^0$ is Koszul and $H_{B^!}=(1+t)^3$ for $a=0$. By Lemma~\ref{dri3} (the corresponding variety $W$ is $\K$), (1) is satisfied.
\end{proof}

\begin{lemma}\label{gb6-beta} Let $A=A^b$ be the quadratic algebra given by the generators $x,y,z$ and the relations $xy+byz+zz$, $zx+(b-1)yz-bzy-zz$ and $yy+yz+bzy+zz$ with $b\in\K$, $b\neq 1$. If $b\neq -3$, then $A^b$ is isomorphic to the quadratic algebra $B=B^\beta$ given by the generators $x,y,z$ and the relations $\beta xy-xz+\beta^2yy-(\beta^3+1)zy+zz$, $yx-\beta zx+\beta yy-(\beta^3+1)zy+\beta zz$ and $yz-\beta zy$ with $\beta\in\K^*$, $\beta^2\neq 1$, where $b$ and $\beta$ are related by $b=-1-\beta-\frac1\beta$.
If $b=-3$, then $A$ is isomorphic to the quadratic algebra $B$ given by the generators $x,y,z$ and the relations $xy+yy-\frac12 zz$, $yx+yy+zx+zy-\frac12zz$ and $yz-zy+\frac12zz$.
In both cases the linear substitution facilitating the isomorphism can be chosen preserving the linear span of $y$ and $z$.
\end{lemma}

\begin{proof} First, assume that $b\neq -3$. Since $\K$ is algebraically closed and $b\neq 1$, we can solve a quadratic equation to find $\beta\in\K^*$ such that $b=-1-\beta-\frac1\beta$. Since $b\notin\{1,-3\}$, we have $\beta^2\neq 1$. We start with the algebra $A=A^b$. The sub $x\to x+by+z$, $y\to y$, $z\to z$ turns the defining relations of $A$ into $xy-(b^2-1)zy-(b-1)zz$, $xy+(b-1)yz$ and $yy+yz+bzy+zz$. We follow up with scaling of $x$: $x\to (b-1)x$, $y\to y$, $z\to z$: the relations become $xy-(b+1)zy-zz$, $zx+yz$ and $yy+yz+bzy+zz$. Finally, we perform the $y,z$ sub $y\to \beta y-z$, $z\to y-\beta z$, $x\to x$ (non-degenerate since $\beta^2\neq 1$). This turns the relations into those of $B^\beta$. Thus the composition of these three linear subs provides an isomorphism between $A^b$ and $B^\beta$. Since each preserves the span of $y$ and $z$, so does the resulting isomorphism.

Now assume $b=-3$. In this case the defining relations of $A$ transformed by the sub $x\to x$, $y\to y$, $z\to y+z$ span the same space as the defining relations of $B$. Again, the span of $y$ and $z$ is preserved by the isomorphism.
\end{proof}

\begin{lemma}\label{gbq6} Let $B=B^\beta$ be the quadratic algebra given by the generators $x,y,z$
and the relations $\beta xy-xz+\beta^2yy-(\beta^3+1)zy+zz$, $yx-\beta zx+\beta yy-(\beta^3+1)zy+\beta zz$ and $yz-\beta zy$ with $\beta\in\K^*$, $\beta^2\neq 1$.
If $\beta^k\neq 1$ for all $k\in\N$, then $B\in\Omega^+$ and $B$ is Koszul. On the other hand, if $k$ is the smallest positive integer for which $\beta^{k+2}=1$, then $\dim B_{k+5}=\frac{(k+6)(k+7)}{2}+1$ and therefore $B\notin\Omega$.
\end{lemma}

\begin{proof} Throughout the proof we use the left-to-right degree lexicographical ordering on $x,y,z$ monomials assuming $x>y>z$. First, we use induction by $k$ to prove that, the equality
\begin{equation}\label{eqxy}\textstyle
xz^ky-\beta^{-k-1}xz^{k+1}+\beta^{k+1}z^kyy-\frac{\beta^3+1}{\beta}z^{k+1}y+\beta^{-k-1}z^{k+2}=0\ \ (k\in\Z_+)
\end{equation}
holds in $B$.

Indeed, the first defining relation yields the validity of (\ref{eqxy}) for $k=0$. Now if we assume that $k\in\N$ and (\ref{eqxy}) holds with $k$ replaced by $k-1$, then we resolve the overlap $xz^{k-1}yz=(xz^{k-1}y)z=xz^{k-1}(yz)$ using (\ref{eqxy}) for $k-1$ and the defining relations. After obvious cancelations (\ref{eqxy}) for $k$ follows.
Next, we show that
\begin{equation}\label{eqxx}
\begin{array}{l}
\textstyle
(1{-}\beta^{k+1})xz^kx{+}(\beta^{k-1}{-}1)(\beta^2{-}\beta^{-k})xz^{k+1}{+}(\beta^{k+3}{-1})(1{-}\beta^k)z^{k+1}x
{+}\beta^{k+2}(\beta^{k+1}{-}1)z^kyy
\\
+(1{-}\beta^{2k+2})(\beta^3{-}\beta{+}1)z^{k+1}y{+}(1{-}\beta^{k+1})
(\beta^2{+}\beta^{-1}{-}\beta^{-k}{-}\beta{-}\beta^{k+2})z^{k+2}=0\ \ (k\in\N)
\end{array}
\end{equation}
holds in $B$.

These equalities are obtained by resolving the overlap $xz^{k-1}yx=(xz^{k-1}y)x=xz^{k-1}(yx)$ using the already verified (\ref{eqxy}) and the defining relations. After obvious cancelations (\ref{eqxx}) follows.

Now assume that
\begin{equation}\label{assu}
\text{$n\in\Z_+$ and $\beta^j\neq 1$ for $1\leq j\leq n+1$.}
\end{equation}
Note that the only monomials that do not have any of the monomials of the shape $xz^jy$ for $j\in\Z_+$, $xz^jx$ for $j\in\N$, $yx$ or $yz$ as submonomials are exactly the monomials of the form $z^ky^m$ or $z^kx^pz^m$ with $k,m,\in\Z_+$, $p\in\N$. Observe that the number of such monomials of degree $s$ is precisely $\frac{(s+1)(s+2)}{2}$.

By Lemmas~\ref{q6dual} and~\ref{gb6-beta}, $\dim B_k\geq\frac{(k+1)(k+2)}{2}$ for all $k$. It follows that if (\ref{assu}) is satisfied, then the left-hand sides of (\ref{eqxy}) and (\ref{eqxx}) for $k\leq n$ together with $yx-\beta zx+\beta yy-(\beta^3+1)zy+\beta zz$ and $yz-\beta zy$ form the degree up to $n+2$ part of a Gr\"obner basis of the ideal of relations of $B$ and $\dim B_k=\frac{(k+1)(k+2)}{2}$ for $k\leq n+2$.

Now assume that $\beta^j\neq 1$ for all $j\in\N$. Then (\ref{assu}) is satisfied for all $n\in\Z_+$. Hence the left-hand sides of (\ref{eqxy}) and (\ref{eqxx}) for all $k$ together with $yx-\beta zx+\beta yy-(\beta^3+1)zy+\beta zz$ and $yz-\beta zy$ form a Gr\"obner basis of the ideal of relations of $B$ and $\dim B_k=\frac{(k+1)(k+2)}{2}$ for $k\leq n+2$. Then $B\in\Omega$. Since none of the leading monomials of the above Gr\"obner basis starts with the smallest variable $z$, $zu\neq 0$ for every non-zero $u\in B$. Hence $zu=yu=0$ fails for every non-zero $u\in B$. By Lemma~\ref{gb6-beta}, $B$ is isomorphic to the quadratic algebra $A$ given by the generators $x,y,z$ and the relations $xy+byz+zz$, $zx+(b-1)yz-bzy-zz$ and $yy+yz+bzy+zz$ with $b\in\K$ satisfying $b=-1-\beta-\frac1\beta$. Since the isomorphism is given by a linear sub preserving the span of $y$ and $z$, we have that $zu=yu=0$ fails for every non-zero $u\in B$. Hence the map $d_3$ from (\ref{q6kos}) is injective. By Lemma~\ref{koal3}, $A$ is Koszul and so is $B$. By Lemma~\ref{non-pbwb-all}, $A$ is not $\rm PBW_B$ and therefore $A$ and $B$ are in $\Omega^+$.

It remains to deal with the case when $\beta$ is a root of $1$. Let $m$ be the smallest positive integer satisfying $\beta^{m+1}=1$. Since $\beta^2\neq 1$, $m\geq 2$. Then (\ref{assu}) is satisfied for $n=m-1$. Then the left-hand sides of (\ref{eqxy}) and (\ref{eqxx}) for $k\leq m-1$ together with $yx-\beta zx+\beta yy-(\beta^3+1)zy+\beta zz$ and $yz-\beta zy$ form the degree up to $m+1$ part of a Gr\"obner basis of the ideal of relations of $B$ and $\dim B_k=\frac{(k+1)(k+2)}{2}$ for $k\leq m+1$. Plugging $\beta^{m+1}=1$ into (\ref{eqxx}) and (\ref{eqxy}) with $k=m$ and using the inequality $\dim B_{m+2}\geq \frac{(m+3)(m+4)}{2}$, we easily see that the left-hand sides of (\ref{eqxy}) and (\ref{eqxx}) for $k=m$ form the degree up to $m+2$ part of a Gr\"obner basis of the ideal of relations of $B$ and $\dim B_{m+2}=\frac{(m+3)(m+4)}{2}$. The said left-hand sides (up to a scalar multiple) are $xz^my-z^{m+1}x+z^myy-\frac{\beta^3+1}{\beta}z^{m+1}y+z^{m+2}$ and $xz^{m+1}-z^{m+1}x$. Resolving the overlap $xz^myx=xz^m(yx)=(xz^my)x$, we get the following equality in $B$:
$$
z^{m+1}\bigl(xx-\beta^2yy+(\beta-\beta^{-1})xz+(\beta^4+\beta^3-\beta^2+1)zy+(\beta^{-1}-1-\beta)zz\bigr)=0.
$$
Now it is easy to see that the number of degree $m+3$ monomials, which do not contain any of $yx$, $yz$, $xz^jy$ for $0\leq j\leq m$, $xz^jx$ for $1\leq j\leq m-1$, $xz^{m+1}$ or $z^{m+1}xx$ as a submonomial
is exactly $\frac{(m+4)(m+5)}{2}$, while the number of degree $m+4$ monomials with the same property is
$\frac{(m+5)(m+6)}{2}+1$. Since $\dim B_{m+3}\geq \frac{(m+4)(m+5)}{2}$, the degree $m+3$ part of the Gr\"obner basis consists of just one element: the left-hand side in the above display and $\dim B_{m+3}= \frac{(m+4)(m+5)}{2}$. Finally, dealing with all (this time) overlaps of the leading monomials of the Gr\"obner basis (so far) of degree $m+4$ (there are $m+3$ of them: $yz^{m+1}xx$, $xz^{m+1}xx$, $z^{m+1}xy$, $xzxz^{m+1}$, $xz^jxz^px$ and $xz^jxz^py$ with $j,p\in\N$, $j+p=m+1$, $p\geq2$), we find that all of them resolve without producing a degree $m+4$ member of the Gr\"obner basis. This part is easy but tedious: we leave the details to an interested reader.

As a result, we have $\dim B_{m+4}=\frac{(m+5)(m+6)}{2}+1$ and therefore $B\notin\Omega$. This completes the proof.
\end{proof}

\begin{lemma} \label{gbq6-3} Let $B$ be the quadratic algebra given by the generators $x,y,z$ and the relations $xy+yy-\frac12 zz$, $yx+yy+zx+zy-\frac12zz$ and $yz-zy+\frac12zz$. If the characteristic of $\K$ is $0$, then $B\in\Omega^+$ and $B$ is Koszul. If $\K$ has prime characteristic, then $B\notin\Omega$.
\end{lemma}

\begin{proof}Throughout the proof we use the left-to-right degree lexicographical ordering on $x,y,z$ monomials assuming $x>y>z$. First, we use induction by $k$ to prove that, the equality
\begin{equation}\label{eqxy1}\textstyle
xz^ky-\frac{k}{2}xz^{k+1}+z^kyy-kz^{k+1}y+\frac{(k-1)(k+2)}{4}z^{k+2}=0\ \ (k\in\Z_+)
\end{equation}
holds in $B$.

Indeed, the first defining relation yields the validity of (\ref{eqxy}) for $k=0$. Now if we assume that $k\in\N$ and (\ref{eqxy1}) holds with $k$ replaced by $k-1$, then we resolve the overlap $xz^{k-1}yz=(xz^{k-1}y)z=xz^{k-1}(yz)$ using (\ref{eqxy1}) for $k-1$ and the defining relations. After obvious cancelations (\ref{eqxy1}) for $k$ follows.

Next, we show that
\begin{equation}\label{eqxx1}
\begin{array}{l}
\textstyle
(k+1)xz^{k}x+\bigl(\frac{k(k+1)}{2}-1\bigr)xz^{k+1}-\bigl(\frac{(k+1)(k+2)}{2}-1\bigr)z^{k+1}x
\\
-(k+1)z^{k}yy+\frac{k(k+1)}{2}z^{k+1}y-\frac{(k-2)(k+1)(k+3)}{4}z^{k+2}=0\ \ (k\in\N)
\end{array}
\end{equation}
holds in $B$.

These equalities are obtained by resolving the overlap $xz^{k-1}yx=(xz^{k-1}y)x=xz^{k-1}(yx)$ using the already verified (\ref{eqxy1}) and the defining relations. After obvious cancelations (\ref{eqxx1}) follows.

Now assume that
\begin{equation}\label{assu1}
\text{$n\in\Z_+$ and ${\rm char}\,\K\notin\{2,\dots,n+1\}$.}
\end{equation}
Note that the only monomials that do not have any of the monomials of the shape $xz^jy$ for $j\in\Z_+$, $xz^jx$ for $j\in\N$, $yx$ or $yz$ as submonomials are exactly the monomials of the form $z^ky^m$ or $z^kx^pz^m$ with $k,m,\in\Z_+$, $p\in\N$. Note also that the number of such monomials of degree $s$ is precisely $\frac{(s+1)(s+2)}{2}$.

By Lemmas~\ref{q6dual} and~\ref{gb6-beta}, $\dim B_k\geq\frac{(k+1)(k+2)}{2}$ for all $k$. It follows that if (\ref{assu1}) is satisfied, then the left-hand sides of (\ref{eqxy1}) and (\ref{eqxx1}) for $k\leq n$ together with $yx+yy+zx+zy-\frac12zz$ and $yz-zy+\frac12zz$ form the degree up to $n+2$ part of a Gr\"obner basis of the ideal of relations of $B$ and $\dim B_k=\frac{(k+1)(k+2)}{2}$ for $k\leq n+2$.

Now assume that $\K$ has characteristic $0$. Then (\ref{assu1}) is satisfied for all $n\in\Z_+$. Hence the left-hand sides of (\ref{eqxy1}) and (\ref{eqxx1}) for all $k$ together with $yx+yy+zx+zy-\frac12zz$ and $yz-zy+\frac12zz$ form a Gr\"obner basis of the ideal of relations of $B$ and $\dim B_k=\frac{(k+1)(k+2)}{2}$ for $k\leq n+2$. Then $B\in\Omega$. Since none of the leading monomials of the above Gr\"obner basis starts with the smallest variable $z$, $zu\neq 0$ for every non-zero $u\in B$. Hence $zu=yu=0$ fails for every non-zero $u\in B$. By Lemma~\ref{gb6-beta}, $B$ is isomorphic to the quadratic algebra $A$ given by the generators $x,y,z$ and the relations $xy-3yz+zz$, $zx-4yz+3zy-zz$ and $yy+yz-3zy+zz$ ($b=-3$). Since the isomorphism is given by a linear sub preserving the span of $y$ and $z$, we have that $zu=yu=0$ fails for every non-zero $u\in B$. Hence the map $d_3$ from (\ref{q6kos}) is injective. By Lemma~\ref{koal3}, $A$ is Koszul and so is $B$. By Lemma~\ref{non-pbwb-all}, $A$ is not $\rm PBW_B$ and therefore $A$ and $B$ are in $\Omega^+$.

It remains to deal with the case when ${\rm char}\,\K=p$ is a prime number. Since we have excluded characteristics $2$ and $3$, $p\geq 5$. Then (\ref{assu1}) is satisfied for $n=p-1\geq 4$. Then the left-hand sides of (\ref{eqxy1}) and (\ref{eqxx1}) for $k\leq p-2$ together with form the degree up to $p$ part of a Gr\"obner basis of the ideal of relations of $B$ and $\dim B_k=\frac{(k+1)(k+2)}{2}$ for $k\leq p$. Using the equality ${\rm char}\,\K=p$, (\ref{eqxx1}) and (\ref{eqxy1}) with $k=p-1$ and using the inequality $\dim B_{p+1}\geq \frac{(p+2)(p+3)}{2}$, we easily see that the left-hand sides of (\ref{eqxy1}) and (\ref{eqxx1}) for $k=p-1$ form the degree $p+1$ part of a Gr\"obner basis of the ideal of relations of $B$ and $\dim B_{p+1}=\frac{(p+2)(p+3)}{2}$. The said left-hand sides have the same linear span as  $xz^{p-1}y+\frac12z^{p}x+z^{p-1}yy+z^{p}y-\frac12z^{p+1}$ and $xz^{p}-z^{p}x$. Resolving the overlap $xz^{p-1}yx=xz^{p-1}(yx)=(xz^{p-1}y)x$, we get the following equality in $B$:
$$
\textstyle z^{p}(xx-xz-yy-zy+\frac32zz)=0.
$$
Now it is easy to see that the number of degree $p+2$ monomials, which do not contain any of $yx$, $yz$, $xz^jy$ for $0\leq j\leq p-1$, $xz^jx$ for $1\leq j\leq p-2$, $xz^{p}$ or $z^{p}xx$ as a submonomial
is exactly $\frac{(p+3)(p+4)}{2}$, while the number of degree $p+3$ monomials with the same property is
$\frac{(p+4)(p+5)}{2}+1$. Since $\dim B_{p+2}\geq \frac{(p+3)(p+4)}{2}$, the degree $p+2$ part of the Gr\"obner basis consists of just one element: the left-hand side in the above display and $\dim B_{p+2}= \frac{(p+3)(p+4)}{2}$. Finally, as in the previous lemma all overlaps of the leading monomials of the Gr\"obner basis (so far) of degree $p+3$ (they are listed in the proof of the previous lemma: one just has to assume  $m+1=p$) resolve. As a result, $\dim B_{p+3}=\frac{(p+4)(p+5)}{2}+1$ and therefore $B\notin\Omega$. This completes the proof.
\end{proof}

\begin{lemma}\label{algS} The algebras in {\rm (S1--S20)} belong to $\Omega$ and are Koszul. The algebras in {\rm (S1--S16)} are $PBW_B$, while the algebras in  {\rm (S17--S20)} are non-$PBW_B$. Algebras in {\rm (S1--S20)} with different labels are non-isomorphic and the isomorphism conditions of Theorem~$\ref{main}$ within algebras with a given label from {\rm (S1--S20)} are satisfied.
\end{lemma}

\begin{proof} Algebras from (S1--S16) belong to $\Omega^0$ by Lemma~\ref{pbwb}. Algebras in (S17) and (S18) are Koszul and belong to $\Omega^+$ by Lemmas~\ref{gbq6-3}, \ref{gbq6} and~\ref{gb6-beta}, while algebras in (S19) and (S20) are Koszul and belong to $\Omega^+$ by Lemma~\ref{q45-1} and Remark~\ref{pnqnrem}. Algebras in (S17--S20) are non-PBW$_{\rm B}$ according to Lemma~\ref{non-pbwb-all}. It remains to deal with isomorphisms. As algebras in (S17--S20) are the only non-PBW$_{\rm B}$ ones of the batch, they can not be isomorphic to any of the algebras from (S1--S16). Thus the families of algebras in (S1--S16) and in (S17--S20) can be treated separately. Note that quasipotentials $Q$ for algebras in (S17--S18) satisfy $E_1(Q)=E_2(Q)$, while $E_1(Q)\neq E_2(Q)$ for algebras in (S19--S20). Thus algebras in (S17--S18) can not be isomorphic to any of the algebras in (S19--S20). Within the family (S17--S18), the required isomorphism statement now follows from Lemmas~\ref{q8-1iso} and~\ref{gb6-beta}, while the same for (S19--S20) is ensured by the isomorphism part of Lemma~\ref{q45-1}. These considerations take (S17--S20) out of the picture, leaving us with (S1--S16). The algebra (S1) is the only one of the lot with quasipotential $Q$ satisfying $E_1(Q)=E_2(Q)$. This singles it out, leaving us with (S2--S16). The quasipotential $Q$ for each algebra in (S2--S16) satisfies $E_1(Q)=\spann\{x,y\}$ and $E_2(Q)=\spann\{y,z\}$. Thus every isomorphism (linear substitution) between two algebras in (S2--S16) must leave each of these two spaces invariant. That is, such an isomorphism must send $y$ to its own scalar multiple, $x$ to a member of $\spann\{x,y\}$ and $z$ to a member of $\spann\{y,z\}$.

Keeping this in mind, observe that algebras in (S14--S16) are the only ones in (S2--S16) with no rank one elements in the space of quadratic relations. Note as well that the defining relations of each algebra in (S14--S16) contain $xz+yy$, one rank 2 relation depending on $y$ and $z$ only and one rank 2 relation depending on $x$ and $y$ only. When a substitution of the above described form is applied to one of the algebras in (S14--S16), the only way for the resulting algebra to still possess the quadratic relation of the form $sxz+tyy$ with $s,t\in\K^*$, is for the substitution to be a scaling. Checking scalings, one easily sees that no scaling provide an isomorphism of algebras from (S14--S16) with different labels and the only scalings providing isomorphisms between two algebras from (S14--S16) with the same label are $x\to sx$, $y\to sy$ and $z\to sz$ with $s\in\K^*$ for algebras from (S14--S15) and $x\to sx$, $y\to \frac{y}{st}$, $z\to tz$ with $s,t\in\K^*$ for algebras from (S16). It is clear that in each case the parameter $a$ is preserved, which proves that algebras in (S14--S16) are pairwise non-isomorphic. These considerations take (S14--S16) out of the picture, leaving us to deal with (S2--S13).

Note that (S2--S5) are singled out from (S2--S13) by having a rank one element $f$ in the space of quadratic relations, which belongs to neither $E_1(Q)^2$ nor $E_2(Q)^2$. Thus algebras from (S2--S5) can not be isomorphic to algebras from (S6--S13). Note that the above element $f$ is actually $xz$ for all algebras in (S2--S5) and that such an element is unique up to a scalar multiple. Thus a substitution providing an isomorphism between two algebras in (S2--S5) must not only preserve $\spann\{x,y\}$ and $\spann\{y,z\}$ but also transform $xz$ to its own scalar multiple. Again, only scalings do that. It is obvious that a scaling can not provide an isomorphism between algebras from (S2--S5) with different labels. As for scalings transforming an algebra from (S2--S5) to another one with the same label, they are $x\to sx$, $y\to sy$ and $z\to sz$ with $s\in\K^*$ for algebras from (S2), $x\to sx$, $y\to sy$ and $z\to tz$ with $s,t\in\K^*$ for algebras from (S3), $x\to sx$, $y\to ty$ and $z\to tz$ with $s,t\in\K^*$ for algebras from (S4) and arbitrary scalings for algebras from (S5). However, it is clear that in each case the parameter $a$ is preserved, which proves that algebras in (S2--S5) are pairwise non-isomorphic. These considerations take (S2--S5) out of the picture, leaving us to deal with (S6--S13).

Now (S6--S9) are singled out from (S6--S13) by having a rank one element $f$ in the space of quadratic relations, which belongs to $E_2(Q)^2$. Thus algebras from (S6--S9) can not be isomorphic to algebras from (S10--S13). Note that the above element $f$ is actually $yz$ for all algebras in (S6--S9) and that such an element is unique up to a scalar multiple. Thus a substitution providing an isomorphism between two algebras in (S6--S9) must not only preserve $\spann\{x,y\}$ and $\spann\{y,z\}$ but also transform $yz$ to its own scalar multiple. The only substitutions, which do this are $x\to \alpha x+ty$, $y\to\beta y$ and $z\to \gamma z$ with $\alpha,\beta,\gamma\in\K^*$ and $t\in \K$. Such substitutions can not give birth or eliminate the presence of $zz$ in the defining relations. Thus algebras from (S6) and (S8) can not be isomorphic to any of the algebras from (S7) and (S9). Since $xy+byx$ and $xy-yx+ayy$ with $a,b\in\K^*$ can not be obtained from one another by an $x,y$ linear substitution (one can use Lemma~\ref{1-dim}), algebras from (S6--S7) can not be isomorphic to any of the algebras from (S8--S9). It follows that algebras from (S6--S9) with different labels are non-isomorphic. Next, $xy+ayx+pyy$ and $xy+byx+qyy$ with different $a,b\in\K^*$ can not be transformed to one another by a substitution of the form $x\to \alpha x+ty$, $y\to\beta y$ with $\alpha,\beta\in\K^*$ and $t\in \K$. It follows that algebras in (S6) and (S7) are pairwise non-isomorphic. Finally any substitution $x\to \alpha x+ty$, $y\to\beta y$ and $z\to \gamma z$ with $\alpha,\beta,\gamma\in\K^*$ and $t\in \K$ is an automorphism of each of the algebras in (S8--S9). As a result, algebras in (S6--S9) are pairwise non-isomorphic. Since algebras in (S10--S13) are isomorphic to algebras from (S6--S9) with the opposite multiplication, algebras in (S10--S13) are pairwise non-isomorphic as well. The proof is now complete.
\end{proof}

If $V$ is a $3$-dimensional vector space and $Q\in V^3$ satisfies $n_1(Q)=n_2(Q)=2$, then, generically, $R_Q$ is $4$-dimensional. In this case though $Q$ is not a quasipotential and therefore is of no interest to us. The key part of our job in the remaining part of this section is identifying $Q$ for which $R_Q$ is $3$-dimensional. We split our search for such $Q$ in two. Note that both $E_1(Q)$ and $E_2(Q)$ are $2$-dimensional subspaces of $V$. Then they may either coincide or intersect by a one-dimensional space.

\begin{lemma}\label{Q22-1} Let $A=A(V,R)\in\Omega$ be such that the corresponding quasipotential $Q=Q_A$ satisfies $n_1(Q)=n_2(Q)=2$ and $E_1(Q)=E_2(Q)$. Then $A$ is isomorphic to an algebra from {\rm (S1,\,S17,\,S18)} of Theorem~$\ref{main}$.
\end{lemma}

\begin{proof} Start by choosing a basis $x,y,z$ in $V$ such that $x$ and $y$ span $M=E_1(Q)=E_2(Q)$. Then $Q=Q_0+Q_1$, where $Q_1\in M^3$ and $Q_0\in MLM$, where $L$ is the one-dimensional space spanned by $z$. If $Q_0=0$, then Lemma~\ref{2-2} ensures that $A\notin \Omega$. This contradiction yields $Q_0\neq 0$. By Lemma~\ref{2-dim}, there is an $x,y$ sub bringing $Q_0$ to one of the following forms: $Q_0=xzy-yzx-yzy$, $Q_0=xzy-\alpha yzx$ with $\alpha\in\K^*$, $Q_0=yzy$ or $Q_0=xzy$.

{\bf Case 1:} $Q_0=xzy-yzx-yzy$.

By a sub $x\to x$, $y\to y$, $z\to sx+ty$ with appropriately chosen $s,t\in\K$ we kill the $xxy$ and $xyy$ coefficients of $Q$. Then $Q$ acquires form
$$
Q=xzy-yzx-yzy+a_1xxx+a_2xyx+a_3yxx+a_4yxy+a_5yyx+a_6yyy
$$
with $a_j\in \K$. Then $R_Q=F_1(Q)+F_2(Q)$ is spanned by $f_1=zy+a_1xx+a_2yx$, $f_2=-zx-zy+a_3xx+a_4xy+a_5yx+a_6yy$, $f_3=-yz+a_1xx+a_2xy+a_3yx+a_5yy$ and $f_4=xz-yz+a_4yx+a_6yy$. The $4\times 4$ matrix $S$ of the $zx$, $xz$, $zy$ and $yz$ coefficients of $f_2$, $f_4$, $f_1$ and $f_3$ (in this order) is
$$
S=\left(\begin{array}{cccc}-1&0&-1&0\\ 0&1&0&-1\\ 0&0&1&0\\ 0&0&0&-1\end{array}\right).
$$
Since $S$ is obviously invertible, $\dim R_Q=4>3$. Hence $Q$ fails to be a quasipotential. Thus Case~1 produces no algebras from $\Omega$.

{\bf Case 2:} $Q_0=xzy-\alpha yzx$ with $\alpha\in\K^*$.

By a sub $x\to x$, $y\to y$, $z\to sx+ty$ with appropriately chosen $s,t\in\K$ we kill the $xxy$ and $xyy$ coefficients of $Q$. Then $Q$ acquires form
$$
Q=xzy-\alpha yzx+a_1xxx+a_2xyx+a_3yxx+a_4yxy+a_5yyx+a_6yyy
$$
with $a_j\in \K$. Then $R_Q=F_1(Q)+F_2(Q)$ is spanned by $f_1=zy+a_1xx+a_2yx$, $f_2=-\alpha zx+a_3xx+a_4xy+a_5yx+a_6yy$, $f_3=-\alpha yz+a_1xx+a_2xy+a_3yx+a_5yy$ and $f_4=xz+a_4yx+a_6yy$. The $4\times 4$ matrix $S$ of the $zx$, $xz$, $zy$ and $yz$ coefficients of $f_2$, $f_4$, $f_1$ and $f_3$ (in this order) is
diagonal with the numbers $-\alpha$, $1$, $1$, $-\alpha$ on the main diagonal. Since $\alpha\neq 0$,  this matrix is invertible. Hence $\dim R_Q=4>3$ and  $Q$ fails to be a quasipotential. Thus Case~2 produces no algebras from $\Omega$.

{\bf Case 3:} $Q_0=yzy$.

By a sub $x\to x$, $y\to y$, $z\to sx+ty$ with appropriately chosen $s,t\in\K$ we kill the $yxy$ and $yyy$ coefficients of $Q$. Then $Q$ acquires form
$$
Q=yzy+a_1xxx+a_2xxy+a_3xyx+a_4xyy+a_5yxx+a_6yyx
$$
with $a_j\in \K$. Then $R_Q=F_1(Q)+F_2(Q)$ is spanned by $f_1=a_1xx+a_2xy+a_3yx+a_4yy$,
$f_2=a_1xx+a_3xy+a_5yx+a_6yy$, $f_3=zy+a_5xx+a_6yx$ and $f_4=yz+a_2xx+a_4xy$. Since $Q$ is a quasipotential,
$\dim R_Q\leq 3$ and therefore $f_j$ are linearly dependent. The monomials $yz$ and $zy$ feature with non-zero coefficients only in $f_4$ and $f_3$ respectively. Hence linear dependence of $f_j$ yields linear dependence of $f_1$ and $f_2$. Next, neither $f_1$ nor $f_2$ is zero (otherwise either $n_1(Q)<2$ or $n_2(Q)<2$). Hence there is $\alpha\in\K^*$ such that $f_2=\alpha f_1\neq0$.

If $a_1\neq 0$, we must have $\alpha=1$. In this case $f_2=\alpha f_1$ reads $a_2=a_3=a_5$ and $a_4=a_6$. By scaling $x$, we can turn $a_1$ into $1$. Then $Q=yzy+xxx+a(xxy+xyx+yxx)+b(xyy+yyx)$ with $a,b\in\K$. By a sub $x\to x-ay$, $y\to y$, $z\to z+sx+ty$ with appropriately chosen $s,t\in\K$, we can (preserving the shape of $Q$) kill $a$. Then $Q=yzy+xxx+b(xyy+yyx)$ with $b\in\K$. If $b=0$, then $xx\in R$, which in view of Lemma~\ref{cube} leads to a contradiction. Thus $b\neq 0$. A further scaling turns $b$ into $1$. Then $Q=yzy+xxx+xyy+yyx$. In this case $R$ is spanned by $xx+yy$, $zy+yx$ and $yz+xy$. A direct Gr\"obner basis computation yields $\dim A_5=22$, which is incompatible with $A\in\Omega$.

It remains to consider the case $a_1=0$. In this case, the equality $f_2=\alpha f_1$ implies that $Q$ has the shape $Q=yzy+a(xxy+\alpha xyx+\alpha^2yxx)+b(xyy+\alpha yyx)$ with $a,b\in\K$ (recall that $\alpha\in\K^*$).
If $a=b=0$, $n_1(Q)=1$. If $a=0$ and $b\neq 0$, $yy\in R$. Since both contradict the assumptions, $a\neq 0$. A normalization turns $a$ into $1$. If $b=0$, then $Q=yzy+xxy+\alpha xyx+\alpha^2yxx$ with $\alpha\in \K^*$. Now $R$ is spanned by $xy+\alpha yx$, $zy+\alpha^2 xx$ and $yz+xx$. A Gr\"obner basis computation gives $\dim A_5=22$ for generic $\alpha$ (with finitely many exceptions). By Lemma~\ref{minhs}, $\dim A_5\geq 22$ for all $\alpha$. Since this is incompatible with $A\in\Omega$, case $b=0$ does not occur. For every $s\in\K$ the sub $x\to x+sy$, $y\to y$, $z\to z+px+qy$ with appropriately chosen $p,q\in\K$ preserves the shape of $Q$ changing the parameter $b$ according to the rule $b\mapsto b+(1+\alpha)s$. Thus we can reduce the general situation to the already considered case $b=0$ unless $\alpha=-1$. Thus it remains to consider the case $b\neq 0$ and $\alpha=-1$. Further scaling turns $b$ into $1$, leaving us with $Q=yzy+xxy-xyx+yxx+xyy-yyx$. Then $R$ is spanned by $xy-yx+yy$, $zy+xx-yx$ and $yz+xx+xy$. A Gr\"obner basis computation yields $\dim A_5=22$, which is incompatible  with $A\in\Omega$. Thus Case~3 yields no algebras from $\Omega$.

{\bf Case 4:} $Q=xzy$.

By a sub $x\to x$, $y\to y$, $z\to sx+ty$ with appropriately chosen $s,t\in\K$ we kill the $xxy$ and $xyy$ coefficients of $Q$. Then $Q$ acquires form
$$
Q=xzy+a_1xxx+a_2xyx+a_3yxx+a_4yxy+a_5yyx+a_6yyy
$$
with $a_j\in \K$. Then $R_Q=F_1(Q)+F_2(Q)$ is spanned by $f_1=a_3xx+a_4xy+a_5yx+a_6yy$,
$f_2=a_1xx+a_2xy+a_3yx+a_5yy$, $f_3=zy+a_1xx+a_2yx$ and $f_4=xz+a_4yx+a_6yy$. Since $Q$ is a quasipotential,
$\dim R_Q\leq 3$ and therefore $f_j$ are linearly dependent. The monomials $xz$ and $zy$ feature with non-zero coefficients only in $f_4$ and $f_3$ respectively. Hence linear dependence of $f_j$ yields linear dependence of $f_1$ and $f_2$. Next, neither $f_1$ nor $f_2$ is zero (otherwise either $n_1(Q)<2$ or $n_2(Q)<2$). Hence there is $\alpha\in\K^*$ such that $f_2=\alpha f_1\neq0$. Solving this system of linear equations, we see that $Q$ must have the form $Q=xzy+a(xxx+\alpha yxx+\alpha^2 yyx+\alpha^3 yyy)+b(xyx+\alpha yxy)$ with $a,b\in\K$. The sub $x\to x$, $y\to\frac1\alpha y$, $z\to z$ turns $\alpha$ into $1$, while preserving the overall shape of $Q$:
$Q=xzy+a(xxx+yxx+yyx+yyy)+b(xyx+yxy)$ with $a,b\in\K$. If $a=0$, then $b\neq 0$ (otherwise $n_1(Q)=1$). By scaling $z$, we turn $b$ into $1$ arriving at $Q=xzy+xyx+yxy$. Then $R$ is spanned by $zy+yx$, $xy$ and $xz+yx$. That is, we have arrived at the algebra (S1). It remains to consider the case $a\neq 0$. By scaling $z$, we can turn $a$ into $1$. Thus $Q=xzy+xxx+yxx+yyx+yyy+bxyx+byxy$ with $b\in\K$. After swapping $x$ and $z$, we get $Q=zxy+zzz+yzz+yyz+yyy+bzyz+byzy$. Then $R$ is spanned by $xy+zz+byz$ $zz+yz+yy+bzy$ $zx+yy+byz$. These relations span the same space as that of the algebra $A^b$ of Lemma~\ref{q8-1iso}. By Lemmas~\ref{gb6-beta}, \ref{gbq6} and \ref{gbq6-3}, $A$ is isomorphic to an algebra from (S17--S18).
\end{proof}

\begin{lemma}\label{Q22-2} Let $A=A(V,R)\in\Omega$ be such that the corresponding quasipotential $Q=Q_A$ satisfies $n_1(Q)=n_2(Q)=2$ and $E_1(Q)\neq E_2(Q)$. Then $A$ is isomorphic to to a $\K$-algebra given by generators $x,y,z$ and three quadratic relations from {\rm (S2--S16)} or {\rm (S19--S20)} of Theorem~$\ref{main}$.
\end{lemma}

\begin{proof} Let $y\in V$ be such that $y$ spans the one-dimensional space $E_1(Q)\cap E_2(Q)$. Pick $x$ and $z$ in $V$ such that $x,y$ is a basis in $E_1(Q)$, while $y,z$ is a basis in $E_2(Q)$. Clearly, $x,y,z$ form a basis of $V$. Then $Q$ has the form
$$
\begin{array}{r}
Q=a_1xxy+a_2xxz+a_3xyy+a_4xyz+a_5xzy+a_6xzz
\\
+a_7yxy+a_8yxz+a_9yyy+a_{10}yyz+a_{11}yzy+a_{12}yzz,
\end{array}
$$
where $a_j\in\K$. Clearly $R_Q$ is spanned by
$$
\begin{array}{l}
f_1=a_1xy+a_2xz+a_3yy+a_4yz+a_5zy+a_6zz,
\\
f_2=a_7xy+a_8xz+a_9yy+a_{10}yz+a_{11}zy+a_{12}zz,
\\
f_3=a_1xx+a_3xy+a_5xz+a_7yx+a_9yy+a_{11}yz,
\\
f_4=a_2xx+a_4xy+a_6xz+a_8yx+a_{10}yy+a_{12}yz.
\end{array}
$$
Since $Q$ is a quasipotential $\dim R_Q\leq 3$. Hence $f_j$ must be linearly dependent. Note also that $f_1$ and $f_2$ must be linearly independent (otherwise $n_2(Q)<2$) and $f_3$ and $f_4$ must be linearly independent (otherwise $n_1(Q)<2$).

The $4\times 4$ matrices $S_1$ of $xz$, $yz$, $zy$ and $zz$ coefficients and $S_2$ of $xx$, $yx$, $xy$ and $xz$ coefficients of $f_1$, $f_2$, $f_3$ and $f_4$ (all in given order) are
$$
S_1=\left(\begin{array}{cccc}a_2&a_4&a_5&a_6\\ a_8&a_{10}&a_{11}&a_{12}\\
a_5&a_{11}&0&0\\ a_6&a_{12}&0&0\end{array}\right)\ \ \text{and}\ \
S_2=\left(\begin{array}{cccc}0&0&a_1&a_2\\ 0&0&a_7&a_8\\ a_1&a_7&a_8&a_5\\ a_2&a_8&a_4&a_6\end{array}\right).
$$
If either of the two matrices
$$
S_3=\left(\begin{array}{cc}a_1&a_7\\ a_2&a_8\end{array}\right)\ \ \text{or}\ \
S_4=\left(\begin{array}{cc}a_5&a_6\\ a_{11}&a_{12}\end{array}\right)
$$
is non-degenerate, then at least one of $S_1$ or $S_2$ is invertible leading to linear independence of $f_j$. Thus both $S_3$ and $S_4$ must be non-invertible.

{\bf Case 1:} $a_2a_6\neq 0$.

A scaling turns $a_2$ and $a_6$ into $1$. Hence, we can assume that $a_2=a_6=1$. A substitution $x\to x+sy$, $y\to y$, $z\to z+ty$ with appropriate $s,t\in \K$ kills $a_1$ and $a_{12}$. Thus we can assume $a_1=a_{12}=0$. Now since $S_3$ and $S_4$ are non-invertible, we have $a_7=a_{11}=0$. Thus $Q=xxz+a_3xyy+a_4xyz+a_5xzy+xzz+a_8yxz+a_9yyy+a_{10}yyz$. In this case $f_1=xz+a_3yy+a_4yz+a_5zy+zz$, $f_2=a_8xz+a_9yy+a_{10}yz$, $f_3=a_3xy+a_5xz+a_9yy$ and $f_4=xx+a_4xy+xz+a_8yx+a_{10}yy$. The monomial $zz$ features with non-zero coefficient only in $f_1$, while $xx$ features with non-zero coefficient only in $f_4$. Hence linear dependence of $f_j$ yields linear dependence of $f_2$ and $f_3$. If $a_3\neq 0$, linear dependence of $f_2$ and $f_3$ occurs only if $f_2=0$. In this case, $f_1$ and $f_2$ are linearly dependent, which is a contradiction. Thus $a_3=0$. If $a_{10}\neq 0$, linear dependence of $f_2$ and $f_3$ occurs only if $f_3=0$. In this case, $f_3$ and $f_4$ are linearly dependent, which is a contradiction. Hence $a_{10}=0$. Plugging $a_3=a_{10}=0$ back into $f_j$, we get $f_1=xz+a_4yz+a_5zy+zz$, $f_2=a_8xz+a_9yy$, $f_3=a_5xz+a_9yy$ and $f_4=xx+a_4xy+xz+a_8yx$. Since $f_2$ and $f_3$ are linearly dependent and neither is $0$, there is $\alpha\in\K^*$ such that $f_3=\alpha f_2\neq0$.

If $a_9\neq 0$, we have $\alpha=1$ and $a_5=a_8$. A scaling turns $a_9$ into $1$. Thus $Q=xxz+xzz+yyy+sxzy+syxz+txyz$ with $s,t\in\K$. Clearly, $R$ is spanned by $xy+yy+syz+tzy$, $zz+sxy$ and $xx+xy+szx+txz$. Then $s\neq 0$ (otherwise $zz\in R$ and Lemma~\ref{cube} provides a contradiction). If $t=0$, then computing the reduced Gr\"obner basis, we see that for generic algebras in our one-parametric family $\dim A_6=29$. By Lemma~\ref{minhs}, $\dim A_6\geq 29$ for all $s$ whenever $t=0$, which is incompatible with the membership in $\Omega$. Hence $t\neq 0$. Now swapping of $y$ and $z$ together with an appropriate scaling turns these relations into those from Lemma~\ref{q45-1}. Now by Lemma~\ref{q45-1} and Remark~\ref{pnqnrem}, our algebra is isomorphic to an algebra from (S19--S20). It remains to deal with the case $a_9=0$. Since $f_2$ and $f_3$ are non-zero $a_5a_8\neq 0$. A scaling turns $a_5$ into $1$. Now $Q=xxz+bxyz+xzy+xzz+ayxz$ with $b\in\K$ and $a\in\K^*$. Then $R$ is spanned by $byz+zy+zz$, $xz$ and $xx+bxy+ayx$. If $b=0$, applying the Gr\"obner basis technique, we see that $\dim A_3>10$, contradicting the assumptions. Thus $b\neq 0$ and we have an algebra from (S2). This concludes Case~1.

{\bf Case 2:} \ $a_2\neq 0$ and $a_6=0$.

Scaling, we can make $a_2=1$. A sub $x\to x$, $y\to y$, $z\to z+sy$ with an appropriate $s\in\K$ kills $a_7$. Since $S_3$ is degenerate, $a_7=0$ and $a_2=1$, we have $a_1=0$. Since $a_6=0$ and $S_4$ is degenerate, we have $a_5a_{12}=0$.

{\bf Case 2a:} \ additionally, $a_5\neq 0$ and $a_{12}=0$.

By a sub $x\to x+ty$, $y\to y$, $z\to z$, we can kill $a_{11}$. Further scaling turns $a_5$ into $1$. Thus $Q=xxz+a_3xyy+a_4xyz+xzy+a_8yxz+a_9yyy+a_{10}yyz$,
$f_1=xz+a_3yy+a_4yz+zy$, $f_2=a_8xz+a_9yy+a_{10}yz$, $f_3=a_3xy+xz+a_9yy$ and $f_4=xx+a_4xy+a_8yx+a_{10}yy$. Since $zy$ features with non-zero coefficient only in $f_1$, while $xx$ features with non-zero coefficient only in $f_4$, linear dependence of $f_j$ yields linear dependence of $f_2$ and $f_3$. Since neither is $0$, there is $\alpha\in\K^*$ such that $f_3=\alpha f_2$. For this to happen, we must have $a_3=a_{10}=0$. Thus $Q=xxz+a_4xyz+xzy+a_8yxz+a_9yyy$, $f_1=xz+a_4yz+zy$, $f_2=a_8xz+a_9yy$, $f_3=xz+a_9yy$ and $f_4=xx+a_4xy+a_8yx$.
If $a_9\neq 0$, we must have $\alpha=1$ and $a_8=1$. Further scaling allows us to turn $a_9$ into $1$. In this case $Q=xxz+axyz+xzy+yxz+yyy$, $f_1=xz+ayz+zy$, $f_2=f_3=xz+yy$, and $f_4=xx+axy+yx$, where $a=a_4\in\K$. If $a=0$, a Gr\"obner basis computation yields $\dim A_4=16$, contradicting the assumptions. Thus $a\neq 0$. Now $A$ is an algebra from (S14).

It remains to deal with the case $a_9=0$. Then $Q=xxz+axyz+xzy+byxz$, $f_1=xz+ayz+zy$, $f_2=bxz$, $f_3=xz$ and $f_4=xx+axy+byx$, where $a=a_4$ and $b=a_8$ are in $\K$. Since $f_2\neq 0$, we must have $b\neq 0$. Now $R$ is spanned by $xz$, $ayz+zy$ and $xx+axy+byz$. Again, if $a=0$, then $\dim A_4=16$, contradicting $A\in\Omega$ (use Gr\"obner basis). Thus $a\neq 0$ and $A$ becomes an algebra from (S3).

{\bf Case 2b:} \ additionally (to assumptions of Case~2), $a_5=0$ and $a_{12}\neq 0$.

By scaling, we can turn $a_{12}$ into $1$. By a sub $x\to x+sy$, $y\to y$, $z\to z$ with an appropriate $s\in\K$, we can kill $a_4$. Then $Q=xxz+a_3xyy+a_8yxz+a_9yyy+a_{10}yyz+a_{11}yzy+yzz$, $f_1=xz+a_3yy$, $f_2=a_8xz+a_9yy+a_{10}yz+a_{11}zy+zz$, $f_3=a_3xy+a_9yy+a_{11}yz$ and $f_4=xx+a_8yx+a_{10}yy+yz$. Since $xx$ features with non-zero coefficient only in $f_4$, while $zz$ features with non-zero coefficient only in $f_2$, linear dependence  of $f_j$ yields linear dependence of $f_1$ and $f_3$. Since $xz$ features in $f_1$ (with non-zero coefficient) but not in $f_3$, the latter fails. Thus Case~2b carries no algebras from $\Omega$.

{\bf Case 2c:} \ additionally (to assumptions of Case~2), $a_5=a_{12}=0$.

By a sub $x\to x+sy$, $y\to y$, $z\to z$ with an appropriate $s\in\K$, we can kill $a_{10}$. Since $a_2=1$ and $a_1=a_5=a_6=a_7=a_{10}=a_{12}=0$, we get  $Q=xxz+a_3xyy+a_4xyz+a_8yxz+a_9yyy+a_{11}yzy$, $f_1=xz+a_3yy+a_4yz$, $f_2=a_8xz+a_9yy+a_{11}zy$, $f_3=a_3xy+a_9yy+a_{11}yz$ and $f_4=xx+a_4xy+a_8yx$. If $a_{11}\neq 0$, then $f_j$ are linearly independent. Indeed, $xx$ features only in $f_4$, $zy$ features only in $f_2$, out of $f_1$ and $f_3$ only $f_1$ sports $xz$ and $f_3$ must be non-zero. Since $f_j$ are linearly dependent, we have $a_{11}=0$. Plugging this back in $Q$ and $f_j$, we get $Q=xxz+a_3xyy+a_4xyz+a_8yxz+a_9yyy$, $f_1=xz+a_3yy+a_4yz$, $f_2=a_8xz+a_9yy$, $f_3=a_3xy+a_9yy$ and $f_4=xx+a_4xy+a_8yx$. Now $a_3a_8\neq 0$. Indeed, otherwise, $f_2=0$ or $f_3=0$ or $yy\in R$. By means of scaling, we can turn $a_3$ and $a_8$ into $1$. If $a_4\neq 0$, then $f_j$ are linearly independent. Indeed, $xx$ features only in $f_4$, $yz$ features only in $f_1$, while $f_2$ and $f_3$ are obviously linearly independent. Hence $a_4=0$. Plugging $a_4=0$ and $a_3=a_8=1$ this back into $Q$ and $f_j$, we get $Q=xxz+xyy+yxz+ayyy$, $f_1=xz+yy$, $f_2=xz+ayy$, $f_3=xy+yy$ and $f_4=xx+yx$. The only case when $f_j$ are linearly dependent is $a=1$. Thus $Q=xxz+xyy+yxz+yyy$ and $R$ is spanned by $xz+yy$, $xy+yy$ and $xx+yx$. A direct computation shows that $\dim A_3=12$, contradicting the assumption $A\in\Omega$.

{\bf Case 3:} \ $a_2=0$ and $a_6\neq 0$.

This case is obtained from Case~2 by passing to opposite multiplication. That is, in this case $A$ must be isomorphic to and algebra from (S4) and (S15). Indeed, up to an isomorphism, algebras in (S4) and (S15) are the algebras from (S3) and (S14) with the opposite multiplication.

{\bf Case 4:} \ $a_2=a_6=0$.

Since $S_3$ and $S_4$ are degenerate, $a_1a_8=a_5a_{12}=0$.

{\bf Case 4a:} \ additionally, $a_1=a_5=0$.

Then $Q=a_3xyy+a_4xyz+a_7yxy+a_8yxz+a_9yyy+a_{10}yyz+a_{11}yzy+a_{12}yzz$, $f_1=a_3yy+a_4yz$, $f_2=a_7xy+a_8xz+a_9yy+a_{10}yz+a_{11}zy+a_{12}zz$, $f_3=a_3xy+a_7yx+a_9yy+a_{11}yz$ and
$f_4=a_4xy+a_8yx+a_{10}yy+a_{12}yz$. Now $a_4\neq 0$ (otherwise either $f_1=0$ or $yy\in R$). We can normalize to make $a_4=1$ and use the sub $x\to x$, $y\to y$, $z\to z+sy$ with an appropriate $s\in\K$ to kill $a_3$.

First, we show that $a_7=0$. Assume the contrary: $a_7\neq 0$. Then we can normalize to make $a_7=1$ and use the sub $x\to x+ty$, $y\to y$ and $z\to z$ to kill $a_9$. Now  $Q=xyz+yxy+a_8yxz+a_{10}yyz+a_{11}yzy+a_{12}yzz$, $f_1=yz$, $f_2=xy+a_8xz+a_{10}yz+a_{11}zy+a_{12}zz$, $f_3=yx+a_{11}yz$ and $f_4=xy+a_8yx+a_{10}yy+a_{12}yz$. If $a_{12}\neq 0$, $f_j$ are easily seen to be linearly independent. Hence $a_{12}=0$. Using this equality, we see that if $a_8\neq 0$, then $f_j$ are still linearly independent. Hence $a_8=0$. Now if $a_{10}\neq 0$, $f_j$ persist in being linearly independent. Then $a_{10}=0$. Plugging all this back, we get $Q=xyz+yxy+a_{11}yzy$, $f_1=yz$, $f_2=xy+a_{11}zy$, $f_3=yx+a_{11}yz$ and $f_4=xy$. Since $xy,yz\in R$, $xyz\in RV\cap VR$. The latter space is supposed to be one-dimensional spanned by $Q$. This contradiction proves that $a_7=0$.

First, we consider the case $a_{10}=0$. Plugging $a_7=a_{10}=0$ back into formulas for $Q$ and $f_j$, we get $Q=xyz+a_8yxz+a_9yyy+a_{11}yzy+a_{12}yzz$, $f_1=yz$, $f_2=a_8xz+a_9yy+a_{11}zy+a_{12}zz$, $f_3=a_9yy+a_{11}yz$ and $f_4=xy+a_8yx+a_{12}yz$. Now $a_9=0$. Indeed, otherwise either $f_3=0$ or $yy\in R$. Since $a_9=0$, we have $a_{11}\neq 0$ (otherwise $f_3=0$). We arrive at $Q=xyz+a_8yxz+yzy+a_{12}yzz$, $f_1=f_3=yz$, $f_2=a_8xz+zy+a_{12}zz$ and $f_4=xy+a_8yx+a_{12}yz$. If $a_8=0$, Lemma~\ref{2-2} says $A\notin\Omega$. Hence $a_8=a\in\K^*$. If $a_{12}\neq 0$, it can be turned into $1$ by scaling. Thus we have two options $Q=xyz+ayxz+yzy+yzz$ or $Q=xyz+ayxz+yzy$  with $a\in\K^*$. In the first case $R$ is spanned by $yz$, $axz+zy+zz$ and $xy+ayx+yz$, while in the second case $R$ is spanned by $yz$, $axz+zy$ and $xy+ayx$ landing us into (S6) and (S7).

Now assume $a_{10}\neq 0$. If $a_8\neq -1$, then we can use the sub $x\to x+ty$, $y\to y$, $z\to z$ with an appropriate $t\in\K$ to kill $a_{10}$, bringing us back to the case $a_{10}=0$, already dealt with. Thus we can assume that $a_8=-1$. Next, $a_{11}\neq 0$ (otherwise $f_3=0$ or $yy\in R$). By scaling, we can turn $a_{11}$ into $1$. Next, $a_9=0$ (otherwise $yy\in R$). If $a_{12}\neq 0$, it can be turned into $1$ by scaling.

Thus we have two options $Q=xyz-yxz+ayyz+yzy+yzz$ and $Q=xyz-yxz+ayyz+yzy$, where $a=a_{10}\in \K^*$. In the first case $R$ is spanned by $yz$, $-xz+ayz+zy+zz$ and $xy-yx+ayy+yz$, while in the second case $R$ is spanned by $yz$, $-xz+ayz+zy$ and $xy-yx+ayy$ and we arrive to (S8) and (S9).

{\bf Case 4b:} \ additionally (to $a_2=a_6=0$), $a_8=a_{12}=0$.

This case is obtained from Case~4a by passing to opposite multiplication. Thus the list of algebras to one of which $A$ must be isomorphic is (S10--S13). Indeed, the classes (S10--S13) can (up to an isomorphism) be obtained from (S6--S9) in this order by passing to the opposite multiplication.

{\bf Case 4c:} \ additionally (to $a_2=a_6=0$), $a_1=a_{12}=0$ and $a_5a_8\neq 0$.

Since $a_5\neq 0$, we can scale to make $a_5=1$. By the sub $x\to x+sy$, $y\to y$, $z\to z$, we can kill $a_{11}$. Since $a_8\neq 0$ the sub $x\to x$, $y\to y$, $z\to z+ty$ with an appropriate $t\in\K$ kills $a_7$. Plugging $a_2=a_6=a_1=a_{12}=a_{11}=a_7=0$ and $a_5=1$ into the formula for $Q$ and $f_j$, we get $Q=a_3xyy+a_4xyz+xzy+a_8yxz+a_9yyy+a_{10}yyz$, $f_1=a_3yy+a_4yz+zy$, $f_2=a_8xz+a_9yy+a_{10}yz$, $f_3=a_3xy+xz+a_9yy$ and $f_4=a_4xy+a_8yx+a_{10}yy$. Since $zy$ features only in $f_1$ and $yx$ features only in $f_4$, linear dependence of $f_j$ yields linear dependence of $f_2$ and $f_3$. Since neither is zero, $f_3=\alpha f_2$ for some $\alpha\in\K^*$. Then $a_3=a_{10}=0$. Plugging this back into the formula for $Q$ and $f_j$, we get $Q=a_4xyz+xzy+a_8yxz+a_9yyy$, $f_1=a_4yz+zy$, $f_2=a_8xz+a_9yy$, $f_3=xz+a_9yy$ and $f_4=a_4xy+a_8yx$. First, consider the case $a_9\neq 0$. In this case a scaling makes $a_9=1$. Then $\alpha=1$ and therefore $a_8=1$. Thus $Q=axyz+xzy+yxz+yyy$, $f_1=ayz+zy$, $f_2=f_3=xz+yy$ and $f_4=axy+yx$ with $a=a_4\in\K$. If $a=0$, one easily sees that $\dim A_3>10$. Thus $a\neq 0$ and we fall under the jurisdiction of (S16). It remains to deal with the case $a_9=0$. Then $Q=axyz+xzy+byxz$, $f_1=ayz+zy$, $f_2=bxz$, $f_3=xz$ and $f_4=axy+byx$, where $a=a_4\in\K$ and $b=a_8\in\K^*$. If $a=0$, one easily sees that $\dim A_3>10$. Thus $a\neq 0$ and we arrive to an algebra from (S5).

{\bf Case 4d:} \ additionally (to $a_2=a_6=0$), $a_5=a_{8}=0$ and $a_1a_{12}\neq 0$.

This case is obtained from Case~4c by passing to opposite multiplication. Since both classes (S5) and (S16) are closed under passing to the opposite multiplication, we again have $A$ isomorphic from an algebra of (S5) or (S16).

Between cases 4a--4d all options for $a_j$ satisfying $a_1a_8=a_5a_{12}=0$ are exhausted.
\end{proof}

Part IV of Theorem~\ref{main} now follows from Lemmas~\ref{Q22-1}, \ref{Q22-2} and \ref{algS}.

\section{Proof of Part~I of Theorem~\ref{main}}

Part I of Theorem~\ref{main} is a rather odd one out and is technically more difficult than each of the other parts. We start with some general comments on algebras in $\Omega$ having a square in the space of quadratic relations.

\begin{lemma}\label{triv}
Let $L$ be a $1$-dimensional subspace of the $3$-dimensional vector space $V$ over $\K$ and let $R$ be a $3$-dimensional subspace of the $5$-dimensional space $LV+VL$ such that $L^2\subset R$. Then the quadratic algebra $A=A(V,R)$ satisfies $\dim A_3\geq 12$. In particular, $A\notin \Omega'$.
\end{lemma}

\begin{proof} Let $z\in V$ be such that $z$ spans $L$. Then $zz\in R$. By passing to the opposite multiplication, if necessary, we can assume that $R\not\subseteq LV$. A standard linear algebra argument yields that $x,y\in V$ can be chosen in such a way that $x,y,z$ form a basis in $V$ and $R$ is spanned by one of the following triples $\{xz-azx,yz-bzy,zz\}$ with $a,b\in\K$, $\{xz-azx-zy,yz-azy,zz\}$ with $a\in\K$, $\{xz-azx,zy,zz\}$ with $a\in\K$, $\{xz-zy,zx,zz\}$ or $\{xz,zx,zz\}$. With respect to the chosen basis, the triples serve as defining relations for $A$. Using the usual left-to-right degree-lexicographical ordering with $x>y>z$, we can compute the members of the reduced Gr\"obner basis in the ideal of relations of $A$ of degree up to $3$ (actually, the defining relations form the Gr\"obner basis already in all cases except for the second last, for which two degree $3$ members of the Gr\"obner basis occur: $zyz$ and $zyx$). This gives $\dim A_3=12$ in all cases except for $\{xz-azx,zy,zz\}$, in which we have $\dim A_3=13$.
\end{proof}

\begin{lemma}\label{zzz} Assume that $A=A(V,R)\in\Omega'$ and the corresponding quasipotential is $Q=z^3$. Then the quadratic algebra $B=A/I$, where $I$ is the ideal generated by $z$, is uniquely $($up to an isomorphism$)$ determined by $A$. For $B=A(V_0,R_0)$, we have $\dim V_0=2$ and $1\leq \dim R_0\leq 2$ and
exactly one of the following holds true
\begin{itemize}\itemsep=-2pt
\item $R_0=\spann\{yy\};$
\item $R_0=\spann\{xy\};$
\item $R_0=\spann\{xy-\alpha yx\}$ with $\alpha\in\K^*$ being uniquely determined up to replacing it by $\frac1\alpha;$
\item $R_0=\spann\{xy-yx-yy\};$
\item $R_0=\spann\{xx,yy\};$
\item $R_0=\spann\{xx-yx,yy\};$
\item $R_0=\spann\{xy,yy\};$
\item $R_0=\spann\{yx,yy\};$
\item $R_0=\spann\{xy-\alpha yx,yy\}$ with $\alpha\in\K^*$ being uniquely determined$;$
\item $R_0=\spann\{xy,yx\};$
\item $R_0=\spann\{xx-xy,yx\};$
\item $R_0=\spann\{xx-\alpha xy-yy,yx\}$ with $\alpha\in\K$, $\alpha^2+1\neq 0$ with $\alpha$ being uniquely determined
\end{itemize}
for some $x,y\in V$ such that $x,y,z$ is a basis in $V$ $(\{x,y\}$ is now naturally interpreted as a basis in $V_0)$.
\end{lemma}

\begin{proof} As the quasipotential for an algebra from $\Omega'$ is unique up to a scalar multiple, $z$ is uniquely determined up to a non-zero scalar multiple as well. Thus $B$ is uniquely determined by $A$. Obviously, for $B=A(V_0,R_0)$, we have $\dim V_0=2$ and $\dim R_0\leq 2$. However, we can not have $\dim R_0=0$ since otherwise $A$ satisfies the assumptions of Lemma~\ref{triv}, which yields $\dim A_3\geq 12$ contradicting the inclusion $A\in\Omega'$. The last statement is a direct application of Lemmas~\ref{1-dim} and~\ref{2-dim}.
\end{proof}

Note that for $A=A(V,R)$ from (R1--R39), the space $R_0$ (of the algebra $B=A(V_0,R_0)$, as defined in Lemma~\ref{zzz}) is spanned by
\begin{itemize}\itemsep=-2pt
\item $\{xy+yx,yy\}$ if $A$ is from (R1--R3);
\item $\{xy-yx,yy\}$ if $A$ is from (R4);
\item $\{xy-ayx,yy\}$ with $a\notin\{0,1,-1\}$ if $A$ is from (R5--R6);
\item $\{xx-xy,yx\}$ if $A$ is from (R7--R8);
\item $\{yx,yy\}$ if $A$ is from (R9--R12);
\item $\{xy,yy\}$ if $A$ is from (R13--R16);
\item $\{xy,yx\}$ if $A$ is from (R17--R19);
\item $\{xx,yy\}$ if $A$ is from (R20);
\item $\{xy-yx-yy\}$ if $A$ is from (R21--R26);
\item $\{xy-ayx\}$ with $a\notin\{0,1\}$ if $A$ is from (R27--R33);
\item $\{xy-yx\}$ if $A$ is from (R34--R36);
\item $\{xy\}$ if $A$ is from (R37--R39).
\end{itemize}
By Lemma~\ref{zzz}, algebras from different groups out of (R1--R3), (R4), (R5--R6), (R7--R8), (R9--R12), (R13--R16), (R17--R19), (R20), (R21--R26), (R27--R33),  (R34--R36) and (R37--R39) are non-isomorphic. This splits the proof into 12 independent parts.

In this section we shall often use the following ordering on monomials in three variables $x,y,z$. To introduce it smoother, for an $x,y,z$ monomial $u$, we denote by $u'$ the (same degree) monomial obtained from $u$ by replacing all occurrences of $y$ by $x$. We also use the symbol $\prec$ to denote the left-to-right degree lexicographical ordering assuming $x>y>z$. The ordering $<$ we shall use is defined as follows. We say that $v<u$ if and only if
\begin{equation}\label{order}
\begin{array}{l}
\text{$u$ has higher degree than $v$,}\\
\text{or $u$ and $v$ have the same degree, but $z$-degree of $v$ is higher,}\\
\text{or $u$ and $v$ have the same degree and the same $z$-degree and $v'\prec u'$,}\\
\text{or $u$ and $v$ have the same degree and the same $z$-degree and $v'=u'$ and $v\prec u$.}
\end{array}
\end{equation}
For instance, for monomials of degree up to $3$ the just defined order looks like
$$
\begin{array}{l}
1<z<y<x<zz<zy<zx<yz<xz<yy<yx<xy<xx<zzz<zzy<zzx<zyz
\\
<zxz<yzz<xzz<zyy<zyx<zxy<zxx<yzy<yzx<xzy<xzx<yyz
\\
<yxz<xyz<xxz<yyy<yyx<yxy<yxx<xyy<xyx<xxy<xxx.
\end{array}
$$
It is easy to see that the ordering (\ref{order}) is compatible with multiplication and therefore can be used in computing Gr\"obner bases. Curiously, for the purpose of dealing with algebras in $\Omega$, whose quasipotential is a cube, this order in most cases proves much more convenient than the degree lexicographical one. On a number of occasions it even yields a finite Gr\"obner basis while the degree lexicographical ordering provides an infinite one.

\subsection{Case $R_0=\spann\{xy+ayx,yy\}$}

\begin{lemma}\label{xy+yxVyy} Let $A$ be a quadratic algebra given by generators $x,y,z$ and relations
\begin{equation}\label{dere2}
\text{$xy-\alpha yx-azx-bzy$, $yy-xz-qyz-czx-dzy$ and $zz$},
\end{equation}
where $\alpha,a,b,c,d,q\in\K$ and $\alpha\neq 0$. Then $A\in\Omega'$. Moreover, $A\in\Omega$ if and only if
\begin{equation}\label{eqeq1}\textstyle
c=\alpha^2,\ a=\alpha(d-q),\ b=\frac{d(d-q)}{\alpha}\ \text{and}\ (\alpha+1)((2-\alpha)d-(\alpha^2-\alpha+1)q)=0.
\end{equation}
Furthermore, two algebras from this family with parameters $\alpha,a,b,c,d,q$ and $\alpha',a',b',c',d',q'$ are isomorphic if and only if either $\alpha=\alpha'\neq1$ and $(a',b',c',d',q')=(at,bt^2,c,dt,qt)$ for some $t\in\K^*$ or $\alpha=\alpha'=1$ and $(a',b',c',d',q')=(at,t^2(b+as),c,t(d+s),t(q+s))$ for some $t\in\K^*$ and $s\in\K$.
\end{lemma}

\begin{proof} First, we deal with the isomorphism question. Assume that two algebras from our family with parameters $\alpha,a,b,c,d,q$ and $\alpha',a',b',c',d',q'$ are isomorphic. Then there is a linear substitution facilitating the isomorphism. Since $z^2$ is the only square in the space of quadratic relations for both algebras, $z$ is mapped to its own scalar multiple by this substitution: $z\to wz$ with $w\in\K^*$. Since $y^2$ is the only square in the space of quadratic relations for $A/I$ with $I$ being the ideal generated by $z$ for both algebras, our substitution must map $y$ to $vy+sz$ with $v\in\K^*$, $s\in\K$. By Lemma~\ref{zzz} (the uniqueness), $\alpha$ is an isomorphism invariant. Thus we must have $\alpha=\alpha'$. First, consider the case $\alpha=\alpha'\neq1$. Then it is easy to see that the shape of the space of quadratic relations of the quotient $A/I$ will not be preserved unless $x\to ux+tz$ with $u\in\K^*$, $t\in\K$. Now applying the substitution $x\to ux+tz$,  $y\to vy+sz$, $z\to wz$, we see that the shape of the space of quadratic relations of the algebra $A$ will not be preserved unless $s=t=0$. We are left with scalings only. Now it is a matter of direct verification to see that a scaling provides a required isomorphism if and only if $(a',b',c',d',q')=(at,bt^2,c,dt,qt)$ for some $t\in\K^*$. It remains to consider the case $\alpha=\alpha'=1$. Similar considerations show that in this case the shape of the space quadratic relations of $A$ is preserved (stays in our family) if and only if the substitution has the form $x\to ux+rz$, $y\to vy$, $z\to wz$ with $u,v,w,r\in\K$, $uvw\neq 0$ and $uw=v^2$. Applying these substitutions, we directly see that one of them provides a required isomorphism if and only if $(a',b',c',d',q')=(at,t^2(b+as),c,t(d+s),t(q+s))$ for some $t\in\K^*$ and $s\in\K$.

Now we deal with the Hilbert series of $A$. Throughout the proof we use the order (\ref{order}) on $x,y,z$ monomials. Resolving the overlaps $yyy$ and $xyy$, we see that the degree 3 part of the Gr\"obner basis of the ideal of relations of $A$ consists of two members
\begin{equation*}
\begin{array}{l}
g_1=yxz{-}xzy{+}cyzx{+}(d{-}q)yzy{-}\alpha czyx{+}(qc{-}d)zxz,
\\
g_2=xxz{+}(c{-}\alpha^2 )xzx{+}(\alpha q{+}d)xzy{-}\alpha(qc{+}\alpha q{+}a)yzx{-}\alpha(b{+}qd{-}q^2)yzy
\\
\qquad{-}\alpha^2czxx{-}\alpha(a+\alpha d {-}\alpha qc)zyx{+}(qa{+}\alpha qd{-}b{-}\alpha q^2c)zxz.
\end{array}
\end{equation*}
Since $g_1$ and $g_2$ are inearly independent, it follows that $\dim A_3=10$ regardless what the values of the parameters are. Thus $A\in\Omega'$. There are $4$ degree $4$ overlaps $yxzz$, $xxzz$, $yyxz$ and $xyxz$. The members of the ideal of relations obtained from the last two overlaps always belong to the linear span of the ones obtained from the first two overlaps. We denote the latter $r_1$ and $r_2$ respectively. The explicit formulae for $r_j$ are as follows:

\begin{equation*}
\begin{array}{l}
\scriptstyle r_1=xzyz{-}cyzxz{+}(q{-}d)yzyz{+}\alpha czxzy{-}\alpha c^2zyzx{+}\alpha c(q{-}d)zyzy;
\\
\scriptstyle r_2=(\alpha^2{-}c)xzxz{+}(\alpha^2q{+}\alpha a{-}cd)yzxz {+}(\alpha b{-}d^2+qd)yzyz{+}\alpha^2c(\alpha^2{-}c)zxzx{+}\alpha(cd{-}\alpha^2qc{-}\alpha cd{+}a{+}\alpha d)zxzy
\\
\quad\scriptstyle
{+}\alpha(\alpha^2qc^2{-}c^2d{+}\alpha^3qc{+}\alpha^2ac{-}ac{-}\alpha cd)zyzx{+}
\alpha(\alpha^2bc+\alpha^qcd-\alpha^2q^2c+qa+\alpha qd+qcd-ad-\alpha d-cd^2)zyzy
\end{array}
\end{equation*}

Note that there are exactly $16$ degree $4$ monomials which do not contain the leading terms $xy$, $yy$, $zz$, $yxz$ and $xxz$ of the members of the Gr\"obner basis of degree up to $3$. For $A$ to be in $\Omega$, we must have $\dim A_4=15$. This happens precisely when the dimension of the space $L$ spanned by $r_1$ and $r_2$ is exactly $16-15=1$. On the other hand, the monomial $xzyz$ features in $r_1$ with non-zero coefficient and does not feature in $r_2$. Hence the only way for the dimension of $L$ to be $1$ is to have $r_2=0$. Thus we have a system of algebraic equations on the parameters coming from the coefficients of $r_2$ being zero. The equation coming from the $xzxz$-coefficient is $c=\alpha^2$. Plugging this into the $yzxz$ one, we get $a=\alpha(d-q)$. Plugging both into $yzyz$ one, we get $b=\frac{d(d-q)}{\alpha}$. Plugging all this into the $zxzy$ coefficient, we get $(\alpha+1)((2-\alpha)d-(\alpha^2-\alpha+1)q)=0$. Now the rest of the equations are automatically satisfied provided these four are. Thus we have that $\dim A_4=15$ if and only if $r_2=0$ if and only if (\ref{eqeq1}) is satisfied.

Now if $r_2=0$, one easily checks that the Gr\"obner basis of the ideal of relations of $A$ actually ends with $r_1$: it consists of the defining relations, $g_1$, $g_2$ and $r_1$ (great advantage of the ordering we picked!). The leading monomials of the elements of the basis are $xy$, $yy$, $zz$, $yxz$, $xxz$ and $xzyz$. Knowing these, we easily confirm that $H_A=(1-t)^{-3}$ and $A$ is indeed in $\Omega$. Thus $A\in\Omega$ if and only if (\ref{eqeq1}) is satisfied.
\end{proof}

The main result of this section is the following lemma.

\begin{lemma}\label{case-R1-6} Let $A=A(V,R)\in\Omega$ be a quadratic algebra such that $\dim V=\dim R=3$ and with respect to some basis $x,y,z$ in $V$, $zz\in R$ and the quadratic algebra $B=A/I$ with $I$ being the ideal generated by $z$, is given by the relations $xy-\alpha yx$ and $yy$ with $\alpha\in\K^*$. Then $A$ is isomorphic to to a $\K$-algebra given by generators $x,y,z$ and three quadratic relations from {\rm (R1--R6)} of Theorem~$\ref{main}$. Furthermore, the algebras in {\rm (R1--R6)} belong to $\Omega^-$ and are pairwise non-isomorphic.
\end{lemma}

\begin{proof} By the assumptions, $R$ is spanned by $zz$, $xy-\alpha yx+f$ and $yy+g$ with $f,g\in\spann\{xz,zx,yz,zy\}$. Using a substitution $x\to x+sz$, $y\to y+tz$, $z\to z$ with appropriate $s,t\in\K$, we can kill the $xz$ and $yz$ coefficients of $f$. We shall see later that the $xz$ coefficient in $g$ must be non-zero. In the meantime though, we view it as an extra assumption.

{\bf Case 1:} the $xz$ coefficient in $g$ is non-zero.

By means of scaling, we can turn this coefficient into $-1$. Then $A$ is given by the generators $x,y,z$ and the relations from (\ref{dere2}) for some $\alpha,a,b,c,d,q\in\K$ and $\alpha\neq 0$. Since $A\in\Omega$, Lemma~\ref{xy+yxVyy} implies that (\ref{eqeq1}) is satisfied.

{\bf Case 1a:} $\alpha\neq1$ and $\alpha\neq-1$. The last equation in (\ref{eqeq1}) reads $(2-\alpha)d=(\alpha^2-\alpha+1)q$. According to the isomorphism part of Lemma~\ref{xy+yxVyy}, we can (by a scaling) multiply $(q,d)$ by any non-zero constant without breaking up the overall shape of relations. Then if $(q,d)\neq (0,0)$, since $(2-\alpha,\alpha^2-\alpha+1)\neq (0,0)$ and $(2-\alpha)d=(\alpha^2-\alpha+1)q$, by doing this we can turn $(q,d)$ into $(2-\alpha,\alpha^2-\alpha+1)$. After this sub $q=2-\alpha$, $d=\alpha^2-\alpha+1$ and plugging this into the rest of the equations in (\ref{eqeq1}), we get $c=\alpha^2$, $a=\alpha(\alpha^2-1)$ and $b=\frac{(\alpha-1)(\alpha^3+1)}{\alpha}$. We have arrived to an algebra from (R5) with $a=\alpha$. If $q=d=0$, the equations in (\ref{eqeq1}) read $c=\alpha^2$ and $a=b=0$. In this case we have an algebra from (R6).

{\bf Case 1b:} $\alpha=-1$. Then the last equation in (\ref{eqeq1}) is satisfied automatically. The rest yield $c=1$, $a=q-d$ and $b=d(q-d)$. As above, a scaling allows to multiply $(q,d)$ by any non-zero constant. Thus if $q\neq d$, we can turn $d-q$ into $1$. Then the defining relations take form $xy+yx+zx+dzy$, $yy-xz(1-d)yz-zx-dzy$ and $zz$ and we have an algebra from (R1). It remains to consider the case $q=d$. Then $a=b=0$ and the defining relations take the form $xy+yx$, $yy-xz-dyz-zx-dzy$ and $zz$. If $d\neq 0$, a scaling transforms them into $xy+yx$, $yy-xz-yz-zx-zy$ and $zz$ and we arrive to the algebra (R2). If $d=0$, we have the algebra (R1).

{\bf Case 1c:} $\alpha=1$. In this case the equations (\ref{eqeq1}) read $a=b=0$, $c=1$ and $d=q$. The defining relations take the form $xy-yx$, $yy-xz-dyz-zx-dzy$ and $zz$. The isomorphism part of Lemma~\ref{xy+yxVyy} in the case $\alpha=1$ implies that all these algebras (when $d$ varies) are isomorphic to each other. Hence they are isomorphic to the algebra with $d=0$, which is (R4).

{\bf Case 2:} the $xz$ coefficient in $g$ is zero, while the $zx$ coefficient is non-zero.

We shall show that this case does not occur. A sub of the form $x\to x+sz$, $y\to y+tz$, $z\to rz$ with appropriately chosen $s,t\in\R$ and $r\in\K^*$ will kill the $zx$ and $zy$ coefficients in $f$ (the monomials $xz$ and $yz$ will creep back in) and turn the non-zero $zx$ coefficient of $g$ into $-1$. The defining relations of $A$ take form $xy-\alpha yx-axz-byz$, $yy-qyz-zx-dzy$ and $zz$ with $\alpha, a,b,d,q\in\K$, $\alpha\neq 0$. Since $A\in\Omega$, so is $A$ with the opposite multiplication. The latter is isomorphic to the quadratic algebra $C$ given by generators $x,y,z$ and relations $yx-\alpha xy-azx-bzy$, $yy-qzy-xz-dyz$ and $zz$. Up to scalar multiples, the relations of $C$ are of the form (\ref{dere2}). Hence the parameters must satisfy (\ref{eqeq1}), the first equation in which yields $0=\frac1{\alpha^2}$. Since this is obviously faulty, Case~2 actually does not occur.

{\bf Case 3:} both $xz$ and $zx$ do not feature in $g$.

In this case Lemma~\ref{2-2} yields $A\notin\Omega$, contradicting the assumptions. Hence this case does not occur as well. It remains to notice that algebras in (R1--R6) are all in the family given by relations (\ref{dere2}). The isomorphism part of  Lemma~\ref{xy+yxVyy} easily implies that algebras in (R1--R6) are pairwise non-isomorphic.
\end{proof}

\subsection{Case $R_0=\spann\{xx-xy,yx\}$}

\begin{lemma}\label{r4-1}Let $A$ be the quadratic algebra given by the generators $x,y,z$ and the relations $xx-xy-yz$, $yx-azx-bzy$ and $zz$ with $a,b\in\K$. If either $b=0$ and $a\neq-1$ or $b\neq 0$ and  $a(1+b+{\dots}+b^k)\neq 1$ for all $k\in\Z_+$, then $A\in\Omega^-$. Otherwise, $A\notin\Omega$. \end{lemma}

\begin{proof} The fact that $A\in\Omega^-$ whenever $A\in\Omega$ follows from Lemma~\ref{LL2}. It remains to prove that $A\in\Omega$ if and only if either $b=0$ and $a\neq-1$ or $b\neq 0$ and  $a(1+b+{\dots}+b^k)\neq 1$ for all $k\in\Z_+$.

First, we get rid of the easy case $b=0$. In this case we use the left-to-right degree lexicographical ordering assuming $z>y>x$. The ideal of relations of $A$ always has finite Gr\"obner basis. If $a\notin\{0,-1\}$, the said basis consists of the defining relations together with $zyx$, $xxz+xxy-xxx$, $yyx+axzx-axxx$ and $xxyx-\frac{a}{a+1}xxxx$. If $a=0$, the basis consists of the defining relations together with $xxz+xxy-xxx$. In both cases it follows that $H_A=(1-t)^{-3}$ and therefore $A\in\Omega$. On the other hand, if $a=-1$, then the Gr\"obner basis consists of the defining relations together with $zyx$, $xxz+xxy-xxx$, $yyx-xzx+xxx$ and $xxxx$, yielding $\dim A_5=23$. Hence $A\notin\Omega$ in this case. For the rest of the proof we shall assume that $b\neq 0$.

We start by a substitution, which turns the defining relations into a more convenient form. After the permutation $z\to y\to x\to z$, the defining relations take the form $yy$, $xz-ayz-byx$ and $xy+zx-zz$. Now we perform the sub $x\to x+ay$, $y\to y$ and $z\to z$, after which the relations take the form $yy$, $xz-byx$ and $xy+zx+azy-zz$. Finally, after scaling $x\to bx$, $y\to y$ and $z\to bz$, the relations acquire the shape
$yy$, $xz-yx$ and $xy+bzx+azy-bzz$, which shows that $A$ is isomorphic to the algebra $B$ given by the generators $x,y,z$ and the relations
\begin{equation}\label{dere-r9}
\text{$yy$, $xz-yx$ and $xy+bzx+azy-bzz$}.
\end{equation}
We shall compute the Gr\"obner basis of the ideal of relations of $B$ with respect to the left-to-right degree lexicographical ordering assuming $x>y>z$. First, a direct computation shows that the only degree $3$ and $4$ members of the reduced Gr\"obner basis of the ideal of relations of $B$ are
$$
\textstyle zzx-\frac{1-a}{b}zzy-zzz\ \ \text{and}\ \  zzyx-\frac{1-a}{b}zzyz-zzzz.
$$
Similarly, the degree $5$ part of the Gr\"obner basis consists of
$$
\text{$(a-1)zzyzz+bz^5$\ \ and\ \ $b^2zzyzx+(ab+a-1)zzyzy-b^2zzyzz+bz^5$.}
$$

Consider the recurrent sequence $\alpha_1=\frac{a-1}{b}$ and $\alpha_{k+1}=\frac{\alpha_k+a}{b}$ for $k\in\N$.
We shall prove inductively the following statement ($k\in\N$):
\begin{itemize}
\item[($G_k$)] If $\alpha_j\neq 0$ for $1\leq j\leq k$, then the complete list of the degree up to $2k+3$ elements of the reduced Gr\"obner basis of the ideal of relations of $B$ consists of
the defining relations $yy$, $xz-yx$, $xy+bzx+azy-bzz$, the degree $3$ element $zzx-\frac{1-a}{b}zzy-zzz$, $zz(yz)^{j-1}yx+\alpha_jzz(yz)^j+z^4u_j$, $zz(yz)^jz+\gamma_jz^{2j+3}$ and $zz(yz)^jx+\alpha_{j+1}zz(yz)^jy+\frac1b z^4u_jy+\gamma_jz^{2k+3}$ for $1\leq j\leq k$, where $u_m=\sum\limits_{j=0}^{m-1}\beta_{m,j}z^{2m-2j}(yz)^j\in B_{2m-2}$ ($\beta_{m,j}\in\K$) and $\gamma_m\in \K$ are defined recurrently: $u_1=-1$ (that is, $\beta_{1,0}=-1$), $\gamma_0=1$, $\gamma_1=-\frac1{\alpha_1}=\frac{b}{1-a}$ and for $m>1$, $\gamma_m=\sum\limits_{j=0}^{m-1}\gamma_j\beta_{m-1,j}$, $u_{m}=\frac1b u_{m-1}yz+\gamma_m z^{2m-2}$ (yielding a recurrent formula for $\beta_{m,j}$ as well).
\end{itemize}

If $\alpha_1\neq 0$ (that is, $a\neq 1$), then the degree up to $5$ elements of the Gr\"obner basis, collected in the above two displays, easily justify ($G_1$). Now assume that $k\in\N$ and ($G_k$) holds. We shall verify ($G_{k+1}$) provided $\alpha_j\neq 0$ for $1\leq j\leq k+1$. The overlap $zz(yz)^kxz$ produces $zz(yz)^kyx+\frac{\alpha_k+a}{b}+\gamma_kz^{2k+3}+\frac1b z^4u_ky$. The rest of degree $2k+4$ overlaps resolve. Using the recurrent definitions of $u_m$ and $\alpha_m$, we see that
\begin{equation}\label{wer1}
\begin{array}{l}
\text{the degree $2k+4$ part of the reduced Gr\"obner basis of $B$}
\\
\text{consists of one element $zz(yz)^{k}yx+\alpha_{k+1}zz(yz)^j+z^4u_{k+1}$.}
\end{array}
\end{equation}
Next,
\begin{equation}\label{wer2}
\text{the overlap $zz(zy)^kyxz$ produces $\alpha_{k+1}zz(yz)^{k+1}z+z^4u_{k+1}z$.}
\end{equation}
By assumption, $\alpha_{k+1}\neq 0$. Using the recurrent definitions of $u_m$ and $\alpha_m$, we see that after dividing by $\alpha_{k+1}$, the expression in (\ref{wer2}) acquires the shape $zz(yz)^{k+1}z+\gamma_{k+1}z^{2k+5}$. Finally,
\begin{equation}\label{wer3}
\text{the overlap $zz(zy)^kyxz$ produces $\textstyle zz(yz)^{k+1}x{+}\alpha_{k+1}zz(yz)^{k+1}y{-}zz(yz)^{k+1}z{+}\frac1b z^4u_{k+1}y$.}
\end{equation}
Since when $\alpha_{k+1}\neq 0$, we already have $zz(yz)^{k+1}z+\gamma_{k+1}z^{2k+5}$ in the ideal of relations and since the other overlaps of degree $2k+5$ resolve, we see that the degree $2k+5$ part of the reduced Gr\"obner basis of the ideal of relations of $B$ consists of two members: $zz(yz)^{k+1}z+\gamma_{k+1}z^{2k+5}$ and $zz(yz)^{k+1}x+\alpha_{k+1}zz(yz)^{k+1}y+\frac1b z^4u_{k+1}y+\gamma_{k+1}z^{2k+3}$. This concludes the proof of ($G_{k+1}$).

Now assume that $a(1+b+{\dots}+b^k)\neq 1$ for all $k\in\Z_+$. It easily follows that $\alpha_k\neq 0$ for every $k\in\N$. Indeed, one easily checks that $\alpha_{m+1}=\frac{a(1+{\dots}+b^m)-1}{b^{m+1}}$. Then the assumption of each ($G_k$) is satisfied and ($G_k$) produce the complete reduced Gr\"obner basis of the ideal of relations of $B$. The set of its leading monomials consists of $xy$, $xz$, $yy$, $zz(yz)^{k-1}x$, $zz(yz)^{k}yx$ and $zz(yz)^{k}z$ for $k\in\N$. As a result, the corresponding normal words are $z^\delta(yz)^my^\epsilon x^k$, $y^\delta(zy)^mz^pyx^q$ and $y^\delta(zy)^mz^p(yz)^ky^\epsilon$, where $\epsilon,\delta\in\{0,1\}$, $k,m\in\Z_+$, $q\in\N$ and $p\geq 2$. Now it is easy to see that the number of normal words of degree $n$ is exactly $\frac{(n+1)(n+2)}{2}$ and therefore $B\in\Omega$. Hence, $A\in\Omega$.

It remains to deal with the case when $a(1+b+{\dots}+b^m)=1$ for some $m\in\N$. Then there is the smallest $k\in\Z_+$ for which $\alpha_{k+1}=0$. Then ($G_k$) is satisfied. Furthermore, (\ref{wer1}--\ref{wer3}) still give the degrees $2k+4$ and $2k+5$ parts of the Gr\"obner basis of the ideal of relations of $B$ (the leading monomial in (\ref{wer2}) is $z^4(yz)^kz$). Now all degree $2k+6$ overlaps resolve except for $zz(yz)^{k+1}xz$, which yields a basis member with $zz(yz)^{k+1}$ as the leading monomial. Knowing the leading monomials of the
Gr\"obner basis members of degrees up to $2k+6$, we get that $\dim A_{2k+6}=\dim B_{2k+6}=\frac{(2k+7)(2k+8)}{2}+1$ and therefore $A\notin\Omega$.
\end{proof}

\begin{lemma}\label{case-R7-8} Let $A=A(V,R)\in\Omega$ be a quadratic algebra such that $\dim V=\dim R=3$ and with respect to some basis $x,y,z$ in $V$, $zz\in R$ and the quadratic algebra $B=A/I$ with $I$ being the ideal generated by $z$, is given by the relations $xx-xy$ and $yx$. Then $A$ is isomorphic to to a $\K$-algebra given by generators $x,y,z$ and three quadratic relations from {\rm (R7--R8)} of Theorem~$\ref{main}$. Furthermore, the algebras in {\rm (R7--R8)} belong to $\Omega^-$ and are pairwise non-isomorphic.
\end{lemma}

\begin{proof} By the assumptions, $R$ is spanned by $zz$, $xx-xy+f$ and $yx+g$ with $f,g\in\spann\{xz,zx,yz,zy\}$. Using a substitution $x\to x+sz$, $y\to y+tz$, $z\to z$ with appropriate $s,t\in\K$, we can kill the $xz$ and $zx$ coefficients of $f$. Then $A$ is given by the generators $x,y,z$ and the relations
\begin{equation}\label{dere3}
\text{$xx-xy-ayz-bzy$, $yx-pxz-qzx-cyz-dzy$ and $zz$},
\end{equation}
where $a,b,c,d,p,q\in\K$.

Then $A^{\rm opp}$, being $A$ with the opposite multiplication, is isomorphic to the algebra $C$ given by the generators $x,y,z$ and the relations $xx-yx-azy-byz$, $xy-pzx-qxz-czy-dyz$ and $zz$. After the substitution $x\to x+(p+c)z$, $y\to x-y+(q+d)z$, $z\to z$, the defining relations of $C$ take the shape $xx-xy+dyz-pzy$, $yx-bxz+(p+q+c+d)zx+(p+b+c)yz+azy$. That is,
\begin{equation}\label{dere3opp}
\begin{array}{l}
\text{$A^{\rm opp}$ is isomorphic to an algebra given by $(\ref{dere3})$ with the}\\ \text{parameters $(-d,p,-p-b-c,-a,b,-p-q-c-d)$ in place of $(a,b,c,d,p,q)$}
\end{array}
\end{equation}

Throughout the proof we again use the order (\ref{order}) on $x,y,z$ monomials. Resolving the overlaps $xxx$ and $yxx$, we see that the degree 3 part of the Gr\"obner basis of the ideal of relations of $A$ consists of two members
\begin{equation}\label{gbd3-2}
\begin{array}{l}
g_1=ayyz{-}pxzx{+}pxzy{-}cyzx{+}(b{+}c)yzy{+}dzyy{-}pdzxz{-}(aq{+}cd)zyz,
\\
g_2=xyy{+}(a{-}c{-}p)xyz{-}qxzx{+}(b{-}d)xzy{-}ayzx{+}ayzy{+}bzyy{-}pbzxz{-}(bc{+}pb)zyz.
\end{array}
\end{equation}

{\bf Case 1:} $a\neq 0$.

By means of scaling, we can turn $a$ into $1$. The leading monomials of $g_1$ and $g_2$ are now $yyz$ and $xyy$. One easily sees that $\dim A_3=10$ regardless what the values of other parameters are. There are $5$ degree $4$ overlaps $yyzz$, $xyyz$, $xyyx$, $xxyy$ and $yxyy$. The members of the ideal of relations obtained from the last two overlaps always belong to the linear span of the ones obtained from the first three. We denote these $r_1$, $r_2$ and $r_3$ respectively. The explicit formulae for $r_j$ (we assume $a=1$) are as follows:

\begin{equation*}
\begin{array}{l}
\scriptstyle r_1=pxzxz{-}pxzyz{+}cyzxz{-}(b{+}c)yzyz{-}pdzxzx{+}pdzxzy{-}cdzyzx{+}d(b{+}c)zyzy;
\\
\scriptstyle r_2={-}(p{+}c)xyzx{+}(p{+}b{+}c)xyzy{+}dxzyy{+}(q{-}pd)xzxz{+}(d{-}b{-}q{-}cd)xzyz{+}yzxz
\\
\quad\scriptstyle
{-}yzyz{-}pbzxzx{+}pbzxzy{-}(bc{+}pb)zyzx{+}b(b{+}c{+}p)zyzy;
\\
\scriptstyle
r_3=(p{+}c{-}q{-}1)xyzx{-}dxyzy{+}qxzxy{+}yzxy{+}(pd{-}pb{-}pq{-}cq)xzxz{+}(q{-}pd)xzyz
\\
\quad\scriptstyle
{-}(p{+}c)yzxz{+}yzyz{+}pbzxzx{+}(bc{+}pb{-}qb)zyzx{-}bdzyzy.
\end{array}
\end{equation*}

Note that there are exactly $17$ degree $4$ monomials which do not contain the leading terms $xx$, $yx$, $zz$, $yyz$ and $xyy$ of the members of the Gr\"obner basis of degree up to $3$. Since $A\in\Omega$, we must have $\dim A_4=15$. This happens precisely when the dimension of the space $L$ spanned by $r_1$, $r_2$ and $r_3$ is exactly $17-15=2$. On the other hand, $yzxy$ features in $r_3$ with coefficient $1$ and does not feature in $r_1$ or $r_2$. Thus $\dim L=2$ implies that the space $L_1$ spanned by $r_1$ and $r_2$ is one-dimensional.
The $2\times 2$ matrix $S_1$ of $yzxz$ and $yzyz$ coefficients in $r_1$ and $r_2$ is
$$
S_1=\left(\begin{array}{cc}c&-b-c\\ 1&-1\end{array}\right).
$$
Since $S_1$ must have rank at most $1$, we have $b=0$. Plugging this back into $r_1$ and $r_2$, we see that $L_1$ is one-dimensional precisely means that the rank of $S_2$ is $1$, $S_2$ being the matrix of coefficients of $r_1$ and $r_2$:
$$
S_2=\left(\begin{array}{ccccccc}0&0&p&-p&c&pd&cd\\ p+c&d&q-pd&d-q-cd&1&0&0\end{array}\right).
$$
Now it is elementary to see that rank of $S_2$ equals $1$ precisely when either $p=c=0$ or $d=p+c=q+1=0$. In the case $d=p+c=q+1=0$, the defining relations of $A$ take the shape $xx-xy-yz$, $yx-pxz+zx+pyz$ and $zz$. Using Gr\"obner basis, one easily sees that for generic $p$, $\dim A=23$ in this case. By Lemma~\ref{minhs}, $\dim A\geq23$ for all $p$, which violates the assumption $A\in\Omega$. Thus the case $d=p+c=q+1=0$ does not occur. This leaves us with the option $p=c=0$. Then the defining relations of $A$ take the shape $xx-xy-yz$, $yx-qzx-dzy$ and $zz$. By Lemma~\ref{r4-1}, $A$ falls into (R7) any algebra in (R7) belongs to $\Omega^-$. This concludes Case~1.

{\bf Case 2:} $a=0$ and $d\neq 0$.

By means of scaling, we can turn $d$ into $-1$. By (\ref{dere3opp}), $A^{\rm opp}$ is isomorphic to
an algebra given by $(\ref{dere3})$ with the parameters $(1,p,-p-b-c,0,b,-p-q-c+1)$ in place of $(a,b,c,d,p,q)$. The latter is under the jurisdiction of Case~1. From Case~1 we know that the inclusion $A^{\rm opp}\in\Omega$ yields $p=b=-p-b-c=0$. That is $p=b=c=0$. Plugging $a=b=c=p=0$ and $d=-1$ back into
(\ref{dere3}), we see that the defining relations of $A$ take the shape $xx-xy$, $yx-qzx+zy$ and $zz$. If additionally $q=0$, a direct Gr\"obner basis computation yields $\dim A_5=23$, which is incompatible with $A\in\Omega$. Thus $q\neq 0$ and we fall into (R8). Plugging $a=b=c=p=0$ and $d=-1$ into (\ref{gbd3-2}), we get $g_1=-zyy$ and $g_2=xyy-qxzx+xzy$. Continuing the computation of the Gr\"obner basis, we see that it actually turns out finite comprising the defining relations, $g_1$, $g_2$, $qzyzx-zyzy$ and $qxyzx-xyzy-qxzxy$. The leading monomials of the members of the Gr\"obner basis are $xx$, $yx$, $zz$, $zyy$, $xyy$, $zyzx$ and $xyzx$. This allows to compute the Hilbert series of $A$: $H_A=(1-t)^{-3}$. Thus algebras in (R8) are in $\Omega$. By Lemma~\ref{LL2}, they are in $\Omega^-$.

{\bf Case 3:} $a=d=0$ and $p\neq 0$.

By means of scaling, we can turn $p$ into $1$. Plugging $a=d=0$ and $p=1$ into (\ref{dere3}) and (\ref{gbd3-2}), we see that the defining relations of $A$ are $xx-xy-bzy$, $yx-xz-qzx-cyz$ and $zz$, while the degree $3$ members of the Gr\"obner basis of the ideal of relations of $A$ are $xzx-xzy+cyzx-(b+c)yzy$ and $xyy-(c+1)xyz+(b-q)xzy+qcyzx-q(b+c)yzy+bzyy-bzxz-b(c+1)zyz$. Resolving the degree $4$ overlaps of the leading monomials, we find that the degree $4$ part of the Gr\"obner basis is spanned by $4$ elements $r_1$, $r_2$, $r_3$ and $r_4$ (they come from the overlaps $yxzx$, $xxzx$, $xzxx$ and $xyyx$ respectively), where
\begin{equation*}
\begin{array}{l}
\scriptstyle r_1={-}cyyyz{+}(b{+}c)yyzy{+}qczyzx{-}q(b{+}c)zyzy;
\\
\scriptstyle r_2={-}(c{+}1)xyzx{+}(b{+}c{+}1)xyzy{-}bzyzx{+}bzyzy;
\\
\scriptstyle r_3={-}xzyy{-}(b{+}c)yzyy{+}(c{+}1)xzyz{+}byzxz{+}(c{+}1)(b{+}c)yzyz;
\\
\scriptstyle r_4=(c+1-q)xyzx+qxzyy-qcyzxy+q(b+c)yzyy-(b+q(c+1))xzyz+(qc(c+1)-b^2)yzxz
\\
\quad\scriptstyle
{-}q(c+1)(b+c)yzyz+bzxzy+b(1-q)zyzx+b(b+c)zyzy.
\end{array}
\end{equation*}
Note that there are exactly $18$ degree $4$ monomials which do not contain the leading terms $xx$, $yx$, $zz$, $xyy$ and $xzx$ of the members of the Gr\"obner basis of degree up to $3$. Since we have assumed that $A\in\Omega$, we must have $\dim A_4=15$. This happens precisely when the dimension of the space $L$ spanned by $r_1,r_2,r_3,r_4$ is exactly $18-15=3$. If $b\neq 0$, one easily sees that $r_j$ are linearly independent (just look at the $4\times 6$ matrix of $zxzy$, $xyzx$, $xyzy$, $yyzx$, $yyzy$ and $xzyy$ coefficients of $r_j$). Thus $b=0$. Next, it is easy to see that if $c$ is neither $0$ nor $-1$, $r_j$ are still linearly independent. Thus $c=0$ or $c=-1$. If $c=-1$, the defining relations of $A$ are $xx-xy$, $yx-xz-qzx+yz$ and $zz$. Using Gr\"obner basis technique, one easily sees that in this case for generic $q$, $\dim A_5=22$. By Lemma~\ref{minhs}, $\dim A_5\geq 22$ for all $q$, contradicting $A\in\Omega$. If $c=0$, the defining relations of $A$ are $xx-xy$, $yx-xz-qzx$ and $zz$. Using Gr\"obner basis technique, one easily sees that in this case for generic $q$, $\dim A_5=23$. By Lemma~\ref{minhs}, $\dim A_5\geq 23$ for all $q$, contradicting $A\in\Omega$. Thus this case does not occur.

{\bf Case 4:} \ $a=d=p=0$ and $c\neq 0$.

By means of scaling, we can turn $c$ into $1$. Plugging $a=d=p=0$ and $c=1$ into (\ref{dere3}) and (\ref{gbd3-2}), we see that the defining relations of $A$ are $xx-xy-bzy$, $yx-qzx-yz$ and $zz$, while the degree $3$ members of the Gr\"obner basis of the ideal of relations of $A$ are $yzx-(b+1)yzy$ and
$xyy-xyz-qxzx+bxzy+bzyy-bzyz$. Resolving the degree $4$ overlaps of the leading monomials, we find that the degree $4$ part of the Gr\"obner basis is spanned by $2$ elements $r_1$ and $r_2$ (they come from the overlaps $yzxx$ and $xyyx$ respectively), where
\begin{equation*}
\scriptstyle r_1=(b{+}1)(yzyy{-}yzyz)\ \ \text{and}\ \
r_2=(1{-}q)(b{+}1)(xyzy{+}bzyzy){+}q(xzxy{-}xzxz).
\end{equation*}
Note that there are exactly $16$ degree $4$ monomials which do not contain the leading terms $xx$, $yx$, $zz$, $xyy$ and $yzx$ of the members of the Gr\"obner basis of degree up to $3$. Since we have assumed that $A\in\Omega$, we must have $\dim A_4=15$. This happens precisely when the dimension of the space $L$ spanned by $r_1$ and $r_2$ is exactly $16-15=1$. This can only happen when $b=-1$. Thus the defining relations of $A$ are $xx-xy+zy$, $yx-qzx-yz$ and $zz$. Using Gr\"obner basis technique, one easily sees that in this case for generic $q$, $\dim A_5=23$. By Lemma~\ref{minhs}, $\dim A_5\geq 23$ for all $q$, contradicting $A\in\Omega$. Thus this case does not occur.

{\bf Case 5:} \ $a=d=p=c=0$.

The defining relations of $A$ are $xx-xy-bzy$, $yx-qzx$ and $zz$. If $b=0$, one easily sees that $\dim A_3>10$, which is incompatible with $A\in\Omega$. This $b\neq 0$. By scaling, we can make $b=1$. The defining relations of $A$ become $xx-xy-zy$, $yx-qzx$ and $zz$. Using Gr\"obner basis technique, one easily sees that in this case for generic $q$, $\dim A_5=22$. By Lemma~\ref{minhs}, $\dim A_5\geq 22$ for all $q$, contradicting $A\in\Omega$. Thus this final case does not occur.

It remains to show that algebras in (R7--R8) are pairwise non-isomorphic. For $B=A/I$ for algebras $A$ from (R7--R8), $x$ is up to a scalar multiple the only degree $1$ for which there exist degree $1$ elements $v$ and $w$ satisfying $vx=xw=0$ in $B$. Next, $y$ is up to a scalar multiple the only degree $1$ element satisfying $yx=0$ in $B$. Since $zz$ is the only square in $R$, a linear substitution providing an isomorphism between to algebras in (R7--R8) must be of the form $x\to ux+tz$, $y\to vy+sz$, $z\to wz$ with $u,v,w\in\K^*$ and $s,t\in\K$. Applying these to the relations in (R7--R8), we see that the only way such a sub transforms a space of quadratic relations corresponding to an algebra from (R7--R8) to another such space is to be of the form  $x\to ux$, $y\to uy$, $z\to uz$ with $u$. Since this substitution provides an automorphism of each algebra in (R7--R8), we see that  algebras in (R7--R8) are pairwise non-isomorphic.
\end{proof}

\subsection{Cases $R_0=\spann\{yy,yx\}$ and $R_0=\spann\{yy,xy\}$}

\begin{lemma}\label{AlA} Let $A$ be a quadratic algebra given by generators $x,y,z$ and relations
\begin{equation}\label{dere5}
\text{$yx-pxz-azx-bzy$, $yy-qxz-czx-dzy$ and $zz$},
\end{equation}
where $a,b,c,d,p,q\in\K$. Then $A\in\Omega$ if and only if
\begin{equation}\label{eqeq2}\textstyle
c\neq 0,\ \ q=0\ \ \text{and}\ \ bc=ad.
\end{equation}
Furthermore, two algebras from this family with parameters $a,b,c,d,p,q$ and $a',b',c',d',p',q'$ are isomorphic if and only if
$(a',b',c',d',p',q')=\bigl(ta,\frac{tb}{s},tsc,td,tp,tsq\bigr)$ for some $t,s\in\K^*$.
\end{lemma}

\begin{proof} Throughout the proof we again use the order (\ref{order}) on $x,y,z$ monomials. Resolving the overlaps $yyy$ and $yyx$, we see that the degree 3 part of the Gr\"obner basis of the ideal of relations of $A$ consists of two members
\begin{equation}\label{gbd5}
\begin{array}{l}
g_1={-}czxy{-}qxzy{+}cyzx{+}dyzy{+}q(a{-}d)zxz{+}qbzyz;\\ g_2={-}czxx{-}qxzx{+}ayzx{+}byzy{+}p(a{-}d)zxz{+}pbzyz.
\end{array}
\end{equation}

{\bf Case 1:} $c=0$.

By Lemma~\ref{2-2}, $A\notin \Omega$ if $q=0$. Assume that $q\neq 0$. By means of scaling, we can turn $q$ into $1$: $q =1$. Plugging $c=0$ and $q=1$ into (\ref{dere5}) and (\ref{gbd5}), we see that the defining relations of $A$ are $yx-pxz-azx-bzy$, $yy-xz-dzy$ and $zz$, while the degree 3 elements of the Gr\"obner basis of the ideal of relations of $A$ are $-xzy+dyzy+(a{-}d)zxz+bzyz$ and $-xzx+ayzx+byzy+p(a{-}d)zxz+pbzyz$. The degree $4$ part of the Gr\"obner basis of the ideal of relations is spanned by $2$ elements $r_1$ and $r_2$ arising from the overlaps $xzyy$ and $xzyx$ (other overlaps resolve). The explicit formulae for $r_j$ (we assume $q=1$) are as follows:
\begin{equation*}
\scriptstyle r_1=(d{-}a)yzxz{-}byzyz{+}(b{+}ad{-}d^2)zyzy\ \ \text{and}\ \ r_2=(b{+}a^2{-}ad)zyzx{+}((b{-}pd)(a{-}d){-}pb)zyzy.
\end{equation*}
Note that there are exactly $16$ degree $4$ monomials which do not contain the leading terms $yx$, $yy$, $zz$, $xzy$ and $xzx$ of the members of the Gr\"obner basis of degree up to $3$. Since $A\in\Omega$, we must have $\dim A_4=15$. This happens precisely when the dimension of the space $L$ spanned by $r_1$ and $r_2$ is exactly $16-15=1$. That is, the rank of the matrix $S$ of coefficients of $r_1$ and $r_2$ is $1$:
$$
{\rm rk}\,S=1,\ \ \text{where}\ \ S=\left(\begin{array}{cccc}d-a&-b&0&b+d(a-d)\\
0&0&b+a(a-d)&(b-pd)(a-d)-pb \end{array}\right).
$$
It is a routine exercise to see that $S$ has rank 1 only if $p=a$ and $b=-a(a-d)$. Now the defining relations of $A$ are $yx-axz-azx+a(a-d)zy$, $yy-xz-dzy$ and $zz$. A direct Gr\"obner basis computation shows that for generic $a,d$, $\dim A_5=22$. By Lemma~\ref{minhs}, $\dim A_5\geq 22$ for all $a,d$. Hence in Case~1, $A\notin\Omega$.

{\bf Case 1:} $c\neq 0$.

By means of scaling, we can turn $c$ into $1$: $c=1$.  One easily sees that $\dim A_3=10$ regardless what the values of other parameters are. The leading monomials of $g_1$ and $g_2$ are $zxy$ and $zxx$ respectively. The degree $4$ part of the Gr\"obner basis of the ideal of relations is spanned by $2$ elements $r_1$ and $r_2$ arising from the overlaps $zzxx$ and $zzxy$ (other overlaps resolve). The explicit formulae for $r_j$ (we assume $c=1$) are as follows:
\begin{equation*}
\scriptstyle r_1={-}qzxzx{+}azyzx{+}bzyzy\ \ \text{and}\ \ r_2={-}qzxzy{+}zyzx{+}dzyzy.
\end{equation*}
Note that there are exactly $16$ degree $4$ monomials which do not contain the leading terms $yx$, $yy$, $zz$, $zxy$ and $zxx$ of the members of the Gr\"obner basis of degree up to $3$. If $A\in\Omega$, we must have $\dim A_4=15$. This happens precisely when the dimension of the space $L$ spanned by $r_1$ and $r_2$ is exactly $16-15=1$. The latter happens precisely when $r_1$ is a scalar multiple of $r_2$. This, in turn, happens if and only if $q=0$ and $b=ad$, which corresponds to $q=0$ and $bc=ad$ before scaling. Thus $\dim A_4=15$ if and only if (\ref{eqeq2}) is satisfied. In this case, computing the next step of the Gr\"obner basis, we see that no new terms appear: the defining relations, $g_1$, $g_2$ and $r_2$ form the Gr\"obner basis of the ideal of relations of $A$. The leading monomials of the members of the basis are $yx$, $yy$, $zz$, $zxy$, $zxx$ and $zyzx$. This allows to compute the Hilbert series of $A$ and confirm that $H_A=(1-t)^{-3}$. Thus, $A\in\Omega$ if and only if (\ref{eqeq2}) is satisfied.

Now standard argument shows that only scalings transform the space of quadratic relations $R$ for any given algebra from (\ref{dere5}) to another space like that. Applying scalings to relations in  (\ref{dere5}), we easily confirm the required isomorphism statement.
\end{proof}

\begin{lemma}\label{AlA1} Let $A$ be a quadratic algebra given by generators $x,y,z$ and relations $(\ref{dere5})$. Then $A\in\Omega$ if and only if $A$ is isomorphic to an algebra  given by generators $x,y,z$ and relations
\begin{equation}\label{dere5-1}
\text{$yx-pxz-azx-abzy$, $yy-zx-bzy$ and $zz$},
\end{equation}
where $\alpha,a,b,p\in\K$. Furthermore, two algebras from this family with parameters $a,b,p$ and $a',b',p'$ are isomorphic if and only if $(a',b',p')=(ta,tb,tp)$ for some $t\in\K^*$.
\end{lemma}

\begin{proof} By Lemma~\ref{AlA}, $A\in\Omega$ if and only if $c\neq 0$, $q=0$ and $bc=ad$. A scaling allows to turn $c$ into $1$, which brings $A$ into the subfamily (\ref{dere5-1}). By the same criterion for $A$ to be in $\Omega$, every algebra given by the relations (\ref{dere5-1}) is in $\Omega$. As for isomorphisms, the result immediately follows from the isomorphism part of Lemma~\ref{AlA}.
\end{proof}

\begin{lemma}\label{case-R9-12}
Let $A=A(V,R)\in\Omega$ be a quadratic algebra such that $\dim V=\dim R=3$ and with respect to some basis $x,y,z$ in $V$, $zz\in R$ and the quadratic algebra $B=A/I$ with $I$ being the ideal generated by $z$, is given by the relations $yy$ and $yx$. Then $A$ is isomorphic to to a $\K$-algebra given by generators $x,y,z$ and three quadratic relations from {\rm (R9--R12)} of Theorem~$\ref{main}$. Furthermore, the algebras in {\rm (R9--R12)} belong to $\Omega^-$ and are pairwise non-isomorphic.
\end{lemma}

\begin{proof} The assumptions imply that $R$ is spanned by $zz$, $yx+f$ and $yy+g$ with $f,g\in\spann\{xz,zx,yz,zy\}$. Using a linear substitution $x\to x+sz$, $y\to y+tz$, $z\to z$ with appropriate $s,t\in\K$, we can kill the  $yz$ coefficients of $f$ and $g$. Then $A$ is given by the generators $x,y,z$ and the relations (\ref{dere5}). By Lemma~\ref{AlA1}, $A$ is isomorphic to an algebra given by relations (\ref{dere5-1}). If $b\neq 0$, then a scaling turns $b$ into $1$ and we have an algebra from (R9). If $b=0$ and $a\neq 0$, a scaling turns $a$ into $1$, while keeping the equality $b=0$ and we have an algebra from (R10). If $a=b=0$ and $p\neq 0$, a scaling turns $p$ into $1$, while keeping the equality $a=b=0$ and we have the algebra (R11). Finally, if $a=b=p=0$, we already have the algebra (R12). Finally, Lemma~\ref{AlA1} immediately implies that algebras in (R9--R12) belong to $\Omega$ and are pairwise non-isomorphic. They belong to $\Omega^-$ by Lemma~\ref{LL2}.
\end{proof}

\begin{lemma}\label{case-R13-16}
Let $A=A(V,R)\in\Omega$ be a quadratic algebra such that $\dim V=\dim R=3$ and with respect to some basis $x,y,z$ in $V$, $zz\in R$ and the quadratic algebra $B=A/I$ with $I$ being the ideal generated by $z$, is given by the relations $yy$ and $xy$. Then $A$ is isomorphic to to a $\K$-algebra given by generators $x,y,z$ and three quadratic relations from {\rm (R13--R16)} of Theorem~$\ref{main}$. Furthermore, the algebras in {\rm (R13--R16)} belong to $\Omega^-$ and are pairwise non-isomorphic.
\end{lemma}

\begin{proof}
This lemma is obviously equivalent to Lemma~\ref{case-R9-12}. One just has to pass to the opposite multiplication.
\end{proof}

\section{Cases $R_0=\spann\{xy,yx\}$}

\begin{lemma}\label{r9-1}Let $A=A^a$ be the quadratic algebra given by the generators $x,y,z$ and the relations $xy-azx-zy$, $yx-xz$ and $zz$ with $a\in\K$. If for some $k\in\N$, $1+a+{\dots}+a^k=0$, then $A\notin\Omega$. On the other hand, if $1+a+{\dots}+a^k\neq 0$ for each $k\in\N$, then $A\in \Omega^-$. Furthermore, the algebras $A^a$ for $a\in\K$ are pairwise non-isomorphic.
\end{lemma}

\begin{proof} The isomorphism statement is easy and routine. We have done similar verifications many times now. The fact that $A\in\Omega^-$ whenever $A\in\Omega$ follows from Lemma~\ref{LL2}. It remains to prove that $A\in\Omega$ if and only if $1+a+{\dots}+a^k\neq 0$ for each $k\in\N$.

We start by a substitution, which turns the defining relations into a more convenient form. After swapping $x$ and $z$, the defining relations take the form $xx$, $yz-zx$ and $xy+axz-zy$. Now we perform an extra sub $x\to x$, $y\to y-az$ and $z\to z$, which shows that $A$ is isomorphic to the algebra $B$ given by the generators $x,y,z$ and the relations
\begin{equation}\label{dere-r4}
\text{$xy-zy+azz$, $yz-zx-azz$ and $xx$}.
\end{equation}
We shall compute the Gr\"obner basis of the ideal of relations of $B$ with respect to the left-to-right degree lexicographical ordering assuming $x>y>z$. For the sake of convenience, we denote:
$$
\text{$r_{-1}=0$, $r_1=1$ and $r_k=1+a+{\dots}+a^k$ for $k\in\N$}.
$$

We shall prove inductively the following statement:
\begin{itemize}
\item[($G_k$)] If $r_j\neq 0$ for $1\leq j\leq k$, then the complete list of the degree up to $2k+3$ elements of the reduced Gr\"obner basis of the ideal of relations of $B$ consists of $yz-zx-azz$, $xz^{2j}x$, $xz^{2j}y-\frac{ar_{j-1}}{r_j}xz^{2j+1}-\frac1{r_j}z^{2j+1}y+\frac{a}{r_j}z^{2j+2}$, $xz^{2j+1}y-axz^{2j+2}$ and $xz^{2j+1}x+\frac{a^{j}(a-1)}{r_j}xz^{2j+2}-\frac1{r_j}z^{2j+2}x$
for $0\leq j\leq k$.
\end{itemize}

We start with the basis of induction. The overlap $xxx$ resolves (produces no Gr\"obner basis elements of degree $3$). The overlap $xyz$ produces $xzx+(a-1)xz^{2j+2}-z^{2j+2}x$, while the only remaining overlap $xxy$ yields $xzy-axzz$. Thus the complete list of the degree up to $3$ elements of the reduced Gr\"obner basis of the ideal of relations of $B$ consists of the defining relations together with $xzx+(a-1)xz^{2j+2}-z^{2j+2}x$ and $xzy-axzz$. This is exactly the statement ($G_0$). Now assume that ($G_k$) holds. We shall verify ($G_{k+1}$). According to ($G_k$), we already know the members of the basis of degrees up to $2k+3$. The only degree $2k+4$ overlaps are $xxz^{2k+1}x$, $xxz^{2k+1}y$, $xz^{2k+1}xx$, $xz^{2k+1}xy$ and $xz^{2k+1}yz$.
Each of $xxz^{2k+1}x$, $xz^{2k+1}xx$ and $xz^{2k+1}yz$ produce the same element $xz^{2k+2}x$ of the reduced Gr\"obner basis, while $xxz^{2k+1}y$ resolves. Finally, $xz^{2k+1}xy$ produces
\begin{equation}\label{one1}
r_{k+1}xz^{2k+2}y-ar_{k}xz^{2k+3}-z^{2k+3}y+az^{2k+4}.
\end{equation}
Provided $r_{k+1}\neq 0$, we get $xz^{2k+2}y-\frac{ar_{k}}{r_{k+1}}xz^{2k+3}-\frac1{r_{k+1}}z^{2k+3}y+\frac{a}{r_{k+1}}z^{2k+4}$.
Now at degree $2k+4$ we have just two members of the reduced Gr\"obner basis: $xz^{2k+2}x$ and $xz^{2k+2}y-\frac{ar_{k}}{r_{k+1}}xz^{2k+3}-\frac1{r_{k+1}}z^{2k+3}y+\frac{a}{r_{k+1}}z^{2k+4}$.
The only degree $2k+5$ overlaps are $xxz^{2k+2}x$, $xxz^{2k+2}y$, $xz^{2k+2}xx$, $xz^{2k+2}xy$ and $xz^{2k+2}yz$. Now $xxz^{2k+2}x$ and $xz^{2k+2}xx$ resolve, both $xxz^{2k+2}y$ and $xz^{2k+2}xy$ produce
$xz^{2k+3}y-axz^{2k+4}$, while $xz^{2k+2}yz$ produces $xz^{2k+3}x+\frac{a^{k+1}(a-1)}{r_{k+1}}xz^{2k+4}-\frac1{r_{k+1}}z^{2k+4}x$. This proves ($G_{k+1}$) and concludes the inductive proof of ($G_j$).

Now assume that $1+a+{\dots}+a^k\neq 0$ for each $k\in\N$ (that is, $r_k\neq 0$ for $k\in\N$). Then the assumption of each ($G_k$) is satisfied and ($G_k$) produce the complete reduced Gr\"obner basis of the ideal of relations of $B$. The set of its leading monomials consists of $yz$, $xz^kx$ and $xz^ky$ with $k\in\Z_+$. As a result, the corresponding normal words are $z^my^n$ and $z^my^nxz^j$ for $k,m,j\in\Z_+$. One easily sees that the number of these words of degree $k$ is $\frac{(k+1)(k+2)}{2}$ and therefore $H_B=H_A=(1-t)^{-3}$. Hence $A\in\Omega$.

It remains to deal with the case when $r_j=0$ for some $j\in\N$. Then we can pick the minimal $k\in\Z_+$ such that $r_{k+1}=0$. Then the assumptions of ($G_k$) are satisfied. By  ($G_k$) the complete list of the degree up to $2k+3$ elements of the reduced Gr\"obner basis of the ideal of relations of $B$ consists of $yz-zx-azz$, $xz^{2j}x$, $xz^{2j}y-\frac{ar_{j-1}}{r_j}xz^{2j+1}-\frac1{r_j}z^{2j+1}y+\frac{a}{r_j}z^{2j+2}$, $xz^{2j+1}y-axz^{2j+2}$ and $xz^{2j+1}x+\frac{a^{j}(a-1)}{r_j}xz^{2j+2}-\frac1{r_j}z^{2j+2}x$
for $0\leq j\leq k$. We already know that the degree $2k+4$ part of the basis consists of $xz^{2k+2}x$ and (\ref{one1}). Using the equation $r_{k+1}=0$, we can (up to scaling) rewrite (\ref{one1}) as $xz^{2k+3}-z^{2k+3}x+az^{2k+4}$. Continuing the process, we obtain the degrees $2k+5$, $2k+6$ and $2k+7$ parts of the Gr\"obner basis, which comprise $z^{2k+3}(yy-azy-azx)$ and $zxz^{2k+2}yy-a^2zxz^{2k+4}-az^{2k+4}yx+a^2z^{2k+5}x$ (one element in each of degrees $2k+5$ and $2k+6$ and nothing of degree $2k+7$). Now it is easy to verify that $\dim A_{2k+7}=\frac{(2k+8)(2k+9)}{2}+1$ and therefore $A\notin\Omega$.
\end{proof}

\begin{lemma}\label{case-R17-19}
Let $A=A(V,R)\in\Omega$ be a quadratic algebra such that $\dim V=\dim R=3$ and with respect to some basis $x,y,z$ in $V$, $zz\in R$ and the quadratic algebra $B=A/I$ with $I$ being the ideal generated by $z$, is given by the relations $yx$ and $xy$. Then $A$ is isomorphic to to a $\K$-algebra given by generators $x,y,z$ and three quadratic relations from {\rm (R17--R19)} of Theorem~$\ref{main}$. Furthermore, the algebras in {\rm (R17--R19)} belong to $\Omega^-$ and are pairwise non-isomorphic.
\end{lemma}

\begin{proof} Algebras in (R19) are pairwise non-isomorphic by Lemma~\ref{r9-1}. The two algebras from (R17) and (R18) are easily seen to be non-isomorphic to each other and to any of the algebras from (R19). Thus we can forget about isomorphisms and concentrate on $A$.

The assumptions imply that $R$ is spanned by $zz$, $xy+f$ and $yx+g$ with $f,g\in\spann\{xz,zx,yz,zy\}$. Using a substitution $x\to x+sz$, $y\to y+tz$, $z\to z$ with appropriate $s,t\in\K$, we can kill the  $xz$ coefficient of $f$ and the $yz$ coefficient of $g$. Then $A$ is given by the generators $x,y,z$ and the relations
\begin{equation}\label{dere7}
\text{$xy-pyz-azx-bzy$, $yx-qxz-czx-dzy$ and $zz$},
\end{equation}
where $a,b,c,d,p,q\in\K$.

Note also that $A^{\rm opp}$ is isomorphic to the quadratic algebra $C$ given by the generators $x,y,z$ and the relations $yx-pzy-axz-byz$, $xy-qzx-cxz-dyz$ and $zz$. After the sub $x\to x+bz$, $y\to y+cz$ and $z\to z$, the relations of $C$ take the shape $xy-dyz-qzx+bzy$, $yx-axz+czx-pzy$ and $zz$. Hence
\begin{equation}\label{dere7opp2}
\begin{array}{l}
\text{$A^{\rm opp}$ is isomorphic to an algebra given by $(\ref{dere7})$ with the}\\ \text{parameters $(q,-b,-c,p,d,a)$ in place of $(a,b,c,d,p,q)$}
\end{array}
\end{equation}

Throughout the proof we again use the order (\ref{order}) on $x,y,z$ monomials.
Resolving the overlaps $yxy$ and $xyx$, we see that the degree 3 part of the Gr\"obner basis of the ideal of relations of $A$ consists of two members
\begin{equation}\label{gbd7}
g_1=pyyz{-}qxzy{+}ayzx{+}byzy{-}dzyy{-}pczyz\ \text{and}\
g_2=qxxz{+}cxzx{+}dxzy{-}pyzx{-}azxx{-}qbzxz.
\end{equation}

{\bf Case 1:} $pq\neq 0$.

Scaling $x$ and $y$, we can turn $p$  and $q$ into $1$: $p=q=1$. One easily sees that $\dim A_3=10$ regardless what the values of other parameters are. The leading monomials of $g_1$ and $g_2$ are $yyz$ and $xxz$ respectively. The degree $4$ part of the Gr\"obner basis of the ideal of relations is spanned by $2$ elements $r_1$ and $r_2$ arising from the overlaps $xxzz$ and $yyzz$ (other overlaps resolve). The explicit formulae for $r_j$ (we assume $p=q=1$) are as follows:
\begin{equation*}
\scriptstyle r_1={-}cxzxz{-}dxzyz{+}yzxz{-}aczxzx{-}adzxzy{+}azyzx;\ \
\scriptstyle r_2=xzyz{-}ayzxz{-}byzyz{+}dzxzy{-}adzyzx{-}bdzyzy.
\end{equation*}
Note that there are exactly $16$ degree $4$ monomials which do not contain the leading terms $yx$, $xy$, $zz$, $xxz$ and $yyz$ of the members of the Gr\"obner basis of degree up to $3$. Since $A\in\Omega$, we must have $\dim A_4=15$. This happens precisely when the dimension of the space $L$ spanned by $r_1$ and $r_2$ is exactly $16-15=1$. It is easy to observe that $r_1$ and $r_2$ are linearly independent if $b\neq 0$ or $c\neq 0$. This yields $b=c=0$. Now $L$ has dimension $1$ precisely when the vectors $(-d,1,-ad,a)$ and $(1,-a,d,-ad)$ are proportional. The latter occurs exactly when $a^2=d^2=ad=1$. That is, $(a,d)=(1,1)$ or $(a,d)=(-1,1)$ and we end up with two algebras from (R17) and (R18). After checking that in the case $p=q=1$, $b=c=0$ and $(a,d)\in\{(1,1),(-1,-1)\}$, the Gr\"obner basis of the ideal of relations of $A$ is finite and consists of defining relations, $g_1$, $g_2$ and $r_2$, one gets the complete set of leading monomials of the members of the basis: $xy$, $yx$, $zz$, $xxz$, $yyz$ and $xzyz$. This allows to compute the Hilbert series of $A$ and confirm that $H_A=(1-t)^{-3}$. Thus, indeed, $A\in\Omega$ for $A$ from (R17) and (R18). By Lemma~\ref{LL2}, $A\in\Omega^-$ for $A$ from (R17) and (R18).

{\bf Case 2:} $ad\neq 0$.

According to (\ref{dere7opp2}), $A^{\rm opp}$ falls under the jurisdiction of Case~1. Since both algebras in (R17) and (R18) are isomorphic to their own opposites, in this case we again get only algebras from (R17) and (R18).

{\bf Case 3:} $ad=pq=0$.

There are options here. We have either $a=p=0$, or $q=d=0$, or $p=d=0$, or $q=a=0$.

{\bf Case 3a:} $a=p=0$.

In this case the defining relations of $A$ are $xy-bzy$, $yx-qxz-czx-dzy$ and $zz$, while $g_1$ and $g_2$ take the form $g_1=-qxzy+byzy-dzyy$ and $g_2=qxxz+cxzx+dxzy-qbzxz$. The leading monomials of $g_1$ and $g_2$ now are $xzy$ and $xxz$ respectively. The degree $4$ part of the Gr\"obner basis of the ideal of relations is spanned by $3$ elements $r_1$, $r_2$ and $r_3$ arising from the overlaps $xxzz$, $xzyx$ and $xxzy$ (other overlaps resolve). The explicit formulae for $r_j$ (we assume $q=1$) are as follows:
\begin{equation*}
\scriptstyle r_1={-}cxzxz{-}bdyzyz{+}d^2zyyz;\ \ r_2={-}xzxz{+}byzxz{-}dczyzx{-}d^2zyzy;
\ \ r_3=byyzy{-}dyzyy{-}(d{+}bc)zyzy.
\end{equation*}
Note that there are exactly $17$ degree $4$ monomials which do not contain the leading terms $yx$, $xy$, $zz$, $xzy$ and $xxz$ of the members of the Gr\"obner basis of degree up to $3$. Since $A\in\Omega$, we must have $\dim A_4=15$. This happens precisely when the dimension of the space $L$ spanned by $r_1$, $r_2$ and $r_3$ is exactly $17-15=2$. One easily sees that $r_1$, $r_2$ and $r_3$ are linearly independent provided $d\neq 0$. Thus $d=0$. Again, it is easily observed that $r_1$, $r_2$ and $r_3$ are linearly independent if $bc\neq 0$.
Hence $bc=0$. Next, if $b=d=0$, then $L$ is one-dimensional. Hence $b\neq 0$. The only option left is $c=d=0$ and $b\neq 0$. Now the defining relations of $A$ are $xy-bzy$, $yx-xz$ and $zz$ with $b\in\K^*$. After the sub $x\to x$, $y\to \frac yb$ and $z\to \frac zb$, the space of defining relations is spanned by $xy-zy$, $yx-xz$ and $zz$ and we have an algebra from (R19) corresponding to the parameter $a$ equal $0$.

It remains to deal with the case $q=0$. That is, $a=p=q=0$. Then the defining relations of $A$ are $xy-bzy$, $yx-czx-dzy$ and $zz$, while $g_1$ and $g_2$ take the form $g_1=byzy-dzyy$ and $g_2=cxzx+dxzy$. If $bc\neq 0$, we can turn $b$ and $c$ into $1$ by scaling the variables. The leading monomials of $g_1$ and $g_2$ are now $xzx$ and $yzy$. The degree $4$ part of the Gr\"obner basis now is spanned by $dxzyy$ and $d(zyzx+dzyzy)$. If $d=0$, this yields $\dim A_4=16$, while if $d\neq 0$, we have $\dim A_4=14$. Since both are incompatible with $A\in\Omega$, we must have $bc=0$. If $b=0$ and $c\neq 0$, we turn $c$ into $1$ by scaling. Then the defining relations of $A$ are $xy$, $yx-zx-dzy$ and $zz$, while $g_1$ and $g_2$ take the form $g_1=-dzyy$ and $g_2=xzx+dxzy$. If $d=0$, we have $\dim A_3=11$. Hence $d\neq 0$. The degree $4$ part of the Gr\"obner basis now is spanned by $xzyy$ and $zyzx+dzyzy$ yielding $\dim A_4=14$. Since $A\in\Omega$, this can not occur. If $b\neq 0$ and $c=0$, we turn $b$ into $1$ by scaling. The defining relations of $A$ are $xy-zy$, $yx-dzy$ and $zz$, while $g_1$ and $g_2$ take the form $g_1=yzy-dzyy$ and $g_2=dxzy$. Again $\dim A_3=11$ if $d=0$. Hence $d\neq 0$. Computing the degree $4$ part of the Gr\"obner basis, we get $\dim A_4=16$, contradicting $A\in\Omega$. Finally, if $b=c=0$, then the defining relations of $A$ are $xy$, $yx-dzy$ and $zz$, while $g_1$ and $g_2$ take the form $g_1=-dzyy$ and $g_2=dxzy$. If $d=0$, $\dim A_3=12$, while if $d\neq 0$, then $\dim A_4=16$. Since both are incompatible with $A\in\Omega$, we conclude that the case $q=0$ does not occur.

{\bf Case 3b:} $q=d=0$.

This case is obtained from Case~3a by the sub swapping $x$ and $y$. Thus no new algebras occur here.

{\bf Case 3c:} $p=d=0$.

If $q=0$ or $a=0$, we fall into the already considered Cases~3a or~3b. Thus we can assume that $qa\neq 0$. By scaling, we can turn $q$ into $1$ (can not normalize $a$ at the same time). In this case the defining relations of $A$ are $xy-azx-bzy$, $yx-xz-czx$ and $zz$ with $a\neq 0$, while $g_1$ and $g_2$ take the form $g_1=-xzy+ayzx+byzy$ and $g_2=xxz+cxzx-azxx-bzxz$. The leading monomials of $g_1$ and $g_2$ now are $xzy$ and $xxz$ respectively. The degree $4$ part of the Gr\"obner basis of the ideal of relations is spanned by $3$ elements $r_1$, $r_2$ and $r_3$ arising from the overlaps $xxzz$, $yxxz$ and $xzyx$ (other overlaps resolve). The explicit formulae for $r_j$ are as follows:
\begin{equation*}
\scriptstyle r_1=c(xzxz{+}azxzx);\ \ r_2=ayzxx{-}xzxz{+}byzxz;\ \ r_3=ayyzx{+}byyzy{-}aczyzx{-}bczyzy.
\end{equation*}
Note that there are exactly $17$ degree $4$ monomials which do not contain the leading terms $yx$, $yy$, $zz$, $xzy$ and $xxz$ of the members of the Gr\"obner basis of degree up to $3$. Since $A\in\Omega$, we must have $\dim A_4=15$. This happens precisely when the dimension of the space $L$ spanned by $r_1$, $r_2$ and $r_3$ is exactly $17-15=2$. That is, the rank of the matrix $S$ of the coefficients of $r_j$ must be $2$:
$$
{\bf rk}\,S=2,\ \ \text{where}\ \ S=\left(\begin{array}{cccccccc}
0&0&0&c&0&ac&0&0\\ 0&0&a&-1&b&0&0&0\\ a&b&0&0&0&0&-ac&-bc
\end{array}\right).
$$
Recall that $a\neq 0$. If $c\neq 0$, then ${\rm rk}\,S=3$. This leaves only the case $c=0$ in which the rank of $S$ is indeed 2. Then the defining relations of $A$ are $xy-azx-bzy$, $yx-xz$ and $zz$. If $b=0$, a direct Gr\"obner basis computation shows that $\dim A_5=22$ for generic $a$. By Lemma~\ref{minhs}, $\dim A_5\geq 22$ for all $a$, contradicting $A\in\Omega$. Hence $b\neq 0$. An extra scaling turns $b$ into $1$. Then the defining relations of $A$ are $xy-azx-zy$, $yx-xz$ and $zz$ with $a\in\K^*$. By Lemma~\ref{r9-1}, we fall into (R19) and $A\in\Omega^-$.

{\bf Case 3d:} $q=a=0$.

This case is obtained from Case~3c by swapping $x$ and $y$. Thus no new algebras occur here. The proof is now complete.
\end{proof}

\subsection{Case $R_0=\spann\{xx,yy\}$}

\begin{lemma}\label{case-R20} Let $A=A(V,R)\in\Omega$ be a quadratic algebra such that $\dim V=\dim R=3$ and with respect to some basis $x,y,z$ in $V$, $zz\in R$ and the quadratic algebra $B=A/I$ with $I$ being the ideal generated by $z$, is given by the relations $xx$ and $yy$. Then $A$ is isomorphic to to a $\K$-algebra given by generators $x,y,z$ and three quadratic relations from {\rm (R20)} of Theorem~$\ref{main}$. Furthermore, the isomorphism conditions in {\rm (R20)} are satisfied and all algebras in {\rm (R20)} belong to $\Omega^-$.
\end{lemma}

\begin{proof} The assumptions imply that $R$ is spanned by $zz$, $xx+f$ and $yy+g$ with $f,g\in\spann\{xz,zx,yz,zy\}$. Using a substitution $x\to x+sz$, $y\to y+tz$, $z\to z$ with appropriate $s,t\in\K$, we can kill the  $xz$ coefficient of $f$ and the $yz$ coefficient of $g$. Then $A$ is given by the generators $x,y,z$ and the relations
\begin{equation}\label{dere8}
\text{$xx-pyz-azx-bzy$, $yy-qxz-czx-dzy$ and $zz$},
\end{equation}
where $a,b,c,d,p,q\in\K$.

Note also that $A^{\rm opp}$ is isomorphic to the quadratic algebra $C$ given by the generators $x,y,z$ and the relations $xx-pzy-axz-byz$, $yy-qzx-cxz-dyz$ and $zz$. After the sub $x\to x+az$, $y\to y+dz$ and $z\to z$, the relations of $C$ take the shape $xx-byz+azx-pzy$, $yy-cxz-qzx+dzy$ and $zz$. Hence
\begin{equation}\label{dere7opp1}
\begin{array}{l}
\text{$A^{\rm opp}$ is isomorphic to an algebra given by $(\ref{dere8})$ with the}\\ \text{parameters $(-a,p,q,-d,b,c)$ in place of $(a,b,c,d,p,q)$}
\end{array}
\end{equation}

Another useful observation is that swapping $x$ and $y$ preserves the shape of the relations and changes the parameters by $(a,b,c,d,p,q)\to (d,c,b,a,q,p)$

Throughout the proof we again use the order (\ref{order}) on $x,y,z$ monomials.
Resolving the overlaps $xxx$ and $yyy$, we see that the degree 3 part of the Gr\"obner basis of the ideal of relations of $A$ consists of two members:
\begin{equation}\label{gbd71}
g_1=pxyz{+}axzx{+}bxzy{-}pyzx{-}bzyx{-}pazyz;\ \
g_2=qyxz{-}qxzy{+}cyzx{+}dyzy{-}czxy{-}qdzxz.
\end{equation}

{\bf Case 1:} $pq\neq 0$.

By a scaling, we can turn $p$  and $q$ into $-1$: $p=q=-1$. One easily sees that $\dim A_3=10$ regardless what the values of other parameters are. The leading monomials of $g_1$ and $g_2$ are $xyz$ and $yxz$ respectively. The degree $4$ part of the Gr\"obner basis of the ideal of relations is spanned by $2$ elements $r_1$ and $r_2$ arising from the overlaps $xyzz$ and $yxzz$ (other overlaps resolve). The explicit formulae for $r_j$ (we assume $p=q=1$) are as follows:
\begin{equation*}
\scriptstyle r_1=axzxz{+}bxzyz{+}yzxz{-}bzxzy{-}bczyzx{-}bdzyzy;
\ \  r_2=xzyz{+}cyzxz{+}dyzyz{-}aczxzx{-}bczxzy{-}czyzx.
\end{equation*}
Note that there are exactly $16$ degree $4$ monomials which do not contain the leading terms $xx$, $yy$, $zz$, $xyz$ and $yxz$ of the members of the Gr\"obner basis of degree up to $3$. Since $A\in\Omega$, we must have $\dim A_4=15$. This happens precisely when the dimension of the space $L$ spanned by $r_1$ and $r_2$ is exactly $16-15=1$. It is easy to observe that $r_1$ and $r_2$ are linearly independent if $a\neq 0$ or $d\neq 0$. This yields $a=d=0$. Now $L$ has dimension $1$ precisely when the vectors $(b,1,b,bc)$ and $(1,c,bc,c)$ are proportional. The latter occurs exactly when $bc=1$. That is, $b=-\alpha$ and $c=-\frac1\alpha$ for some $\alpha\in\K^*$. Then the defining relations of $A$ take the form $xx+yz+\alpha zy$, $yy+xz+\frac1\alpha zx$ and $zz$ with $\alpha\in\K^*$ landing us into (R20). Swapping of $x$ and $y$ provides an isomorphism between two such algebra parameters $\alpha$ and $\frac1\alpha$. Observing that the shape (\ref{dere8}) of relations is preserved only by scaling of the variables and scaling combined with swapping of $x$ and $y$. This allows to easily verify that there are no other isomorphic algebras in (R20) than indicated in Theorem~\ref{main}. Next, as soon as we have gotten into (R20) the Gr\"obner basis of the ideal of relations of $A$ became finite consisting of defining relations, $g_1$, $g_2$ and $r_2$. The complete set of leading monomials of the members of the basis is: $xx$, $yy$, $zz$, $xyz$, $yxz$ and $xzyz$. This allows to compute the Hilbert series of $A$ and confirm that $H_A=(1-t)^{-3}$. Thus $A\in\Omega$. By Lemma~\ref{LL2}, $A\in\Omega^-$.

{\bf Case 2:} $bc\neq 0$.

According to (\ref{dere7opp1}), $A^{\rm opp}$ falls under the jurisdiction of Case~1. Since the class (R20) is closed under passing to the opposite multiplication, in this case we again get only algebras from (R20).

{\bf Case 3:} $b=p=0$ or $c=q=0$.

In this case Lemma~\ref{2-2} yields $A\notin \Omega$. This contradiction shows that the case does not occur.
What is left to consider are the cases $p=c=0$ and $q=b=0$.

{\bf Case 4:} $q=b=0$.

If $p=0$ or $c=0$, we fall into the already considered Cases~3. Thus we can assume that $pc\neq 0$. By scaling, we can turn $p$ and $c$ into $1$. In this case the defining relations of $A$ are $xx-yz-azx$, $yy-zx-dzy$ and $zz$, while $g_1$ and $g_2$ take the form $g_1=xyz+axzx-yzx-azyz$ and $g_2=yzx+dyzy-zxy$.
The leading monomials of $g_1$ and $g_2$ now are $xyz$ and $yzx$ respectively. The degree $4$ part of the Gr\"obner basis of the ideal of relations is spanned by $3$ elements $r_1$, $r_2$ and $r_3$ arising from the overlaps $xyzz$, $yyzx$ and $yzxx$ (other overlaps resolve). The explicit formulae for $r_j$ are as follows:
\begin{equation*}
\scriptstyle r_1=axzxz{+}dyzyz{+}azxzx{+}dzyzy;\ \ r_2=zxzx{-}d^2zyzy;\ \
r_3={-}zxyx{+}dyzyx{+}yzyz.
\end{equation*}
Note that there are exactly $17$ degree $4$ monomials which do not contain the leading terms $xx$, $yy$, $zz$, $xyz$ and $yzx$ of the members of the Gr\"obner basis of degree up to $3$. Since $A\in\Omega$, we must have $\dim A_4=15$. This happens precisely when the dimension of the space $L$ spanned by $r_1$, $r_2$ and $r_3$ is exactly $17-15=2$. That is, the rank of the matrix $S$ of the coefficients of $r_j$ must be $2$:
$$
{\bf rk}\,S=2,\ \ \text{where}\ \ S=\left(\begin{array}{cccccc}
-1&d&0&1&0&0\\ 0&0&0&0&1&-d^2\\ 0&0&a&d&a&d
\end{array}\right).
$$
If $d\neq 0$, then ${\rm rk}\,S=3$. Hence $d=0$. Now if $a\neq 0$, then ${\rm rk}\,S=3$. Hence $a=0$.
Then the defining relations of $A$ are $xx-yz$, $yy-zx$ and $zz$. A direct Gr\"obner basis computation shows that $\dim A_6=31$, which is incompatible with $A\in\Omega$. Hence the case does not occur.

{\bf Case 5:} $p=c=0$.

This case is obtained from Case~4 by the sub swapping $x$ and $y$.
\end{proof}

\subsection{Case $R_0=\spann\{xy-yx-yy\}$}

\begin{lemma}\label{r12-1}Let $A=A^{a,b,p,q}$ be the quadratic algebra given by the generators $x,y,z$ and the relations $xy-yx-yy-azx-bzy$, $xz-pzx-qzy$ and $zz$ with $a,b,p,q\in\K$, while $B=B^{a,b,q}$ be the quadratic algebra given by the generators $x,y,z$ and the relations $xy-yx-yy-azx-bzy$, $zx-qyz$ and $zz$ with $a,b,q\in\K$. Then $A\in\Omega$ if and only if $A\in\Omega^-$ if and only if $q+np\neq 0$ for every $n\in\Z_+$. Similarly, $B\in\Omega$ if and only if $B\in\Omega^-$ if and only if $q\neq 0$. Furthermore, $A^{a,b,p,q}$ and $A^{a',b',p',q'}$ are isomorphic if and only if $(a',b',p',q')=(ta,tb,p,q)$ with $t\in\K^*$, $B^{a,b,q}$ and $B^{a',b',q'}$ are isomorphic if and only if $(a',b',q')=(ta,tb,q)$ with $t\in\K^*$ and none of algebras $A$ is isomorphic to an algebra $B$.
\end{lemma}

\begin{proof} Any isomorphism between two of algebras $A$, two of algebras $B$ or an algebra $A$ and an algebra $B$ must send $z$ to its scalar multiple (indeed $zz$ is the only square in the space of defining relations up to a scalar multiple). The same isomorphism, after $z$ is factored out, must preserve $xy-yx-yy$ up to a scalar multiple. Thus the substitution facilitating such an isomorphism must be of the form $x\to ux+\alpha y+\beta z$, $y\to uy+\gamma z$, $z\to vz$ with $u,v\in\K^*$ and $\alpha,\beta,\gamma\in\K$. Now the isomorphism statements of the lemma are easily verified.

It remains to deal with the Hilbert series statement. First, we deal with algebras $A$. We use the left-to-right degree lexicographical ordering assuming $x>y>z$. If $k\in\Z_+$ and $q+jp\neq 0$ for $0\leq j<k$, then we can easily compute the members of the reduced Gr\"obner basis of the ideal of relations of $A$ of degrees up to $k+3$. Indeed, the only degree $3$ overlap other than $zzz$ (this one resolves producing nothing) is $xzz$. It resolves producing the member $qzyz$ of the ideal of relations. If $k>0$, then $q\neq 0$ and $zyz$ comprises the degree $3$ part of the said basis. The only degree $4$ overlaps are $zzyz$, $zyzz$ and $xzyz$. The first two resolve to $0$, while the last produces $(q+p)zyyz$. If $k>1$, then $q+p\neq 0$ and $zyyz$ comprises the degree $4$ part of the Gr\"obner basis. Proceeding in the same manner, we find that the the members of the reduced Gr\"obner basis of the ideal of relations of $A$ of degrees up to $k+2$ are the defining relations together with $zyz$, $zyyz,\dots,zy^{k}z$. If $q+kp\neq 0$, the only degree $k+3$ member is $zy^{k+1}z$, while if $q+kp=0$, then there are no degree $k+3$ members.

Now observe that the only monomials, which do not contain any of $xy$, $xz$ and $zy^jz$ for $j\in\Z_+$ as submonomials are $y^mx^n$ and $y^jzy^mx^n$ with $j,m,n\in\Z_+$. The number of such words of degree $n$ is exactly $\frac{(n+1)(n+2)}{2}$. Thus if $q+kp\neq 0$ for all $k\in\Z_+$, then $H_A=(1-t)^{-3}$ and $A\in\Omega$. On the other hand, if $k$ is the smallest non-negative integer satisfying $q+kp=0$, then $\dim A_{k+3}=\frac{(k+4)(k+5)}{2}+1$ (the preceding dimensions match that of an algebra from $\Omega$) and therefore $A\notin\Omega$. Since $zz$ is a relation of $A$, Lemma~\ref{LL2} yields that $A\in\Omega^-$ provided $A\in\Omega$.

Now we deal with algebras $B$. One way is to repeat the above process using the right-to-left degree lexicographical ordering instead. However, there is a quicker way. If $q=0$, then $B\notin\Omega$ according to Lemma~\ref{2-2}. Thus we are left with showing that $B\in\Omega$ provided $q\neq 0$ (non-Koszulity and membership on $\Omega^-$ are dealt with in exactly the same way as with algebras $A$). Clearly, $B$ with the opposite multiplication $B^{\rm opp}$ is isomorphic to the quadratic algebra $C$ given by the generators $x,y,z$ and the relations $yx-xy-yy-axz-byz$, $xz-qzy$ and $zz$ with $a,b\in\K$, $q\in\K^*$ (each monomial in the defining relations of $B$ is reversed). After the sub $x\to x$, $y\to -y$ and $z\to z$, the relations of $C$ take the form $xy-yx-yy-axz+byz$, $xz+qzy$ and $zz$. We follow up with the sub $x\to x+(b-a)z$, $y\to y+az$, $z\to z$, which turns the defining relations of $C$ into $xy-yx-yy-azx+(b-2a)zy$, $xz+qzy$ and $zz$. These are the defining relations of one of the algebras $A$, which we have already dealt with. By the criterion of the $\rm PBW_S$ condition for algebras $A$, verified in the first part of the proof, $C\in\Omega$. Since $B$ and $C$ share the Hilbert series, $B\in\Omega$.
\end{proof}

\begin{lemma}\label{case-R21-26} Let $A=A(V,R)\in\Omega$ be a quadratic algebra such that $\dim V=\dim R=3$ and with respect to some basis $x,y,z$ in $V$, $zz\in R$ and the quadratic algebra $B=A/I$ with $I$ being the ideal generated by $z$, is given by one relation $xy-yx-yy$. Then $A$ is isomorphic to to a $\K$-algebra given by generators $x,y,z$ and three quadratic relations from {\rm (R21--R26)} of Theorem~$\ref{main}$. Furthermore, the algebras in {\rm (R17--R19)} belong to $\Omega^-$ and are pairwise non-isomorphic.
\end{lemma}

\begin{proof} The fact that algebras in {\rm (R17--R19)} belong to $\Omega^-$ and are pairwise non-isomorphic follows directly from Lemma~\ref{r12-1}. Thus we can concentrate on the algebras $A$. The assumptions imply that $R$ is spanned by $zz$, $xy-yx-yy+f$ and $g$ with $f,g\in\spann\{xz,zx,yz,zy\}$, $g\neq 0$.

{\bf Case 1:} $xz$ features in $g$ with non-zero coefficient.

A sub $x\to x+uy$, $y\to y$, $z\to z$ with an appropriately chosen $u\in\K$ kills the $yz$ coefficient in $g$. After doing this, we perform a sub $x\to x+vz$, $y\to y+wz$ with appropriately chosen $v,w\in\K$ to kill the $xz$ and $yz$ coefficients in $f$. Now the defining relations of $A$ take the form $xy-yx-yy-azx-bzy$, $xz-pzx-qzy$ and $zz$, where $a,b,p,q\in\K$. If $a\neq 0$, Lemma~\ref{r12-1} implies that $A$ is isomorphic to an algebra given by relations of the same shape only with $a=1$. By Lemma~\ref{r12-1}, $A$ is now isomorphic to an algebra from (R21). If $a=0$ and $b\neq 0$, same argument shows that $A$ is isomorphic to an algebra from (R22). Similarly, if $a=b=0$, $A$ is isomorphic to an algebra from (R23).

{\bf Case 2:} $xz$ does not feature in $g$.

If $zx$ does not feature in $g$ as well, Lemma~\ref{2-2} yields that $A\notin\Omega$, providing a contradiction. Thus $zx$ does feature in $g$ with non-zero coefficient. A sub $x\to x+uy$, $y\to y$, $z\to z$ with an appropriately chosen $u\in\K$ kills the $zy$ coefficient in $g$. Now if $yz$ does not feature in $g$, Lemma~\ref{2-2} kicks in again furnishing a contradiction. Hence $yz$ features in $g$ with non-zero coefficient. After an appropriate scaling, the defining relations of $A$ take the form $xy-yx-yy-azx-bzy$, $zx-pyz$ and $zz$, where $p,a,b\in\K$, $p\neq 0$. If $a\neq 0$, Lemma~\ref{r12-1} ensures that $A$ is isomorphic to an algebra from (R24). Same lemma implies that $A$ is isomorphic to an algebra from (R25) if $a=0$ and $b\neq 0$ and to an algebra from (R29) if $a=b=0$. The proof is now complete.
\end{proof}

\subsection{Case $R_0=\spann\{xy-\alpha yx\}$ with $\alpha\neq 0$, $\alpha\neq 1$}

\begin{lemma}\label{r14-15-ome}Let $A$ be the quadratic algebra given by the generators $x,y,z$ and the relations $xy-\alpha yx-azx-bzy$, $sxz+pyz+qzx+rzy$ and $zz$ with $a,b,s,p,q,r\in\K$ and $\alpha\in\K^*$. Then $A\in\Omega$ if and only if $A\in\Omega^-$ if and only if $sr-\alpha^npq\neq 0$ for all $n\in\Z_+$.
\end{lemma}

\begin{proof} The proof is exactly the same as the proof in Lemma~\ref{r12-1} of the condition for membership of $A$ (from the said lemma) in $\Omega$. The only difference is that the coefficient popping up in front of $zy^{k+1}z$ when the overlap $xzy^kz$ is considered is now $sr-\alpha^kpq$ (instead of $q+kp$ of Lemma~\ref{r12-1}).
\end{proof}

\begin{lemma}\label{r14-iso}Let $A=A^{\eta}$ be the quadratic algebra given by the generators $x,y,z$ and the relations $xy-\alpha yx-azx-bzy$, $sxz+pyz+qzx+rzy$ and $zz$ with $\eta=(\alpha,a,b,s,p,q,r)\in\K^7$, $\alpha\notin\{0,1\}$. Then $A^{\eta}$ and $A^{\eta'}$ are isomorphic if and only if $\eta$ and $\eta'$ belong to the same orbit of the group action generated by the involution $(\alpha,a,b,s,p,q,r)\mapsto \bigl(\frac1\alpha,-\frac{b}{\alpha},-\frac{a}{\alpha},p,s,r,q\bigr)$ and the maps $(\alpha,a,b,s,p,q,r)\mapsto(\alpha,t_1a,t_2b,t_1t_3s,t_2t_3p,t_1t_3q,t_2t_3r)$ with $t_1,t_2,t_3\in\K^*$.
\end{lemma}

\begin{proof} Any isomorphism between two algebras $A$ must send $z$ to its scalar multiple ($zz$ is the only square in the space of defining relations up to a scalar multiple). The same isomorphism, after $z$ is factored our, must send $xy-\alpha yx$ to a scalar multiple of an expression of the same shape (with different $\alpha$ perhaps).  Thus the substitution facilitating such an isomorphism must be of the form $x\to ux+\alpha z$, $y\to vy+\beta z$, $z\to wz$ with $u,v,w\in\K^*$ and $\alpha,\beta\in\K$ or the same composed with swapping of $x$ and $y$. The swapping yields the transformation of parameters $(\alpha,a,b,s,p,q,r)\mapsto \bigl(\frac1\alpha,-\frac{b}{\alpha},-\frac{a}{\alpha},p,s,r,q\bigr)$. As for the first collection of subs, exactly as in Lemma~\ref{r16-iso}, one must have $\alpha=\beta=0$. This leaves us with scaling only, resulting in transformations of parameters of the form $(\alpha,a,b,s,p,q,r)\mapsto(\alpha,t_1a,t_2b,t_1t_3s,t_2t_3p,t_1t_3q,t_2t_3r)$ with $t_1,t_2,t_3\in\K^*$.
The result follows.
\end{proof}

\begin{lemma}\label{case-R27-33} Let $A=A(V,R)\in\Omega$ be a quadratic algebra such that $\dim V=\dim R=3$ and with respect to some basis $x,y,z$ in $V$, $zz\in R$ and the quadratic algebra $B=A/I$ with $I$ being the ideal generated by $z$, is given by one relation $xy-\alpha yx$ with $\alpha\in\K^*$, $\alpha\neq1$. Then $A$ is isomorphic to to a $\K$-algebra given by generators $x,y,z$ and three quadratic relations from {\rm (R27--R33)} of Theorem~$\ref{main}$. Furthermore, algebras with different labels are non-isomorphic, the isomorphism conditions in {\rm (R27--R33)} are satisfied and all algebras in {\rm (R27--R33)} belong to $\Omega^-$.
\end{lemma}

\begin{proof} The fact that algebras with different labels from (R27--R33) are non-isomorphic, the isomorphism conditions in {\rm (R27--R33)} are satisfied and all algebras in (R27-R33) belong to $\Omega^-$ follows from Lemmas~\ref{r14-15-ome} and~\ref{r14-iso}. Now we can focus on $A$. By the assumptions, $R$ is spanned by $zz$, $xy-\alpha yx+f$ and $g$ with $f,g\in\spann\{xz,zx,yz,zy\}$, $g\neq 0$.

A sub $x\to x+vz$, $y\to y+wz$, $z\to z$ with appropriately chosen $v,w\in\K$ to kill the $xz$ and $yz$ coefficients in $f$. Now the defining relations of $A$ take the form $xy-\alpha yx-azx-bzy$, $sxz+pyz+qzx+rzy$ and $zz$, where $\alpha,a,b,s,p,q,r\in\K$, $\alpha\neq 0$, $\alpha\neq 1$. Using Lemmas~\ref{r14-15-ome} and~\ref{r14-iso} and considering options for possible distribution of zeros among the parameters, one easily sees that $A$ must be isomorphic to one of the algebras from (R27--R33).
\end{proof}

\subsection{Case $R_0=\spann\{xy-yx\}$}

\begin{lemma}\label{r15-iso}Let $A=A^{a,b}$ be the quadratic algebra given by the generators $x,y,z$ and the relations $xy-yx-azx-bzy$, $xz-zy$ and $zz$ with $a,b\in\K$. Then $A^{a,b}$ and $A^{a',b'}$ are isomorphic if and only if $(a',b')=(ta,tb)$ for some $t\in\K^*$.
\end{lemma}

\begin{proof} Any isomorphism between two algebras $A$ must send $z$ to its scalar multiple ($zz$ is the only square in the space of defining relations up to a scalar multiple). It should also keep $xz-zy$ in the space of relations. These two properties are satisfied only for subs of the form $x\to ux+\alpha z$, $y\to uy+\beta z$, $z\to vy$ with $u,v\in\K^*$ and $\alpha,\beta\in\K$. After applying this sub to $A^{a,b}$, the space of defining relations becomes spanned by $xy-yx-tazx-tbzy$, $xz-zy$ and $zz$ with $t=\frac vu$. The result follows.
\end{proof}

\begin{lemma}\label{case-R34-36} Let $A=A(V,R)\in\Omega$ be a quadratic algebra such that $\dim V=\dim R=3$ and with respect to some basis $x,y,z$ in $V$, $zz\in R$ and the quadratic algebra $B=A/I$ with $I$ being the ideal generated by $z$, is given by one relation $xy-yx$. Then $A$ is isomorphic to to a $\K$-algebra given by generators $x,y,z$ and three quadratic relations from {\rm (R34--R36)} of Theorem~$\ref{main}$. Furthermore, algebras from {\rm (R34--R36)} belong to $\Omega^-$ and are pairwise non-isomorphic.
\end{lemma}

\begin{proof}
Lemmas~\ref{r14-15-ome} and~\ref{r15-iso} imply that algebras from {\rm (R34--R36)} belong to $\Omega^-$ and are pairwise non-isomorphic. Now we can focus on $A$. By the assumptions, $R$ is spanned by $zz$, $xy-yx+f$ and $g$ with $f,g\in\spann\{xz,zx,yz,zy\}$, $g\neq 0$. A sub $x\to x+vz$, $y\to y+wz$ with appropriately chosen $v,w\in\K$ to kill the $xz$ and $yz$ coefficients in $f$. Now the defining relations of $A$ take the form $xy-yx-azx-bzy$, $sxz+pyz+qzx+rzy$ and $zz$, where $a,b,s,p,q,r\in\K$.
By Lemma~\ref{r14-15-ome}, we must have $sr-pq\neq 0$. It easily follows that there is a substitution of the form $x\to u_1x+u_2y+u_3z$, $y\to v_1x+v_2y+v_3z$ and $z\to wz$, which preserves the overall shape of relations and turns the second one into $xz-zy$. Hence the defining relations of $A$ become $xy-yx-azx-bzy$, $xz-zy$ and $zz$ with $a,b\in\K$. Now by Lemmas~\ref{r14-15-ome} and~\ref{r15-iso}, $A$  is isomorphic to an algebra from (R34) if $a\neq 0$, $A$  is isomorphic to an algebra from (R35) if $a=0$ and $b\neq 0$ and $A$  is isomorphic to the algebra (R36) if $a=b=0$.
\end{proof}

\subsection{Case $R_0=\spann\{xy\}$}

\begin{lemma}\label{r16-iso}Let $A=A^{a,b,s,p,q,r}$ be the quadratic algebra given by the generators $x,y,z$ and the relations $xy-azx-byz$, $sxz+pyz+qzx+rzy$ and $zz$ with $a,b,s,p,q,r\in\K$. Then $A^{a,b,s,p,q,r}$ and $A^{a',b',s',p',q',r'}$ are isomorphic if and only if $(a',b',s',p',q',r')=(t_1a,t_2b,t_1t_3s,t_2t_3p,t_1t_3q,t_2t_3r)$ with $t_1,t_2,t_3\in\K*$. That is, $6$-tuples of parameters give rise to isomorphic algebras if they are in the same orbit of the action $(a,b,s,p,q,r)\mapsto (t_1a,t_2b,t_1t_3s,t_2t_3p,t_1t_3q,t_2t_3r)$ of $\K^*\times \K^*\times \K^*$.
\end{lemma}

\begin{proof} Any isomorphism between two algebras $A$ must send $z$ to its scalar multiple ($zz$ is the only square in the space of defining relations up to a scalar multiple). The same isomorphism, after $z$ is factored our, must preserve $xy$ up to a scalar multiple. Thus the substitution facilitating such an isomorphism must be of the form $x\to ux+\alpha z$, $y\to vy+\beta z$, $z\to wz$ with $u,v,w\in\K^*$ and $\alpha,\beta\in\K$. Taking into account the specifics of the first relation (absence of $xz$ and $zy$), one sees that our sub preserves the general shape of the space of relations only if $\alpha=\beta=0$. This leaves us with scaling only and the verification of the desired description of isomorphic algebras $A$ becomes trivial.
\end{proof}

\begin{lemma}\label{r16-ome}Let $A$ be the quadratic algebra given by the generators $x,y,z$ and the relations $xy-azx-byz$, $sxz+pyz+qzx+rzy$ and $zz$ with $a,b,s,p,q,r\in\K$. Then $A\in\Omega$ if and only if $A\in\Omega^-$ if and only if $sr\neq 0$ and $sr-pq\neq 0$.
\end{lemma}

\begin{proof} First, we eliminate the case $s=r=0$. In this case Lemma~\ref{2-2} implies that $A\notin\Omega$ if $pq=0$. Thus we can assume that $s=r=0$ and $pq\neq 0$. After an appropriate scaling of the variables, the space $R$ of defining relations of $A$ is spanned either by $xy$, $yz-zx$ and $zz$ or by $xy-zx$, $yz-zx$ and $zz$. In both cases an easy Gr\"obner basis calculation yields $\dim A_4=17$, which is incompatible with the membership in $\Omega$. Thus $A\notin \Omega$ if $s=r=0$.

The two cases $s\neq 0$ and $r\neq 0$ are reduced to one another by passing to the opposite multiplication followed up by swapping of $x$ and $y$: this procedure results in an algebra with relations of the same shape, parameters being transformed according to the rule $(a,b,s,p,q,r)\mapsto (b,a,r,q,p,s)$. Thus for the rest of the proof we can assume that $s\neq 0$. We use the left-to-right degree lexicographical ordering assuming $x>y>z$. The only degree $3$ overlaps of the leading monomials are $zzz$ and $xzz$. The first resolves, while the second results in $(sr-pq)zyz$. If $sr-pq=0$, we have $\dim A_3=11$ and therefore $A\notin\Omega$. From now on, we assume $sr-pq\neq 0$. In this case the degree $3$ part of the reduced Gr\"obner basis of the ideal of relations of $A$ consists of $zyz$. The only degree $4$ overlaps are $zzyz$, $zyzz$ and $xzyz$. The first two resolve, while the third produces $rzyyz$. If $r=0\iff sr=0$, $\dim A_4=16$ and therefore $A\notin\Omega$. It remains to show that $A\in\Omega$ provided $sr\neq 0$ and $sr-pq\neq 0$. In this case the degree $4$ part of the Gr\"obner basis consists of $zyyz$. Now we proceed inductively to show that for each $k\in\N$, the degree $k+2$ part of the reduced Gr\"obner basis of the ideal of relations of $A$ consists of $zy^kz$. We already have the basis of induction. Assume that $k\geq 3$ and the statement holds for all smaller $k$. Then the only degree $k+2$ overlaps are $zzy^{k-1}z$, $zy^{k-1}zz$ and $xzy^{k-1}z$. The first two resolve, while the last is easily seen to produce $zy^kz$ (we have to use $r\neq 0$ here). Thus the reduced Gr\"obner basis of the ideal of relations of $A$ consists of the defining relations together with $zy^kz$ for $k\in\N$. The corresponding normal words are $y^mx^n$ and $y^jzy^mx^n$ with $j,m,n\in\Z_+$. It easily follows that $H_A=(1-t)^{-3}$: the number of normal words of degree $n$ is $\frac{(n+1)(n+2)}{2}$. By Lemma~\ref{LL2}, $A\in\Omega^-$, whenever $A\in\Omega$, which completes the proof.
\end{proof}

\begin{lemma}\label{case-R37-39} Let $A=A(V,R)\in\Omega$ be a quadratic algebra such that $\dim V=\dim R=3$ and with respect to some basis $x,y,z$ in $V$, $zz\in R$ and the quadratic algebra $B=A/I$ with $I$ being the ideal generated by $z$, is given by one relation $xy$. Then $A$ is isomorphic to a $\K$-algebra given by generators $x,y,z$ and three quadratic relations from {\rm (R37--39)} of Theorem~$\ref{main}$. Furthermore, algebras from {\rm (R34--R36)} belong to $\Omega^-$ and are pairwise non-isomorphic.
\end{lemma}

\begin{proof} The fact that algebras with different labels from (R27--R33) are non-isomorphic, the isomorphism conditions in {\rm (R27--R33)} are satisfied and all algebras in (R27--R33) belong to $\Omega^-$ follows from Lemmas~\ref{r16-ome} and~\ref{r16-iso}. Now we can focus on $A$. The assumptions imply that $R$ is spanned by $zz$, $xy+f$ and $g$ with $f,g\in\spann\{xz,zx,yz,zy\}$, $g\neq 0$. A sub $x\to x+vz$, $y\to y+wz$ with appropriately chosen $v,w\in\K$ kills the $xz$ and $zy$ coefficients in $f$. Now the defining relations of $A$ take the form $xy-ayz-bzx$, $sxz+pyz+qzx+rzy$ and $zz$, where $a,b,s,p,q,r\in\K$. By Lemma~\ref{r16-ome}, $sr\neq 0$ and $sr-pq\neq 0$.

By scaling we can turn $s$ into $1$.

Thus we fall into (R15). The inclusion $A\in\Omega^-$, non-Koszulity of $A$ along with the isomorphism statements follow from Lemmas~\ref{r16-ome} and~\ref{r16-iso}. Usin the same lemmas and considering all options for possible distribution of zeros among the parameters, one easily sees that $A$ must be isomorphic to one of the algebras from (R37--R39).
\end{proof}

Note that now we have run out of algebras in the first part of Theorem~\ref{main}, but we still have unexplored options for $R_0$. This is due to the fact that the latter provide no algebras from $\Omega$.

\subsection{Case $R_0=\spann\{yy\}$}

\begin{lemma}\label{case12} Let $A=A(V,R)\in\Omega$ be a quadratic algebra such that $\dim V=\dim R=3$ and with respect to some basis $x,y,z$ in $V$, $zz\in R$ and the quadratic algebra $B=A/I$ with $I$ being the ideal generated by $z$, is given by one relation $yy$. Then $A\notin\Omega$.
\end{lemma}

\begin{proof} The assumptions imply that $R$ is spanned by $zz$, $yy+f$ and $g$ with $f,g\in\spann\{xz,zx,yz,zy\}$, $g\neq 0$.
If both $xz$ and $zx$ do not feature in $g$, the result follows from Lemma~\ref{2-2}. Passing to the opposite multiplication
does not change membership in $\Omega$ and leads to an algebra of the same shape with the $xz$ and $zx$ coefficients in $g$
(among other things) swapped. Thus for the rest of the proof we can without loss of generality assume that $xz$ features in $g$
with non-zero coefficient. A sub $x\to x+sy$, $y\to y$, $z\to z$ with an appropriate $s\in\K$ kills the $yz$ term in $g$.
Subtracting $g$ with an appropriate coefficient from the first relation, we can assume that $xz$ does not feature in $f$.
Now if either $zy$ does not feature in $g$ or $zx$ does not feature in $f$, the result follows from Lemma~\ref{2-2}. Thus we can assume
that both $zy$ in $g$ and $zx$ in $f$ have non-zero coefficients. By scaling, we can turn these coefficients into $-1$. Thus the
defining relations of $A$ acquire the form $yy-zx-azy$, $xz-bzx-zy$ and $zz$ with $a,b\in\K$. We use the ordering (\ref{order}) on the
monomials. The members of the reduced Gr\"obner basis of the ideal of relations of $A$ of degrees three and four are easily seen to be
$zyz$, $yzx+ayzy-zxy$, $zxyz$ and $zxzx+azxzy$, which yields $\dim A_4=16$. Hence $A\notin\Omega$.
\end{proof}

\subsection{Case $R_0=\spann\{xx-yx,yy\}$}

\begin{lemma}\label{case4} Let $A=A(V,R)\in\Omega$ be a quadratic algebra such that $\dim V=\dim R=3$ and with respect to some basis $x,y,z$ in $V$, $zz\in R$ and the quadratic algebra $B=A/I$ with $I$ being the ideal generated by $z$, is given by the relations $xx-yx$ and $yy$. Then $A\notin\Omega$.
\end{lemma}

\begin{proof} By the assumptions, $R$ is spanned by $zz$, $xx-yx+f$ and $yy+g$ with $f,g\in\spann\{xz,zx,yz,zy\}$. Using a substitution $x\to x+sz$, $y\to y+tz$, $z\to z$ with appropriate $s,t\in\K$, we can kill the $xz$ and $zx$ coefficients of $f$. Then $A$ is given by the generators $x,y,z$ and the relations
\begin{equation}\label{dere4}
\text{$xx-yx-ayz-bzy$, $yy-pxz-qzx-cyz-dzy$ and $zz$},
\end{equation}
where $a,b,c,d,p,q\in\K$.

Then $A^{\rm opp}$, being $A$ with the opposite multiplication, is isomorphic to the algebra $C$ given by the generators $x,y,z$ and the relations $xx-xy-azy-byz$, $yy-pzx-qxz-czy-dyz$ and $zz$. After the substitution
$x\to x-y$, $y\to -y$ and $z\to z$, the defining relations of $C$ take the shape
$xx-yx+byz+ayz$, $yy-qxz-pzx+(q+d)yz+(p+c)zy$ and $zz$. That is,
\begin{equation}\label{dere4opp}
\begin{array}{l}
\text{$A^{\rm opp}$ is isomorphic to an algebra given by $(\ref{dere4})$ with the}\\ \text{parameters $(-b,-a,-q-d,-p-c,q,p)$ in place of $(a,b,c,d,p,q)$}
\end{array}
\end{equation}

Throughout the proof we again use the order (\ref{order}) on $x,y,z$ monomials. Resolving the overlaps $yyy$ and $xxx$, we see that the degree 3 part of the Gr\"obner basis of the ideal of relations of $A$ consists of two members
\begin{equation}\label{gbd3-1}
\begin{array}{l}
g_1=pyxz{-}pxzy{+}qyzx{+}(d{-}c)yzy{-}qzxy{+}(qc{-}pd)zxz,
\\
g_2=xyx{+}axyz{-}pxzx{+}bxzy{-}(a+c)yzx{-}byzy{-}(b{+}d{+}q)zyx{-}aqzxz{-}a(q{+}d)zyz.
\end{array}
\end{equation}

If $p=q=0$, the result follows from Lemma~\ref{2-2}. Thus we can assume that either $p$ or $q$ is non-zero. According to (\ref{dere4opp}), the cases $p\neq 0$ and $q\neq 0$ are reduced to each other by passing to the opposite multiplication. Thus we can assume that $p\neq 0$. Via scaling, we can turn $p$ into $1$.
The leading monomials of $g_1$ and $g_2$ are now $yxz$ and $xyx$. One easily sees that $\dim A_3=10$ regardless what the values of other parameters are. The degree $4$ part of the Gr\"obner basis of the ideal of relations is spanned by $3$ elements $r_1$, $r_2$ and $r_3$ arising from the overlaps $yxzz$, $xyxz$ and $xyxx$ (other overlaps resolve). The explicit formulae for $r_j$ (we assume $p=1$) are as follows:

\begin{equation*}
\begin{array}{l}
\scriptstyle r_1=xzyz{-}qyzxz{+}(c{-}d)yzyz{+}qzxyz;
\\
\scriptstyle r_2=qxyzx{-}(c{-}d)xyzy{-}qxzxy{+}qyzxy{+}(qc{-}d)xzxz{-}(b{+}c)xzyz{+}(a{+}d)yzxz
\\
\quad\scriptstyle
{+}(b{-}c(c{-}d))yzyz{-}qczxyz{-}q^2zxzx{+}(b{+}q(c{-}d))zxzy{-}q(b{+}d)zyzx{+}((b{+}d)(c{-}d){-}b)zyzy;
\\
\scriptstyle
r_3={-}(a{+}c)xyzx{-}bxyzy{-}(q{+}b{+}d)xzyx{+}(q{+}a{+}b{+}d)yzyx{-}qaxzxz{+}a(1{-}d{-}q)xzyz{+}a(q{+}a{+}c)yzyz
\\
\quad\scriptstyle
{+}qazxyz{+}(b{+}q(a{+}c))zxzx{+}qbzxzy{+}(d(a{+}c){+}b(c{-}1))zyzx{+}b(b{+}d)zyzy.
\end{array}
\end{equation*}
Note that there are exactly $17$ degree $4$ monomials which do not contain the leading terms $xx$, $yy$, $zz$, $yxz$ and $xyx$ of the members of the Gr\"obner basis of degree up to $3$.

Assume the contrary. That is, $A\in\Omega$. Then we must have $\dim A_4=15$. This happens precisely when the dimension of the space $L$ spanned by $r_1$, $r_2$ and $r_3$ is exactly $17-15=2$.

{\bf Case 1:} \ $q\neq 0$.

In this case $yzxy$ features in $r_2$ with non-zero coefficient and does not feature in each of $r_1$ or $r_3$. Since $\dim L=2$, we now have that the span of $r_1$ and $r_3$ is one-dimensional. Now $yzxz$ features in $r_3$ with non-zero coefficient and does not feature in $r_1$. Hence we must have $r_1=0$. This only happens if $a=b=c=0$ and $d=-q$. Then the defining relations of $A$ acquire the form $xx-yx$, $yy-xz+dzx-dzy$ and $zz$ with $d\in\K$. Using Gr\"obner basis technique, one easily sees that in this case for generic $d$, $\dim A_5=22$. By Lemma~\ref{minhs}, $\dim A_5\geq 22$ for all $q$, contradicting $A\in\Omega$.

{\bf Case 2:} \ $q=0$.

In this case it is easy to observe that $r_1$, $r_2$ and $r_3$ are linearly independent if $b\neq 0$ or if $b=0$ and $d\neq 0$ or if $b=d=0$ and $a\neq 0$ or if $b=d=a=0$ and $c\neq 0$. This only leases us the case $a=b=c=d=q=0$ in which case $r_1=r_2=0$ and $L$ is one-dimensional. This contradiction completes the proof.
\end{proof}

\subsection{Case $R_0=\spann\{xx-\alpha xy-yy,yx\}$ with $\alpha\in\K$, $\alpha^2+1\neq 0$}

This is the most technically annoying case. To add an insult to injury, it turns out that there are no algebras from $\Omega$ in it. The plan of action is the following. We represent algebras in this case as a multi-parametric family of quadratic algebras and observe that we always have $\dim A_3=10$. Using Gr\"obner basis technique, we identify the algebras of the family satisfying $\dim A_4=15$. These form several one-parametric families of quadratic algebras. However, members of each of them satisfy $\dim A_5>21$, which is incompatible with $A\in\Omega$. We shall use the order (\ref{order}) on the monomials. Anyone curious can try to do the same using the degree-lexicographical ordering to see how much messier things get.

\begin{lemma}\label{case1} Let $A=A(V,R)$ be a quadratic algebra such that $\dim V=\dim R=3$ and with respect to some basis $x,y,z$ in $V$, $zz\in R$ and the quadratic algebra $B=A/I$ with $I$ being the ideal generated by $z$, is given by the relations $xy-\alpha yx-yy$ and $yy$ with $\alpha\in\K$ and $\alpha^2+1\neq 0$. Then $A\notin\Omega$.
\end{lemma}

\begin{proof} Assume the contrary: $A\in\Omega$. The assumptions imply that $R$ is spanned by $zz$, $xy-\alpha yx-yy+f$ and $yy+g$ with $f,g\in\spann\{xz,zx,yz,zy\}$. Using a linear substitution, which leaves $z$ intact and replaces $x$ and $y$ by $x+sz$ and $y+tz$ respectively with appropriate $s,t\in\K$, we can kill the $xz$ and $yz$ coefficients of $f$. After this substitution, $A$ is given by the generators $x,y,z$ and the relations
\begin{equation}\label{dere1}
\text{$xx-\alpha xy-yy-azx-bzy$, $yx-pxz-qyz-czx-dzy$ and $zz$},
\end{equation}
where $\alpha,a,b,c,d,p,q\in\K$ and $\alpha^2+1\neq 0$.

Throughout the proof we will use the order (\ref{order}) on $x,y,z$ monomials. Resolving the overlaps $yxx$ and $xxx$, we see that the degree 3 part of the Gr\"obner basis of the ideal of relations of $A$ consists of two members
\begin{equation*}
\begin{array}{l}
g_1=yyy{-}pxzx{{+}}\alpha pxzy{-}(q{-}a)yzx{+}(\alpha q{+}b)yzy{+}(\alpha d{-}c)zyy{-}pdzxz{-}qdzyz,
\\
g_2=(\alpha^2{+}1)xyy{-}\alpha(q{+}\alpha p)xyz{-}(q{+}\alpha p)yyz{+}(\alpha p{+}a{-}\alpha c)xzx{+}(b{-}\alpha^2p{-}\alpha d)xzy{+}(\alpha q{-}\alpha a{-}c)yzx
\\
\qquad{-}(\alpha^2q{+}\alpha b{+}d)yzy{-}(\alpha^2 d{-}\alpha b{-}\alpha c{+}a)zyy{-}p(\alpha a{-}\alpha d{+}b{+}c)zxz{-}(\alpha pb{-}\alpha qd{+}pd{+}qb)zyz.
\end{array}
\end{equation*}
One easily sees that $\dim A_3=10$ regardless what the values of the parameters are. There are $6$ degree $4$ overlaps $xyyx$, $xyyy$, $yyyy$, $yyyx$, $xxyy$, $yxyy$. The members of the ideal of relations obtained from the last two overlaps always belong to the linear span of the ones obtained from the first four overlaps. We denote these $r_1,r_2,r_3$ and $r_4$ respectively. The explicit formulae for $r_j$ are as follows:
\begin{equation*}
\begin{array}{l}
\scriptstyle r_1=(\alpha(q{+}\alpha p){-}(\alpha^2{+}1)c)xyzx{-}(\alpha^2{+}1)dxyzy{+}(q{+}\alpha p)yyzx{+}\alpha(\alpha c{-}\alpha p{-}a)xzxy{+}(\alpha c{-}\alpha p{-}a)xzyy{+}\alpha(\alpha a{-}\alpha q{+}c)yzxy
\\
\quad\scriptstyle {+}(\alpha a{-}\alpha q{+}c)yzyy{+}(\alpha^2p^2{+}\alpha pd{-}pb{-}(\alpha^2{+}1)pc{-}\alpha qc{+}\alpha pq{+}qa)xzxz{-}(\alpha^2{+}1)pdxzyz{+}(\alpha^2pq{+}\alpha pb{+}pd{-}\alpha qa{+}\alpha q^2{-}qc)yzxz
\\
\quad\scriptstyle {+}p(b{+}\alpha a{-}\alpha d{+}c)zxzx{+}(pd{+}\alpha pb{+}qb{-}\alpha qd{+}ac{-}\alpha bc{-}\alpha c^2{+}\alpha^2cd)zyzx{+}(ad{-}\alpha bd{-}\alpha cd{+}\alpha^2d^2)zyzy,
\\
\scriptstyle r_2=(\alpha^2{+}1)(a{-}q{-}\alpha p)xyzx{+}(\alpha(\alpha^2{+}2)(q{+}\alpha p){+}(\alpha^2{+}1)b)xyzy{-}(\alpha^2{+}1)pyyzx{+}(\alpha(\alpha^2{+}2)p{+}q)yyzy{+}(\alpha c{-}\alpha p{-}a)xzxy
\\
\quad\scriptstyle
{+}(\alpha^2p{+}\alpha d{-}b{-}(\alpha^2{+}1)(c{-}\alpha d))xzyy{+}(\alpha a{-}\alpha q{+}c)yzxy{+}(\alpha^2q{+}\alpha b{+}d)yzyy{-}(\alpha^2{+}1)pdxzxz{-}(\alpha^2{+}1)qdxzyz
\\
\quad\scriptstyle
{-}\alpha p(\alpha{-}\alpha d{+}b{+}c)zxzx{+}(\alpha^2{+}1)p(\alpha{-}\alpha d{+}b{+}c){+}((q{-}a)(\alpha^2d{-}\alpha b{-}\alpha c{+}a){-}(\alpha^2{+}1)pb)zyzx
\\
\quad\scriptstyle
{+}(\alpha(\alpha^2{+}2)pb{-}\alpha qd{+}pd{+}qb{-}(\alpha q{+}b)(\alpha^2d{-}\alpha b{-}\alpha c{+}a))zyzy,
\\
\scriptstyle r_3=(q{-}a)yyzx{-}(\alpha q{+}b)yyzy{-}pxzxy{+}\alpha pxzyy{-}(q{-}a)yzxy{+}(\alpha q{-}\alpha d{+}b{+}c)yzyy{+}pdyzxz{+}qdyzyz{+}\alpha pd zxzx
\\
\quad\scriptstyle
{-}(1{+}\alpha^2)pdzxzy{+}(pd{+}(q{-}a)(\alpha d{-}c))zyzx{-}(\alpha pd{+}(\alpha^2{+}1)qd{+}\alpha bd{-}\alpha qc{-}bc)zyzy,
\\
\scriptstyle r_4={-}cyyzx{-}dyyzy{+}\alpha p xzxy{+}pxzyy{+}\alpha(q{-}a)yzxy{+}(q{-}a)yzyy{-}p(q{+}\alpha p)xzxz{+}(aq{-}q^2{-}pc{-}\alpha pq{-}pb)yzxz
\\
\quad\scriptstyle
{-}(pd{+}2q(b{+}\alpha q))yzyz{+}pd zxzx{+}(qd{+}c^2{-}\alpha cd)zyzx{+}d(c{-}\alpha d)zyzy.
\end{array}
\end{equation*}
Note that there are exactly $18$ degree $4$ monomials which do not contain the leading terms $xx$, $yx$, $zz$, $xyy$ and $yyy$ of the members of the Gr\"obner basis of degree up to $3$. Since we have assumed that $A\in\Omega$, we must have $\dim A_4=15$. This happens precisely when the dimension of the space $L$ spanned by $r_1,r_2,r_3,r_4$ is exactly $18-15=3$.

Unfortunately, there are values of parameters for which precisely this happens. Obviously,
$$
S=\{(\alpha,a,b,c,d,p,q)\in\K^7:\dim L \leq 3\}
$$
is an affine variety, while
$$
S_0=\{(\alpha,a,b,c,d,p,q)\in\K^7:\dim L\leq 2\}
$$
is a subvariety of $S$. By the above remark $\dim A_4=15$, which only happens if $(\alpha,a,b,c,d,p,q)\in S\setminus S_0$ and $\alpha^2+1\neq 0$. Note also that if we scale $z$, then the relations retain their form, $\alpha$ stays put, while the vector $(a,b,c,d,p,q)$ is scaled by the same constant as $z$. This allows us to treat $S$ and $S_0$ as subvarieties of $\K\times \K P^5$, reducing the number of free parameters by $1$. We shall see that $S$, as a subvariety of $\K\times \K P^5$ has dimension one and splits into the union of several irreducible one-dimensional varieties plus a finite set. In order to deal with the case $(\alpha,a,b,c,d,p,q)\in S\setminus S_0$, we have to go one step further in computing the Gr\"obner basis to see that $\dim A_5>21$, which is incompatible with $A\in \Omega$.

Now we shall sketch the procedure of pinpointing the variety $S$.

{\bf Case 1:} $d\neq 0$. By scaling $z$, we can without loss of generality assume that $d=1$.

{\bf Case 1a:} $p\neq 0$ (in addition to $d=1$).

Set $\kappa=c+b+\alpha a-\alpha$. The matrix of $zxzx$, $zxzy$ and $xzyz$ coefficients of $r_1-\kappa r_4$, $r_2-\kappa r_3$, $r_3$, $r_4$ now is
$$
\left(\begin{array}{ccc}0&0&-(\alpha^2+1)p\\ 0&0&-(\alpha^2+1)q\\0&-(\alpha^2+1)p&0\\ p&0&0\end{array}\right).
$$
Since $p\neq 0$ and $\alpha^2+1\neq 0$, it follows that the only way for $r_j$ to be linearly dependent is for the equation
\begin{equation}\label{depepq}
\rho=-qr_1+pr_2+p\kappa r_3+q\kappa r_4=0
\end{equation}
to be satisfied. This happens precisely when the coefficients of $\rho$ front of all 16 monomials featuring in $r_j$ vanish, giving a system of algebraic equations on the parameters involved. Next, we observe that $\alpha\kappa q\neq 0$. Indeed, if $\kappa=0$, then vanishing of the $yyzy$  coefficient of $\rho$ yields $q=-\alpha(\alpha^2+2)p$. Plugging this into the equation provided by vanishing of the $yyzx$ coefficient of $\rho$, we get $p^2(\alpha^2+1)^2=0$, which is impossible. Thus $\kappa\neq0$. If $q=0$, the  equation $\rho=0$ reads $r_2+\kappa r_3=0$. The $yzxz$ coefficient of $r_2+\kappa r_3$ is $\kappa p$. Thus $\kappa p=0$, which is impossible. Finally, if $\alpha=0$, it is an easy exercise to see that $r_j$ are linearly independent (the matrix of coefficients of $r_j$ simplifies dramatically if $\alpha=0$). Thus we can assume that $\alpha\kappa q\neq 0$. Equating the $yzyz$ coefficients of $\rho$ and $0$, we get $2\kappa q^2(\alpha q+b)=0$. Since $q$ and $\kappa$ are non-zero, we get $b=-\alpha q$.

First, we consider the case $p=\alpha q$. In this case, the equation $\rho=0$ resolves rather smoothly. looking at $xyzy$ coefficient, we get $(\alpha^2+1)q+\alpha pq+\alpha^2(\alpha^2+2)p^2=0$. Together with $p=\alpha q$, this yields $q=-\frac1{\alpha^2(\alpha^2+1)}$, $p=-\frac1{\alpha(\alpha^2+1)}$ (we use the fact that $\alpha\neq 0$ here). Solving the equations arising from $xyzx$ and $xzyy$ of $\rho$ and using the above expressions for $p$ and $q$, we find $a=-1-\frac1{\alpha^2}$ and $c=\alpha-\frac1{\alpha(\alpha^2+1)}$. Finally, $b=-\alpha q=\frac1{\alpha(\alpha^2+1)}$.
Summarising, we get
\begin{equation}\label{IC1}
\textstyle a=-1-\frac1{\alpha^2},\ \ b=\frac1{\alpha(\alpha^2+1)},\ \ c=\alpha-\frac1{\alpha(\alpha^2+1)},\ \ d=1,\ \
p=-\frac1{\alpha(\alpha^2+1)}\ \ \text{and}\ \ q=-\frac1{\alpha^2(\alpha^2+1)}.
\end{equation}
For these values of parameters $r_j$ span a $3$-dimensional space.

Now we consider the case $p\neq \alpha q$. The equations arising from $xzxy$ and $yzxy$ coefficients of $\rho$ read
$(p-\alpha q)(\alpha c-a-p(\kappa+\alpha))=0$ and $(p-\alpha q)(c-(q-a)(\kappa+\alpha))=0$. Since $p\neq \alpha q$, we get
$$
\alpha c-a-p(\kappa+\alpha)=c-(q-a)(\kappa+\alpha)=0.
$$
Now the $yyzy$ coefficient yields $p(\alpha(\alpha^2+2)p+q)-q\kappa=0$, from which we have $\kappa=p(1+\alpha(\alpha^2+2)\frac{p}{q})$. The equation arising from $xyzy$ reads $q(\alpha^2+1)+\alpha pq+\alpha^2(\alpha^2+2)p^2=0$. This implies $q=-\frac{\alpha^2(\alpha^2+2)p^2}{\alpha p+\alpha^2+1}$. Plugging this into the above expression for $\kappa$ and simplifying, we get $\kappa=-\frac{\alpha^2+1}{\alpha}$ (delightfully, $p$ cancels out). Hence $\kappa+\alpha=-\frac1{\alpha}$. Plugging this into the above display, we get $a-\alpha c=\frac{p}\alpha=q$. Hence $p=\alpha q$ and we have arrived to a contradiction. This concludes Case~1a.

{\bf Case 1b:} $p=0$ and $q\neq 0$ (in addition to $d=1$).

Considering the $xyzy$ and $xzyz$ coefficients of $r_j$, we see that in this case $r_1$ and $r_2$ are linearly independent modulo the span of $r_3$ and $r_4$. Since $r_j$ are linearly dependent, it follows that $r_3$ and $r_4$ must be linearly dependent. If, additionally, $a\neq q$, then the $2\times 2$ matrix of $yzxz$ and $yzyz$ coefficients of $r_3$ and $r_4$ is non-degenerate providing a contradiction. Thus $a=q$. Now the $2\times 2$ matrix of $yyzx$ and $yzyz$ coefficients of $r_3$ and $r_4$ is non-degenerate unless $c=0$. Thus $c=0$. Finally, now the $2\times 2$ matrix of $yzyz$ and $zyzx$ coefficients of $r_3$ and $r_4$ is non-degenerate unless $q=0$. Thus $q=0$, which is a contradiction.

{\bf Case 1c:} $p=q=0$ (in addition to $d=1$).

Now the matrix of coefficients of $r_j$ becomes rather simple and it is an elementary linear algebra exercise to see that $r_j$ span a $3$-dimensional space only when either $a=c=b-\alpha=0$ or $b=a=c-\alpha=0$ or $b=c=\alpha$ and $a=\alpha^2$ (recall that we have $p=q=0$ and $d=1$). This provides $3$ curves sitting in $S$:
\begin{align}\label{IC2}
&a=0,\ \ b=\alpha,\ \ c=0,\ \ d=1,\ \ p=0,\ \ q=0;
\\
&a=0,\ \ b=0,\ \ c=\alpha,\ \ d=1,\ \ p=0,\ \ q=0; \label{IC3}
\\
&a=\alpha^2,\ \ b=\alpha,\ \ c=\alpha,\ \ d=1,\ \ p=0,\ \ q=0. \label{IC4}
\end{align}
This concludes Case~1, in which we have identified $4$ one-dimensional irreducible components of $S$.

{\bf Case 2:} $d=0$ and $p\neq 0$. By scaling $z$, we can without loss of generality assume that $p=1$.

If $\alpha a+b+c\neq 0$, then by looking at $zxzx$ and $zxzy$ coefficients, we see that $r_1$ and $r_2$ are linearly independent modulo span of $r_3$ and $r_4$. By looking at $xzxy$ and $xzyy$ coefficients, we see that $r_3$ and $r_4$ are linearly independent. Hence $r_j$ span a $4$-dimensional space. Thus we must have $\alpha a+b+c=0$. That is, $c=-\alpha a-b$.

{\bf Case 2a:} $q\neq 0$ (on top of $d=0$, $p=1$ and $c=-\alpha a-b$).

First, consider the case $b+\alpha q\neq 0$. Looking at the $4\times 3$ matrix of the $yzxy$, $yzyy$ and $yzyz$ coefficients of $r_j$, we see that in this case $r_1$ and $r_4$ are linearly independent modulo the span of $r_2$ and $r_3$. Thus $r_2$ and $r_3$ must be linearly dependent. Considering the determinants of four of $2\times 2$ submatrices of the coefficients of $r_2$ and $r_3$, we get
$$
\begin{array}{l}
(\alpha^2+1)(q-a)-\alpha(\alpha^2+1)=\alpha(\alpha^2+2)q+(\alpha^2+1)b+\alpha^2(\alpha^2+2)
\\ \qquad =(q-a)(b+\alpha q+\alpha^2+1)=((\alpha^2+1)a+\alpha b+\alpha)(q-a)-(b+\alpha q)=0.
\end{array}
$$
One of the equations yields that either $q=a$ or $b+\alpha q=-1-\alpha^2$. If $q\neq a$, then $b+\alpha q=-1-\alpha^2$. Plugging this into the rest of the equations, we easily see that the system is incompatible. Thus $q=a$. Then the last equation yields $b+\alpha q=0$, which contradicts the assumption.

Thus $b+\alpha q=0$. That is, $b=-\alpha q$. First, assume that $q\neq a$. After plugging this in, the shape of $yzxy$ and $yzyy$ coefficients of $r_j$ tells us that $r_3$ and $r_4$ are linearly independent modulo the span of $r_1$ and $r_2$. If $q+\alpha(\alpha^2+2)\neq 0$, $yyzy$ features with non-zero coefficient only in $r_2$. Thus we must have $r_1=0$, which is not the case under the assumption $q\neq 0$. Thus the only option is $q=-\alpha(\alpha^2+2)$. Looking at $yzxy$ and $yzyy$ coefficients, we see that the only way for $r_1$ and $r_2$ to be linear dependent is to have $(\alpha^2+1)a-\alpha^2q+\alpha=0$. This together with $q=-\alpha(\alpha^2+2)$ yields $a=-\alpha(\alpha^2+1)$. Plugging in the rest of the data we have
\begin{equation}\label{IC5}
\textstyle
a=-\alpha(\alpha^2+1),\ \ b=\alpha^2(\alpha^2+2),\ \ c=-\alpha^2,\ \ d=0, p=1, q=-\alpha(\alpha^2+2).
\end{equation}
For these values of parameters $r_j$ span a $3$-dimensional space. This concludes the case $q\neq a$. Now we assume $a=q$. In this case after eliminating zero columns as well as ones being obviously linear combinations of the ones present, the matrix of coefficients of $r_j$ reduces to the form we are finally not embarrassed to present in full:
$$
\left(\begin{array}{cccc}q+\alpha&0&0&q+\alpha\\ -(\alpha^2+1)&q+\alpha(\alpha^2+2)&q+\alpha&0\\ 0&0&1&0\\ 0&0&0&-1\end{array}\right).
$$
The only way for it to have rank $3$ is for the equality $(q+\alpha)(q+\alpha(\alpha^2+2))=0$ to be satisfied, which
yields two more one-parametric families of algebras:
\begin{align}\label{IC6}
&a=-\alpha,\ \ b=\alpha^2,\ \ c=0,\ \ d=0,\ \ p=1,\ \ q=-\alpha;
\\
&a=-\alpha(\alpha^2+2),\ \ b=\alpha^2(\alpha^2+2),\ \ c=0,\ \ d=0,\ \ p=1,\ \ q=-\alpha(\alpha^2+2). \label{IC7}
\end{align}
This concludes Case 2a.

{\bf Case 2b:} $q=0$ (on top of $d=0$, $p=1$ and $c=-\alpha a-b$).

It is easy to see that for three specific cases
\begin{equation}\label{IC8}
\begin{array}{l}
\alpha=0,\ \ a=0,\ \ b=0,\ \ c=0,\ \ d=0,\ \ p=1,\ \ q=0,
\\
\alpha^2+2=0,\ \ a=0,\ \ b=0,\ \ c=0,\ \ d=0,\ \ p=1,\ \ q=0,
\\
\alpha^2+2=0,\ \ a=\alpha,\ \ b=0,\ \ c=2,\ \ d=0,\ \ p=1,\ \ q=0,
\end{array}
\end{equation}
$r_j$ span a $3$-dimensional space, thus we have specified $5$ points in $S$ (zero dimensional irreducible components, actually, when $S$ is interpreted as a subvariety of $\K\times \K P^5$).

On the other hand, if our parameters are not the ones provided in (\ref{IC8}), one easily sees that $r_2$ and $r_3$ are linearly independent modulo linear span of $r_1$ and $r_4$. Thus for $r_j$ to span a $3$-dimensional space, $r_1$ and $r_4$ must be linearly dependent. The latter is easily seen to never happen. This concludes Case~2b and Case~2.

{\bf Case 3:} $p=d=0$ and $q\neq 0$. By scaling $z$, we can without loss of generality assume that $q=1$.

If $(a-\alpha c)(a-1)\neq 0$, $r_j$ are linearly independent. Indeed the $4\times 4$ matrix of the $yyzx$, $yyzy$, $xzxz$ and $yzxz$ coefficients of $r_j$ is invertible (with the proper ordering of rows and columns, it is triangular with non-zero diagonal entries). Hence we must have $(a-\alpha c)(a-1)=0$. First, assume that $a=\alpha c$. In this case, treating the cases $b+\alpha=0$ and $b+\alpha\neq 0$ separately, it is straightforward to verify that $r_j$ are still always linearly independent. Thus we have $a-\alpha c\neq 0$. Since $(a-\alpha c)(a-1)=0$, we must have $a=1$. In this case, again, $r_j$ are linearly independent, unless $c=0$ and $b=-\alpha$. In the latter case $r_j$ span a $3$-dimensional space,
providing yet another piece of $S$:
\begin{equation}\label{IC9}
\textstyle
a=1,\ \ b=-\alpha,\ \ c=0,\ \ d=0,\ \  p=0,\ \ q=1.
\end{equation}
This concludes Case 3.

{\bf Case 4:} $p=d=q=0$.

In this case, the matrix of coefficients of $r_j$ becomes so simple that we just give the answer. If $a=b=c=0$, all $r_j$ vanish (and have no chance to span a $3$-dimensional space). If $b\neq 0$ and $a=c=0$, $r_j$ span a $2$-dimensional space (yielding $\dim A_4=16$). In all other cases $r_j$ are linearly independent. Thus this final case provides no contribution into $S\setminus S_0$.

As a result, the equality $\dim A=15$ only happens if the parameters (after appropriate scaling) fall into one of the sets described in (\ref{IC1}--\ref{IC9}). Now a direct computation of the degree $5$ part of the Gr\"obner basis of the ideal of relations of $A$ yields $\dim A_5>21$ for $5$ algebras corresponding to (\ref{IC8}) and for generic $\alpha$ (with finitely many possible exceptions when the leading monomials differ from the ones in generic case) for each of the one-parametric families. As a matter of fact, $\dim A_5=22$ for the family (\ref{IC1}) and $\dim A_5=23$ for all other families. Anyway, by Lemma~\ref{minhs}, $\dim A_5\geq 22$ for every $\alpha$, which is incompatible with $A\in\Omega$.
\end{proof}

Now Part~I of Theorem~\ref{main} follows straight away from Lemmas~\ref{zzz}, \ref{case-R1-6}, \ref{case-R7-8}, \ref{case-R9-12}, \ref{case-R13-16}, \ref{case-R17-19}, \ref{case-R20}, \ref{case-R21-26}, \ref{case-R27-33}, \ref{case-R34-36}, \ref{case-R37-39}, \ref{case12}, \ref{case4} and \ref{case1}.

\section{Proof of Part XI of Theorem~\ref{main}}

Throughout this section we work with the left-to-right degree-lexicographical ordering assuming $x>y$. We start by ruling out an annoying case with no algebras from $\Lambda$ in it.

\begin{lemma}\label{xxx-gen} Let $A$ be the cubic algebra given by generators $x$ and $y$ and relations
\begin{equation*}
y^3, \ \ x^3-axxy+(1-a)xyx-(a+b)yxx-(c-a^2)xyy-dyxy-(1-bc+a-a^2)yyx
\end{equation*}
for some $a,b,c,d\in\K$. Then $A\notin\Lambda$.
\end{lemma}

\begin{proof} Assume the contrary. That is $A\in\Lambda$ for some $a,b,c,d\in\K$. Members of the Gr\"obner basis of the ideal of relations of $A$ of degrees up to $5$ consist of the two degree three defining relations, one degree four element
$$
\begin{array}{l}
\scriptstyle g=xxyx-cxxyy-(1+b)xyxx-(d-a(1-a))xyxy-(1+a-c(1+b))xyyx-((1-a)(a+b)-d)yxyx\\
\scriptstyle -(ad-(c-a^2)(a+b))yxyy+(a(1+b)+b(a+b)+(1-bc))yyxx-(a(1-bc)+a^2(1-a)-d(a+b))yyxy
\end{array}
$$
and one degree $5$ element
$$
\begin{array}{l}
\scriptstyle h=xxyyx+(1+b)(d+a(a+b))xyxxy-[-bd+b(1-a)(1+b+a)]xyxyx-[-b(1+b)(c-a^2)-(d-a(1-a))(d+a^2)]xyxyy\\
\scriptstyle -[-b-(b^2+b+1)(a+b-c)]xyyxx-[-d((b-c)(1+b)+2)-a^2(1-c(1+b))-ac]xyyxy-(1+b)[a(a+b)+d]yxyxx\\
\scriptstyle -[-((1-a)(a+b)-d)(d+a^2)+bd]yxyxy-[+(1+a)(a+b)-(1+b)a^2(a+b)+d-ad(1+b)]yxyxy\\
\scriptstyle -[+(a(1+b)+b(a+b))(d+a(a+b))+(1-bc)(d+a^2)]yxyxy-[-b(1-a)(a+b)(a+b)-ab(1-a)-bd(a+b)-b(1-bc)]yyxyx\\
\scriptstyle -[+(c-a^2)b[+b(a+b)+a(1+b)]-(1-bc)[-bc+ad+a^2(a+b)]-a^2(1-a)(d+a^2)+d(a+b)(d+a^2)]yyxyy.
\end{array}
$$
Note that there are exactly $16$ degree $6$ monomials in $x,y$, which do not contain the leading monomials $x^3$, $y^3$, $x^2yx$ and $x^2y^2x$ of the members of the Gr\"obner basis of degree up to $5$. Since $16$ is also the $t^6$-coefficient of the series $(1+t)^{-1}(1-t)^{-3}$ from the definition of $\Lambda$, the only way for $A$ to fall into $\Lambda$ is for all degree $6$ overlaps of the above leading monomials to resolve (produce no degree $6$ members of the Gr\"obner basis).

Now, this does not occur. The overlap $x^2y(x^3)=(x^2yx)x^2$ yields a degree $6$ homogeneous element $f$ of the free algebra with coefficients being polynomials in $a,b,c,d$. Since $A\in\Lambda$, we must have $f=0$. This gives a system of algebraic equations on $a,b,c,d$. This system has no solutions at all, which is easier to confirm by hand than one might think. For starters, the $xyx^2y^2$-coefficient in $f$ is $(1+b)(d+a(a+b))(d-c^2+bc+2ac)$. Thus we must have $b=-1$ or $d=-a(a+b)$ or $d=c^2-bc-2ac$. Plugging these one at a time into other coefficients of $f$ we not only reduce the number of parameters by one but cause massive cancellations each time. In each case we get another coefficient which is a product of low degree polynomials and repeat the procedure. Anyway, we arrive to a contradiction, which completes the proof.
\end{proof}

\begin{lemma}\label{xxx-1} Let $V$ be a $2$-dimensional vector space over $\K$ and $R$ be a $2$-dimensional subspace of $V^3$ such that $y^3\in R$ and $R$ is not contained in the ideal generated by $y$. Then the cubic algebra $A=B(V,R)$ belongs to $\Lambda$ if and only if there is $x\in V$ such that $x,y$ form a basis in $V$ and $R$ is spanned by $y^3$ and $x^3-xy^2-ayxy-y^2x$ for some $a\in \K^*$.
\end{lemma}

\begin{proof} We start by picking an arbitrary $x\in V$ such that $x,y$ form a basis in $V$. By assumptions, $R$ is spanned by $y^3$ and $f=x^3+b_1x^2y+b_2xyx+b_3yx^2+b_4xy^2+b_5yxy+b_6y^2x$ for some $b\in \K^6$. If $b_1\neq b_2$ it is a routine exercise to see that a substitution of the form $x\to sx+ry$, $y\to ty$ with $s,t\in\K^*$, $r\in\K$ transforms $f$ into $x^3-axxy+(1-a)xyx-(a+b)yxx-(c-a^2)xyy-dyxy-(1-bc+a-a^2)yyx$ for some $a,b,c,d\in\K$ up to a scalar multiple. By Lemma~\ref{xxx-gen}, $A\notin\Lambda$. Thus $A\notin\Lambda$ if $b_1\neq b_2$. The case $b_2\neq b_3$ transforms into the case $b_1\neq b_2$ when we pass to the opposite multiplication. Thus $A\notin\Lambda$ unless $b_1=b_2=b_3$. Next, a substitution of the same form as above kills $b_1$, $b_2$ and $b_3$. After this substitution, $R$ is spanned by $y^3$ and $f=x^3-sxy^2-ayxy-by^2x$ for some $s,a,b\in\K$.

{\bf Case 1:} $s\neq 0$.

In this case an additional scaling turns $s$ into $1$. Thus we can assume that $s=1$ and therefore $R$ is spanned by $y^3$ and $f=x^3-xy^2-ayxy-by^2x$ for some $a,b\in\K$.

{\bf Case 1a:} $a=0$ in addition to $s=1$.

The elements of the Gr\"obner basis of the ideal of relations of $A$ of degree up to $5$ turn out to be $y^3$, $x^3-xy^2-by^2x$, $x^2y^2+(b-1)xy^2x-by^2x^2$, $(1-b)xy^2xy+by^2x^2y$ and $(1-b+b^2)xy^2x^2-b^2y^2xy^2$. If $b\neq 1$ and $b^2-b+1\neq 0$, the leading monomials of the elements of the Gr\"obner basis of degrees up to $5$ are $y^3$, $x^3$, $x^2y^2$, $xy^2xy$ and $xy^2x^2$. There are exactly 16 degree $6$ monomials which do not contain the above leading monomials as subwords. Thus in order for $A$ to be a member of $\Lambda$ all degree $6$ overlaps of leading monomials must resolve (same argument as in the proof of Lemma~\ref{xxx-gen}). A direct computation shows that this happens if and only if either $b=0$ or $b^2+1=0$. Thus the only values of the parameter $b$ for which $A$ has a chance to belong to $\Lambda$ are solutions of the equation $b(b-1)(b^2+1)(b^2-b+1)=0$. It remains to deal with finitely many specific algebras (6 to be precise). A direct Gr\"obner basis computation in each case shows that for all these algebras $\dim A_7\geq 21$, which is incompatible with the membership in $\Lambda$. Thus no algebras from $\Lambda$ feature in this case.

{\bf Case 1b:} $a\neq 0$ in addition to $s=1$.

The elements of the Gr\"obner basis of the ideal of relations of $A$ of degree up to $5$ turn out to be
$$
\begin{array}{l}\text{$y^3$, $x^3-xy^2-ayxy-by^2x$, $x^2y^2+axyxy+(b-1)xy^2x-ayxyx-by^2x^2$,}\\ \text{$xyxy^2+\frac{b-1}{a}xy^2xy-yxyxy-\frac{b}{a}y^2x^2y$ and $x^2yxy-bxyxyx-\frac{1-b+b^2}{a}xy^2x^2+(b-1)yxyx^2+\frac{b^2}{a}y^2xy^2$.}
\end{array}
$$
Again, there are exactly 16 degree 6 monomials, which do not contain the leading monomials $y^3$, $x^3$, $x^2y^2$, $xyxy^2$ and $x^2yxy$ as subwords. Thus in order for $A$ to be a member of $\Lambda$ all degree $6$ overlaps of leading monomials must resolve. A direct computation shows that this happens precisely when $b=1$. Now in the case $b=1$, the five elements of the Gr\"obner basis from the above display constitute the entire Gr\"obner basis of the ideal of relations of $A$ (all higher degree overlaps resolve as well). Now the five monomials $y^3$, $x^3$, $x^2y^2$, $xyxy^2$ and $x^2yxy$ are all leading monomials of members of a Gr\"obner basis. This allows us to compute the Hilbert series of $A$, which is $H_A=(1+t)^{-1}(1-t)^{-3}$. Thus in this case $A\in\Lambda$ precisely when $b=1$. In this case the defining relations of $A$ take the shape $y^3$ and $x^3-xy^2-ayxy-y^2x$ for some $a\in \K^*$. That is, $A$ coincides with $A^a$.

{\bf Case 2:} $s=0$ and $b\neq 0$.

In this case an additional scaling turns $b$ into $1$. Then $R$ is spanned by $y^3$ and $x^3-ayxy-y^2x$ for some $a\in\K$. Then $A^{\rm opp}$ being $A$ with the opposite multiplication is isomorphic to the algebra given by generators $x$ and $y$ and relations $y^3$ and $x^3-ayxy-xy^2$. Thus $A^{\rm opp}$ is under the jurisdiction of Case~1a and therefore $A^{\rm opp}\notin\Lambda$. Hence $A\notin\Lambda$. It remains to consider the final case.

{\bf Case 3:} $s=b=0$.

If $a=0$, $R$ is spanned by $x^3$ and $y^3$. Then $\dim A_4=14$ and $A$ fails to be in $\Lambda'$, let alone $\Lambda$. If $a\neq 0$, after a scaling , which turns $a$ into $1$, $R$ becomes spanned by $y^3$ and $x^3-yxy$. For this single algebra the coefficients of the Hilbert series coincide with that of $(1+t)^{-1}(1-t)^{-3}$ up to $t^8$ inclusive. Still a direct Gr\"obner computation (easy in this case because of the simplicity of relations) yields $\dim A_9=31$, which is greater by $1$ than the $t^9$-coefficient of $(1+t)^{-1}(1-t)^{-3}$. Thus we have no algebras from $\Lambda$ in this case.
\end{proof}

\begin{lemma}\label{xxy-1} Let $A$ be the cubic algebra given by generators $x$ and $y$ and relations $y^3$ and $f$, where $f$ is a linear combination of $xyx$, $xy^2$, $yxy$ and $y^2x$. Then $A\notin\Lambda$.
\end{lemma}

\begin{proof} {\bf Case 1:} $xyx$ features in $f$ with non-zero coefficient.

By a substitution of the form $x\to ux+ty$, $y\to y$ with $u\in\K^*$, $t\in\K$, we can turn the $xyx$ coefficient in $f$ into $1$ and kill the $xy^2$ coefficient. Thus without loss of generality $f=xyx-ayxy-by^2x$ with some $a,b\in\K$. If $a\neq 0$, an additional scaling turns $a$ into $1$. Then $f=xyx-yxy-by^2x$ with $b\in\K$. Computing the Gr\"obner basis in the ideal of relations for degrees up to $6$, we see that apart from the defining relations only $xy^2xy-yxy^2x$ and $yxy^2x^2$ turn up. Still this yields $\dim A_6=15$, which is incompatible with $A$ being in $\Lambda$. If $a=0$, after scaling we are left with two options $f=xyx-y^2x$ and $f=xyx$. In both cases a Gr\"obner basis computation yields $\dim A_5=13$, again incompatible with membership in $\Lambda$.

{\bf Case 2:} $xyx$ does not feature $f$.

Now, $f=axy^2+byxy+cy^2x$ with some $a,b,c\in\K$. An easy Gr\"obner basis calculation shows that $\dim A_5\geq 14$ no matter what the parameters are. Hence $A\notin\Lambda$.
\end{proof}

\begin{lemma}\label{xxy-2} Let $A$ be the cubic algebra given by generators $x$ and $y$ and relations $y^3$ and $x^2y-axyx-byx^2-yxy-cy^2x$ with $a,b,c\in\K$. Then $A\in\Lambda$ if and only if either $(a,b,c)=(1,-1,0)$ or $(a,b,c)=(-1,-1,0)$ or $a=0$ and $b+b^2+{\dots}+b^k+c\neq 0$ for all $k\in\N$.
\end{lemma}

\begin{proof} {\bf Case 1:} $a\neq 0$.

A direct computation shows that the only members of the reduced Gr\"obner basis in the ideal of relations of $A$ of degree at most $5$ are $y^3$, $x^2y-axyx-byx^2-yxy-cy^2x$ and $xyxy^2+byxyxy+b^2y^2xyx+\frac{b+c}{a}y^2xy^2$. Since there are exactly $16$ degree $6$ monomials that do not contain any of $x^2y$, $y^3$ and $xyxy^2$ as submonomials, $A$ can only fall into $\Lambda$ if all overlaps of degree $6$ of these three leading monomials resolve. However the overlap $(x^2y)xy^2=x(xyxy^2)$ resolves precisely when $a^2=1$, $b=-1$ and $c=0$. Now if $(a,b,c)=(\pm1,-1,0)$,  $y^3$, $x^2y\mp xyx+yx^2-yxy$ and $xyxy^2-yxyxy+y^2xyx\pm y^2xy^2$ form the entire reduced Gr\"obner basis in the ideal of relations of $A$. This allows to confirm that $H_A=(1+t)^{-1}(1-t)^{-3}$ and therefore $A\in\Lambda$ in these two cases.

{\bf Case 2:} $a=0$.

This time the only members of the reduced Gr\"obner basis in the ideal of relations of $A$ of degree at most $5$ are $y^3$ and $x^2y-byx^2-yxy-cy^2x$ if $b+c=0$ and
$y^3$, $x^2y-byx^2-yxy-cy^2x$ and $y^2xy^2$ if $b+c\neq 0$. If $b+c=0$, we have $\dim A_5=13$  and therefore $A\notin\Lambda$. Assume now that $b+c\neq 0$. There are no elements of degree $6$ of the Gr\"obner basis. The only non-zero degree $7$ member can arise from the overlap $(x^2y)yxy^2=x^2(y^2xy^2)$, which produces $(c+b+b^2)y^2xyxy^2$. If $c+b+b^2=0$, then $\dim A_7=21$ exceeds by $1$ the $t^7$ coefficient of $(1+t)^{-1}(1-t)^{-3}$. Thus to keep $A$ in $\Lambda$, we must have $c+b+b^2\neq 0$. In this case the members of the Gr\"obner basis of degrees up to $7$ are $y^3$, $x^2y-byx^2-yxy-cy^2x$, $y^2xy^2$ and $y^2xyxy^2$. This pattern goes on. If $b+b^2+{\dots}+b^k+c\neq 0$ for $1\leq k\leq m$, then the members of the Gr\"obner basis of degrees up to $2m+4$ are $y^3$, $x^2y-byx^2-yxy-cy^2x$ and $y^2(xy)^jxy^2$ for $0\leq j\leq m-1$ and $H_A$ and $(1+t)^{-1}(1-t)^{-3}$ match up to $t^{2m+4}$. The only non-zero degree $2m+5$ member of the Gr\"obner basis can arise only from the overlap $x^2y^2(xy)^{m-1}xy^2$, which produces $(c+b+b^2+{\dots}+b^{m+1})y^2(xy)^mxy^2$. If $c+b+b^2+{\dots}+b^{m+1}=0$, $\dim A_{2m+5}$ exceeds by $1$ the $t^{2m+5}$ coefficient of $(1+t)^{-1}(1-t)^{-3}$ and $A\notin\Lambda$. Otherwise it matches and show goes on. As a result, $A$ belongs to $\Lambda$ if and only if $b+b^2+{\dots}+b^k+c\neq 0$ for all $k\in\N$. Note that in this case the full reduced Gr\"obner basis in the ideal of relations of $A$ consists of the defining relations $y^3$ and $x^2y-byx^2-yxy-cy^2x$ together with infinitely many monomials $y^2(xy)^jxy^2$ for $j\in\Z_+$.
\end{proof}

\begin{lemma}\label{xxy-3} Let $A$ be the cubic algebra given by generators $x$ and $y$ and relations $y^3$ and $f=x^2y-axyx-byx^2-y^2x$ with $a,b\in\K$. Then $A\in\Lambda$ if and only if $a=0$.
\end{lemma}

\begin{proof} {\bf Case 1:} $a\neq 0$.

A direct computation shows that the only members of the reduced Gr\"obner basis in the ideal of relations of $A$ of degree at most $5$ are $y^3$, $x^2y-axyx-byx^2-y^2x$ and $xyxy^2+byxyxy+b^2y^2xyx+\frac{1}{a}y^2xy^2$. Since there are exactly $16$ degree $6$ monomials that do not contain any of $x^2y$, $y^3$ and $xyxy^2$ as submonomials, $A$ can only fall into $\Lambda$ if all overlaps of degree $6$ of these three leading monomials resolve. However, the overlap $(x^2y)xy^2=x(xyxy^2)$ is easily seen to never resolve. Hence $A\in\Lambda$.

{\bf Case 2:} $a=0$.

This time the only members of the reduced Gr\"obner basis in the ideal of relations of $A$ of degree at most $5$ are $y^3$, $x^2y-byx^2-y^2x$ and $y^2xy^2$. The pattern of the further Gr\"bner basis construction here is the same as in Case~2 of Lemma~\ref{xxy-2}, only the monomials $y^2(xy)^jxy^2$ pop up with coefficients $1$ only. The full reduced Gr\"obner basis in the ideal of relations of $A$ consists of the defining relations $y^3$ and $x^2y-byx^2-yxy-cy^2x$ together with infinitely many monomials $y^2(xy)^jxy^2$ for $j\in\Z_+$. This yields $H_A=(1+t)^{-1}(1-t)^{-3}$ and therefore $A\in\Lambda$.
\end{proof}

\begin{lemma}\label{xxy-4} Let $A$ be the cubic algebra given by generators $x$ and $y$ and relations $y^3$ and $f=x^2y-axyx-byx^2$ with $a,b\in\K$. Then $A\in\Lambda$ if and only if $a\neq 0$ and $b=-a^2$.
\end{lemma}

\begin{proof} {\bf Case 1:} $a=0$.

The only members of the reduced Gr\"obner basis in the ideal of relations of $A$ of degree at most $5$ are the defining relations $y^3$ and $x^2y-byx^2$. This yields $\dim A_5=13$ and therefore $A\notin\Lambda$.

{\bf Case 2:} $a\neq 0$.

A direct computation shows that the only members of the reduced Gr\"obner basis in the ideal of relations of $A$ of degree at most $5$ are $y^3$, $x^2y-axyx-byx^2$ and $xyxy^2+byxyxy+b^2y^2xyx$. Since there are exactly $16$ degree $6$ monomials that do not contain any of $x^2y$, $y^3$ and $xyxy^2$ as submonomials, $A$ can only fall into $\Lambda$ if all overlaps of degree $6$ of these three leading monomials resolve. This is easily seen to happen precisely when $a^2+b=0$. If $a\neq 0$ and $a^2+b=0$, then $y^3$, $x^2y-axyx+a^2yx^2$ and $xyxy^2-a^2yxyxy+a^4y^2xyx$ form the full Gr\"obner basis of the ideal of relations of $A$, yielding $H_A=(1+t)^{-1}(1-t)^{-3}$.

Thus $A\in\Lambda$ precisely when $a\neq 0$ and $b=-a^2$.
\end{proof}

\begin{lemma}\label{xxy-5} Let $A$ be the cubic algebra given by generators $x$ and $y$ and relations $y^3$ and $f=x^2y-xyx-ayx^2-xy^2-byxy-cy^2x$ with $a,b,c\in\K$. Then $A\in\Lambda$ if and only if $a=c=-1$.
\end{lemma}

\begin{proof} {\bf Case 1:} $a\neq 0$ and $a\neq -1$.

In this case $A^{\rm opp}$ is isomorphic to the algebra given by generators $x$ and $y$ and relations $y^3$ and $x^2y+\frac1axyx-\frac1ayx^2+\frac{c}axy^2+\frac{b}ayxy+\frac1ay^2x$. Applying the substitution $y\to y$, $x\to x-\frac{c}{1+a}y$, we see that $A^{\rm opp}$ is isomorphic to the algebra given by generators $x$ and $y$ and relations $y^3$ and $x^2y+ sxyx-syx^2+tyxy+sy^2x$, where $s=\frac1a$ and $t=\frac{c(1-a)+b(1+a)}{a(1+a)}$. If $t\neq 0$, a scaling turns the defining relations into $y^3$ and $x^2y+ sxyx-syx^2-yxy-\frac{s}{t}sy^2x$, which falls under the jurisdiction of Lemma~\ref{xxy-2}. If $t=0$, a scaling turns the defining relations into $y^3$ and $x^2y+sxyx-syx^2-y^2x$, which falls under the jurisdiction of Lemma~\ref{xxy-3}. Applying these two lemmas, we see that $A\notin\Lambda$.

{\bf Case 2:} $a=0$.

A direct computation shows that the only members of the reduced Gr\"obner basis in the ideal of relations of $A$ of degree at most $5$ are $y^3$, $x^2y-xyx-xy^2-byxy-cy^2x$ and $xyxy^2+cy^2xy^2$. Since there are exactly $16$ degree $6$ monomials that do not contain any of $x^2y$, $y^3$ and $xyxy^2$ as submonomials, $A$ can only fall into $\Lambda$ if all overlaps of degree $6$ of these three leading monomials resolve. This does not happen with the overlap $x^2yxy^2$ though. Hence $A\notin\Lambda$.

{\bf Case 3:} $a=-1$.

A direct computation shows that the only members of the reduced Gr\"obner basis in the ideal of relations of $A$ of degree at most $5$ are $y^3$, $x^2y-xyx+yx^2-xy^2-byxy-cy^2x$ and $xyxy^2-yxyxy+y^2xyx+(c+1-b)y^2xy^2$. Since there are exactly $16$ degree $6$ monomials that do not contain any of $x^2y$, $y^3$ and $xyxy^2$ as submonomials, $A$ can only fall into $\Lambda$ if all overlaps of degree $6$ of these three leading monomials resolve. This is easily seen to happen precisely when $c+1=0$. If $a=c=-1$, then $y^3$, $x^2y-xyx+yx^2-xy^2-byxy+y^2x$ and $xyxy^2-yxyxy+y^2xyx-by^2xy^2$ form the full Gr\"obner basis of the ideal of relations of $A$, yielding $H_A=(1+t)^{-1}(1-t)^{-3}$.

Thus $A\in\Lambda$ precisely when $a=c=-1$.
\end{proof}

Now we are ready to prove Part XI of Theorem~\ref{main}. The fact that all algebras $A$ in (Z1--Z10) are in $\Lambda$ follows from Lemmas~\ref{xxx-1} and \ref{xxy-2}--\ref{xxy-5}. Indeed, either $A$ itself or $A^{\rm opp}$ features in these lemmas as an algebra from $\Lambda$. Next, all algebras in (Z1--Z10) are pairwise non-isomorphic. Indeed, since $y^3$ is the only cube (up to a scalar multiple) among the cubic relations of any algebra in (Z1--Z10), a linear substitution providing an isomorphism between two algebras in (Z1--Z10) must be of the form $x\to ux+ty$, $y\to vy$ with $u,v\in\K^*$ and $t\in \K$. Since the subs $x\to sx$, $y\to sy$ with $s\in\K^*$ do not change any of the spaces of cubic relations, we can restrict ourselves to the case $u=1$. Now it is a matter of routine verification to see that for two algebras from (Z1--Z10), a substitution $x\to x+ty$, $y\to vy$ transforms the space of cubic relations of one algebra to the space of cubic relations of the other only if we had the two copies of the same item from (Z1--Z10) to begin with. It remains to verify that if $A\in\Lambda$ and the quasipotential $Q=Q_A$ is the fourth power of a degree $1$ element, then $A$ is isomorphic to an algebra form (Z1--Z10). Note that the algebras in (Z1--Z10)  put together form a class of algebras closed under passing to the opposite multiplication. Thus it is enough to show that either $A$ or $A^{\rm opp}$ is isomorphic to an algebra from (Z1--Z10).

Without loss of generality, we can assume that $y^4$ is the quasipotential of $A$, where $y\in V$. If the space $R$ of cubic relations of $A$ is not contained in the ideal generated by $y$, Lemma~\ref{xxx-1} guarantees that $A$ is isomorphic to an algebra from (Z1). Thus we can assume that $R$ is contained in the ideal generated by $y$. Then $A$ is given by generators $x,y$ and relations $y^3$ and $f$, which is a non-zero linear combination of $x^2y$, $xyx$, $yx^2$, $xy^2$, $yxy$ and $y^2x$. If both $x^2y$ and $yx^2$ do not feature in $f$, Lemma~\ref{xxy-1} implies that $A\notin\Lambda$ yielding a contradiction. Hence either $x^2y$ or $yx^2$ feature in $f$ with a non-zero coefficient. These two options transform to one another when we pass to the opposite multiplication. Since it is enough to verify that either $A$ or $A^{\rm opp}$ is isomorphic to an algebra from (Z1--Z10), we can without loss of generality assume that $x^2y$ features in $f$ with non-zero coefficient.

{\bf Case 1:} $xy^2$ does not feature in $f$ but $yxy$ does.

In this case a scaling turns $x^2y$ and $yxy$ coefficients into $1$ and $-1$ respectively. That is, $f=x^2y-axyx-byx^2-yxy-cy^2x$ with $a,b,c\in\K$. By Lemma~\ref{xxy-2}, $A$ is isomorphic to an algebra from (Z2--Z4) or (Z6).

{\bf Case 2:} $xy^2$ and $yxy$ do not feature in $f$ but $y^2x$ does.

In this case a scaling turns $x^2y$ and $y^2x$ coefficients into $1$ and $-1$ respectively. That is, $f=x^2y-axyx-byx^2-y^2x$ with $a,b\in\K$. By Lemma~\ref{xxy-3}, $A$ is isomorphic to an algebra from (Z7).

{\bf Case 3:} None of $xy^2$, $yxy$ or $y^2x$ feature in $f$.

In this case a scaling turns $x^2y$-coefficient into $1$. That is, $f=x^2y-axyx-byx^2$ with $a,b\in\K$. By Lemma~\ref{xxy-4}, $A$ is isomorphic to an algebra from (Z9).

Note that Cases~1--3 cover all options when $xy^2$ does not make an appearance in $f$.

{\bf Case 4:} $xy^2$ features in $f$ with non-zero coefficient.

If the sum of $x^2y$ and $xyx$ coefficients in $f$ is non-zero, a substitution $x\to x+sy$, $y\to y$ with an appropriate $s\in\K$ kills $xy^2$ in $f$, after which Cases~1--3 kick in. Thus we can assume that $x^2y$ and $xyx$ coefficients in $f$ add up to zero. A scaling turns $x^2y$ and $xy^2$ coefficients into $1$ and $-1$ respectively. Then $f=x^2y-xyx-ayx^2-xy^2-byxy-cy^2x$ for some $a,b,c\in\K$. By Lemma~\ref{xxy-5}, $A$ is isomorphic to an algebra from (Z10). This completes the final case and the proof of Part~XI of Theorem~\ref{main}.

\section{Proof of Part XII of Theorem~\ref{main}}

\begin{lemma}\label{isomoXII} The algebras from {\rm (Y1--Y8)} of Theorem~$\ref{main}$ with different labels are non-isomorphic. There are no isomorphism between two algebras from {\rm (Y1--Y8)} with the same label apart from the label {\rm (Y1)}, where the isomorphic ones are precisely those specified in {\rm (Y1)}.
\end{lemma}

\begin{proof} Note that every algebra from (Y1--Y8) is quasipotential with the corresponding quasipotential being of the form $Q=ufv$, where $u,v\in V$ are linearly independent and the non-zero $f\in V^2$ is not a square of a degree one element. If we have two isomorphic algebras $A$ and $A'$ from the list (Y1--Y8), then a linear substitution facilitating the isomorphism must transform the quasipotential $Q=ufv$ of $A$ to the quasipotential $Q'=u'f'v'$ of $A'$ up to a non-zero scalar multiple. It follows that $u$, $v$ and $f$ are transformed to scalar multiples of $u'$, $v'$ and $f'$ respectively. Now using Lemma~\ref{1-dim} (including the isomorphism part), we easily see that algebras in (Y1--Y8) with different labels can not be isomorphic and that algebras in (Y2--Y8) are pairwise non-isomorphic. As for two algebras from (Y1), Lemma~\ref{1-dim} yields that the only substitutions which can transform an algebra from (Y1) to an algebra from from (Y1) are $x\to sx$, $y\to sy$ or $x\to sy$, $y\to sx$ with $s\in\K^*$. The first type of substitutions do not change the parameters $a,b$, while the second type (the swap) transforms an algebras from (Y1) with parameters $a,b$ to the algebra from (Y1) with parameters $(a^{-1},b^{-1})$.
\end{proof}

\begin{lemma}\label{e12}
Let $A\in\Lambda'$ be such that the corresponding quasipotential $Q$ is not a fourth power of degree $1$ element and satisfies $n_1(Q)=n_2(Q)=1$ and $E_1(Q)=E_2(Q)$. Then $A\notin\Lambda$.
\end{lemma}

\begin{proof} Let $y\in V$ be such that $y$ spans the one-dimensional space $E_1(Q)$. Take $x\in V$ such that $x$ and $y$ are linearly independent. Since $y$ spans $E_1(Q)=E_2(Q)$, we have $Q=yfy$, where $f\in V^2$ (a quadratic element). Clearly $f$ is a linear combination of $x^2$, $xy$, $yx$ and $y^2$. Since $Q$ is not a fourth power, $f$ can not be a scalar multiple of $y^2$.

{\bf Case 1:} $x^2$ features in $f$ with non-zero coefficient.

A substitution of the form $x\to ux+sy$, $y\to y$ with appropriate $u\in\K^*$, $s\in\K$ kills $xy$ in $f$. Then $f$ becomes a linear combination of $x^2$, $yx$ and $y^2$. If $yx$ still features in $f$ after the substitution, we turn the $x^2$ and $yx$ coefficients in $f$ into $1$ and $-1$ respectively by an appropriate scaling. Then $f=x^2-yx-ay^2$ for some $a\in\K$. Hence $Q=y(x^2-yx-ay^2)y$ and therefore $R_Q$ is spanned by $x^2y-yxy-ay^3$ and $yx^2-y^2x-ay^3$. Since these two are linearly independent, $A$ is presented by generators $x$ and $y$ and relations $x^2y-yxy-ay^3$ and $yx^2-y^2x-ay^3$. If $yx$ does not features in $f$ after the above substitution, we turn the $x^2$ coefficient of $f$ into $1$ by a scaling and observe that $f=x^2-ay^2$ for some $a\in\K$. Hence $Q=y(x^2-ay^2)y$ and therefore $R_Q$ is spanned by $x^2y-ay^3$ and $yx^2-ay^3$. Since these two are linearly independent, $A$ is presented by generators $x$ and $y$ and relations $x^2y-ay^3$ and $yx^2-ay^3$. In both cases, an easy Gr\"obner basis computation (apart from the defining relations there is only one other degree $\leq 5$ member) yields $\dim A_5\geq 13$, which is incompatible with $A$ being in $\Lambda$.

{\bf Case 2:} $x^2$ does not feature in $f$, but the sum of the $xy$ and $yx$ coefficients is non-zero.

Clearly either $xy$ or $yx$ (or both) feature in $f$. By passing to the opposite multiplication, if necessary, we can without loss of generality assume that $xy$ features in $f$. A substitution of the form $x\to ux+sy$, $y\to vy$ with appropriate $u,v\in\K^*$, $s\in\K$ kills $y^2$ in $f$ and turns the $xy$ coefficient into $1$. Then $f=xy-ayx$ for some $a\in\K$. Hence $Q=y(xy-ayx)y$ and therefore $R_Q$ is spanned by $xy^2-ayxy$ and $yxy-ay^2x$. Since these two are linearly independent, $A$ is presented by generators $x$ and $y$ and relations $xy^2-ayxy$ and $yxy-ay^2x$. Same argument as in Case~1 yields $\dim A_5\geq 13$ and therefore $A\notin\Lambda$.

{\bf Case 3:} $x^2$ does not feature in $f$ and the $xy$ and $yx$ coefficients in $f$ sum up to zero.

Since $f$ can not be a scalar multiple of $y^2$, both $xy$ and $yx$ feature in $f$. An appropriate scaling turns $f$ into either $xy-yx$ or $xy-yx-y^2$. This gives $Q=y(xy-yx)y$ or $Q=y(xy-yx-y^2)y$. Same argument as in the previous cases shows that $A$ is presented by the generators $x$ and $y$ and either the relations $xy^2-yxy$ and $yxy-y^2x$ or the relations
$xy^2-yxy-y^3$ and $yxy-y^2x-y^3$. Again, a Gr\"obner basis calculation gives $\dim A_5=13$ in both cases and therefore $A\notin\Lambda$.
\end{proof}

\begin{lemma}\label{SS} Let $A\in\Lambda$ be such that the corresponding quasipotential $Q$ is not a fourth power of degree $1$ element and satisfies $n_1(Q)=n_2(Q)=1$. Then there is a basis $x,y$ in $V$ with respect to which $Q$ has one of the following forms$:$ $(x-by)(xy-ayx)(x-y)$ with $b\neq 1$, $a\neq 0$ and $a\neq 1$, $(x-y)(xy-yx-ayy)x$ with $a\neq 0$, $x(xy-yx)y$, $(x-ay)xy(x-y)$ with $a\neq 1$, $(x-y)(xy-ayx)y$ with $a\neq 1$, $x(xy-ayx)y$ with $a\neq 1$, $x(xy-yx-yy)y$ or $y(xy-yx-yy)x$.
\end{lemma}

\begin{proof} In this proof we only use the left-to-right degree-lexicographical ordering on $x,y$ monomials assuming $x>y$. By Lemma~\ref{e12}, $Q=ufv$, where $u,v$ are linearly independent elements of $V$ and $f\in V^2$ is non-zero. Clearly, $R_Q$ is spanned by $uf$ and $fv$. Using the fact that $u$ and $v$ are linearly independent, it is easy to see that $uf$ and $fv$ are linearly independent. Hence $R_Q$ is the two-dimensional space spanned by $uf$ and $fv$ and therefore the ideal of relations of $A$ is generated by $uf$ and $fv$. First, observe that $f$ can not be a square of an element of $V$. Indeed, if $f=y^2$ for $y$ being a non-zero element of $V$, then the ideal of relations $I$ of $A$ is generated by $uy^2$ and $y^2v$. Clearly $I$ is contained in the ideal $J$ generated by $y^2$ and therefore $B=F(V)/J$ is a quotient of $A$. On the other hand, $B$ has exponential growth and therefore $A$ has exponential growth. Since the latter is incompatible with the membership in $\Lambda$, we arrive to a contradiction. Thus $f$ is not a square. By Lemma~\ref{1-dim}, there is a basis $x,y$ in $V$ such that $f=xy$ or $f=xy-ayx$ with $a\in\K^*$ or $f=xy-yx-yy$. We go through these options case by case.

{\bf Case 1:} $f=xy$.

If $u$ is not a scalar multiple of $y$ and $v$ is not a scalar multiple of either $x$ or $y$, a scaling (of $x$ and $y$) turns $Q=uxyv$ into $(x-ay)xy(x-y)$ with $a\in\K$. We also have $a\neq 1$ since $u$ and $v$ are linearly independent. If $u$ is not a scalar multiple of $y$ and $v$ is a scalar multiple of $y$, a scaling turns $Q$ into either $(x-y)xyy$ or $xxyy$ (depending on whether $u$ is a scalar multiple of $x$ or not), which are $(x-y)(xy-ayx)y$ or $x(xy-ayx)y$ with $a=0$. If $u$ is not a scalar multiple of $y$ and $v$ is a scalar multiple of $x$, a scaling turns $Q$ into $(x-ay)xyx$ with $a\in\K$. Since $u$ and $v$ are linearly independent, $a\neq 0$. Then the defining relations of $A$ are $x^2y-ayxy$ and $xyx$. A direct computation shows that there are no other members of degrees up to $5$ in the reduced Gr\"obner basis of the ideal of relations of $A$ apart from $x^2y-ayxy$, $xyx$ and $xy^2xy$. This yields $\dim A_5=13$, which is incompatible with $A\in\Lambda$. It remains to consider the case when $u$ is a scalar multiple of $y$. Now a scaling turns $Q$ into $yxy(x-ay)$ with $a\in\K$. Then the defining relations of $A$ are $xyx-axy^2$ and $yxy$. A direct computation shows that there are no other members of degrees up to $5$ in the reduced Gr\"obner basis of the ideal of relations of $A$ apart from $x^2y-ayxy$, $xyx$ and $xy^3$. This yields $\dim A_5=13$, which is incompatible with $A\in\Lambda$. Thus the last two cases do not occur.

{\bf Case 2:} $f=xy-yx$.

In this case it is easy to see that $Q$ is a scalar multiple of $u(uv-vu)v$. In the basis $x=u$, $y=v$, this yields $Q=x(xy-yx)y$.

{\bf Case 3:} $f=xy-ayx$ with $a\in\K^*$, $a\neq 1$.

If $u$ is not a scalar multiple of $y$ and $v$ is not a scalar multiple of either $x$ or $y$, a scaling (of $x$ and $y$) turns $Q=uxyv$ into $(x-by)(xy-ayx)(x-y)$ with $b\in\K$. Since $u$ and $v$ are linearly independent, we have $b\neq 1$. If $u$ is not a scalar multiple of either $x$ or $y$ and $v$ is a scalar multiple of $y$, a scaling turns $Q$ into $(x-y)(xy-ayx)y$. If $u$ is a scalar multiple of $x$ and $v$ is a scalar multiple of $y$, a scaling turns $Q$ into $x(xy-ayx)y$. If $u$ is not a scalar multiple of either $x$ or $y$ and $v$ is a scalar multiple of $x$, a scaling turns $Q$ into $(x-y)(xy-ayx)x$. Swapping $x$ and $y$ and a further scaling transforms $Q$ into $(x-y)(xy-a^{-1}yx)y$. All cases of $u$ not being a scalar multiple of $y$ are taken care of (some do not feature since $u$ and $v$ are linearly independent). If $u$ is a scalar multiple of $y$ and $v$ is not a scalar multiple of either $x$ or $y$, a scaling (of $x$ and $y$) turns $Q$ into $y(xy-ayx)(x-y)$. Swapping $x$ and $y$ and a further scaling transforms $Q$ into $x(xy-a^{-1}yx)(x-y)=(x-by)(xy-a^{-1}yx)(x-y)$ with $b=0$. If $u$ is a scalar multiple of $y$ and $v$ is a scalar multiple of $x$, a scaling turns $Q$ into $y(xy-ayx)x$. Swapping $x$ and $y$ and a further scaling transforms $Q$ into $x(xy-a^{-1}yx)y$.

{\bf Case 4:} $f=xy-yx-y^2$.

If neither $u$ nor $v$ is a scalar multiple of $y$, a substitution of the form $x\to \alpha x+\beta y$, $y\to \gamma y$ with $\alpha,\gamma\in\K^*$ and $\beta\in\K$ transforms $Q$ into
$(x-y)(xy-yx-ayy)x$ with $a\neq 0$. If $v$ is a scalar multiple of $y$, then $u$ is not a scalar multiple of $y$ and a substitution of the same form turns $Q$ into $x(xy-yx-yy)y$.
If $u$ is a scalar multiple of $y$, then $v$ is not a scalar multiple of $y$ and a substitution of the same form turns $Q$ into $y(xy-yx-yy)x$.
\end{proof}

\begin{lemma}\label{SS1} Let $A\in\Lambda'$ be the quasipotential algebra on generators $x,y$ with the quasipotential $Q$ being one of the following$:$ $(x-ay)xy(x-y)$ with $a\neq 1$, $(x-y)(xy-ayx)y$ with $a\neq 1$, $x(xy-ayx)y$ with $a\in\K$, $x(xy-yx-yy)y$ or $y(xy-yx-yy)x$. Then $A\in\Lambda$.
\end{lemma}

\begin{proof} We use the left-to-right degree-lexicographical ordering on $x,y$ monomials assuming $x>y$. If $Q=(x-y)(xy-ayx)y$, the defining relations of $A$ are $x^2y-axyx-yxy+ay^2x$ and $xy^2-ayxy$. If $Q=x(xy-ayx)y$, the defining relations of $A$ are $x^2y-axyx$ and $xy^2-ayxy$. If $Q=x(xy-yx-yy)y$, the defining relations of $A$ are $x^2y-xyx-yxy-y^3$ and $xy^2-yxy-y^3$. In all these cases, the defining relations themselves form the reduced Gr\"obner basis of the ideal of relations of $A$ with the leading monomials of the members being $x^2y$ and $xy^2$. This yields $H_A=(1+t)^{-1}(1-t)^{-3}$ and therefore $A\in\Lambda$. The algebra with $Q=y(xy-yx-yy)x$ is isomorphic to $B^{\rm opp}$, where $B$ is the algebra with quasipotential $x(xy-yx-yy)y$. Since $B$ and $B^{\rm opp}$ have the same Hilbert series, $A\in\Lambda$ as well. This leaves only the case $Q=(x-ay)xy(x-y)$ with $a\neq 0$. Then the defining relations of $A$ are $x^2y-ayxy$ and $xyx-xy^2$. An easy computation shows that the reduced Gr\"obner basis of the ideal of relations of $A$ consists of $x^2y-ayxy$, $xy^{2n-1}x-xy^{2n}$ for $n\in\N$ and $xy^{2n}xy$ for $n\in\N$. The leading monomials of the members of the basis are $x^2y$ and $xy^{2n-1}x$, $xy^{2n}xy$ for $n\in\N$. Knowing these, we get $H_A=(1+t)^{-1}(1-t)^{-3}$ and therefore $A\in\Lambda$.
\end{proof}

The next two lemmas deal with quasipotentials $(x-by)(xy-ayx)(x-y)$ and $(x-y)(xy-yx-ayy)x$. We have to confess that we could not find any nice pattern in the leading monomials of the members of the Gr\"obner basis arising from the natural presentations of the corresponding algebras. However, we found different presentations of the same algebras (with three generators rather than two) for which we can compute the basis and hence the Hilbert series.

\begin{lemma}\label{SS2} Let $A\in\Lambda'$ be the quasipotential algebra on generators $x,y$ with the quasipotential $Q=(x-by)(xy-ayx)(x-y)$ with $b\neq 1$, $a\neq 0$ and $a\neq 1$.  Then $A\in\Lambda$ if and only if $a^nb\neq 1$ for all $n\in\N$.
\end{lemma}

\begin{proof} Clearly, $A$ can be presented by generators $x,y,f$ and relations $xy-ayx-f$, $xf-byf$, $fx-fy$. To make $A$ graded in the same way as it was originally, we assign degree $1$ to $x$ and $y$ and degree $2$ to $f$. Now we order the monomials in $x,y,f$ in the following way. A monomial of greater degree is greater. For two monomials of the same total degree, the one of greater $x$-degree (with more $x$ in it) is greater. For two monomials of the same degree and the same $x$-degree, the one with greater $y$-degree is greater. Finally, for two monomials of the same degree, same $x$-degree and same $y$-degree (they have the same number of $x$, the same number of $y$ and therefore the same number of $f$ as well), we break ties using the left-to-right lexicographical order assuming $f>x>y$.

As an illustration, we list low degree terms of the reduced Gr\"obner basis in the ideal of relations of $A$ as presented above. The degrees up to $3$ terms are the defining relations $xy-ayx-f$, $xf-byf$ and $fx-fy$. The only degree $4$ term arises from the overlap $fxy$ and is $fyx-\frac1a fy^2+\frac1a f^2$. Two degree $5$ terms arise from the overlaps $fxf$ and $fyxy$ and are $fyf$ and $fy^2x-\frac1{a^2} fy^3+\frac1{a^2} f^2y$ respectively. The degree $6$ overlap $fyxf$ produces $(1-ab)fy^2f-f^3$. Now, who is the leading monomial of this one depends on whether $ab$ equals $1$ or not. If $ab\neq 1$, we have the degree $6$ element $fy^2f-\frac1{1-ab}f^3$. Another degree 6 overlap $fy^2xy$ now yields $fy^3x-\frac1{a^3} fy^4+\frac1{a^3} f^2y^2+\frac{1}{a(1-ab)}f^3$. All other degree $6$ overlaps resolve. Degree 7 overlap $fy^3xy$ produces $fy^4x-\frac1{a^4} fy^5+\frac1{a^4} f^2y^3+\frac{1}{a^2(1-ab)}f^3y$. Other degree $7$ overlaps resolve except for $fy^2xf$, which produces $(1-a^2b)fy^3f$. If $a^2b=1$ the latter disappears from the list. Otherwise we have the monomial $fy^3f$.

{\bf Case 1:} $a^nb\neq 1$ for all $n\in\N$.

An easy inductive argument shows that the reduced Gr\"obner basis in the ideal of relations of $A$ consists of $xy-ayx-f$, $xf-byf$, $fx-fy$, $fyx-\frac1a fy^2+\frac1a f^2$ and two elements of each degree $m\geq 5$, which have the form $fy^{m-3}x-h$ with $h$ being a linear combination of $fy^{m-4},f^2y^{m-2},\dots$ and $fy^{m-4}f$ if $m$ is odd or $fy^{m-4}f-s_mf^{m/2}$ with $s_m\in\K$ if $m$ is even. These two terms come from the overlaps $fy^{m-4}xy$ and $fy^{m-5}xf$ respectively, while all other overlaps of degree $m$ resolve. As a result, the complete list of leading monomials of the members of the basis is $xy$, $xf$, $fy^jx$ for $j\in\Z_+$ and $fy^jf$ for $j\in\N$. Thus the corresponding normal words are $y^jx^k$ and $y^jf^my^k$ with $j,k,m\in\Z_+$. Counting the number of these words of given degree, we easily confirm that $H_A=(1+t)^{-1}(1-t)^{-3}$ and therefore $A\in\Lambda$.

{\bf Case 2:} $n\in\N$ is the smallest positive integer for which $a^nb=1$.

We look at the whole family of algebras corresponding to $a\in\K^*$, $a\neq 1$ and $b=a^{-n}$. Up to degree $n+4$ inclusive the Gr\"obner basis elements are the same as in Case~1.

{\bf Case 2a:} $n$ is even.

Then we have just one (as opposed to two in Case~1) degree $n+5$ element of the Gr\"obner basis: $fy^{n+1}f$ disappears. We acquire one extra normal word $fy^{n+1}f$ of degree $n+5$.  Hence $\dim A_{n+5}$ is greater by $1$ than in Case~1 and therefore $A\notin\Lambda$.

{\bf Case 2b:} $n$ is odd.

In this case the degree $n+5$ element of the Gr\"obner basis coming from the overlap $fy^{m}xf$ is a scalar multiple of $f^{k}$ where $k=\frac{n+5}{2}$. It vanishes for finitely many exceptional values of $a$, but for a (Zarisski) generic $a$ this makes $f^k$ a member of the Gr\"obner basis. At degree $n+5$, we have 'lost' one normal word $f^k$ and 'acquired' one normal word $fy^{n+1}f$ when compared to Case~1. The changes balance themselves yielding the same $\dim A_{n+5}$ as in Case~1. Performing two more steps of the Gr\"obner basis calculation, we see that for a generic $a$, in degree $n+6$ we lose two normal words $yf^k$, $f^ky$ and acquire two normal words $yfy^{n+1}f$ and $fy^{n+1}fy$ as compared with Case~1. Still the changes balance and $\dim A_{n+6}$ is the same as in Case~1. However in degree $n+7$, we lose four normal words $y^2f^k$, $yf^ky$, $f^ky^2$ and $f^{k+1}$ and acquire five $y^2fy^{n+1}f$, $yfy^{n+1}fy$, $fy^{n+1}fy^2$, $f^2y^{n+1}f$ and $fy^{n+1}f^2$. Thus for a generic $a$, $\dim A_{n+7}$ is greater by $1$ than in Case~1. By Lemma~\ref{minhs}, for an arbitrary $a$, $\dim A_{n+7}$ is greater than in Case~1 by at least $1$. Hence $A\notin\Lambda$.
\end{proof}

\begin{lemma}\label{SS3} Let $A\in\Lambda'$ be the quasipotential algebra on generators $x,y$ with the quasipotential $Q=(x-y)(xy-yx-ayy)x$ with $a\neq 0$. Then $A\in\Lambda$ if and only if $na+1\neq0$ for all $n\in\N$.
\end{lemma}

\begin{proof} Clearly, $A$ can be presented by generators $x,y,f$ and relations $xy-yx-ayy-f$, $xf-yf$, $fx$. We use the same grading and the same order on $x,y,f$ monomials as in the proof Lemma~\ref{SS2}.

As an illustration, we list low degree terms of the reduced Gr\"obner basis in the ideal of relations of $A$ as presented above. The degrees up to $3$ terms are the defining relations $xy-yx-ayy-f$, $xf-yf$ and $fx$. The only degree $4$ term arises from the overlap $fxy$ and is $fyx+afy^2+f^2$. Two degree $5$ terms arise from the overlaps $fxf$ and $fyxy$ and are $fyf$ and $fy^2x+2afy^3+f^2y$ respectively. The degree $6$ overlap $fyxf$ produces $(1+a)fy^2f+f^3$. Now, who is the leading monomial of this one depends on whether $1+a$ equals $0$ or not. If $1+a\neq 0$, we have the degree $6$ element $fy^2f+\frac1{1+a}f^3$. Another degree 6 overlap $fy^2xy$ now yields $fy^3x+3afy^4+f^2y^2-\frac{1}{1+a}f^3$. All other degree $6$ overlaps resolve. Degree 7 overlap $fy^3xy$ produces $fy^4x+4afy^5+f^2y^3-\frac{1}{1+a}f^3y$. Other degree $7$ overlaps resolve except for $fy^2xf$, which produces $(1+2a)fy^3f$. If $1+2a=0$, the latter disappears from the list. Otherwise we have the monomial $fy^3f$.

{\bf Case 1:} $na+1\neq0$ for all $n\in\N$.

An easy inductive argument shows that the reduced Gr\"obner basis in the ideal of relations of $A$ consists of $xy-yx-ayy-f$, $xf-yf$, $fx$, $fyx+afy^2+f^2$ and two elements of each degree $m\geq 5$, which have the form $fy^{m-3}x-h$ with $h$ being a linear combination of $fy^{m-4},f^2y^{m-2},\dots$ and $fy^{m-4}f$ if $m$ is odd or $fy^{m-4}f-s_mf^{m/2}$ with $s_m\in\K$ if $m$ is even. These two terms come from the overlaps $fy^{m-4}xy$ and $fy^{m-5}xf$ respectively, while all other overlaps of degree $m$ resolve. The leading monomials of memebrs of the Gr\"obner basis as well as normal words are the same as in Case~1 in the proof of Lemma~\ref{SS2}. Hence $A\in\Lambda$.

{\bf Case 2:} ${\rm char}\,\K=0$ and $a=-\frac1n$ for some $n\in\N$.

Exactly the same argument as in Case~2 in the proof of Lemma~\ref{SS2} shows that $\dim A_{n+5}$ is greater by $1$ than the $t^{n+5}$ coefficient of $(1+t)^{-1}(1-t)^{-3}$ if $n$ is even and that $\dim A_{n+7}$ is greater by $1$ than the $t^{n+7}$ coefficient of $(1+t)^{-1}(1-t)^{-3}$ if $n$ is odd. In any case $A\notin\Lambda$.

{\bf Case 3:} ${\rm char}\,\K=p$ ($p$ is a prime) and $a=-\frac1n$ for some $n\in\N$, $n<p$.

The defining relations of $A$ are $nxy-nyx+yy-nf$, $xf-yf$ and $fx$. If we treat them as relations defining a $\Q$-algebra $B$, Case 2 yields that $\dim B_k$ is greater by $1$ than the $t^k$ coefficient of $(1+t)^{-1}(1-t)^{-3}$, where $k=n+5$ if $n$ is even and $k=n+7$ if $n$ is odd. If we consider $nxy-nyx+yy-nf$, $xf-yf$ and $fx$ as defining relations of a $\Z_p$-algebra $C$, an argument similar to the one from the proof of Lemma~\ref{minhs} shows that $\dim C_j\geq \dim B_j$ for all $j\in\Z_+$. On the other hand, $H_A=H_C$. Hence $\dim A_k$ is greater by at least $1$ than the $t^k$ coefficient of $(1+t)^{-1}(1-t)^{-3}$ and therefore $A\notin\Lambda$.
\end{proof}

Now Part~XII of Theorem~\ref{main} is just an amalgamation of Lemmas~\ref{isomoXII} and \ref{SS}--\ref{SS3}.

\section{More remarks}

Cases ${\rm char}\,\K=2$ and ${\rm char}\,\K=3$ differ from the case ${\rm char}\,\K\notin\{2,3\}$ we considered (and differ from each other) in a whole lot of ways. Some parts of our results hold in these cases, however more require adjustment. It would be interesting to have similar classifications of $\Omega$ and $\Lambda$ in the cases ${\rm char}\,\K=2$ and ${\rm char}\,\K=3$.

Some entries in the tables from Theorem~\ref{main} specify a single algebra. In many cases, this is an artifact of parametrization: the algebra should actually be a member of, say, a variety of algebras from $\Omega$ featuring in another row of the same table, only excluded for one reason or another. For example, if all members of a variety except for one algebra are non-$\rm PBW_{\rm B}$, while the exceptional algebra is $\rm PBW_{\rm B}$, the latter will occupy its own row, while the parameters corresponding to it will feature as exceptional in the row describing the variety. Sometimes the reason is different. If a variety of algebras is naturally parameterized by the projective plane $\K{\rm P}^2$, then it will occupy three rows in our table: one parameterized by an affine plane, one parameterized by an affine line and a single point. Although soundness of such practice is debatable, we have decided to parameterize by numbers rather than equivalence classes of any sort. However there are few single algebras in the tables from Theorem~\ref{main}, which a genuinely isolated points of the 'scheme' of isomorphism classes of algebras from $\Omega$ or $\Lambda$. For example the only 'isolated' twisted potential algebras in $\Omega$ are (T12--T15). Note also that the algebra in (T12) is isomorphic to the opposite of (T13) and the same holds for the pair (T14) and (T15). Thus we essentially have twodifferent algebras here: (T12) and (T14). We wonder if there is something special about them apart from being isolated in $\Omega$.

\subsection{Regular algebras in $\Omega$ and $\Lambda$}

It turns out that regular (in the  Artin--Schelter sense \cite{AS}) algebras in $\Omega$ and $\Lambda$ are precisely the twisted potential ones (including potential). This is not really surprising since it is an easy consequence of the definition of regular algebras. However there is a surprising bit as well. Artin, Tate and Van~den~Bergh \cite{TVB} show that Sklyanin algebras (=algebras from (P1)) are all domains. One can use the same technique (and/or elementary arguments in most cases) to show that all twisted potential algebras featuring in Theorem~\ref{main} are domains: these are (P1--P8), (F1--F4), (T1--T18) and (G1--G10). Curiously, no other domains occur in Theorem~\ref{main}. All other algebras have zero divisors of degree 1 or in rare cases 2. Thus for algebras in $\Omega$ or $\Lambda$, being a domain is the same as being twisted potential. We believe there should be a way to prove this equivalence other than working through the tables in Theorem~\ref{main}. We also conjecture that this fact goes beyond $\Omega$ and $\Lambda$. Namely, we conjecture that if $(n,k)\in\N^2$, $n\geq 2$, $k\geq 2$ and $(n,k)\neq (2,2)$, then a quasipotential algebra $A$ with the Hilbert series $H_A=(1-nt+nt^{k}-t^{k+1})^{-1}$ is a domain if and only if it is twisted potential. By the way, we might as well keep the case $(n,k)=(2,2)$ in. It was excluded because the class of algebras in question is empty (it is non-empty in all other cases). The above comment shows that the answer is affirmative in the cases $(n,k)=(3,2)$ and $(n,k)=(2,3)$.

\subsection{Hilbert series of algebras in $\Omega'$}

In \cite{HiS}, the authors have shown that there are just finitely many series (11 to be precise), featuring as Hilbert series of quadratic algebras $A$ satisfying $\dim A_1=\dim A_2=3$. This class of algebras coincides with the class of duals of quadratic algebras $A=A(V,R)$ with $\dim V=\dim R=3$. However, along the way we have stumbled upon enough algebras to conclude that infinitely many different Hilbert series occur for $A\in\Omega'$. This fact was already applied by the authors. For example, the exceptions for the family (N1) exhibit infinitely many different Hilbert series. Furthermore, we have already applied the  classification of Theorem~\ref{main}. Namely, we constructed \cite{autom} (found among algebras in (N1) to be precise) an automaton algebra, which fails to have a finite Gr\"obner basis for any choice of generators and a compatible order on monomials.

{\bf Acknowledgements}

We are grateful to IHES and MPIM for hospitality, support, and excellent research atmosphere.
This work was partially  funded by the ERC grant 320974, EPSRC grant EP/M008460/1 and the ESC Grant N9038.


\small\rm

%
%
%
%
%
%
%

\end{document}